\documentclass[11 pt]{amsart}
\usepackage[margin=1.0in]{geometry}

\usepackage[T1]{fontenc}
\usepackage[utf8]{inputenc}
\usepackage{lmodern}
\usepackage{hyperref, amsmath, amssymb, mathtools}
\usepackage{verbatim}
\usepackage{amsthm}
\usepackage{amsfonts}
\usepackage{amsmath}
\usepackage{mathrsfs}  
\usepackage{comment}
\usepackage{chngcntr}
\usepackage{blkarray}
\usepackage{tikz-cd}
\usepackage{ctable}
\usepackage{enumitem}

\newtheorem{thm}{Theorem}
\newtheorem*{thm*}{Theorem}
\newtheorem{lemma}[thm]{Lemma}
\newtheorem{corollary}[thm]{Corollary}
\newtheorem{conjecture}[thm]{Conjecture}
\newtheorem{prop}[thm]{Proposition}
\newtheorem*{prop*}{Proposition}

\theoremstyle{definition}
\newtheorem{defn}[thm]{Definition}
\theoremstyle{remark}
\newtheorem{remark}[thm]{Remark}
\newtheorem{example}[thm]{Example}
\numberwithin{thm}{section}
\numberwithin{equation}{section}

\newcommand{\R}{\mathbb{R}}
\newcommand{\Q}{\mathbb{Q}}
\newcommand{\Z}{\mathbb{Z}}
\newcommand{\F}{\mathbb{F}}

\newcommand{\C}{\mathbb{C}}
\newcommand{\A}{\mathbb{A}}

\newcommand{\GL}{\text{GL}}

\newcommand{\calA}{\mathcal{A}}
\newcommand{\calB}{\mathcal{B}}

\newcommand{\calG}{\mathcal{G}}

\newcommand{\calI}{\mathcal{I}}
\newcommand{\calU}{\mathcal{U}}
\newcommand{\calO}{\mathcal{O}}

\newcommand{\bP}{\mathbb{P}}
\newcommand{\bx}{\mathbf{x}}
\newcommand{\by}{\mathbf{y}}
\newcommand{\bc}{\mathbf{c}}

\newcommand{\bv}{\mathbf{v}}

\newcommand{\bd}{\mathbf{d}}
\newcommand{\ba}{\mathbf{a}}

\newcommand{\bk}{\mathbf{k}}
\newcommand{\bD}{\mathbf{D}}

\newcommand{\fb}{\mathfrak{b}}
\newcommand{\fp}{\mathfrak{p}}
\newcommand{\frakP}{\mathfrak{P}}
\newcommand{\ft}{\mathfrak{t}}
\newcommand{\fa}{\mathfrak{a}}
\newcommand{\fk}{\mathfrak{k}}
\newcommand{\fS}{\mathfrak{S}}

\newcommand{\Frob}{\textrm{Frob}}

\newcommand{\BM}{\bigtriangledown}

\newcommand{\fraka}{\mathfrak{a}}
\newcommand{\frakp}{\mathfrak{p}}

\newcommand{\ord}{\text{ord}}

\newcommand{\vol}{\text{vol}}

\newcommand{\Gal}{\text{Gal}}
\newcommand{\sym}{\text{sym}}
\newcommand{\Res}{\textrm{Res}}

\newcommand{\legendre}[2]{\ensuremath{\left( \frac{#1}{#2} \right) }}

\title{On Manin's Conjecture for Ch\^atelet surfaces}
\author[K. Woo]{Katharine Woo}
\address{Department of Mathematics, Stanford University, Stanford, CA 94305}
\email{katywoo@gmail.com}
\date{\today}

\begin{document}

\maketitle

\begin{abstract}
    We resolve Manin's conjecture for all Ch\^atelet surfaces over $\Q$. 
\end{abstract}

\begin{small}
\setcounter{tocdepth}{1}
\tableofcontents
\end{small}

\section{Introduction}
Let $\Delta\in \Q^\times$ satisfy $\sqrt{-\Delta}\not\in \Q$, and let $f(z)\in \Q[z]$ be a separable polynomial of degree $3$ or $4$.
A Ch\^atelet surface $X_{\Delta,f}$ over $\Q$ is defined as a proper smooth model of the affine surface:
 $$x^2+\Delta y^2=f(z).$$
The purpose of this paper is to prove Manin's conjecture for all Ch\^atelet surfaces, including when $\Delta<0.$ 

Ch\^atelet surfaces were first studied by Ch\^atelet in \cite{Chatelet2,Chatelet1}, who was interested in parameterizing their rational points over number fields (see \cite{CTChateletsurvey}). Later, these surfaces were the first examples of surfaces that are not homogeneous spaces for which the Hasse principle, weak approximation, and other interesting problems about rational points were resolved. For certain pairs $(\Delta, f(z))$, $X_{\Delta,f}$ may fail the Hasse principle due to a Brauer-Manin obstruction. In \cite{ColliotSSDI,ColliotSSDII}, Colliot-Th\'el\`ene, Sansuc, and Swinnerton-Dyer show that the Brauer-Manin obstruction is the only possible cause of why such a surface may fail the Hasse principle; these results hold not only over $\Q$, but over any number field. 

We first present Manin's conjecture, as applied to Ch\^atelet surfaces. There exists a map $\psi:X_{\Delta,f}\rightarrow\bP^4$, whose image is a singular del Pezzo surface of degree four (as described by Browning in \cite{BrowningLinear}). This map then induces a height function on $X_{\Delta,f}$ by pulling back the natural height $H$ on $\bP^4$ via $\psi.$ We define the point-counting function: 
$$\Tilde{N}(X_{\Delta,f},B) = \#\{x\in X_{\Delta,f}(\Q): H(x)\leq B\}.$$
Manin's conjecture, originally written as \cite[Conjecture C']{BatyrevManin} by Batyrev and Manin and stated over $\Q$ explicitly by Browning as \cite[Conjecture 2.3]{Browning-book}, predicts the behavior of $\Tilde{N}(X_{\Delta,f},B)$ as $B\rightarrow\infty$; additionally, in \cite{PeyreConstant}, Peyre gives an interpretation of the constant in terms of local densities and the Brauer group. In the following conjecture, the exponent is worked out in \cite[Lemma 1]{BrowningLinear} by Browning. 
\begin{conjecture}[Manin's conjecture for Ch\^atelet surfaces]\label{conj: manin}
Let $\Delta\in \Q^\times$ satisfy that $\sqrt{-\Delta}\not\in \Q$ and let $f(z)\in \Z[z]$ be a separable polynomial of degree 3 or 4 that decomposes as a product of $r$ irreducible polynomials $f=f_1\hdots f_r$. Let $X_{\Delta,f}/\Q$ be a Ch\^atelet surface defined by $x^2+\Delta y^2 = f(z)$. Assume that $X_{\Delta,f}(\Q)\neq \emptyset$. Then as $B\rightarrow\infty$, 
   $$\tilde{N}(X_{\Delta,f},B) \sim C_{\Delta,f} B\log(B)^{\rho_{\Delta,f}-1},$$
   where $C_{\Delta,f}>0$ and $$\rho_{\Delta,f} = 2+ \#\{1\leq i\leq r: \sqrt{-\Delta}\in \Q[z]/(f_i(z))\}.$$
\end{conjecture}

\begin{remark}
    This formulation of Manin's conjecture over $\Q$ is equivalent to the same statement under the assumption that $X_{\Delta,f}(\Q)$ is Zariski-dense, thanks to the fact that if $X_{\Delta,f}(\Q) \neq  \emptyset$ then $X_{\Delta,f}(\Q)$ is Zariski-dense. Indeed by \cite[Proposition 9.8.(iii)]{ColliotSSDII}, if $X_{\Delta,f}(\Q)\neq \emptyset$ then $X_{\Delta,f}$ is birational to a geometrically integral nonconical cubic surface in $\bP^{3}$ with a pair of conjugate singular points. Then by \cite[Theorem A]{CorayTsfasman} of Coray and Tsfasman, $X_{\Delta,f}$ must be unirational and thus $X_{\Delta,f}(\Q)$ is Zariski-dense. 
    
    Moreover, the question of when $X_{\Delta,f}(\Q)\neq \emptyset$ is settled by the seminal result of Colliot-Th\'el\`ene, Sansuc, and Swinnerton-Dyer \cite{ColliotSSDII}; it demonstrates that $X_{\Delta,f}(\Q)=\emptyset$ if and only if there is either a Brauer-Manin obstruction or a local obstruction to the existence of rational points. Furthermore, as discussed in \cite{ColliotSSDII}, a Brauer-Manin obstruction to the Hasse principle may occur only when $f(z)$ is the product of two irreducible quadratics. There is a finite and effective computation for determining when $X_{\Delta,f}(\Q)=\emptyset;$ this computation uses auxiliary varieties that we will introduce in (\ref{eq: def of variety X_Delta, f_1,...,f_r}).
\end{remark}

Numerous instances of Manin's conjecture for Ch\^atelet surfaces have been resolved. In \cite{BrowningLinear}, Browning established an upper bound of the correct order of magnitude for Ch\^atelet surfaces with $\Delta>0.$ In \cite{IwaniecMunshiChatelet}, Iwaniec and Munshi derive a lower bound when $f(z)$ is an irreducible cubic polynomial and $\Delta$ can be either positive or negative. In their seminal result, de la Bret\`eche, Browning, and Peyre \cite{delaBrownPeyre-maninChatelet} introduced the use of the torsor method, among other tools, to prove Manin's conjecture for $x^2+y^2=f(z)$ when $f$ decomposes as a product of distinct linear factors. Later, de la Bret\`eche and Browning \cite{delaBrowningChatelet} showed Manin's conjecture for the case that $\Delta=1$ and $f(z)$ contains an irreducible cubic factor. In \cite{delaTenenbaum-Manin}, the asymptotic for Manin's conjecture is established for $x^2+y^2=f(z)$ where $f$ is an irreducible quartic or a product of two irreducible quadratic factors over $\Q(i)$ by de la Bret\`eche and Tenenbaum; these factorization cases are often regarded as the trickiest cases. Finally, Destagnol \cite{Destagnol} resolved Manin's conjecture for $\Delta=1$ by proving the asymptotic for  $x^2+y^2 = f(z)$ where $f$ is a product of two linearly independent linear factors and an irreducible quadratic factor, the remaining unknown factorization. In the above cases, the techniques are amenable to replacing $x^2+y^2$ with any positive-definite quadratic form of class number one.  

Ch\^atelet surfaces belong to a wider class of surfaces; indeed, Ch\^atelet surfaces are singular del Pezzo surfaces of degree four. Manin's conjecture was established by Batyrev and Tschinkel \cite{BatyrevTschinkel} for all del Pezzo surfaces of degree $\geq 7$ and nonsingular del Pezzo surfaces of degree 6, as these are toric varieties. Additionally, Manin's conjecture is known for many more examples of del Pezzo surfaces of lower degree, as is discussed by Browning in \cite[Chapter 2]{Browning-book}.\\

In \cite[Lemma 2]{BrowningLinear}, Browning shows that 
\begin{multline*}
    \tilde{N}(X_{\Delta,f},B) = \frac{1}{2}\cdot \#\Big\{((x,y),(u,v),t) \in \Z^2\times \Z^2\times \Z_+: x^2 + \Delta y^2 = t^2 F(u,v) \\
    \gcd(x,y,t) = \gcd(u,v)=1,
    \max(|tu^2|,|tv^2|, |x|,|y|)\leq B\Big\},
\end{multline*}
where $F(u,v)$ denotes the homogenization of $f(z)$.
We modify this height function slightly to better suit our techniques. Define: 
\begin{multline}\label{eq: def of the height}
    N(X_{\Delta,f},B) := \frac{1}{2}\cdot \#\Big\{ ((x,y),(u,v),t)\in \Z^2\times \Z^2 \times \Z_+: x^2 + \Delta y^2 = t^2 F(u,v) \\ \gcd(x,y,t)=\gcd(u,v)=1, \max(|tu^2|,|tv^2|,\|x\pm y\sqrt{-\Delta}\|)\leq \nu_{\Delta,f}B\Big\},
\end{multline}
where $\|.\|$ denotes the Euclidean norm on $\C$ and $$\nu_{\Delta,f} = \begin{cases}
    \left(\max_{\by\in [-1,1]^2}|F(\by)|\right)^{1/2}, & \Delta>0 \\ 1, & \Delta<0.
\end{cases}$$ This new height is bounded up to a constant (depending on $\Delta$ and $f$) both above and below by
the natural height on $X_{\Delta,f}$ induced by $\psi$ and hence $N(X_{\Delta,f},B) \asymp \tilde{N}(X_{\Delta,f},B)$; however, this choice simplifies our computations. 

\begin{thm}\label{thm: Manin}
    Let $\Delta\in \Z$ be a squarefree integer satisfying that $\sqrt{-\Delta}\not\in \Q$. Let $f(z)\in \Z[z]$ be a separable polynomial of degree 3 or 4. Let $X_{\Delta,f}/\Q$ be the corresponding Ch\^atelet surface. Then as $B\rightarrow \infty$, there exists a constant $c_{\Delta,f}$ such that $$N(X_{\Delta,f},B) = c_{\Delta,f}B\log(B)^{\rho_{\Delta,f}-1}+O_{\Delta,f}(B\log(B)^{\rho_{\Delta,f}-1-10^{-10}}).$$
    Moreover, if $c_{\Delta,f}$ is zero then there is a Brauer-Manin obstruction or a local obstruction to rational points on $X_{\Delta,f}.$
\end{thm}
\begin{corollary}\label{cor: Manin}
    Manin's conjecture holds for all Ch\^atelet surfaces over $\Q.$
\end{corollary}
\begin{remark}
    Observe that if there is a Brauer-Manin obstruction or a local obstruction to the existence of rational points on $X_{\Delta,f}$, then we automatically know that $N(X_{\Delta,f},B) = 0$. Hence, Theorem \ref{thm: Manin} can be upgraded to $c_{\Delta,f}=0$ if and only if there is a Brauer-Manin or local obstruction. 
\end{remark}

In Theorem \ref{thm: Manin}, we do not assume a priori that $X_{\Delta,f}(\Q)\neq \emptyset.$ Our proof of the aforementioned asymptotic result gives an explicit, albeit complicated, expression for $c_{\Delta,f}$. Notably, $c_{\Delta,f}$ is closely related to auxiliary varieties; these varieties appear in the work of Colliot-Th\'el\`ene, Coray, and Sansuc in \cite{ColliotTheleneCoraySansuc} when $f(z)$ is the product of two irreducible quadratics. Since our expression for $c_{\Delta,f}$ is complicated, we determine when it vanishes but do not pursue an adelic interpretation along the lines of Peyre's prediction \cite{PeyreConstant}. In Appendix \ref{sec: appendix peyre}, we show that the constant matches Peyre's prediction in a simple, but novel case. 

Let $f(z) = f_1(z) \hdots f_r(z)$ be a decomposition into irreducible factors and let $\alpha_1,\hdots,\alpha_r\in \Q^\times$ satisfy that $\alpha_1\cdots \alpha_r=1$. We define the auxiliary variety:
    \begin{equation}\label{eq: def of variety X_Delta, f_1,...,f_r}
    X^*_{\Delta,f_1,\hdots,f_r, \alpha_1,\hdots, \alpha_r} := \{(z,(x_i,y_i)_{i=1}^r): x_i^2 + \Delta y_i^2 = \alpha_i f_i(z)\neq 0\}.
\end{equation}
The relationship between $X^*_{\Delta,f_1,...,f_r,\alpha_1,...,\alpha_r}$ and $X_{\Delta,f}$ will be expanded upon further in \S\ref{sec: constant}. For now, let us define the following condition: 
\begin{enumerate}[label = (\Alph*)]
    \item There exists a tuple $(\alpha_1,...,\alpha_r)\in (\Q^\times)^r$ satisfying $\alpha_1\cdots\alpha_r = 1$ such that $$X^*_{\Delta,f_1,...,f_r,\alpha_1,...,\alpha_r}(\Q_v)\neq \emptyset$$ for all places $v$ of $\Q$.
\end{enumerate}
We will show that if condition (A) holds then our constant $c_{\Delta,f}$ is nonzero (see Proposition \ref{prop: leading constant auxiliary}). Additionally, if $f(z)$ is irreducible, then a consequence is that $c_{\Delta,f}$ is nonzero if there are no local obstructions to rational points on $X_{\Delta,f}$.

In this way, Theorem \ref{thm: Manin} gives an \textit{analytic proof of the existence of many rational points on $X_{\Delta,f}$ via direct counting} if condition (A) holds. Now, by the work of Colliot-Th\'el\`ene, Coray, and Sansuc, and an application of Harari's ``formal lemma'' \cite{Harari}, one can interpret condition (A) in terms of the Brauer-Manin obstruction and local obstructions of $X_{\Delta,f}$. In particular, this allows us to conclude that if $c_{\Delta,f}$ is zero, then $X_{\Delta,f}(\Q) = \emptyset$ as well.

Finally, we remark on the possible application of these techniques to achieve an analytic proof of weak approximation for $X_{\Delta,f}$. That $x^2+\Delta y^2$ is a binary quadratic form corresponding to principal ideals is not of much relevance to our proof; indeed one can recreate the entire argument with a positive-definite or indefinite binary quadratic form. Adding arbitrary congruence conditions on $(u,v)$ will also be of little consequence for the analytic techniques. These points will be expanded on and are necessary ingredients in forthcoming work on del Pezzo surfaces of degree one, where both arbitrary congruence conditions and arbitrary positive-definite binary quadratic forms are required. So, at least in the case when $f(z)$ is irreducible, the methods can be adapted to establish weak approximation for $X_{\Delta,f}$.
\\

Finally, we highlight the new ingredients that allow us to push beyond $\Delta=1$. In \S2, we provide a detailed outline of our proof. First, we decompose our count into contributions from Eisenstein series and cusp forms. 
Consider the following function for $\Delta$ a positive squarefree integer: $$\theta(z) = \sum_{\substack{I\subset \calO_{\Q(\sqrt{-\Delta})}\\ I \textrm { is principal}}} \exp(2\pi i z \cdot N(I)).$$
This function is a modular form of weight one and has a decomposition into Eisenstein and cuspidal components. If $\Delta=1$, then the fact that $\Q(i)$ has class number one implies that $\theta(z)$ is purely Eisenstein, i.e. the cuspidal part is zero. We will find that in general the cuspidal contribution is negligible in comparison with the Eisenstein contribution. 

When $\Delta<0$, i.e. the indefinite case, a similar picture forms with careful analysis; the main additional obstacle in this case is that $\Q(\sqrt{-\Delta})$ has an infinite unit group. To mitigate this difficulty, we restrict our point-count to mildly-short boxes and intervals of length scaled by $\ell = \log(B)^{-10^{-10}}.$ This restriction allows us to estimate accurately the count $$\#\{|x\pm y\sqrt{-\Delta}|\leq B: x^2 + \Delta y^2 = t^2 F(u,v)\}$$
and convert our sum into a question about principal ideals $I$ with norm $t^2F(u,v).$ Of course, summing over short intervals and boxes complicates our estimates and sometimes adds an extra factor of $\ell^{-O(1)}$, where this exponent can be made explicit. Fortunately, the error terms achieved by previous results on Manin's conjecture, whose techniques we mimic, achieve savings of $\log(B)^{-10^{-5}}$ which will be more powerful than $\ell^{-O(1)}.$ Once the sum is written in terms of principal ideals of $\Q(\sqrt{-\Delta})$, we can again decompose the count in terms of Eisenstein and cuspidal contributions as before; this is equivalent to decomposing the Maass form into Eisenstein and cuspidal ingredients.  \\

Second, we demonstrate that the cuspidal contribution is always negligible in comparison with the Eisenstein contribution. 
Let $C(z)$ denote the cuspidal part of $\theta(z)$ and let $\lambda_C(n)$ denote its Fourier coefficients. Our cuspidal contribution reduces to estimating sums of the form
$$\sum_{t\leq B} \sum_{|u|,|v|\leq (B/t)^{1/2}} \left|\lambda_C(t^2F(u,v))\right|,$$
and of the form
$$\sum_{t\leq B}\sum_{d\mid t}\mu(d) \sum_{|u|,|v|\leq (B/t)^{1/2}} \lambda_C\left(t^2F(u,v)/d^2\right).$$ 

In the first case, we study correlation sums of the absolute values of its Fourier coefficients $\lambda_C(n)$ along polynomial values; although we sacrifice cancellation in the sign of $\lambda_C(n)$, we still gain a saving of a power of log. In particular, we use a Nair-type sieve \cite{Nair} to reduce the problem to computing the first few Sato-Tate moments of the corresponding $L$-functions (see Rankin's \cite{Rankin} work on the moments of the Ramanujan $\tau$-function). We remark that this method was used by Holowinsky \cite{HolowinskyShiftedConvolution} to study correlation sums of Fourier coefficients of cusp forms with quadratic polynomials. Heath-Brown \cite{HBclassnum} also used this technique to study the equidistribution of solutions to certain Ch\^atelet surfaces and his work inspired our proof techniques. While this technique handles most cases, there are delicate issues concerning the cuspidality of the base change lift of $C(z)$ to the splitting field of $f(z)$; in these cases, the absolute values are too expensive of a sacrifice. 

Fortunately, in the above exceptional cases, we can appeal to the more robust method of handling the second sum. We determine that since $C(z)$ is a cusp form, we can derive savings for the sum $$\sum_{d\leq X} \mu(d) \sum_{t\leq X/d} \lambda_C(at^2).$$ Moreover, our cusp forms are induced from class group characters of $\Q(\sqrt{-\Delta}).$ This extra structure allows us to track explicitly the dependency in the above bound on $a$. Thus, we can derive a polylogarithmic saving on $$ \sum_{|u|,|v|\leq B^{1/2}}\sum_{d\leq B/\max(|u|,|v|)^2} \mu(d) \sum_{t\leq B/d\max(|u|,|v|)^2} \lambda_C(t^2F(u,v)).$$ 
The inner sum is controlled by the poles of the symmetric square $L$-function of $C(z)$ (see \cite{Shimura} for a discussion of the poles); if there is no pole, then the desired cancellation follows from the inner sum. If $C(z)$ is dihedral and the symmetric square $L$-function has a pole, then an asymptotic for the inner sum is well-understood and the cancellation follows from the M\"obius sum. \\

For the Eisenstein contribution, we follow the methods of \cite{delaBrownPeyre-maninChatelet, delaBrowningChatelet,delaTenenbaum-Manin, Destagnol, HB-linear} that resolved the $\Delta=1$ case. The third new feature is that for more general $\Delta$, due to Gauss genus theory (see \cite{FomenkoQuad}), the Eisenstein Fourier coefficients now contain an extra summation over genus characters. This extra complication can still be handled by the aforementioned techniques for the $\Delta=1$ case and we achieve an explicit expression for the leading constant $c_{\Delta,f}.$

The fourth and last new ingredient is the analysis of the vanishing of the leading constant $c_{\Delta,f}$. This is resolved in \S\ref{sec: constant} and involves the auxiliary varieties introduced in (\ref{eq: def of variety X_Delta, f_1,...,f_r}). In particular, our leading term brings us to a non-negative sum which is approximately of the form
$$\sum_{(\alpha_1,...,\alpha_r)/\sim} \prod_{p} \mu_p(X^*_{\Delta,f_1,...,f_r,\alpha_1,...,\alpha_r}(\Q_p)).$$
where if $\mu_p(X^*_{\Delta,f_1,...,f_r,\alpha_1,...,\alpha_r}(\Q_p))\geq 0$ is zero then $X^*_{\Delta,f_1,...,f_r,\alpha_1,...,\alpha_r}(\Q_p)=\emptyset.$ Here, the sum over $(\alpha_1,...,\alpha_r)/\sim$ is over a sufficient and finite set of auxiliary varieties. Consequently, we show that if $c_{\Delta,f}=0$ then there is a local obstruction for every auxiliary variety.

\section*{Acknowledgments}
The author would like to thank her advisor Peter Sarnak for his guidance and support, Tim Browning for many enlightening conversations about Manin's conjecture and his helpful feedback, and Jean-Louis Colliot-Th\'el\`ene for his insightful comments on auxiliary varieties and for explaining the proof of Theorem \ref{thm: descent}. The author would also like to thank Niven Achenjang, R\'egis de la Bret\`eche, Ulrich Derenthal, Kevin Destagnol, Trajan Hammonds, Roger Heath-Brown, and Will Sawin for comments on earlier versions of this paper.

This material is based upon work supported by the National Science Foundation under grants DGE-2039656 and DMS-2502864. Any opinions, findings, and conclusions or recommendations expressed in this material are those of the author and do not necessarily reflect the views of the National Science Foundation.

\section{Outline}\label{sec: outline}
In this section, we outline our proof strategy; our goal is to highlight the new techniques and convey the essential ideas, without the technical detail provided in the rest of the paper. By leaving out some of the more notationally-heavy complications, we hope to provide a clear roadmap for the rest of the paper. We warn the reader that in this section some of the claims are vague (and will be made precise in later sections) and some are technically false (but morally true and the modifications needed to correct the statements will be given later).\\ 

Let us start with the case when $\Delta>0$. If $\Delta=1$, Theorem \ref{thm: Manin} is known; so assume that $\Delta>1$ (for example, we can take $\Delta=5$). The main challenge to overcome is when $\Q(\sqrt{-\Delta})$ has class number greater than one. From our definition (\ref{eq: def of the height}), we see that: $$N(X_{\Delta,f},B) = \sum_{0\leq t\leq B} \sum_{|u|,|v|\leq (B/t)^{1/2}} r_{x^2+\Delta y^2}(t^2F(u,v)),$$
where $r_{x^2+\Delta y^2}(n) = \#\{(x,y)\in \Z^2: x^2+\Delta y^2 = n\}.$
Consider the binary theta series $$\theta(z) = \sum_{n} r_{x^2+\Delta y^2}(n) \exp(2\pi i n z) = \sum_{x,y} \exp(2\pi i z (x^2+\Delta y^2),$$
which is a \textbf{modular form of weight one} in $M_1(\Gamma_0(-\Delta),\chi)$ (see \cite[\S22.3]{IwaniecKowalski}), where $\chi$ denotes the real quadratic character for $\Q(\sqrt{-\Delta}).$ We can decompose $\theta(z)$ in terms of its Eisenstein and cuspidal contributions $\theta(z) = E(z) +C(z)$. Note that since $\theta(z)$ is a weight one modular form, the Fourier coefficients of the Eisenstein component and the cuspidal component are expected to be of the same magnitude.

The Fourier coefficients of the Eisenstein part can be written explicitly in terms of \textbf{genus characters}, i.e. the class group character of order $\leq 2$. As such, we have a nice description of the Fourier coefficients of the Eisenstein part: $$\lambda_E(n) \approx (1\star \chi)(n) := \sum_{d\mid n} \chi(d).$$

For the cuspidal part, we can write $\lambda_C(n)$ as a fixed linear combination of Fourier coefficients $\lambda_{\Psi}(n)$, where $\Psi$ is the automorphic cuspidal representation of $\GL_2(\A_\Q)$ given by a class group character $\psi$ of order $\geq 3$. The critical fact that we can restrict to those $\psi$ of order $\geq 3$ follows from Gauss genus theory for binary quadratic forms, as discussed by Fomenko in \cite{FomenkoQuad}. We can describe these coefficients $\lambda_\Psi(n)$ explicitly: 
\begin{equation}\label{eq: outline lambda xi}
    \lambda_\Psi(n) = \sum_{\substack{\fa\subset \calO_{\Q(\sqrt{-\Delta})}\\ N(\fa)=n}}\psi(\fa).
\end{equation}

From this decomposition of $r_{x^2+\Delta y^2}(n)$, we can see that Theorem \ref{thm: Manin} follows from gaining asymptotics on $$\sum_{t\leq B} \sum_{|u|,|v|\leq (B/t)^{1/2}} (1\star \chi)(t^2F(u,v))$$
and sufficient upper bounds on $$\sum_{t\leq B} \sum_{|u|,|v| \leq (B/t)^{1/2}} \lambda_\Psi(t^2F(u,v)).$$
Note that we have left out the $\gcd$-conditions here. \\ 

Before analyzing the two sums above, let us discuss what happens when $\Delta<0$. There are two challenges involved in the indefinite case: again, we run into issues when $\Q(\sqrt{-\Delta})$ has class number greater than one; we also must handle the fact that there are infinitely many solutions to the equation $x^2+\Delta y^2 = 1.$ To handle the second issue, we restrict to counting in short boxes. 

Let $\ell, L = w(B)$, a small value going slowly to zero as $B\rightarrow\infty.$ We eventually take $w(B) = \log(B)^{-10^{-10}}.$ Define the box $$\calB(\bx_0,L):= \bx_0 + [0,L]^2,$$
and the interval
$$\calI(\theta,\ell) := \{\exp(2\pi i \theta'): \theta'\in [\theta,\theta+\ell]\}.$$
Then we consider the restricted counting function:
\begin{multline*}
N(X_{\Delta,f},\calB(\bx_0,L),\calI(\theta,\ell)):=   \#\left\{ ((x,y),(u,v),t)\in \Z^2\times \Z^2\times \Z_+: \begin{array}{cc}
     x^2+\Delta y^2 = t^2 F(u,v), \\
     (u,v)\in (B/t)^{1/2}\calB(\bx_0,L),\\
     \|x\pm y\sqrt{-\Delta}\|\leq B ,\\ 
     \arg(x+y\sqrt{-\Delta}) \in \calI(\theta,\ell), \\ 
    \gcd(x,y,t) = \gcd(u,v) = 1
\end{array} 
       \right\}.
\end{multline*}
In other words, we restrict $(u,v)$ to a mildly shorter box $(B/t)^{1/2}\calB(\bx_0,L)$ and constrain the argument $\arg(x+y\sqrt{-\Delta})\in \calI(\theta,\ell).$ Here we define the argument as $$\arg(x+y\sqrt{-\Delta}) = \frac{|x+y\sqrt{-\Delta}|}{|x-y\sqrt{-\Delta}|} \bmod |\varepsilon|^2,$$
where $\varepsilon$ is the fundamental unit to Pell's equation $x^2+\Delta y^2=1.$ Our strategy is to produce an asymptotic for $N(X_{\Delta,f},\calB(\bx_0,L),\calI(\theta,\ell))$; this asympotic would then recover an estimate for $N(X_{\Delta,f},B)$ by splitting the region $\{|x|,|y|\leq B, |u|,|v|\leq (B/t)^{1/2}\}$ into short boxes and arguments in $\calI(\theta,\ell)$, and then summing over $\bx_0$ and $\theta$. This computation is done in \S\ref{subsec: proof of main thm}.

By restricting to the shorter box and interval, we can make precise the count $$\#\{|x\pm y\sqrt{-\Delta}|\leq B: x^2+y^2 = t^2F(u,v), \arg(x+y\sqrt{-\Delta})\in \calI(\theta,\ell)\}.$$
In particular, we can find an explicit constant $$\kappa(\bx_0,\theta) = c_{\Delta}\cdot \max(-\log|F(\bx_0)|,0) \cdot \frac{|\calI(\theta,\ell)|}{4\pi}$$ satisfying that \begin{multline*}N(X_{\Delta,f},\calB(\bx_0,L),\calI(\theta,\ell)) = (\kappa(\bx_0,\theta)+O(w(B))) \\ \times \sum_{t\leq B} \sum_{(u,v)\in (B/t)^{1/2}\calB(\bx_0,L)} \#\{I\subset \Q(\sqrt{-\Delta}): I\textrm{ a principal ideal and }N(I) = t^2F(u,v)\} \\ + E(X_{\Delta,f},\calB(\bx_0,L),\calI(\theta,\ell)),
\end{multline*}
where $E(X_{\Delta,f},\calB(\bx_0,L),\calI(\theta,\ell))$ is a linear combination of remainder terms of the form \begin{equation*}\sum_{t\leq B} \sum_{|u|,|v|\leq (B/t)^{1/2}} |\lambda_\Xi(t^2 F(u,v))|\end{equation*}
for $\Xi$ a cuspidal automorphic representation of $\GL_2(\A_\Q)$ given by an unitary nontrivial Hecke Gr\"ossencharacter $\xi$ on $\Q(\sqrt{-\Delta})$ of infinite order (see Definition \ref{def: fundamental character}). 

For the main term, we again use a decomposition into Eisenstein and cuspidal components. In particular, we can express $$
\#\{I\subset \Q(\sqrt{-\Delta}): I\textrm{ principal and }N(I) = n\} = \frac{1}{|C|}\sum_{\psi \in \hat{C}} \sum_{\substack{\fa\subset \calO_{\Q(\sqrt{-\Delta})}\\ N(\fa) = n}} \psi(\fa),$$
where $C$ is the class group of $\Q(\sqrt{-\Delta}).$ Again, from Gauss's genus theory, the Eisenstein contribution comes from the class group characters of order $\leq 2$; the Fourier coefficient for the Eisenstein part is approximately $$\lambda_E(n) \approx(1\star \chi)(n).$$
Similarly, the cuspidal Fourier coefficient is a linear combination of sums $\lambda_\Psi$ where $\Psi$ is the representation given by $\psi$ a class group character of order $\geq 3.$ This matches the process completed in the positive definite case. 

In the end, Theorem \ref{thm: Manin} reduces to estimating the sum 
\begin{equation}\label{eq: outline eisenstein}\sum_{t\leq B} \sum_{(u,v)\in (B/t)^{1/2}\calB(\bx_0,L)} (1\star \chi)(t^2F(u,v)),\end{equation}
and providing a suitable upper bound on the sums \begin{equation}\label{eq: outline cusp}\sum_{t\leq B} \sum_{|u|,|v|\leq (B/t)^{1/2}} \lambda_\Xi(t^2F(u,v)),\end{equation}
\begin{equation}\label{eq: outline cups abs}
    \sum_{t\leq B}\sum_{|u|,|v|\leq (B/t)^{1/2}}|\lambda_\Xi(t^2F(u,v))|.
\end{equation}
where $\Xi$ is the representation of $\GL_2(\A_\Q)$ induced by a Hecke character $\xi$. Additionally, (\ref{eq: outline cusp}) will only be used when $\xi$ is a class group character of order $\geq 3$ (and hence has finite image), whereas (\ref{eq: outline cups abs}) only occurs when $\xi$ is a Hecke Gr\"ossencharacter (which has infinite image). We remark that it is likely that with more careful analysis we could remove the absolute values in (\ref{eq: outline cups abs}); however, we achieve the necessary savings with the absolute values added. \\ 

Our method for estimating the Eisenstein contribution follows closely the previous approaches when $\Delta=1$; indeed, when $\Delta=1$, $\theta(z) = E(z)$ and so the whole sum is contained in the Eisenstein part. We particularly mimic the proofs of de la Bret\`eche and Tenenbaum in \cite{delaTenenbaum-Manin} and the proof of Heath Brown in \cite{HB-linear}. However, we must consider when $(u,v)\in (B/t)^{1/2}\calB(\bx_0,L)$ lies in a short box -- fortunately, all of the techniques are either amenable to this issue or we achieve a savings that is far larger than the loss incurred from summing over $|u|,|v|\leq (B/t)^{1/2}$ instead of the smaller region. 

Using estimates on lattice points in convex regions, we establish a near-optimal \textbf{level of distribution} result for squarefree binary forms in \S\ref{sec: LoD}. Write $F(x,y) = \prod_{i=1}^r F_i(x,y)$ as the decomposition into irreducible forms. We show that $$\sum_{d_i\leq D_i} \left|\sum_{\substack{(u,v)\in (B/t)^{1/2}\calB(\bx_0,L)\\ d_i\mid F_i(u,v)}} 1 - \frac{BL^2}{t}\cdot \frac{\#\{\mathbf{x}\bmod d_1\cdots d_r: d_i \mid F_i(\bx)\}}{(d_1\cdots d_r)^2}\right| \ll D + \sqrt{\frac{BD}{t}}\exp(\sqrt{\log\log(B)}),$$
where $D = D_1\hdots D_r.$ This result allows us to take $D\leq \frac{B}{t} \log(B/t)^{-\varepsilon}$ for any $\epsilon>0$ and still obtain a reasonable bound. This level of distribution result is obtained in \cite{delaTenenbaum-Manin} for $F(x,y)$ irreducible or a product of two irreducible quadratics; we closely follow their techniques. Also, Marasingha \cite{Marasingha-almostprimes} studies a general level of distribution result for squarefree binary forms and achieves a savings for $D\leq \frac{B}{t} \log(B/t)^{-A}$ for any large enough integer $A>0$. We remark that if $F(x,y)$ has linear factors then the shape of the above bound is slightly different. 

This step is the primary obstacle to studying $F(x,y)$ of higher degree; for $\deg(F)\leq 3$, the level of distribution is sufficient and for $\deg(F)=4$, it is almost sufficient (up to a power of $\log(B)$). For $\deg(F)>4$, the level of distribution would be a power of $B$ away from necessary and this gap can not be covered by the argument as written below. \\

The other main input into our proof comes directly from de la Bret\`eche and Tenenbaum in \cite{delaTenenbaum-delta} when $F(x,y)$ is irreducible or the product of two irreducible quadratics; we describe their idea in detail in \S\ref{sec: large moduli} and discuss how to handle other factorizations of $F(x,y)$. Define the \textbf{twisted Hooley $\Delta$-function}:
$$\Delta(n,\chi) := \sup_{\substack{D\in \R \\ 0\leq v\leq 1}} \left|\sum_{\substack{d\mid n\\ d\in [D,e^v D]}}\chi(d)\right|.$$
Applying Nair's sieve and \cite[Theorem 1.1]{delaTenenbaum-delta}, we achieve that for $D\geq \frac{B}{t}\log(B/t)^{-10^{-5}},$ $$ \sum_{\substack{|u|,|v|\leq (B/t)^{1/2}\\ \exists d_i\sim D_i: d_i\mid F_i(u,v)}} \Delta(F(u,v),\chi) \ll \frac{B}{t} \cdot \log(B/t)^{\rho_{\Delta,f}-2-10^{-5}}.$$
Note that since we take $L=w(B) = \log(B)^{-10^{-10}}$, this bound is still $\ll \frac{BL^2}{t} \log(B)^{\rho_{\Delta,f}-2-10^{-7}}.$ This procedure handles the large moduli. \\ 

Together, these results allow us to re-express the Eisenstein contribution as:
\begin{multline*}
    BL^2\sum_{t\leq B}\frac{1}{t}\sum_{(a_i)\bmod \Delta} \prod_{i=1}^r (1+\chi(a_i)) \cdot \#\{\bx\bmod \Delta: F_i(\bx)\equiv a_i \bmod \Delta\} \\ \times  \sum_{d_i\leq D_i} \frac{\chi(d_1\hdots d_r) \cdot \#\{\bx\bmod d_1\cdots d_r: d_i\mid F_i(\bx)\}}{(d_1\cdots d_r)^2}\\ 
    + O\left(\sum_{t\leq B}\frac{B}{t}\log(B/t)^{\rho_{\Delta,f}-2-10^{-5}}\right).
\end{multline*}
However, the above expression is an oversimplification -- in particular, it assumes that $(1\star \chi)(n)$ is a totally multiplicative function, which it is not. Instead, we must substitute the identity that $$(1\star \chi)(mn) = \sum_{c\mid \gcd(m,n)} \mu(c)\chi(c)(1\star \chi)(n/c)(1\star \chi)(m/c).$$
Consequently, we re-express our main term as:
\begin{multline*}
    \sum_{(u,v)\in (B/t)^{1/2}\calB(\bx_0,L)} (1\star \chi)(F(u,v)) = \sum_{(c_1,\hdots,c_r)}' \mu(c_1\cdots c_r)\chi(c_1\cdots c_r) \sum_{\substack{(u,v)\in (B/t)^{1/2}\calB(\bx_0,L) \\ c_i\mid F_i(u,v)}} \prod_{i=1}^r (1\star \chi)(F_i(u,v)/c_i).
\end{multline*}
Here the sum over $(c_1,\hdots,c_r)$ is restricted to a finite set of possible vectors where the entries divide the pairwise resultants of the irreducible factors of $f(z)$. 

Proceeding as before with our level of distribution result and bounds on the large moduli with the twisted Hooley $\Delta$-function, our main term is actually of the form:
\begin{multline}\label{eq: outline main term}
    BL^2 \sum_{(c_1,\hdots,c_r)}' \mu(c_1\cdots c_r)\chi(c_1\cdots c_r) \sum_{t\leq B} \frac{1}{t} \cdot \sum_{(a_i)\bmod \Delta} \prod_{i=1}^r (1+\chi(a_i))\#\{\bx\bmod \Delta: F_i(\bx)\equiv c_ia_i \bmod \Delta\} \\
     \times \sum_{d_i\leq D_i} \frac{\chi(d_1\cdots d_r) \#\{\bx\bmod c_1\cdots c_r\cdot d_1\cdots d_r: c_id_i\mid F_i(\bx)\}}{(c_1\cdots c_r\cdot d_1\cdots d_r)^2}.
\end{multline}
Before analyzing this leading term further, let us discuss the cuspidal contribution. 
\\ 

First, we start with the finite image case. To study (\ref{eq: outline cusp}), we first change the order of summation:
$$\sum_{t\leq B} \sum_{|u|,|v|\leq (B/t)^{1/2}} \lambda_\Xi(t^2F(u,v)) = \sum_{|u|,|v|\leq B^{1/2}} \sum_{t\leq B/\max(|u|,|v|)^2} \lambda_\Xi(t^2F(u,v)).$$
At this point, we need to address a previous false claim about cuspidal contribution in \eqref{eq: outline cusp} -- because we are counting those $(x,y)$ such that $\gcd(x,y,t)=1$, we need to restrict our sum to ideals satisfying that $N(I) = t^2F(u,v)$ where $I = \fa\ft$, $\gcd(\ft,\overline{\ft})=\gcd(\fa,\overline{\ft})=1$. In other words, our sum (\ref{eq: outline cusp}) should be of the form:
$$\sum_{|u|,|v|\leq B^{1/2}}\sum_{\substack{N(\fa) = F(u,v) }}\xi(\fa) \sum_{\substack{N(\ft)\leq B/\max(|u|,|v|)^2\\ \gcd(\fa,\overline{\ft})=1\\ \gcd(\ft,\overline{\ft})=1}}\xi(\ft)^2.$$
Now, since $\xi^2$ is a nontrivial class group character, we achieve a power saving upper bound in the sum over ideals $\ft$. In particular, we claim that 
$$\sum_{\substack{N(\ft)\leq B/\max(|u|,|v|)^2\\ \gcd(\fa,\overline{\ft})=1\\ \gcd(\ft,\overline{\ft})=1}}\xi(\ft)^2\ll \prod_{\fp\mid \fa} \left(1+\frac{1}{N(\fp)^{2/3}}\right) \cdot \left(\frac{B}{\max(|u|,|v|)^2}\right)^{5/6}.$$
The above is a consequence of the prime ideal theroem with nontrivial Hecke characters.

To complete our analysis of the finite image cuspidal contribution, we apply a generalization of Nair's sieve for non-negative multiplicative functions along polynomial values by de la Bret\`eche and Tenenbaum \cite{delaTenenbaum-sieve}. In particular, this sieve gives us that 
\begin{multline*}\sum_{|u|,|v|\leq B^{1/2}} \sum_{\substack{N(\fa)=F(u,v)}} \prod_{\fp\mid \fa} \left(1+\frac{1}{N(\fp)^{2/3}}\right) \\ \ll B \cdot \exp\left(\sum_{p\leq B^{1/2}} \frac{\#\{x\bmod p: f(x)=0\bmod p\}}{p} \cdot \left((1\star\chi)(1+p^{-2/3})-1\right)\right)
\ll B\log(B)^{\rho_{\Delta,f}-2}.
\end{multline*}
Using partial summation, we achieve that $$\sum_{\substack{N(\ft)\leq B/\max(|u|,|v|)^2 \\ \gcd(\ft,\overline{\ft})=1}} \xi(\ft)^2 \sum_{|u|,|v|\leq (B/N(\ft))^{1/2}} \sum_{\substack{N(\fa)=F(u,v)\\ \gcd(\fa,\overline{\ft})=1}}\xi(\fa) \ll B\log(B)^{\rho_{\Delta,f}-2}. $$ 
This is negligible compared to the main term of size $B\log(B)^{\rho_{\Delta,f}-1}.$ \\

We can also view this sum more generally through the perspective of the Fourier coefficients of cusp forms. Let $C(z)$ be a weight one cusp form with central character $\chi$; the gcd condition $\gcd(x,y,t)=1$ leads us to the sum $$\sum_{t\leq B} \sum_{d\mid t} \mu(d) \sum_{|u|,|v|\leq (B/t)^{1/2}} \lambda_C\left(\frac{t^2F(u,v)}{d^2}\right) = \sum_{|u|,|v|\leq B^{1/2}}\sum_{d\leq B/\max(|u|,|v|)^2} \mu(d) \sum_{t\leq B/(d\max(|u|,|v|)^2)} \lambda_C(t^2F(u,v)).$$
This naturally brings us to analyzing the sum $$\sum_{d\leq X} \mu(d) \sum_{n\leq X/d} \lambda_C(an^2)$$
and wanting to track the dependency on $a$ above. 
Using Perron's formula, we can rewrite the above in terms of $L(s,\sym^2(C))$ and $\zeta(s)^{-1}$. Since $C(z)$ is a cusp form, $L(s,\sym^2(C))$ can have at most a pole of order one at $s=1$ (see \cite{Shimura} for a criteria for when there is a pole); thus, $L(s,\sym^2(C))\zeta(s)^{-1}$ converges at $s=1$. This allows us to derive a bound $$\sum_{d\leq X} \mu(d) \sum_{n\leq X/d} \lambda_C(an^2) \ll \mathfrak{S}_a X \log(X)^{-10},$$
where $\mathfrak{S}_a$ is a finite product over the primes dividing $a$. When $C(z)$ is induced by a class group character, we have explicit bounds on $\mathfrak{S}_a$; in particular, the following holds for primes $p$:
$$\mathfrak{S}_p = \begin{cases}
    2+O(1/p^{2/3}), & \chi(p)=1,\\
    O(1/p^{2/3}), & \chi(p)=-1.
\end{cases}$$

Plugging the asymptotic above in and applying Nair's sieve as before, we achieve that the contribution from the cuspidal components with finite image is bounded by $$\ll B \sum_{|u|,|v|\leq B^{1/2}} \frac{\mathfrak{S}_{F(u,v)}}{\max(|u|,|v|)^2} \log(B/\max(|u|,|v|)^2)^{-10} \ll B\log(B)^{\rho_{\Delta,f}-2}.$$
This is again negligible compared to the main term of size $B\log(B)^{\rho_{\Delta,f}-1}.$ 
\\ 

In the Gr\"ossencharacter case, i.e. (\ref{eq: outline cups abs}), we aim to acheive a small power of $\log(B)$ savings from the \textbf{absolute values of the Fourier coefficients of $\Xi$}. For this argument, we apply the same version of Nair's sieve above with our non-negative multiplicative function now taken to be $|\lambda_\Xi(n)|$. Consequently, we have that $$\sum_{|u|,|v|\leq (B/t)^{1/2}} |\lambda_\Xi(t^2F(u,v))| \ll \frac{B}{t}\cdot \prod_{p\ll B/t}\left(1+\frac{(|\lambda_\Xi(p)|-1) \cdot \#\{x\bmod p: f(x) \equiv 0 \bmod p\}}{p}\right).$$
To understand the above product, we study the sum \begin{equation}\label{eq: outline sum over ideals}\sum_{\substack{\fp \subset \calO_{\Q(\sqrt{-\Delta})}\\ N(\frakp) \ll B/t}} \frac{\#\{x\bmod N(\fp): f(x)\equiv 0 \bmod N(\fp)\}}{2N(\fp)}\cdot |\xi(\fp)+\xi(\fp)^{-1}|.\end{equation}
One can show that when $p$ splits in $\Q(\sqrt{-\Delta})$, 
\begin{equation}\label{eq: outline cosine}|\lambda_\Xi(p)|-1 \leq \frac{1}{2}+\frac{1}{4}\lambda_{\Xi^2}(p).\end{equation}
When $p$ does not split in $\Q(\sqrt{-\Delta})$, by definition $|\lambda_\Xi(p)|-1 = -1$. Using (\ref{eq: outline cosine}), we can reduce (\ref{eq: outline sum over ideals}) to analyzing $$\left(\frac{-1}{4}+\frac{3}{4}(\rho_{\Delta,f}-2)\right)\log\log(B/t) + O(1) + \sum_{\substack{\fp\subset \calO_{\Q(\sqrt{-\Delta})}\\ N(\fp)\ll B/t}}\frac{\xi^2(\fp) \#\{x \bmod N(\fp): f(x)\equiv 0 \bmod N(\fp)\}}{N(\fp)}.$$

It remains to understand the final sum. At this point, we recall that we only analyze this sum $\Xi$ is induced by a Gr\"ossencharacter $\xi$, and hence $\xi^2$ is nontrivial. We must then lift the nontrivial Hecke character $\xi^2$ to $K_f$, the splitting field of $f(z)$ over $\Q(\sqrt{-\Delta})$, to understand the sum above; in other words, we need to understand the \textbf{base change} of $\xi^2$ to $K_f$. This new Hecke character $\tilde{\xi}=\xi^2 \circ N_{K_f/\Q(\sqrt{-\Delta})}$ on $K_f$ is necessarily nontrivial since $\xi$ had infinite order; it is, however, not true that the base change of $\xi$ to $K_f$ is always nontrivial if we remove the assumption that $\xi$ has infinite order.

The prime ideal theorem for Hecke Gr\"ossencharacters determines that 
$$\sum_{\substack{\fp\subset \calO_{\Q(\sqrt{-\Delta})}\\ N(\fp)\ll B/t}}\frac{\xi^2(\fp) \#\{x \bmod N(\fp): f(x)\equiv 0 \bmod N(\fp)\}}{N(\fp)}=o(\log\log(B/t)).$$
Consequently, for such character $\xi$, $$\sum_{t\leq B} \sum_{|u|,|v|\leq (B/t)^{1/2}} |\lambda_\Xi(t^2F(u,v))| \ll \sum_{t\leq B} \frac{B}{t}\cdot \log(B)^{\rho_{\Delta,f}-2-1/5} \ll  B\log(B)^{\rho_{\Delta,f}-1-1/5}.$$
We remark that this approach was given in \cite{HBclassnum} by Heath-Brown and our approach is strongly inspired by his work. We also note that our savings here is large enough to overcome the fact that we have moved from summing over $(u,v)\in (B/t)^{1/2}\calB(\bx_0,L)$ to $|u|,|v|\leq (B/t)^{1/2}$ as long as $w(B) \leq \log(B)^{-1/49}.$

We would like to highlight that this approach to the Gr\"ossencharacter case can also be used to approach (\ref{eq: outline cusp}) in most cases. In particular, we only require that the base change of $\xi^2$ to $K_f$ to be a nontrivial Hecke character to achieve a small polylogarithmic savings; this lift will be acceptable unless $K_f$ is not independent from the Hilbert class field of $\Q(\sqrt{-\Delta}).$

Finally, we comment on how this technique can be applied to a wider class of cusp forms. For instance, let $C(z)\in S_1(\Gamma_0(-\Delta),\chi)$ be a cusp form and consider the sum $$\sum_{|u|,|v|\leq X} |\lambda_C(t^2F(u,v))| \ll \frac{B}{t} \cdot \prod_{p\ll B/t} \left(1+ \frac{(|\lambda_C(p)|-1)\cdot \#\{x\bmod p: f(x) \equiv 0 \bmod p\}}{p}\right).$$
We have again applied Nair's sieve here. We can replace (\ref{eq: outline cosine}) with Hecke relations: 
$$|\lambda_C(p)|-1 \leq -\frac{1}{18}+\frac{4}{9}\lambda_{\sym^2(C)}(p) - \frac{1}{18}\lambda_{\sym^4(C)}(p).$$
We can then proceed with the argument above under the assumption of certain analytic properties of the base change of $\sym^2(C)$ and $\sym^4(C)$ to $K_f$. In summary, the savings in this bound, despite the addition of absolute values around the Fourier coefficients, comes from the early Sato-Tate moments of $C(z)$ and its base change to $K_f$. This idea will be expanded upon in forthcoming work. 
\\

Finally, we return to the Eisenstein contribution, i.e. the main term and how we determine when $c_{\Delta,f}$ is zero. Let us define $\xi(s;f,\chi)$ to be the following Dirichlet series for $\Re(s)>1:$
$$\xi(s;f,\chi) = \sum_{d_i} \frac{\chi(d_1\cdots d_r)\#\{x\bmod d_1\cdots d_r: d_i\mid f_i(x)\}}{(d_1\cdots d_r)^s}.$$
For every irreducible factor $f_i(z)$ of $f(z)$, let us take $K_i = \Q[z]/f_i(z)$. Then we will see that $$\xi(s;f,\chi) = A_{\Delta,f}(s) \cdot \prod_{i=1}^r L_{K_i}(s,\chi), \hspace{1cm} A_{\Delta,f}(1)\neq 0$$
where $L_{K_i}(s,\chi)$ is the quadratic twist of the Dedekind zeta function $\zeta_{K_i}(s)$ by $\chi$ and $A_{\Delta,f}(s)$ converges for $\Re(s)>3/4$. We will see that the order of the pole at $s=1$ of $L_{K_i}(s,\chi)$ is one if $\sqrt{-\Delta}\in K_i$ and zero otherwise. As a consequence, $\xi(s;f,\chi)$ has a pole of order $\rho_{\Delta,f}-2$ at $s=1$. 

Since $\xi(s;f,\chi)$ is intimately tied to our sum over $d_i\leq D_i$ in our main term (\ref{eq: outline main term}), we find that (\ref{eq: outline main term}) can be re-written as
\begin{multline*}BL^2\log(B)\sum_{(c_1,...,c_r)}' \mu(c_1\cdots c_r)\chi(c_1\cdots c_r)  \sum_{(a_i)\bmod \Delta} \prod_{i=1}^r (1+\chi(a_i)) \#\{\bx\bmod \Delta: F_i(\bx)\equiv c_ia_i\bmod \Delta\} \\   \times \Res_{s=1} \xi(s;f,\chi) \log(B)^{\rho_{\Delta,f}-2} \times \prod_{p\mid c_1\cdots c_r }\mathfrak{S}_p,
\end{multline*}
where $\mathfrak{S}_p$ are local factors. It should be noted that $c_1\cdots c_r$ itself must divide the product of resultants $\prod_{i\neq j} \Res(f_i,f_j)$ and so this is a finite product. 

In \S\ref{sec: constant}, we will show that if $$\sum_{a_i\bmod \Delta} \prod_{i=1}^r (1+\chi(a_i))\cdot  \#\{\bx\bmod \Delta: F_i(\bx)\equiv c_ia_i \bmod \Delta\}=0$$
then for some $p\mid \Delta$, $X^*_{\Delta,f_1,...,f_r,c_1,...,c_r}(\Q_p)=\emptyset$ (the expression above is an oversimplification, but we later notate the more complicated relevant expression as $\BM(\bc)$). We remark that the true (more complicated) computation involves a sum over genus characters as well. After some manipulations, we further see that if the contribution from a particular tuple $(c_1,\dots,c_r)$ to the sum above is zero, then for some $p\mid c_1\dots c_r\Delta$, we indeed have that $X^*_{\Delta,f_1,\dots,f_r,c_1,\dots,c_r}(\Q_p)=\emptyset.$
 Moreover, we show that these are all of the possible local obstructions to $X^*_{\Delta,f_1,...,f_r,c_1,...,c_r}$ and that the finite set of $(c_1,...,c_r)$ in the restricted sum are the only ones necessary to consider. 

Consequently, we have the following statement: \textit{if our leading constant $c_{\Delta,f}$ is zero then for every vector $(c_1,...,c_r)$, there is a local obstruction to rational points on the torsor $X^*_{\Delta,f_1,...,f_r,c_1,...,c_r}$.}
Finally, the results of Colliot-Th\'el\`ene, Coray, and Sansuc \cite{ColliotTheleneCoraySansuc,CTSSchinzel} and Harari's ``formal lemma''\cite{Harari} (see Theorem \ref{thm: descent}) allows us to interpret this result in terms of the Brauer-Manin obstruction; specifically, if the leading constant $c_{\Delta,f}=0$ then there is either a Brauer-Manin obstruction or local obstruction to rational points on $X_{\Delta,f}$. We remark that in \S\ref{subsec: constant irred case} we discuss in detail the case when $f(z)$ is irreducible and there is only one auxiliary variety -- $X_{\Delta,f}$ itself; in this case, our leading constant is a single product. In Appendix \ref{sec: appendix peyre}, we show that in the case that $\Delta>0$ is prime and $f(z)$ is irreducible over $\Q(\sqrt{-\Delta})$, this product is the expected product of local densities.

\section{Background and notation}\label{sec: background notation}

\subsection{Notation}\label{subsec: notation}
First, let us establish the notation that we will use for the remainder of the paper. 
\begin{itemize}
    \item $f$ will always denote a single variable polynomial and $F$ a binary form,
    \item $f(z) = f_1...f_r(z)$ is the decomposition of $f(z)$ into irreducible polynomials,
    \item $d\sim D$ if $d\in [D,2D]$,
    \item $e(\alpha) = \exp(2\pi i \alpha)$,
    \item $e_p(\alpha) = \exp(2\pi i\alpha/p),$
    \item $\log_k(x) = \log\log\log...\log(x)$ where there are $k$-iterated logarithms,
    \item If $\bv = (v_i)$ denotes a vector, then the lower-case unbolded letter $v$ denotes the product of the entries of $\bv$,  
    \item $K=\Q(\sqrt{-\Delta})$ is the quadratic field extension generated by $x^2+\Delta y^2$,
    \item $N(x+y\sqrt{-\Delta }) = x^2 + \Delta y^2$,
    \item $\calO_K$ is the ring of integers of $K$,
    \item $I$,$\fraka,\frakp,\ft,\fk$ will denote ideals of $\calO_K$,
    \item $\calU$ will denote the ring of units inside of $\calO_K$,
    \item $C$ will denote the class group of $K$,
    \item $\hat{C}$ will denote the group of class group characters of $K$, which can also be viewed as Hecke characters of $K$, 
    \item $\chi$ will denote the real quadratic character such that $\zeta_K(s) = \zeta(s)L(s,\chi)$, and we will denote the modulus of the character as $-\Delta;$ in other words, 
    \begin{equation*}
        -\Delta := \begin{cases}
            4\Delta, & \Delta >0, \\
            -\Delta, & \Delta <0,
        \end{cases}
    \end{equation*}
    \item We define the $\Delta$-part of an integer $n$ as: $$p_{-\Delta}(n) = \prod_{\substack{p\textrm{ prime}\\ p^e\| n\\ p\mid -\Delta}}p^e.$$
    We also define the part of $n$ coprime to $\Delta$ as $$p_{\neg -\Delta}(n) = n/p_{-\Delta}(n) = \prod_{\substack{p\textrm{ prime}\\ p^e \| n \\ p\nmid -\Delta}} p^e,$$
    \item $K_{f_i}$ denotes splitting field of $f_i$ over $K$ and $K_f=\prod_{i=1}^r K_{f_i}$,
    \item $\frakP$ will denote a prime ideal on $\calO_{K_f}$,
    \item For a binary form $F(u,v)$, we denote $$\varrho_F(n) := \#\{(u,v)\bmod n: F(u,v)\equiv 0 \bmod n\}$$
    and the related count
    $$\varrho_F^*(n):= \#\{(u,v)\bmod n: \gcd(u,v,n)=1, F(u,v) \equiv 0 \bmod n\}.$$
    We also consider the single-variable version of this local count: 
    $$\varrho_f(n) := \#\{x\bmod n: f(x)\equiv 0 \bmod n\}.$$
    \item We define the Dirichlet series for $\Re(s)>1$: 
    \begin{equation}\label{eq: def of xi(s;f,chi)}\xi(s;f,\chi) = \prod_{p} \left(1+\frac{\chi(p)\varrho_f(p)}{p^{s}}\right).\end{equation}
\end{itemize}
We also introduce some notation specific to the indefinite case, i.e. when $\Delta<0$. In this case, $K$ is a real quadratic field. 
\begin{defn}
    Let $(x_0,y_0)$ denote the fundamental solution to Pell's equation $x^2+\Delta y^2 = 1$. Then we write $\varepsilon = x_0 + y_0\sqrt{-\Delta}$ as the \textbf{fundamental unit}. Every element of $u\in \calU\subset \calO_K$ will satisfy that for some $k\in \Z$, $$u = \varepsilon^k.$$
\end{defn}

Next, let us fix $(x,y)\in \Z$. Then we define two new variables: 
$$\alpha = x+y\sqrt{-\Delta}, \textrm{  } \overline{\alpha} = x-y\sqrt{-\Delta}.$$
Then we say that $(x,y)$ (or $\alpha$) is \textbf{minimal} if $x\geq 0$ and $$1\leq \left|\frac{\alpha}{\overline{\alpha}}\right|\leq |\varepsilon^2|.$$
We note that for any principal ideal $I$, there will be an unique minimal generator $(x,y)$ such that $(x+y\sqrt{-\Delta}) = (\alpha) = I$. We denote the minimal element of $\calO_K$ equivalent to $(x,y)$ up to a unit as $(x,y) \bmod \varepsilon$. We may also denote this element as $\alpha \bmod\varepsilon.$\\ 

Finally, let us define our Hecke Gr\"ossencharacter $\Psi$ that will detect the ``argument'' of $\alpha$, as alluded to in the outline. This definition is inspired by the choice of Gr\"ossencharacter taken by Heath-Brown in \cite{HBclassnum}.
\begin{defn}\label{def: fundamental character}
    Define the Hecke character $\Psi$ on principal ideals $I=(\alpha)$ to be:
    $$\Psi(I) = \textrm{sgn}(\alpha\overline{\alpha}) \exp\left(\pi i \cdot \frac{\log|\alpha| - \log|\overline{\alpha}|}{\log|\varepsilon|}\right).$$
    $\Psi$ can then be lifted trivially to the set of all ideals of $\calO_K.$
\end{defn}
\begin{remark}
    Note that $\Psi$ is indeed well-defined on ideals as the following holds: $$\log |\varepsilon\alpha|-\log |\overline{\varepsilon\alpha}| = 2\log|\varepsilon| + \log|\alpha|-\log|\overline{\alpha}|.$$
    Additionally, we know that $|\Psi(\fa)|=1$ for any ideal $\fa\subset \calO_K,$ so $\Psi$ is a nontrivial unitary Hecke Gr\"ossencharacter.
\end{remark}

\subsection{Hooley's $\Delta$-function}
Next, we recall the definition and state our notation for the following arithmetic functions: 
\begin{defn}
    The divisor function will be denoted as $$\tau(n) := \sum_{d\mid n} 1.$$
\end{defn}
\begin{defn}\label{def: Hooley}
    Hooley's $\Delta$-function is defined as $$\Delta(n) := \sup_{\substack{D\in \R, 0\leq v\leq 1}} \sum_{\substack{d\mid n\\ d\in [D,e^v D]}} 1.$$
\end{defn}
We note that the average of $\tau(n)$ satisfies $$\sum_{n\leq X} \tau(n) \sim X\log(X).$$
On the other hand, the average of the $\Delta$-function has a small saving: 
$$\sum_{n\leq X} \Delta(n) \ll X\log(X)^{o(1)}.$$
This feature is particularly useful in our level of distribution result proved in \S\ref{sec: LoD}, since we will need to be careful about tracking any and all powers of $\log(X)$. We will use the following case of a more general theorem proved by de la Bret\`eche and Tenenbaum in \cite{delaTenenbaumHooley24}.
\begin{thm}[de la Bret\`eche and Tenenbaum, {\cite[(1.7)]{delaTenenbaumHooley24}}]\label{thm: hooley untwisted}
    Let $f\in \Z[x]$ be an irreducible polynomial.
    Then the following bound holds: 
    $$\sum_{n\leq X} \Delta(n) \varrho_f(n) \ll_f X\log\log(X)^{5/2}.$$
\end{thm}

Next, we introduce the twisted Hooley $\Delta$-function by the quadratic character $\chi$. 
\begin{defn}\label{def: hooley twisted}
    The Hooley $\Delta$-function twisted by $\chi$ is defined as $$\Delta(n,\chi) := \sup_{\substack{D\in \R \\ 0\leq v\leq 1}} \left|\sum_{\substack{d\mid n\\ d\in [D,e^v D]}} \chi(d)\right|.$$
    Observe that $\Delta(n,\chi)\leq \Delta(n)\leq \tau(n)$ for all $n$. 
\end{defn}
\begin{remark}\label{rem: relate Delta and dyadic}
    We also note that the following inequalities hold:
    $$\Delta(n) \geq \sup_{D\in \R} \left|\sum_{\substack{d\mid n\\ d\sim D}} 1\right| =:\Delta_2(n),$$
    $$\Delta(n,\chi) \geq  \sup_{D\in \R} \left|\sum_{\substack{d\mid n\\ d\sim D}} \chi(d)\right| =: \Delta_{2}(n,\chi).$$
    This falls from the fact that $2\leq e^{v}$ for some $0\leq v\leq 1.$
\end{remark}

In \cite{delaTenenbaum-delta}, de la Bret\`eche and Tenenbaum showed that the average of $\Delta(n,\chi)^2$ is smaller than the average of the divisor function. Moreover, this bound also holds for correlations of $\Delta(n,\chi)^2$ with the number of local solutions of a polynomial. 
\begin{thm}[de la Bret\`eche and Tenenbaum, {\cite[Theorem 1.1]{delaTenenbaum-delta}}]\label{thm: hooley character bound}
For $\chi$ a nonprincipal quadratic character mod $M$ and $f\in \Z[x]$ a polynomial irreducible over $\Q(\sqrt{-M})$. Then, the following bound holds for an explicit constant $c_{f,M}>0$:
$$\sum_{n\leq X}\Delta(n,\chi)^2 \varrho_f(n) \ll_{\chi,F} X \exp(c_{f,M}\sqrt{\log_2(X)\log_3(X)}).$$
\end{thm}

\subsection{A sieve on binary forms}
We will often need a sieve for averages of non-negative multiplicative functions along polynomial values. The concept for this sieve was first introduced by Nair in \cite{Nair} for the average of multiplicative functions along single-variable polynomials; later, this work was generalized by Nair and Tenenbaum in \cite{NairTenenbaum}. 

\begin{thm}[Nair \cite{Nair}]\label{thm: Nair single sieve}
Let $f\in \Z[x]$ be a squarefree polynomial with no fixed prime divisors. Let $g$ be a non-negative multiplicative function that is uniformly bounded on prime values. Then we have that $$\sum_{n\leq X} g(|f(n)|) \ll_g X \prod_{p\leq X} \left(1-\frac{\varrho_f(p)}{p}\right)\exp\left(\sum_{p\leq X} \frac{g(p)\varrho_f(p)}{p}\right).$$
\end{thm}

 This sieve was first generalized to binary forms by de la Bret\`eche and Browning  \cite{dlB-Browning-sieve} for irreducible forms. The most general version for binary forms was established by de la Bret\`eche and Tenenbaum in \cite{delaTenenbaum-sieve}, and this is the result that we shall use. 

Define $M_k(A,B,\epsilon)$ to be a class of sub-multiplicative functions in $k$ variables satisfying for $G\in M_k(A,B,\epsilon)$, $$G(a_1,...,a_k)\leq \min(A^{\Omega(a_1...a_k)},B(a_1...a_k)^\epsilon).$$
We note that $\tau(n),\Delta(n),\Delta(n,\chi)$ all belong to $M_1(2,B(x) = x,\epsilon).$

\begin{thm}[de la Bret\`eche and Tenenbaum, {\cite[Theorem 1.1]{delaTenenbaum-sieve}}]\label{thm: Nair sieve}
    Assume $G\in M_k(A,B,\epsilon)$ for a fixed $A,B\geq 1$. Let $F(x,y) = \prod_{i=1}^k F_i(x,y)$ be a squarefree product of primitive binary forms. Fix $\alpha>0$ and assume that $Y\geq X^\alpha$ and $x_0,y_0\in [X,2X]$. Then we have that $$\sum_{x,y\in [x_0,x_0+Y]\times [y_0,y_0+Y]} G(|F_1(x,y)|,...,|F_k(x,y)|) \ll_{A,B,\alpha} Y^2 E(X+Y) \prod_{p\leq X+Y}\left(1-\frac{\varrho_F(p)}{p^2}\right),$$
    where we define $$E(X+Y) = \sum_{s=s_1...s_k\leq X+Y} G(s_1,...,s_k) \frac{\#\{x,y \bmod s: s_i\mid F_i(x,y) \forall i\}}{s^2}.$$
\end{thm}
In particular, if $g(n)$ is a non-negative multiplicative function bounded by $\tau(n)$ and $F(x,y) = \prod_{i=1}^k F_i(x,y)$ is a squarefree primitive binary form, then we have that $$\sum_{x,y\leq X} \prod_{i=1}^k g(F_i(x,y)) \ll X^2 \prod_{i=1}^k\prod_{p\leq X} \left(1+\frac{\varrho_{F_i}(p)(g(p)-1)}{p^2}\right).$$
Sometimes, we will wish to apply the following corollary of Theorem \ref{thm: Nair sieve} where we specify certain divisors of $F(x,y)$:
\begin{corollary}\label{cor: dlB-Tenenbaum with divisor factor}
    Let $F(x,y)\in \Z[x,y]$ be an irreducible primitive binary form and $g\in M_1(A,B,\epsilon)$. We say that $g(t) = 0$ if $t\not\in \Z$. Then we have that $$\sum_{x,y\leq X} g(|F(x,y)/c|) \ll_{A,B} \frac{X^2(A^{\deg(F)}\cdot\deg(F))^{\Omega(c)}}{c}\prod_{p\ll X} \left(1+\frac{\varrho_F(p)(g(p)-1)}{p^2}\right).$$
\end{corollary}
\begin{proof}
We rearrange our sum as: 
$$\sum_{x,y\leq X} g(|F(x,y)/c|) = \sum_{\substack{\alpha\bmod c\\ c\mid F(\alpha)}} \sum_{x,y\ll X/c} g(|F_\alpha(x,y)/c|),$$
where we define $F_\alpha(x,y) = F(cx+\alpha_1,cy+\alpha_2).$ Since $c\mid F(\alpha)$, we have that $F_\alpha(x,y)/c$ is itself an irreducible binary form with integer coefficients. Moreover, $F_\alpha(x,y)/c = c^{e_\alpha} \tilde{F}_\alpha (x,y)$ for some primitive binary form $\tilde{F}_\alpha$ and power $e_\alpha\leq \deg(F)$. So, we use submultiplicativity and apply Theorem \ref{thm: Nair sieve} to the inside sums. Hence, we have the relation:
\begin{equation*}
    \sum_{x,y\ll X/c} g(|F_\alpha(x,y)/c|) \ll \frac{X^2A^{e_\alpha \Omega(c)}}{c^2} \prod_{p\ll X}\left(1+\frac{\varrho_{F_\alpha/c}(p)(g(p)-1)}{p^2}\right). 
\end{equation*}
Since $g\in M_1(A,B,\epsilon)$, we can see the contribution from the primes $p\mid c$ is bounded by $A^{\omega(c)}$. So, let us assume from now on that $p\nmid c$. Now, we consider 
$$\varrho_{F_\alpha/c}(p) = \#\{\bx\bmod p: F_\alpha(x,y)/c \equiv 0 \bmod p\} = \#\{\bx\bmod p: F(cx+\alpha_1,cy+\alpha_2)\equiv 0 \bmod p\}.$$
For $p\nmid c$, we know that $c$ is invertible in $\F_p$ and thus we get that $\varrho_{F_\alpha/c}(p) = \varrho_F(p).$ 

So, our original sum can be bounded by 
\begin{equation*}
    \sum_{x,y\leq X}g(|F(x,y)/c|) \ll X^2 \cdot \frac{A^{\deg(F)\cdot \Omega(c)}}{c^2} \cdot \prod_{p\ll X} \left(1+\frac{\varrho_F(p)(g(p)-1)}{p^2}\right)\sum_{\substack{\alpha \bmod c\\ c\mid F(\alpha)}} 1.
\end{equation*}
Finally, we recall that $\varrho_F(c) \leq c\deg(F)^{\omega(c)}$. So, we have our desired bound. 

\end{proof}

\subsection{Upper bounds on fixed sizes of divisors}
We introduce two lemmas that bound the contribution of large divisors of a polynomial. These results will be particularly useful in \S\ref{sec: large moduli}.

First, we write a result that bounds the number of integers with particularly large divisors; the following formulation was stated in \cite{HB-linear} by Heath-Brown.
\begin{lemma}[Heath-Brown, {\cite[Lemma 5.1]{HB-linear}}]\label{lem: small log saving }
    There exists an explicit constant $\eta>0$ such that $$\#\{m\ll X: \exists d\sim D, d\mid m\} \ll \frac{X}{\log(D)^{\eta}}.$$
    \end{lemma}

Next, we introduce a modification of the above lemma for binary forms that was proved by de la Bret\`eche and Tenenbaum in \cite{delaTenenbaum-Manin}. 
\begin{lemma}[de la Bret\`eche and Tenenbaum, {\cite[(7.41)]{delaTenenbaum-Manin}}]\label{lem: small log saving binary}
    Let $F(x,y)$ be a binary quartic. Then there exists an explicit constant $\eta>0$ such that $$\#\{|x|,|y|\leq X: \exists d\sim D: d\mid F(x,y)\} \ll X^2 \cdot \left(\frac{X^2/D + D/X^2 +\log\log(X)}{\log(D)^{\eta}}\right).$$
\end{lemma}
We remark that in both of the cases above $\eta$ is explicit and greater than 1/1000.

\subsection{Dirichlet series and Rankin-Selberg $L$-functions}\label{subsec: Dirichlet and Rankin selberg}
In this subsection, we will discuss the connection between $\xi(s;f,\chi)$, and a quadratic twist of a Dedekind zeta function. First, we want to establish another similar Dirichlet series that will appear in our computations later. Recall from the notation section that for a vector $\bd = (d_1,\hdots,d_r)$, we write the product of the entries as $d = d_1\cdots d_r.$ Let $\Re(s)>1$ and define the following
\begin{equation}\label{eq: def of xi(s;F,chi)}\xi(s;F,\chi) = \sum_{d_i=1}^\infty \frac{\chi(d) \varrho_F^*(\bd)}{d^{2s}}.\end{equation}
For vectors $(p_1,\hdots,p_r)$ with prime entries $p_i$, we know that $\varrho_F^*((p_i)) \leq 4^r$; so, $\xi(s;F,\chi)$ converges for all $\Re(s)>1$. We also note that we have an Euler product for $\xi(s;F,\chi)$: 
\begin{multline*}\prod_{p} \Bigg(1+\frac{\chi(p)}{p^s}\cdot (\varrho_F^*((p,1,\hdots,1))+\hdots+\varrho_F^*((1,\hdots,p))) \\ + \frac{\chi(p)^2}{p^{2s}}\cdot (\varrho_F^*((p,p,1,\hdots,1)) + \hdots + \varrho_F^*((1,\hdots,p,p)))+\hdots\Bigg).\end{multline*}
This expansion allows us to see that $$\xi(s;F,\chi) = S_{F,\chi}(s) \prod_{p} \left(1+\frac{\chi(p)}{p^s}\cdot (\varrho_F^*((p,1,\hdots,1))+\hdots+\varrho_F^*((1,\hdots,p)))\right),$$
where $S_{F,\chi}(s)$ converges for $\Re(s)>1/2.$ 

We note that at all but finitely many primes $p$, we have that $$\varrho_F^*((p,1,\hdots,1)) + \hdots + \varrho_F^*((1,\hdots,p)) = \varrho_f(p) = \sum_{i=1}^r \varrho_{f_i}(p).$$
Here we recall from \S\ref{subsec: notation} that $\varrho_f(p) = \#\{x \in \F_p: f(x) = 0 \bmod p\}$.
Thus, we get that $$\xi(s;F,\chi) = R_{f,\chi}(s) S_{F,\chi}(s)S_{f,\chi}(s)\prod_{i=1}^r \xi(s;f_i,\chi),$$
where we have that $R_{f,\chi}(s)$ a product over finitely many primes $p\mid F(1,0)$ and $S_{f,\chi}(s)$ converges for $\Re(s)>1/2.$ Now, this tells us that $\xi(s;F,\chi)$ has a pole at $s=1$ if and only if $\xi(s;f,\chi)$ has a pole at $s=1$. The pole at $s=1$ of both of these functions will be determined by the pole of a quadratic twist of a Dedekind zeta function. \\ 

Define the field $L_{f_i} = \Q[z]/f_i(z)$. Let us consider the $\zeta$-function of $L_{f_i}$:
$$\zeta_{L_{f_i}}(s) = \prod_{\frakP\subset \calO_{L_{f_i}}} \left(1+N(\frakP)^{-s} + N(\frakP)^{-2s}+\hdots \right).$$
It is known that $\zeta_{L_{f_i}}(s)$ has a meromorphic continuation to $\C$, a functional equation, and a simple pole at $s=1$. Since $\deg(f_i)\leq 4$, we know that the Galois group of $L_{f_i}$ is solvable. Thus, by the resolution of Dedekind's conjecture for solvable groups by Uchida \cite{Uchida} and van der Waall \cite{vanderWaall}, we have that $$\zeta_{L_{f_i}}(s) = \zeta(s) \prod_{\pi_j} L(s,\pi_j),$$
where $\pi_j$ are nontrivial irreducible representations and $L(s,\pi_j)$ are entire. 

Next, we can view the Dirichlet $L$-function $L(s,\chi)$ as a $L$-function of a nontrivial irreducible representation, which we denote by $\chi$ as well. Then the quadratic twist of $\zeta_{L_{f_i}}(s)$ by $\chi$ is given by: 
$$L(s; f_i\times \chi) = L(s,\chi) \prod_{\pi_j} L(s,\pi_j\times \chi).$$ We then have the following Dirichlet series and Euler product expansions for $\Re(s)>1:$
$$L(s;f_i\times \chi) = \sum_{\mathfrak{A}\subset \calO_{L_{f_i}}} \chi(N(\mathfrak{A}))N(\mathfrak{A})^{-s} = \prod_{\frakP} \left(1+\chi(N(\frakP))N(\frakP)^{-s}+ \chi(N(\frakP))^{2}N(\frakP)^{-2s}+\hdots \right).$$
Note that since $L_{f_i}=\Q[z]/f_i(z)$, the following relation is true: 
$$\varrho_{f_i}(p) = \#\{\frakP\subset \calO_{L_{f_i}}: N(\frakP) = p\}.$$
This tells us that $$\xi(s;f_i,\chi) = S_{L_{f_i},\chi}(s) L(s;f_i\times \chi),$$
where $S_{L_{f_i},\chi}(s)$ converges for $\Re(s)>1/2.$ Thus, we get the following lemma:
\begin{lemma}\label{lem: poles have same order}
    The Dirichlet series $\xi(s;F,\chi)$ and $\xi(s;f,\chi)$ have meromorphic continuations to $\Re(s)>1/2.$ Additionally, assume that $L(s;f_i\times \chi)$ has a pole of order $m_i$ at $s=1$. Then the order of the pole at $s=1$ of both $\xi(s;F,\chi)$ and $\xi(s;f,\chi)$ is given by $m_1+\hdots + m_r$.
\end{lemma}
\begin{corollary}\label{cor: order of pole is varrho}
    The order of the pole of $\xi(s;F,\chi)$ and $\xi(s;f,\chi)$ at $s=1$ is $\rho_{\Delta,f}-2,$ where $\rho_{\Delta,f}$ is defined in Conjecture \ref{conj: manin}.
\end{corollary}
\begin{proof}
    It suffices to show that the order of the pole at $s=1$ of $L(s;f_i\times \chi)$ is given by 
    $$\begin{cases}
        1, & \sqrt{-\Delta} \in L_{f_i} \\ 
        0, & \textrm{else.}
    \end{cases}$$
    Since $\chi$ is nontrivial and $\pi_j$ are nontrivial irreducible representations, $L(s;f_i\times \chi)$ has a pole if and only if $\pi_j = \chi$ for some $j$. This phenomenon occurs if and only if $\sqrt{-\Delta}\in \Q[z]/(f_i(z))= L_{f_i}.$
\end{proof}

\subsection{Theta series and class group characters}\label{subsec: class group}
In this subsection, we will discuss the connection between this problem and modular forms, as alluded to in \S\ref{sec: outline}. Note that if $C$ is the class group of $K$, then 
\begin{equation}\label{eq: class group principal} \mathbf{1}_{\fa \textrm{ is principal}} = \frac{1}{|C|} \sum_{\psi \in \hat{C}} \psi(\fa).\end{equation}
Consequently, we get that $$\sum_{N(\fa) = n} \mathbf{1}_{\fa\textrm{ is principal}} = \frac{1}{|C|}\sum_{\psi\in \hat{C}} \sum_{N(\fa)=n} \psi(\fa).$$

On the other hand, consider the following function: 
$$\theta(z) = \sum_{\substack{N(\fa)=n\\ \fa \textrm{ is principal}}} e(z \cdot N(\fa)).$$
As commented in \cite[Equation (22.55)]{IwaniecKowalski}, for $\Delta>0$, $\theta(z)$ will be a Mellin transform of an Epstein zeta function and will be a modular form of weight one in $M_1(\Gamma_0(-\Delta),\chi).$ In fact, if $\Delta>0$, then this is a recognizable theta series: 
$$\theta(z) = \sum_{x,y} e(z\cdot (x^2+\Delta y^2)).$$

Since $\theta(z)\in M_1(\Gamma_0(-\Delta),\chi)$, we can decompose it into its Eisenstein and cuspidal part. We present the following explicit decomposition of the Fourier coefficients, as shown by Blomer in \cite[Lemma 2.2]{Blomer} for $\Delta>0$. This decomposition is also discussed by Fomenko in \cite{FomenkoQuad}.

\begin{lemma}\label{lem: decomp eisenstein cusp}
    Let $\theta(z) = E(z) + C(z)$, where $E(z)$ is an Eisenstein series and $C(z)$ is a cusp form. Let $\lambda_E(n)$ (resp. $\lambda_C(n)$) denote the Fourier coefficients of $E(z)$ (resp. $C(z)$). Then we have that $$\lambda_E(n) = \frac{1}{|C|}\sum_{\ord(\psi)\leq 2} \sum_{N(\fa)=n} \psi(\fa),$$
    $$\lambda_C(n) = \frac{1}{|C|}\sum_{\ord(\psi)\geq 3} \sum_{N(\fa)=n} \psi(\fa).$$
\end{lemma}

Additionally, Iwaniec discusses in \cite[\S12.5]{IwaniecTopics} how to express the Eisenstein term in even more detail when $K$ is a quadratic field. Because we are working with \textit{binary quadratic forms}, the Eisenstein component contains not only the trivial component, but the contribution from all of the genus characters. Furthermore, we can describe these genus characters explicitly. Let $q_1q_2=-\Delta$ and $\gcd(q_1,q_2)=1$. Then all of our genus characters are of the form:
\begin{equation}\label{eq: genus characters}
    \psi(\fp) = \chi_{q_1,q_2}(\fp) = \begin{cases}
        \chi_{q_1}(N(\fp)), & \fp\nmid q_1, \\ \chi_{q_2}(N(\fp)), & \fp \nmid q_2.
    \end{cases}
\end{equation}
These correspond to a basis of Eisenstein series with Fourier coefficients given by $$\varepsilon_{q_1,q_2}(n) = \sum_{d\mid n} \chi_{q_1}(d) \chi_{q_2}(n/d),$$
with $\chi_{q_1},\chi_{q_2}$ real quadratic characters satisfying $\chi_{q_1}\chi_{q_2}=\chi$ when $\gcd(n,-\Delta)=1.$ Note that the trivial character corresponds to $q_1=1$ and $q_2=-\Delta.$
Thus, we get that \begin{equation}\label{eq: restricted sum for genus characters}\lambda_E(n) = \frac{1}{|C|} \sum_{q_1q_2=-\Delta}' \varepsilon_{q_1,q_2}(n).\end{equation}
Here the restricted sum is due to the fact that while all genus characters are of the form $\chi_{q_1,q_2}(\fp)$, not all decompositions $q_1q_2=-\Delta$ with $\gcd(q_1,q_2)=1$ correspond to genus characters. 

In particular, let $$|\Delta| = \prod_{i=1}^r p_i.$$
Then the genus field is given by $$K(\sqrt{p_1^*},\hdots,\sqrt{p_r^*})$$
where if $p_i$ is an odd prime then 
$$p_i^* = \begin{cases}
    p, & p\equiv 1 \bmod 4, \\ 
    -p, & p\equiv 3 \bmod 4,
\end{cases}$$
and when $p_i=2$, 
$$2^* = \begin{cases}
    1, & \Delta\equiv 1 \bmod 4, \\
    -4, & \Delta\equiv 3 \bmod 4, \\
    8, & \Delta\equiv 2 \bmod 8, \\
    -8, & \Delta \equiv -2\bmod 8.
\end{cases}$$
The genus characters included in the restricted sum for $\lambda_E(n)$ are therefore those that correspond to characters of the genus field. 
We will apply the above decompositions in \S\ref{sec: reductions}.

\section{Upper bounds}\label{sec: upper bounds}
In this section, we will establish some upper bounds for $N(X_{\Delta,f},B)$ that are necessary for our proof. First, we recall that an upper bound of the correct order of magnitude was established by Browning \cite{BrowningLinear} for the positive-definite case. 

\begin{thm}[Browning, \cite{BrowningLinear}]\label{thm: Browning upp bound}
    Assume that $\Delta>0$ is a nonsquare integer. Then as $B\rightarrow\infty$, we have that $$N(X_{\Delta,f},B) \ll B\log(B)^{\rho_{\Delta,f}-1}.$$
\end{thm}

We will need a similar upper bound for this counting function on short boxes $N(X_{\Delta,f},\calB(\bx_0,L),S^1)$; more explicitly, we will need upper bounds on \begin{multline*}S(X_{\Delta,f},\calB(\bx_0,L),S^1) := \frac{1}{2}\cdot \#\Big\{ ((x,y),(u,v),t)\in \Z^2\times \Z^2\times \Z_+: x^2+\Delta y^2 = t^2 F(u,v),\\ 0<t\leq B/\exp(\sqrt{\log(B)}), 
    (u,v)\in (B/t)^{1/2}\calB(\bx_0,L), \|x+y\sqrt{-\Delta}\|,\|x-y\sqrt{-\Delta}\|\leq \nu_{\Delta,f}B, \\  \gcd(x,y,t) = \gcd(u,v) = 1\Big\}.
\end{multline*}
In the above sum, we have already removed the large torsors -- this process will occur in \S\ref{subsec: large torsors} before our applications of the lemmas below.

\begin{lemma}\label{lem: upp bound pos def}
    Assume that $\Delta>0$ is a nonsquare integer. If $L\gg \exp(-(\log(B))^{1/3})$, then we have that $$S(X_{\Delta,f},\calB(\bx_0,L),S^1) \ll L^2 B\log(B)^{\rho_{\Delta,f}-1}$$ 
\end{lemma}
\begin{proof}
    We first expand the definition of $S(X_{\Delta,f},\calB(\bx_0,L),S^1)$:
    \begin{equation*}
        S(X_{\Delta,f},\calB(\bx_0,L),S^1) \ll \sum_{t\leq B} \sum_{\substack{(u,v)\in (B/t)^{1/2}\calB(\bx_0,L)}} \#\{x^2+\Delta y^2 = t^2 F(u,v), \gcd(x,y,t)=1\},
    \end{equation*}

    Our first step is to rewrite this inner count with a sum over ideals. Consider the principal ideal $I = (x+y\sqrt{-\Delta})$. Since $\gcd(x,y,t)=1$ and $N(I) = t^2F(u,v)$, we know that for any prime $p$ dividing $t$, we must have that $(p) = \frakp \overline{\frakp}$ splits in $K$ and that $(p)$ does not divide $I$. As a result, we can write $$I = \ft^2 \fraka$$
where $N(\ft)=t$, $N(\fraka) = F(u,v)$, and $\gcd(\ft,\overline{\ft}) =  1.$ Additionally, we have that $\gcd(t,I,\overline{I})=1,$ which occurs if and only if $\gcd(\fa,\overline{\ft})=1.$ Here we write that $\gcd(\fraka,\fb)=1$ if $\fraka$ and $\fb$ do not share any common prime ideal factors. We point out that from now on we use the convention that $N(\ft) = t$.

So, we can bound the above sum by: 
    $$S(X_{\Delta ,f},\calB(\bx_0,L),S^1) \ll \sum_{N(\ft)\leq B} \sum_{(u,v)\in (B/t)^{1/2}\calB(\bx_0,L)} \#\{\fraka: N(\fraka) = F(u,v)\}.$$
    Next, we can bound this sum of ideals as follows: $$\#\{\fraka: N(\fraka)=n\} \ll (1\star \chi)(n),$$
    and henceforth replace $\#\{\fraka:N(\fraka)=F(u,v)\}$ by $(1\star \chi)(F(u,v))$ in the sum above.

Note that $(1\star \chi)(n)$ is a non-negative multiplicative function that is bounded by $\tau(n)$.
    Applying Theorem \ref{thm: Nair sieve}, we get the following upper bound: 
    \begin{align*}
        S(X_{\Delta ,f},\calB(\bx_0,L),S^1) &\ll \sum_{N(\ft)\leq B}  \sum_{(u,v)\in (B/t)^{1/2} \calB(\bx_0,L)} (1\star \chi)(F(u,v)) \\
        &\ll \sum_{N(\ft)\leq B} \frac{1}{N(\ft)} \cdot B L^2 \prod_{p\ll B/t}\left(1+\frac{\varrho_f(p)\chi(p)}{p}\right).
    \end{align*}
    We recognize this product over primes as the factors in the Euler product of $\xi(s;f,\chi)$. Since $\xi(s;f,\chi)$ has a pole of order $\rho_{\Delta,f}-2$ at $s=1$ (Corollary \ref{cor: order of pole is varrho}), we achieve that 
    \begin{align*}
        S(X_{\Delta,f},\calB(\bx_0,L),S^1) &\ll L^2 B\sum_{N(\ft)\leq B} \frac{1}{t} \log(B/t)^{\rho_{\Delta,f}-2}\\
        &\ll L^2 \cdot B\log(B)^{\rho_{\Delta,f}-1}. 
    \end{align*}
\end{proof}

For the indefinite case, we can similarly achieve an optimal upper bound. 
\begin{lemma}\label{lem: upp bound indefinite}
    Assume that $\Delta<0$ is an integer. If $L\gg \exp(-(\log(B))^{1/3})$, then we have that 
    $$S(X_{\Delta,f},\calB(\bx_0,L),S^1) \ll \left(\left|\frac{\log|F(\bx_0)|}{\log|\varepsilon|}\right|+O(1)\right)L^2 B\log(B)^{\rho_{\Delta,f}-1}.$$
\end{lemma}
\begin{proof}
    We first expand the definition of $S(X_{\Delta,f},\calB(\bx_0,L),S^1):$
    \begin{equation*}
        S(X_{\Delta,f},\calB(\bx_0,L),S^1) \ll \sum_{t\leq B} \sum_{(u,v)\in (B/t)^{1/2}\calB(\bx_0,L)} \sum_{\substack{x+y\sqrt{-\Delta} \bmod \varepsilon \\ x^2+\Delta y^2 = t^2F(u,v)}} \#\{k\in \Z: |\varepsilon^k \alpha|, |\overline{\varepsilon}^k\overline{\alpha}|\leq B\}.
    \end{equation*}
    Consider the count of units $$\#\{k\in \Z: |\varepsilon^k \alpha|, |\overline{\varepsilon}^k\overline{\alpha}|\leq B\}.$$
    We can see that $$|\varepsilon^k\alpha|\leq B \iff k\leq  \frac{\log(B)}{\log|\varepsilon|}-\frac{\log|\alpha|}{\log|\varepsilon|}.$$
    Since $\alpha$ is minimal, we have by definition that $$1\leq \frac{|\alpha|}{|\overline{\alpha}|} = \frac{|\alpha|^2}{\alpha\overline{\alpha}} \leq |\varepsilon|^2.$$
    Because $\alpha\overline{\alpha}=t^2F(u,v)$, we can see that $$\frac{\log|\alpha|}{\log|\varepsilon|} = \frac{\log|t^2 F(u,v)|}{2\log|\varepsilon|}+O(1).$$
    On the other hand, since $(u,v)\in (B/t)^{1/2}\calB(\bx_0,L)$, we have that $$\frac{\log|t^2F(u,v)|}{2\log|\varepsilon|} = \frac{\log(B)}{\log|\varepsilon|}+ \frac{\log|F(\bx_0)|}{2\log|\varepsilon|}+O(L).$$
    Thus, we have that $$|\varepsilon^k\alpha|\leq B \iff k\leq \left|\frac{\log|F(\bx_0)|}{2\log|\varepsilon|}\right| +O(1).$$
    A similar computation with $|\overline{\varepsilon}^k\overline{\alpha}|$ reveals an analogous lower bound. 
    So, we have that $$\#\{k\in \Z: |\varepsilon^k\alpha|,|\overline{\varepsilon}^k\overline{\alpha}|\leq B\} = \left|\frac{\log|F(\bx_0)|}{\log|\varepsilon|}\right|+O(1).$$
    
    Plugging this in above, we can see the following: 
    \begin{align*}
        S(X_{\Delta,f},&\calB(\bx_0,L),S^1) \ll \left(\frac{\log|F(\bx_0)|}{\log|\varepsilon|}+O(1)\right) \sum_{t\leq B} \sum_{\substack{(u,v)\in (B/t)^{1/2} \calB(\bx_0,L)}} \#\{\alpha \bmod \varepsilon: \alpha\overline{\alpha} = t^2F(u,v)\} \\
        &\ll \left(\frac{\log|F(\bx_0)|}{\log|\varepsilon|}+O(1)\right) \sum_{N(\ft)\leq B} \sum_{(u,v)\in (B/t)^{1/2}\calB(\bx_0,L)} \#\{\fraka\subset \calO_K: N(\fraka) = F(u,v)\}.
    \end{align*}
    Now, as in the positive-definite case, we can again use the following bound: 
    $$\#\{N(\fraka) = n\} \ll (1\star \chi)(n).$$
    So, we can apply Theorem \ref{thm: Nair sieve}:
    \begin{align*}
        S(X_{\Delta,f},\calB(\bx_0,L), S^1) &\ll \left(\frac{\log|F(\bx_0)|}{\log|\varepsilon|}+O(1)\right) \sum_{N(\ft)\leq B} \sum_{(u,v)\in (B/t)^{1/2}\calB(\bx_0,L)} (1\star \chi)(t^2F(u,v)) \\
        &\ll \left(\frac{\log|F(\bx_0)|}{\log|\varepsilon|}+O(1)\right) \sum_{N(\ft)\leq B} \frac{1}{N(\ft)}\cdot BL^2 \prod_{p\ll B/t} \left(1+\frac{\varrho_f(p)\chi(p)}{p}\right) \\
        &\ll \left(\frac{\log|F(\bx_0)|}{\log|\varepsilon|}+O(1)\right)L^2 B \sum_{N(\ft)\leq B} \frac{1}{N(\ft)}\log(B/t)^{\rho_{\Delta,f}-2} \\
        &\ll \left(\frac{\log|F(\bx_0)|}{\log|\varepsilon|}+O(1)\right)L^2 \cdot B \log(B)^{\rho_{\Delta,f}-1}.
    \end{align*}
    Again, we have used that the product over primes is related to $\xi(s;f,\chi)$, which has a pole of order $\rho_{\Delta,f}-2$ at $s=1$ (Corollary \ref{cor: order of pole is varrho}). 
    
\end{proof}

\begin{remark}
    Intuitively, we expect for a generic $\alpha$ that $$\#\{k\in \Z: |\varepsilon^k \alpha|,|\overline{\varepsilon}^k\overline{\alpha}|\leq B\} = O(1).$$
    However, we also expect that for certain $\alpha$, this quantity can be as large as $\log(B)$. This discrepancy is now hidden in the $\log|F(\bx_0)|/\log|\varepsilon|$ term, as $\bx_0$ ranges over values in $[-1,1]^2$. For most $\bx_0$, $|F(\bx_0)|$ is of approximately constant size, but for a small subset of $\bx_0$, the value $|F(\bx_0)|$ can be quite small. 
\end{remark}

\section{Initial reductions}\label{sec: reductions}
In \S\ref{sec: outline}, we discussed that a key ingredient in our proof is a restriction to short boxes and the use of Hecke Gr\"ossencharacters. In this section, we will demonstrate explicitly why these tools are crucial for our method. First, we recall some notation. Let $\bx_0\in [-1,1]^2$ and $L$ be length to be determined later. Recall that we denote: 
$$\calB(\bx_0,L) := \bx_0+[0,L]^2.$$

\subsection{Removal of large torsors}\label{subsec: large torsors}
We will first remove the contribution from the large torsors -- when $t\in [B/\exp(\sqrt{\log(B)}),B]$. 
\begin{lemma}\label{lem: large torsors}
    Assume that $\Delta$ is a squarefree integer and $\sqrt{-\Delta}\not\in \Q$. Define the contribution from large $t$ as: 
    \begin{multline}
    L(X_{\Delta,f},B) := \#\Big\{ ((x,y),(u,v),t)\in \Z^2\times \Z^2\times \Z_+: x^2+\Delta y^2 = t^2 F(u,v),\\ 
    |u|,|v|\in (B/t)^{1/2}, \|x+y\sqrt{-\Delta}\|,\|x-y\sqrt{-\Delta}\|\leq \nu_{\Delta,f} B, \\ t\in [B/\exp(\sqrt{\log(B)},B], \gcd(x,y,t) = \gcd(u,v) = 1\Big\}.
\end{multline}
Then we have the upper bound:
$$L(X_{\Delta,f},B) \ll B\log(B)^{\rho_{\Delta,f}-3/2}.$$
\end{lemma}
\begin{proof}
Fix a $t\in [B/\exp(\sqrt{\log(B)}),B]$ for now.
First, assume that $\Delta>0$; then, we have that 
\begin{multline*}\#\{x^2+\Delta y^2 = t^2F(u,v): |u|,|v|\leq (B/t)^{1/2}, \gcd(x,y,t)=\gcd(u,v)=1\} \\ \ll \sum_{|u|,|v|\leq (B/t)^{1/2}} \#\{\fa: N(\fa)=t^2 F(u,v)\} \ll (1\star \chi)(t)\sum_{|u|,|v|\leq (B/t)^{1/2}}\#\{\fa: N(\fa)= F(u,v)\}.\end{multline*}
The above comes from the initial analysis of the sum in Lemma \ref{lem: upp bound pos def}. Since $\#\{\fa:N(\fa)=n\} \ll (1\star \chi)(n)$ and this is a non-negative multiplicative function, we apply Theorem \ref{thm: Nair sieve} to achieve 
\begin{align*}
    \#\{x^2+\Delta y^2 = t^2F(u,v):& |u|,|v|\leq (B/t)^{1/2}, \gcd(x,y,t)=\gcd(u,v)=1\}\\ &\ll \frac{B\cdot (1\star\chi)(t)}{t} \prod_{p\ll B/t} \left(1+\frac{\varrho_f(p)\chi(p)}{p}\right) \ll \frac{B\cdot (1\star\chi)(t)}{t}\log(B/t)^{\rho_{\Delta,f}-2}.
\end{align*}

Next, consider when $\Delta<0$. We claim that the same upper bound holds for a fixed large $t$. In Lemma \ref{lem: upp bound indefinite}, we saw that \begin{multline*}\#\{x^2+\Delta y^2 = t^2F(u,v): |u|,|v|\leq (B/t)^{1/2}, \gcd(x,y,t)=\gcd(u,v)=1\} \\ \ll \sum_{|u|,|v|\leq (B/t)^{1/2}} \sum_{\substack{\alpha = x+y\sqrt{-\Delta}\bmod \varepsilon\\ x^2+\Delta y^2 = t^2F(u,v)}} \#\{k\in \Z: |\varepsilon^k\alpha|,|\overline{\varepsilon}^k\overline{\alpha}|\leq B\}
\\ \ll (1\star\chi)(t) \sum_{|u|,|v|\leq (B/t)^{1/2}} \left(\frac{\log (B)}{\log|\varepsilon|}-\frac{\log|t^2F(u,v)|}{2\log|\varepsilon|}+O(1)\right)\#\{N(\fa)= F(u,v)\}\\
\ll \frac{B\cdot (1\star \chi)(t)}{t}\log(B/t)^{\rho_{\Delta,f}-2} + \sum_{|u|,|v|\leq (B/t)^{1/2}} \left(\frac{\log (B)}{\log|\varepsilon|}-\frac{\log|t^2F(u,v)|}{2\log|\varepsilon|}\right)\#\{N(\fa)=  F(u,v)\}.\end{multline*}
Since $|u|,|v|\leq (B/t)^{1/2}$, we have that $t^2F(u,v) = B^2 F(\by)$ for some $\by\in [-1,1]^2$. So, we get that $$\frac{\log(B)}{\log|\varepsilon|}-\frac{\log|t^2F(u,v)|}{2\log|\varepsilon|} = -\frac{\log|F(\by)|}{2\log|\varepsilon|}.$$
Assume that $1\geq L_t\geq (t/B)^{1/3}$ and divide $[-1,1]^2$ into boxes of sidelength $L_t$. Then for $\by\in \calB(\bx_0,L_t)$, we have that $$-\frac{\log|F(\by)|}{2\log|\varepsilon|} = -\frac{\log|F(\bx_0)|}{2\log|\varepsilon|}+O(L_t).$$
Then if we sum over $(u,v)\in (B/t)^{1/2}\calB(\bx_0,L_t)$: \begin{multline*}\left(-\frac{\log|F(\bx_0)|}{\log|\varepsilon|} +O(1)\right)\sum_{(u,v)\in (B/t)^{1/2}\calB(\bx_0,L_t)} \#\{N(\fa)= F(u,v)\} \\ \ll \left(-\frac{\log|F(\bx_0)|}{\log|\varepsilon|}+O(1)\right) \cdot \frac{BL_t^2\cdot (1\star\chi)(t)}{t}\log(B/t)^{\rho_{\Delta,f}-2}.\end{multline*}
Here we have used that $\#\{N(\fraka)=F(u,v)\}\ll (1\star \chi)(F(u,v))$ and Theorem \ref{thm: Nair sieve}.

For our final step in the indefinite case, we sum over $\bx_0\in [-1,1]^2$ spaced $L_t$ apart. We can see that 
\begin{multline*}
    \#\{x^2+\Delta y^2 = t^2F(u,v): |u|,|v|\leq (B/t)^{1/2}, \gcd(x,y,t)=\gcd(u,v)=1\} \\ \ll \frac{B \cdot (1\star \chi)(t)}{t}\log(B/t)^{\rho_{\Delta,f}-2} + \frac{B\cdot (1\star \chi)(t)}{t}\log(B/t)^{\rho_{\Delta,f}-2} \sum_{\bx_0}  \left(-\frac{\log|F(\bx_0)|}{2\log|\varepsilon|}+O(L_t)\right) L_t^2 \\
    \ll \frac{B\cdot (1\star \chi)(t)}{t}\log(B/t)^{\rho_{\Delta,f}-2} \left(1+\int_{[-1,1]^2} -\frac{\log|F(\by)|}{2\log|\varepsilon|} d\by + O(L_t)\right).
\end{multline*}
Since this integral converges and $L_t \in [(t/B)^{1/3},1]$, we derive that when $\Delta<0$: \begin{multline*}\#\{x^2+\Delta y^2 = t^2F(u,v): |u|,|v|\leq (B/t)^{1/2}, \gcd(x,y,t)=\gcd(u,v)=1\} \\ \ll \frac{B\cdot(1\star \chi)(t)}{t}\log(B/t)^{\rho_{\Delta,f}-2}.
\end{multline*}
Thus, this bound holds independent of the sign of $\Delta.$

Finally, we sum over the values of $t$, achieving that
\begin{align*}\sum_{t\in [B/\exp(\sqrt{\log(B)}),B]} &\#\{x^2+\Delta y^2 = t^2F(u,v): |u|,|v|\leq (B/t)^{1/2}, \gcd(x,y,t)=\gcd(u,v)=1\} \\ &\ll  \sum_{t\in [B/\exp(\sqrt{\log(B)}),B]} \frac{B\cdot (1\star \chi)(t)}{t}\log(B/t)^{\rho_{\Delta,f}-2} \\ & \ll \sum_{N(\ft)\in [B/\exp(\sqrt{\log(B)}),B]} \frac{B}{N(\ft)} \ll B\log(B)^{\rho_{\Delta,f}-3/2}.
\end{align*}
This completes the proof of this lemma.
\end{proof}
Comparing this to Theorem \ref{thm: Manin}, it is clear that this is an error term; hence, it is safe to \textbf{assume from here on that $t\leq B/\exp(\sqrt{\log(B)})$.}

\subsection{Short boxes and intervals}\label{subsec: short boxes}
From now on, we take $t\leq B/\exp(\sqrt{\log(B)})$ and fix a short box and short interval. Let us recall for $\theta\in [0,1]$ and $\ell\leq 1$ a length to be determined later, we defined the interval: 
$$\calI(\theta,\ell) := \{e(\theta'): \theta' \in [\theta,\theta+\ell]\}\subset S^1.$$
Let us define the following point-counting functions:
if $\Delta>0$, we write
\begin{multline}\label{eq: pos def point counting defn}
    N(X_{\Delta,f},\calB(\bx_0,L),S^1) := \frac{1}{2}\cdot\#\Big\{ ((x,y),(u,v),t)\in \Z^2\times \Z^2\times \Z_+: x^2+\Delta y^2 = t^2 F(u,v),\\ 
    (u,v)\in (B/t)^{1/2}\calB(\bx_0,L), \|x+y\sqrt{-\Delta}\|,\|x-y\sqrt{-\Delta}\|\leq B\cdot \max_{\by\in [-1,1]}|F(\by)|^{1/2}, \\ t\leq B/\exp(\sqrt{\log(B)}), \gcd(x,y,t) = \gcd(u,v) = 1\Big\}.
\end{multline}
If $\Delta<0$, we write
\begin{multline}\label{eq: indef point counting defn}
    N(X_{\Delta,f}, \calB(\bx_0,L), \calI(\theta,\ell)) :=  \frac{1}{2}\cdot\#\Big\{((x,y),(u,v),t)\in \Z^2\times \Z^2\times \Z_+: x^2+\Delta y^2 = t^2 F(u,v), \\ (u,v)\in (B/t)^{1/2}\calB(\bx_0,L), \Psi(x+y\sqrt{-\Delta}) \in \calI(\theta,\ell), |x+y\sqrt{-\Delta}|,|x-y\sqrt{-\Delta}|\leq B, \\ t\leq B/\exp(\sqrt{\log(B)}),\gcd(x,y,t)=\gcd(u,v)=1  \Big\}.
\end{multline}
We remark that restricting $\Psi(x+y\sqrt{-\Delta})\in \calI(\theta,\ell)$ allows us to control the argument of $\alpha = x+y\sqrt{-\Delta}.$ Similarly, taking $(u,v)\in (B/t)^{1/2}\calB(\bx_0,L)$ for $L= o(1)$ gives us precision on the value of $t^2F(u,v).$

In the next few sections, we will establish the following estimates on $N(X_{\Delta,f},\calB(\bx_0,L),I(\theta,\ell))$. 
\begin{prop}\label{prop: pos def Manin short}
    Assume $\Delta>0$ is a squarefree integer and $\sqrt{-\Delta}\not\in \Q$, and let $B$ be a sufficiently large number. Let $\bx_0\in [-1,1]^2$ and $L\gg \log(B)^{-10^{-10}}$. Then there exists a constant $c_{\Delta,f}$ such that as $B\rightarrow\infty,$ $$N(X_{\Delta,f},\calB(\bx_0,L),S^1) =L^2\cdot  c_{\Delta,f} B\log(B)^{\rho_{\Delta,f}-1}(1+ O(\log(B)^{-10^{-10}})).$$
    Moreover, if $c_{\Delta,f}=0$ then there is a Brauer-Manin obstruction or a local obstruction, i.e. $X_{\Delta,f}(\Q)=\emptyset.$
\end{prop}
\begin{prop}\label{prop: indef Manin short}
    Assume $\Delta<0$ is a squarefree integer and $\sqrt{-\Delta}\not\in \Q$, and let $B$ be a sufficiently large number. Let $\bx_0\in [-1,1]^2$, $\theta\in [0,1]$, and $\ell,L\gg \log(B)^{-10^{-10}}.$ Then we have that there exists a constant $b_{\Delta,f}$ such that as $B\rightarrow\infty$, \begin{equation*}N(X_{\Delta,f},\calB(\bx_0,L),\calI(\theta,\ell)) = \kappa(\bx_0,\calI(\theta,\ell)) \cdot b_{\Delta,f} \cdot L^2 B\log(B)^{\rho_{\Delta,f}-1} (1+ O(\log(B)^{-10^{-10}})),
    \end{equation*}
    where we define
    $$\kappa(\bx_0,\calI(\theta,\ell)) = \max(-\log|F(\bx_0)|/\log|\varepsilon|,0) \cdot \frac{|\calI(\theta,\ell)|}{4\pi}.$$
    Moreover, if $b_{\Delta,f}$ is zero then there is a Brauer-Manin obstruction or a local obstruction, i.e. $X_{\Delta,f}(\Q)=\emptyset.$ 
\end{prop}
\begin{remark}
    We reiterate that in \S\ref{sec: constant} we will show that if these constants $c_{\Delta,f}$ and $b_{\Delta,f}$ are zero then for every relevant vector $\bc$, there is a local obstruction to rational points on the variety $X^*_{\Delta,f_1,...,f_r,\bc}$ as defined in (\ref{eq: def of variety X_Delta, f_1,...,f_r}).
\end{remark}

\subsection{Decomposition of the positive-definite case}\label{subsec: pos def class num}
For this subsection, let us assume that $\Delta>0$. We do not genuinely need short boxes in this situation. 

Note that in (\ref{eq: pos def point counting defn}), the conditions that $\|x\pm y\sqrt{-\Delta}\| = \sqrt{x^2 + \Delta y^2} \leq B\cdot (\max_{\by\in [-1,1]} |F(\by)|)^{1/2}$ holds trivially. Thus, we get that
\begin{multline*}N(X_{\Delta,f},\calB(\bx_0,L),S^1) = \frac{1}{2}\cdot\sum_{0<t\leq B/\exp(\sqrt{\log(B)})} \sum_{\substack{(u,v)\in (B/t)^{1/2}\calB(\bx_0,L)\\ \gcd(u,v)=1}} \#\{x^2+\Delta y^2 = t^2F(u,v): \gcd(x,y,t)=1\} .\end{multline*}

Our first step is to rewrite this inner sum with a sum over ideals; we repeat the argument discussed in \S4 for this. Consider the principal ideal $I = (x+y\sqrt{-\Delta})$. Since $\gcd(x,y,t)=1$ and $N(I) = t^2F(u,v)$, we know that for any prime $p$ dividing $t$, we must have that $(p) = \frakp \overline{\frakp}$ splits in $K$ and that $(p)$ does not divide $I$. As a result, we can write $$I = \ft^2 \fraka$$
where $N(\ft)=t$, $N(\fraka) = F(u,v)$, and $\gcd(\ft,\overline{\ft}) =  1.$ Additionally, we have that $\gcd(t,I,\overline{I})=1,$ which occurs if and only if $\gcd(\fa,\overline{\ft})=1.$ Here we write that $\gcd(\fraka,\fb)=1$ if $\fraka$ and $\fb$ do not share any common prime ideal factors.

Consequently, we can rewrite the sum as:
$$N(X_{\Delta,f},\calB(\bx_0,L),S^1) =\frac{|\calU|}{2}\sum_{\substack{N(\ft)\leq B/\exp(\sqrt{\log(B)}) \\ \gcd(\ft,\overline{\ft})=1}} \sum_{\substack{(u,v)\in (B/t)^{1/2}\calB(\bx_0,L) \\ \gcd(u,v)=1}} \sum_{\substack{N(\fa)=F(u,v)\\ \gcd(\fa,\overline{\ft})=1}} \mathbf{1}_{\ft^2\fa \textrm{ is principal}}.$$
Recall that the group of units $\calU$ is finite in this case. From (\ref{eq: class group principal}), we know that we can rewrite the above sum with class group characters: 
$$N(X_{\Delta,f},\calB(\bx_0,L),S^1) = \frac{|\calU|}{2|C|}\sum_{\psi\in \hat{C}}\sum_{\substack{N(\ft)\leq B/\exp(\sqrt{\log(B)}) \\ \gcd(\ft,\overline{\ft})=1}}\psi(\ft)^2 \sum_{\substack{(u,v)\in (B/t)^{1/2}\calB(\bx_0,L) \\ \gcd(u,v)=1}} \sum_{\substack{N(\fa)=F(u,v)\\ \gcd(\fa,\overline{\ft})=1}} \psi(\fa).$$
Following Lemma \ref{lem: decomp eisenstein cusp}, we split the above sum into its Eisenstein and cuspidal parts. In particular, let us define for $\psi\in \hat{C}$ satisfying $\ord(\psi)\geq 3$:
\begin{equation}\label{eq: pos def error term}
E(X_{\Delta,f},\calB(\bx_0,L);0,\psi) =  \sum_{\substack{N(\ft)\leq B/\exp(\sqrt{\log(B)}) \\ \gcd(\ft,\overline{\ft})=1}} \psi(\ft)^2 \sum_{\substack{(u,v)\in (B/t)^{1/2}\calB(\bx_0,L) \\ \gcd(u,v)=1}} \sum_{\substack{N(\fa) = F(u,v) \\ \gcd(\fa,\overline{\ft})=1}}\psi(\fa).
\end{equation}
We can see that $$\sum_{\ord(\psi)\geq 3} E(X_{\Delta,f},\calB(\bx_0,L);0,\psi)$$
will be the contribution to $N(X_{\Delta,f},\calB(\bx_0,L),S^1)$ given by the cuspidal class group characters; this be shown in \S\ref{sec: cuspidal class gp} to be an error term. 

For the Eisenstein contribution, i.e. the genus characters as described in \S\ref{subsec: class group}, we must make an appropriate modification to deal with the condition $\gcd(\fa,\overline{\ft})=1.$ For this we note that
$$\sum_{\substack{N(\fa) = n\\ \gcd(\fa, \overline{\ft})=1}} \chi_{q_1,q_2}(\fa) = \sum_{\fk \mid \overline{\ft}} \mu(\fk) \sum_{\substack{N(\fa)=n/N(\fk)}} \chi_{q_1,q_2}(\fa \fk)= \sum_{\fk\mid \overline{\ft}}\mu(\fk)\chi_{q_1,q_2}(\fk) \sum_{\substack{N(\fa)=n/N(\fk)}} \chi_{q_1,q_2}(\fa \fk).$$
From now on, we will denote the norm of $\fk$ as $N(\fk) = k.$
Since $\fk\mid \ft$ and so $\gcd(\fk,\Delta)=1,$ it follows from the definition of $\chi_{q_1,q_2}(\fa)$ given in \cite[\S12.5]{IwaniecKowalski} that $\chi_{q_1,q_2}(\fk)=1.$ So, for a genus character parameterized by $q_1,q_2$ satisfying $q_1q_2=-\Delta$ and $\gcd(q_1,q_2)=1$, we define:
\begin{multline}\label{eq: pos def main term}
   M_{q_1,q_2}(X_{\Delta,f},\calB(\bx_0,L)) =  \sum_{\substack{N(\ft)\leq B/\exp(\sqrt{\log(B)}) \\ \gcd(\ft,\overline{\ft})=1}} \sum_{\fk\mid \overline{\ft}} \mu(\fk) \sum_{\substack{(u,v)\in (B/t)^{1/2} \calB(\bx_0,L) \\ \gcd(u,v)=1}} \varepsilon_{q_1,q_2}(F(u,v)/k)\\ 
   = \sum_{\substack{N(\fk)\leq B/\exp(\sqrt{\log(B)})\\ \gcd(\fk,\overline{\fk})=1}}\mu(\fk) \sum_{\substack{N(\ft)\leq B/k\exp(\sqrt{\log(B)}) \\ \gcd(\ft,\overline{\ft})=1}} \sum_{\substack{(u,v)\in (B/kt)^{1/2}\calB(\bx_0,L) \\ \gcd(u,v)=1}} \varepsilon_{q_1,q_2}(F(u,v)/k),
\end{multline}
where the Fourier coefficients of the corresponding Eisenstein series are when $\gcd(n,-\Delta)=1$:
\begin{equation}\label{eq: def eps q_1,q_2}
    \varepsilon_{q_1,q_2}(n) = \sum_{\substack{d\mid n}} \chi_{q_1}(d) \chi_{q_2}(n/d).
\end{equation}
\begin{remark}
    We note (and will expand in detail further in \S\ref{subsec: eisenstein to dirichlet convolution}) that $$\varepsilon_{q_1,q_2}(n)\approx (1\star \chi)(n).$$
    This function $(1\star \chi)$ is the one studied by de la Bret\`eche and Tenenbaum in \cite{delaTenenbaum-Manin} to prove the case of $x^2+y^2 = f(z)$ for $f$ an irreducible binary quartic. As such, the analysis from the Eisenstein part will strongly match how previous authors have proved the asymptotic for $\Delta=1$.
\end{remark}

We can now write that:
$$N(X_{\Delta,f}, \calB(\bx_0,L), S^1) = \frac{|\calU|}{2|C|} \sum_{\substack{q_1q_2 = -\Delta}}'M_{q_1,q_2}(X_{\Delta,f},\calB(\bx_0,L)) + \sum_{\substack{\psi\in \hat{C}\\ \ord(\psi)\geq 3}} O(E(X_{\Delta,f},\calB(\bx_0,L);0,\psi)).$$
Above, the restricted sum is as in (\ref{eq: restricted sum for genus characters}) and corresponds to those pairs $(q_1,q_2)$ to do indeed give genus characters for $\Q(\sqrt{-\Delta}).$

\subsection{Decomposition of the indefinite case}\label{subsec: indef reduction}

For this subsection, assume that $\Delta<0$. In this case, the group of units $\calU$ is infinite and generated by the fundamental unit $\varepsilon$. 
We wish to count:
\begin{multline*}
N(X_{\Delta,f},\calB(\bx_0,L),\calI(\theta,\ell)) =\\ \frac{1}{2}\cdot\sum_{t\leq B/\exp(\sqrt{\log(B)})} \sum_{\substack{(u,v)\in (B/t)^{1/2}\calB(\bx_0,L) \\ \gcd(u,v)=1}}  \sum_{\substack{x+y\sqrt{-\Delta}\bmod \varepsilon \\ x^2+\Delta y^2 = t^2F(u,v)\\ \Psi(x+y\sqrt{-\Delta})\in \calI(\theta,\ell) \\ \gcd(x,y,t)=1}} \#\{k\in \Z: |\varepsilon^k\alpha|,|\overline{\varepsilon}^k\overline{\alpha}|\leq B \}
\end{multline*}
Our goal is to express the above sum in terms of $M_{q_1,q_2}(X_{\Delta,f},\calB(\bx_0,L))$ and $E(X_{\Delta,f},\calB(\bx_0,L);0,\psi)$ (and similar functions). To do so, we must first estimate this count $\#\{k\in \Z: |\varepsilon^k \alpha|,|\overline{\varepsilon}^k\overline{\alpha}|\leq B\}$ and remove the condition that $\Psi(x+y\sqrt{-\Delta})\in \calI(\theta,\ell).$

\subsubsection{Counting the number of units}
To estimate the count $\#\{k\in \Z: |\varepsilon^k \alpha|,|\overline{\varepsilon}^k\overline{\alpha}|\leq B\},$ we apply bounds similar to the proofs of Lemma \ref{lem: upp bound indefinite} and Lemma \ref{lem: large torsors}. Note that $|\varepsilon^k\alpha|\leq B$ if and only if $$k\leq \frac{\log(B)}{\log|\varepsilon|}-\frac{\log|\alpha|}{\log|\varepsilon|}.$$
Since $\Psi(\alpha)\in \calI(\theta,\ell)$ and $\alpha$ is minimal, we know that $$\frac{\log|\alpha|-\log|\overline{\alpha|}}{\log|\varepsilon|} = \frac{2\log|\alpha|}{\log|\varepsilon|} - \frac{\log|\alpha\cdot \overline{\alpha}|}{\log|\varepsilon|} \in [\theta,\theta+\ell].$$
Recall that $\alpha\cdot \overline{\alpha} = x^2 + \Delta y^2 = t^2 F(u,v).$ Since $(u,v)\in (B/t)^{1/2}\calB(\bx_0,L)$, we know that $t^2F(u,v) = B\cdot (F(\bx_0) + O(L)).$ So, we can replace this term with $$\frac{\log|\alpha\cdot \overline{\alpha}|}{\log|\varepsilon|} = \frac{\log|t^2F(u,v)|}{\log|\varepsilon|} = \frac{2\log(B)}{\log|\varepsilon|} + \frac{\log|F(\bx_0)|}{\log|\varepsilon|} + O_{\varepsilon}(L).$$
Combining all of these factors together, we get that 
$$k\leq -\frac{\theta}{2} - \frac{\log|F(\bx_0)|}{2\log|\varepsilon|} + O(\ell+L).$$

A similar computation with $|\overline{\varepsilon}^k\overline{\alpha}|\leq B$ reveals that we also need $$k\geq -\frac{\theta}{2} + \frac{\log|F(\bx_0)|}{2\log|\varepsilon|} + O(\ell+L).$$
Thus, we get that 
\begin{equation}\label{eq: indef num units}
 \#\{k\in \Z: |\varepsilon^k\alpha|,|\overline{\varepsilon}^k\overline{\alpha}|\leq B\}= \max(-\log|F(\bx_0)|/\log|\varepsilon|, 0) + O(\ell+L).\end{equation}
 \begin{remark}
     Since $\bx_0\in [-1,1]^2$, we expect $\log|F(\bx_0)|/\log|\varepsilon|$ to often be negative. 
 \end{remark}

 Plugging (\ref{eq: indef num units}) in, we have that 
\begin{multline}\label{eq: indef post-unit estimate}
    N(X_{\Delta,f},\calB(\bx_0,L),\calI(\theta,\ell)) = \left(\max(-\log|F(\bx_0)|/\log|\varepsilon|,0) + O(\ell+L)\right) \\ \times \frac{1}{2}\sum_{t\leq B/\exp(\sqrt{\log(B)})} \sum_{\substack{(u,v)\in (B/t)^{1/2}\calB(\bx_0,L)\\ \gcd(u,v)=1}} \sum_{\substack{x+y\sqrt{-\Delta}\bmod \varepsilon\\ x^2+\Delta y^2 = t^2F(u,v) \\ \Psi(x+y\sqrt{-\Delta})\in \calI(\theta,\ell) \\ \gcd(x,y,t)=1}}1.
\end{multline}

\subsubsection{Applying the Erd\H{o}s-Tur\'an inequality}
In this subsection, we remove the condition that $\Psi(x+y\sqrt{-\Delta})\in \calI(\theta,\ell);$
we follow the proof structure set out by Heath-Brown in \cite{HBclassnum}. 
We recall the following application of the Erd\H{o}s-Tur\'an inequality \cite{ErdosTuran}, as written in \cite[Theorem 1]{MontgomeryTenLectures}:
\begin{lemma}[Erd\H{o}s-Tur\'an, \cite{ErdosTuran}]
Let $\calA_t(n) = \{x+y\sqrt{-\Delta}\bmod \varepsilon: x^2+\Delta y^2 = n, \gcd(x,y,t)=1\}$ and $\calI\subset S^1$. Then we have that as $n\rightarrow\infty$ and for any $H\geq 1$, 
\begin{multline*}\sum_{x+y\sqrt{-\Delta }\in \calA} \mathbf{1}_{\Psi(x+y\sqrt{-\Delta})\in \calI} - \frac{|\calI|}{2\pi} \cdot \#\calA \ll H^{-1} \#\calA \\ + \sum_{1\leq h\leq H}\left(\frac{1}{H+1}+\min(|\calI|,(\pi h)^{-1})\right)\left|\sum_{\substack{x+y\sqrt{-\Delta}\in \calA\\ \gcd(x,y,t)=1}} \Psi(x+y\sqrt{-\Delta})^h\right|.\end{multline*}
\end{lemma}
Applying the above to (\ref{eq: indef post-unit estimate}), we achieve that:
\begin{multline}\label{eq: indef post-erdos}
    N(X_{\Delta,f},\calB(\bx_0,L),\calI(\theta,\ell)) = \frac{|\calI(\theta,\ell)|}{4\pi}\cdot \left(\max(-\log|F(\bx_0)|/\log|\varepsilon| + O(\ell+L)\right) \\ \times \left(M'(X_{\Delta,f},\calB(\bx_0,L))(1+O(\ell^{-1}H^{-1}))+O\left(\ell^{-1}\sum_{1\leq h\leq H}\frac{1}{h} A'(X_{\Delta,f},\calB(\bx_0,L);h)\right)\right),
\end{multline}
where we define 
\begin{equation}\label{eq: indef M'}
    M'(X_{\Delta,f},\calB(\bx_0,L)) := \sum_{t\leq B/\exp(\sqrt{\log(B)})} \sum_{\substack{(u,v)\in (B/t)^{1/2}\calB(\bx_0,L) \\ \gcd(u,v)=1}}\sum_{\substack{x+y\sqrt{-\Delta}\bmod \varepsilon\\ x^2+\Delta y^2 = t^2F(u,v) \\ \gcd(x,y,t)=1}} 1,
\end{equation}
\begin{equation}\label{eq: indef A'}
    A'(X_{\Delta,f},\calB(\bx_0,L);h) = \sum_{t\leq B/\exp(\sqrt{\log(B)})} \sum_{\substack{(u,v)\in (B/t)^{1/2}\calB(\bx_0,L)\\ \gcd(u,v)=1}} \Bigg|\sum_{\substack{x+y\sqrt{-\Delta} \bmod \varepsilon \\ x^2+\Delta y^2 = t^2F(u,v) \\ \gcd(x,y,t)=1}} \Psi(x+y\sqrt{-\Delta})^{h}\Bigg|.
\end{equation}
Let us take $H=\log(B)^{5}.$ Lemma \ref{lem: upp bound indefinite} tells us that $$M'(X_{\Delta,f},\calB(\bx_0,L))\ll L^2 B\log(B)^{\rho_{\Delta,f}-1}.$$
Hence, we can see that $$\ell^{-1}H^{-1} M'(X_{\Delta,f},\calB(\bx_0,L)) \ll \ell^{-1}L^2 B\log(B)^{\rho_{\Delta,f}-6}$$ will be more than a sufficient error term. So, it remains to better understand $M'(X_{\Delta,f},\calB(\bx_0,L))$ and $A'(X_{\Delta,f},\calB(\bx_0,L);h).$

\begin{remark}
    We would like to point out that there is an extra set of absolute values in the definition of $A'(X_{\Delta,f},\calB(\bx_0,L);h)$ compared to $E(X_{\Delta,f},\calB(\bx_0,L);0,\psi)$, despite them both being error terms that come from the cuspidal contribution. This leads us to requiring different types of analysis on each sum. 
\end{remark}

\subsubsection{Class group characters}
We again apply the results of \S\ref{subsec: class group} and rewrite the above expressions in terms of class group characters. 
Consider the following sum for $0\leq h\leq H$: 
$$\sum_{\substack{x+y\sqrt{-\Delta} \bmod \varepsilon \\ x^2+\Delta y^2 = t^2F(u,v) \\ \gcd(x,y,t)=1}} \Psi(x+y\sqrt{-\Delta })^{h}.$$
Since every principal ideal $I=(x+y\sqrt{-\Delta})$ with $N(I) = t^2F(u,v)$ will have a unique generator $x+y\sqrt{-\Delta}$ that is minimal, we can write the above sum as: 
$$\sum_{\substack{N(I) = t^2F(u,v) \\ I\textrm{ is principal} \\ \gcd(x,y,t)=1}} \Psi(I)^h.$$
Next, we again note that since we require $\gcd(x,y,t)=1$, we must have that for any prime $p$ dividing $t$, the ideal $(p) = \frakp \overline{\frakp}$ splits and we can not have $(p)$ dividing $I$. Thus, we can write $$I = \ft^2 \fraka,$$
where $N(\ft)=t$, $N(\fraka) = F(u,v)$, and $\gcd(\ft,\overline{\ft})=\gcd(\fraka,\overline{\ft})=1.$ So, we can write the above sum as $$\sum_{\substack{N(\ft)=t\\ \gcd(\ft,\overline{\ft})=1}} \Psi(\ft)^{2h} \sum_{\substack{N(\fraka) = F(u,v) \\ \gcd(\fraka,\overline{\ft})=1}} \Psi(\fraka)^h \cdot \mathbf{1}_{\ft^2\fraka \textrm{ is principal}}.$$

Applying (\ref{eq: class group principal}), the above sum is equivalent to:
$$\frac{1}{|C|}\sum_{\psi\in \hat{C}}\sum_{\substack{N(\ft)=t\\ \gcd(\ft,\overline{\ft})=1}} \Psi(\ft)^{2h}\psi(\ft)^2 \sum_{\substack{N(\fraka) = F(u,v) \\ \gcd(\fraka,\overline{\ft})=1}} \Psi(\fraka)^h \psi(\fa).$$
From Lemma \ref{lem: decomp eisenstein cusp}, we can decompose further the contributions from the Eisenstein and cuspidal parts. Let us start with the main term given by $M'(X_{\Delta,f},\calB(\bx_0,L))$ above; in particular, we have that 
\begin{multline*}
    M'(X_{\Delta,f},\calB(\bx_0,L)) = \frac{1}{|C|}\sum_{q_1q_2=-\Delta}' M_{q_1,q_2}(X_{\Delta,f},\calB(\bx_0,L)) + \sum_{\substack{\psi\in \hat{C}\\ \ord(\psi)\geq 3}} O(E(X_{\Delta,f}, \calB(\bx_0,L); 0,\psi)),
\end{multline*}
In the above decomposition, $M(X_{\Delta,f},\calB(\bx_0,L))$ and $E(X_{\Delta,f},\calB(\bx_0,L);0,\psi)$ are defined in (\ref{eq: pos def main term}) and (\ref{eq: pos def error term}).

On the other hand, for $A'(X_{\Delta,f},\calB(\bx_0,L);h)$, we add in the corresponding absolute value signs. 
In particular, define for $1\leq h\leq H=\log(B)^{5}$ and $\psi\in \hat{C}$: \begin{equation}\label{eq: def indef A h}
    A(X_{\Delta,f},\calB(\bx_0,L);h,\psi) = \sum_{\substack{N(\ft)\leq B/\exp(\sqrt{\log(B)}) \\ \gcd(\ft,\overline{\ft})=1}}\sum_{\substack{(u,v)\in (B/t)^{1/2}\calB(\bx_0,L) \\ \gcd(u,v)=1}} \sigma(F(u,v), \Psi^h \psi; \overline{\ft}),
\end{equation}
where we have defined $\sigma(n,\psi;t)$ as
\begin{equation}
    \sigma(n,\psi;\ft) := \left|\sum_{\substack{N(\fa) = n\\ \gcd(\fa,\ft)=1}}\psi(\fa)\right|.
\end{equation}
Lemma \ref{lem: decomp eisenstein cusp} tells us that 
$$A'(X_{\Delta,f},\calB(\bx_0,L);h) \leq  \sum_{\psi\in \hat{C}} A(X_{\Delta,f},\calB(\bx_0,L);h,\psi).$$

\subsection{Unifying the positive-definite and indefinite notation}\label{subsec: intro main and error prop}

Let us summarize our findings. For $\Delta>0$, we saw that 
\begin{equation}\label{eq: pos def N in terms of M and E}N(X_{\Delta,f},\calB(\bx_0,L), S^1) = \frac{|\calU|}{2|C|} \sum_{q_1q_2=-\Delta}'  M_{q_1,q_2}(X_{\Delta,f},\calB(\bx_0,L)) + \sum_{\substack{\psi\in \hat{C}\\ \ord(\psi)\geq 3}}O(E(X_{\Delta,f},\calB(\bx_0,L);0,\psi)).\end{equation}
On the other hand, for $\Delta<0$, we had the relation when $H=\log(B)^5$
\begin{multline}\label{eq: indef N in terms of M and E}
    N(X_{\Delta, f},\calB(\bx_0,L), \calI(\theta,\ell)) = \frac{|\calI(\theta,\ell)|}{4\pi}\cdot (\max(-\log|F(\bx_0)|/\log|\varepsilon|,0) + O(\ell+L)) \\
    \times \Bigg(\frac{1}{|C|}\sum_{q_1q_2=-\Delta}'  M_{q_1,q_2}(X_{\Delta,f},\calB(\bx_0,L)) + O(B\log(B)^{\varrho_{\Delta,f}-6})\\  + \frac{1}{|C|} \sum_{\substack{\psi\in \hat{C}\\ \ord(\psi)\geq 3}} O(E(X_{\Delta,f}, \calB(\bx_0,L); 0 ,\psi)) + \ell^{-1}\sum_{1\leq h\leq H} \frac{1}{h} \sum_{\psi\in \hat{C}} O(A(X_{\Delta,f},\calB(\bx_0,L);h,\psi))\Bigg).
\end{multline}
Thus, to prove Proposition \ref{prop: pos def Manin short} and Proposition \ref{prop: indef Manin short}, it will suffice to prove the following results for our cuspidal and Eisenstein parts. 

\begin{prop}[The cuspidal contribution -- the finite image case]\label{prop: error term E}
    Let $L\gg \log(B)^{-1/10}$. For 
    $\ord(\psi)\geq 3$, we have that $$E(X_{\Delta,f},\calB(\bx_0,L);0,\psi) \ll_{\Delta,f} L^2 \cdot B\log(B)^{\rho_{\Delta,f}-1-\frac{1}{3}}.$$
\end{prop}

\begin{prop}[The cuspidal contribution -- the Gr\"ossencharacter case]\label{prop: error term A}
Let $L\gg \exp(-\log(B)^{1/3})$. For $1\leq h\leq \log(B)^{5}$ and $\psi\in \hat{C}$, we have that 
$$A(X_{\Delta,f},\calB(\bx_0,L);h,\psi) \ll_{B,\Delta,f} L^2 \cdot B \log(B)^{\rho_{\Delta,f}-1-1/5}.$$    
\end{prop}

\begin{prop}[The Eisenstein contribution]\label{prop: main term}
    Let us assume that $L\gg \log(B)^{-10^{-10}}.$ There exists a constant $C_{q_1,q_2,f}$ such that as $B\rightarrow\infty$, we have that $$M_{q_1,q_2}(X_{\Delta,f},\calB(\bx_0,L)) = C_{q_1,q_2,f} L^2\cdot B\log(B)^{\rho_{\Delta,f}-1} (1+ O(\log(B)^{\rho_{\Delta,f}-1-10^{-8}})).$$
\end{prop}
\begin{remark}
    We will further discuss $C_{q_1,q_2,f}$ in \S\ref{sec: asymptotic} and \S\ref{sec: constant} -- in particular, we relate $C_{q_1,q_2,f}$ to Brauer-Manin and local obstructions to the Hasse principle. However, it turns out that it is simpler to classify directly when $c_{\Delta,f}$ of Theorem \ref{thm: Manin} is zero, as it will become necessary to consider the sum over genus characters $$\sum_{q_1q_2=-\Delta}' C_{q_1,q_2,f}.$$ 
\end{remark}

\begin{remark}
    We can see that the above two propositions that our size constraints on the length of our short boxes $L$ and short intervals $\ell$ come from the Eisenstein term. Additionally, we can see that the necessity to look at short boxes and intervals comes only from the case when $\Delta<0$ and $x^2+\Delta y^2$ is indefinite (as it was used in \S\ref{sec: reductions}.3 to estimate $\#\{k\in \Z: |\varepsilon^k\alpha|,|\varepsilon^{-k}\overline{\alpha}| \leq B\}$). Our condition in Proposition \ref{prop: error term A} that $L\gg \exp(-\log(B)^{1/3})$ comes simply by the desire to have $(B/t)^{1/2}|\calB(\bx_0,L)| \gg 1.$ The condition $L\gg \log(B)^{-1/10}$ in Proposition \ref{prop: error term E} allows us to compare $L^2 B$ and $LB$ against the corresponding savings. 
\end{remark}

\subsection{Interpretation with cusp forms}\label{subsec: interpret cusp}
At this point, let us point out that if $\xi$ is a Hecke character of $K$, then there is a corresponding automorphic representation $\Xi$ of $\GL_2(\A_\Q)$. Moreover, the Fourier coefficients of $\Xi$ will be given by $$\lambda_\Xi(n) = \sum_{N(\fa)=n} \xi(\fa).$$
Since any class group character is a Hecke character, and $\Psi$ is by definition a Hecke character, we can interpret $E(X_{\Delta,f},\calB(\bx_0,L);0,\psi)$ and $A(X_{\Delta,f},\calB(\bx_0,L);h,\psi)$ in terms of these cuspidal automorphic representations $\Xi$. 

First, let $\Xi$ be the automorphic representation for $\Psi^h \psi$ for $h\geq 1$. Then we will see that $A(X_{\Delta,f},\calB(\bx_0,L);h,\psi)$ is upper bounded by the sum: 
$$\sum_{\substack{N(\ft)\leq B/\exp(\sqrt{\log(B)}) \\ \gcd(\ft,\overline{\ft})=1}}\sum_{\substack{(u,v)\in (B/t)^{1/2}\calB(\bx_0,L) \\ \gcd(u,v)=1}} |\lambda_\Xi(F(u,v))|.$$
In Proposition \ref{prop: error term A}, we will achieve our savings despite the absolute values placed around the Fourier coefficients -- this savings will come from studying the first few Sato-Tate moments of $\Xi$. 

Second, let $\Phi$ be the automorphic representation for the class group character $\psi$ with $\ord(\psi)>2$. 
Our error term $E(X_{\Delta,f},\calB(\bx_0,L);0,\psi)$ can then be written as (approximately): 
$$\sum_{\substack{t\leq B/\exp(\sqrt{\log(B)})}} \sum_{d\mid t} \mu(d)\sum_{\substack{(u,v)\in (B/t)^{1/2}\calB(\bx_0,L) \\ \gcd(u,v)=1}} \lambda_\Phi\left(\frac{t^2F(u,v)}{d^2}\right).$$
This viewpoint will be expanded on in \S\ref{subsec: dihedral interpret cusp form}. In Proposition \ref{prop: error term E}, we achieve our saving by the sum over $\ft$ and $d$ and optimal upper bounds on the contribution from the sum over $(u,v)$. It is important in this argument that unlike $A(X_{\Delta,f},\calB(\bx_0,L);h,\psi)$, there are no absolute values. 

\subsection{Proof of the asymptotic for Theorem \ref{thm: Manin}}\label{subsec: proof of main thm}
In both cases, Lemma \ref{lem: large torsors} allows us to reduce the problem to estimating: 
\begin{multline}
    S(X_{\Delta,f},B) :=  \frac{1}{2}\cdot\#\Big\{ ((x,y),(u,v),t)\in \Z^2\times \Z^2\times \Z_+: x^2+\Delta y^2 = t^2 F(u,v),\\ 
    |u|,|v|\leq (B/t)^{1/2}, \|x\pm y\sqrt{-\Delta}\|\leq \nu_{\Delta,f} B, \\ 0<t\leq B/\exp(\sqrt{\log(B)}), \gcd(x,y,t) = \gcd(u,v) = 1\Big\}.
\end{multline}

For this sum, we split into short intervals. We choose $L = \lceil \log(B)^{10^{-10}}\rceil^{-1}$. If $\Delta>0$, take $\ell=1$ and so $\calI(\theta,\ell)=S^1$; if $\Delta<0$, we let $\ell = \lceil \log(B)^{-10^{-10}}\rceil$. Let $\bx_0$ and $\theta$ then range over the corresponding dividing of the box $[-1,1]^2$ and interval $[0,1].$ We know that $$S(X_{\Delta,f},B) = \sum_{\bx_0} \sum_{\theta} N(X_{\Delta,f},\calB(\bx_0,L),\calI(\theta,\ell)).$$
If $\Delta>0$, we apply Proposition \ref{prop: pos def Manin short}: 
\begin{align*}
    S(X_{\Delta,f},B) &= c_{\Delta,f} B\log(B)^{\rho_{\Delta,f}-1} (1+O(\log(B)^{-10^{-10}})) \cdot \sum_{\bx_0} L^2 \\
    &= 4\cdot c_{\Delta,f} B\log(B)^{\rho_{\Delta,f}-1} (1+O(\log(B)^{-10^{-10}})).
\end{align*}
This completes the positive-definite case. 

On the other hand, if $\Delta<0$, we apply Proposition \ref{prop: indef Manin short} to achieve: 
\begin{align*}
    S(X_{\Delta,f},B) &= b_{\Delta,f} B\log(B)^{\varrho_{a,f}-1}(1+O(\log(B)^{-10^{-7}})) \cdot \sum_{\bx_0}\sum_{\theta} \kappa(\bx_0,\calI(\theta,\ell)) \\
    &= b_{\Delta,f} B\log(B)^{\rho_{\Delta,f}-1}(1+O(\log(B)^{-10^{-7}})) \cdot \sum_{\bx_0}\sum_{\theta} L^2 \cdot \ell  \cdot \max(-\log|F(\bx_0)|/\log|\varepsilon|,0).
\end{align*}
Recall that if $b_{\Delta,f}=0$ then there is a Brauer-Manin or local obstruction.

Finally, we can see that this sum can be approximated by the integral:
\begin{multline*}\sum_{\bx_0}\sum_{\theta} L^2 \cdot \ell \cdot \max(-\log|F(\bx_0)|/\log|\varepsilon|,0) \\ = \int_{-1}^1\int_{-1}^1 \int_{0}^1\max(-\log|F(x,y)|/\log|\varepsilon|,0) d\theta dxdy + O(\log(B)^{-10^{-10}}).\end{multline*}
Since this integral is a constant only depending on $f$ and $\varepsilon$ (which only depends on $\Delta$), we can write this expression as $d_{\Delta,f} + O(\log(B)^{-10^{-10}}).$ Moreover, it is clear that $d_{\Delta,f}>0$. Plugging this back into $S(X_{\Delta,f},B)$, we have that 
$$S(X_{\Delta,f},B) = b_{\Delta,f}d_{\Delta,f} B\log(B)^{\rho_{\Delta,f}-1}(1+O(\log(B)^{-10^{-10}})).$$
Taking $c_{\Delta,f} = b_{\Delta,f}d_{\Delta,f}$ finishes the proof. 
\qed

\subsection{Proof of Corollary \ref{cor: Manin}}

    Note that we have restricted in Theorem \ref{thm: Manin} to when $\Delta\in \Z$ is squarefree and $f(z)\in \Z[z]$, whereas the general definition for Ch\^atelet surfaces allows $\Delta\in \Q$ and $f(z)\in \Q[z]$. Consider the initial surface: 
    $$x^2 + \Delta_0 y^2 = f_0(z)$$
    for $\Delta_0 \in \Q$ a nonsquare and $f_0(z)\in \Q[z]$. After clearing denominators, we can see that rational points on the above surface will be in direct bijection with the rational points on the surface: 
    $$(Qx)^2 + \Delta_1 y^2 = f_1(z)$$
    for $Q,\Delta_1\in \Z$ and $f_1(z)\in \Z[z]$. After modifying the height function accordingly (by a factor of $Q$), we can see that the problem of counting rational points on this surface is equivalent to counting rational points (with the modified height function) on the surface $$x^2 + \Delta_1 y^2 = f_1(z)$$ for $\Delta_1\in \Z$ and $f_1(z)\in \Z[z]$. 

    Next, if $\Delta_1$ is not squarefree, we can write $\Delta_1 = \Delta \delta^2.$ Hence, our surface can be written as $$x_1^2+ \Delta (\delta y)^2 = f_1(z).$$
    After modifying the height function by a factor of $\delta,$ we can see that counting rational points on the above surface is equivalent to counting rational points on $$x^2 + \Delta y^2 = f(z),$$
    where $f(z) = f_1(z)\in \Z[z]$. Thus, Theorem \ref{thm: Manin} indeed resolves Manin's conjecture for all Ch\^atelet surfaces over $\Q.$
    
\qed

\section{The cuspidal contribution -- the finite image case}\label{sec: cuspidal class gp}
In this section, we prove Proposition \ref{prop: error term E}. Let us recall the definition of $E(X_{\Delta,f},\calB(\bx_0,L);0,\psi)$ from (\ref{eq: pos def error term}):
$$\sum_{\substack{N(\ft)\leq B/\exp(\sqrt{\log(B)}) \\ \gcd(\ft,\overline{\ft})=1}} \psi(\ft)^2 \sum_{\substack{(u,v)\in (B/t)^{1/2}\calB(\bx_0,L) \\ \gcd(u,v)=1}} \sum_{\substack{N(\fa) = F(u,v) \\ \gcd(\fa,\overline{\ft})=1}}\psi(\fa).$$
We also recall that $\psi$ is a class group character of order $\geq 3$. As described in the outline, we aim to change the order of summation and achieve savings from the sum over $\ft$; however, our situation is slightly complicated by the fact that we are summing over short-ish boxes $\calB(\bx_0,L).$

Let $\epsilon(B)$ to be determined later. We say that $t\sim_\epsilon T$ if $t\in [T,T(1+\epsilon(B))].$ Consider the space: 
$$D(\bx_0,L) = \cup_{T_i} (B/T_i)^{1/2} \calB(\bx_0,L),$$
where $T_i$ ranges over the ``dyadic interals of difference $(1+\epsilon(B))$'' up to $B/\exp(\sqrt{\log(B)}).$ First, we claim that almost all of the points $(u,v)$ counted in $E(X_{\Delta,f},\calB(\bx_0,L);0,\psi)$ are in $D(\bx_0,L).$
\begin{lemma}
Assume $\epsilon(B)<L$. Let $(u,v)$ be in $(B/t)^{1/2}\calB(\bx_0,L)$ for some $t$, but not contained in $D(\bx_0,L)$. Then $(u,v)$ is $O(\epsilon(B))$ close to the corners of the box $\calB(\bx_0,L)$. In other words, either $$(u,v) \in (B/t)^{1/2}(x_0,y_0+L) + [0,O(\epsilon(B))]^2$$
or $$(u,v) \in (B/t)^{1/2}((x_0+L, y_0))+[0,O(\epsilon(B))]^2.$$
\end{lemma}
\begin{proof}
    Since $(u,v)\in (B/t)^{1/2}\calB(\bx_0,L)$, we know that $$\frac{t^{1/2}}{B^{1/2}}(u,v)\in \calB(\bx_0,L).$$
    We also know that if $T_1\leq t\leq T_2$, then $$\frac{T_1^{1/2}}{B^{1/2}}(u,v) \not\in \calB(\bx_0,L), \frac{T_2^{1/2}}{B^{1/2}}(u,v) \not\in \calB(\bx_0,L).$$
    Let us assume that $u\leq v$. Then we must have that $\frac{T_1^{1/2}}{B^{1/2}} u < x_0$ and $\frac{T_2^{1/2}}{B^{1/2}}v > y_0+L.$
    Thus, we get that $\frac{t^{1/2}}{B^{1/2}} u - x_0 \leq O(\epsilon(B) \cdot \frac{t^{1/2}}{B^{1/2}} u)$. Similarly, we achieve that $y_0+L- \frac{t^{1/2}}{B^{1/2}} v = O(\epsilon(B) \cdot \frac{t^{1/2}}{B^{1/2}}v).$
\end{proof}

Next, we justify changing the order of summation for $E(X_{\Delta,f},\calB(\bx_0,L);0,\psi).$
\begin{lemma}\label{lem: change of var}
    Let us define $$D(X_{\Delta,f},\calB(\bx_0,L);\psi) = \sum_{T_i} \sum_{\substack{(u,v)\in (B/T_i)^{1/2}\calB(\bx_0,L)\\ \gcd(u,v)=1}} \sum_{N(\fa)=F(u,v)} \psi(\fa) \sum_{\substack{N(\ft) \sim_\epsilon T_i \\ \gcd(\ft,\overline{\ft})=1 \\ \gcd(\fa,\overline{\ft})=1}} \psi(\ft)^2.$$
    Here the sum over $T_i$ is a sum over powers of $(1+\epsilon(B))$ up to $B/\exp(\sqrt{\log(B)}).$
    Assume that $\epsilon(B) = o(L)$. Then we can bound $$|E(X_{\Delta,f},\calB(\bx_0,L);0,\psi)-D(X_{\Delta,f},\calB(\bx_0,L);\psi)|\ll L\epsilon(B) B\log(B)^{\rho_{\Delta,f}-1}$$
\end{lemma}

\begin{proof}
    First, we compare $D(X_{\Delta,f},\calB(\bx_0,L);\psi)$ to an intermediate sum: 
    $$I(X_{\Delta,f},\calB(\bx_0,L);\psi) = \sum_{T_i} \sum_{\substack{(u,v)\in (B/T_i)^{1/2}\calB(\bx_0,L)\\ \gcd(u,v)=1}} \sum_{N(\fa)=F(u,v)} \psi(\fa) \sum_{\substack{N(\ft) \sim_\epsilon T_i \\ (u,v)\in (B/t)^{1/2}\calB(\bx_0,L) \\ \gcd(\ft,\overline{\ft})=1 \\ \gcd(\fa,\overline{\ft})=1}} \psi(\ft)^2.$$
    Assume $(u,v)\in (B/T_i)^{1/2}\calB(\bx_0,L)$, but $(u,v)\not\in (B/t)^{1/2}\calB(\bx_0,L)$ for some $t\in [T_i,(1+\epsilon(B))T_i]$. Since $\epsilon(B)\leq L$, we can isolate such points as satisfying that they are within $O(\epsilon(B))$ of the edges of $(B/T_i)^{1/2}\calB(\bx_0,L)$. Thus, we have that (using the trivial bound $|\psi^2(\ft)|\leq 1$)
    \begin{align*}
        |D(X_{\Delta,f},\calB(\bx_0,L);\psi)-I(X_{\Delta,f},\calB(\bx_0,L);\psi)| &\ll \sum_{T_i} \epsilon(B) T_i \cdot \sum_{\substack{(u,v) \in (B/T_i)^{1/2}\partial \calB(\bx_0,L)+O(\epsilon(B)) }} (1\star \chi)(F(u,v))  \\
        &\ll \epsilon(B)\sum_{T_i}  T_i \cdot \frac{B}{T_i} \cdot L\epsilon(B) \cdot \log(B)^{\rho_{\Delta,f}-2} \\
        &\ll L\epsilon(B)^2 B\log(B)^{\rho_{\Delta,f}-2} \cdot \frac{\log(B)}{\epsilon(B)}\ll L\epsilon(B) B\log(B)^{\rho_{\Delta,f}-1}.
    \end{align*}
    Here we have used Theorem \ref{thm: Nair sieve} as before with the non-negative multiplicative function $(1\star\chi)(n)$. 
    Also, observe that $(B/T_i)^{1/2}(\partial \calB(\bx_0,L) + O(\epsilon(B))\supset (B/T_i)^{1/2}\partial\calB(\bx_0,L) + O(\epsilon(B))$ and the larger one has area $BT_i^{-1}L\epsilon(B).$
    Here we have used that there are $\ll \log(B)/\epsilon(B)$ such values of $T_i$. 

    Next, we want to compare $I(X_{\Delta,f},\calB(\bx_0,L);\psi)$ and $E(X_{\Delta,f},\calB(\bx_0,L);0,\psi)$. We can see that $$I(X_{\Delta,f},\calB(\bx_0,L);\psi) = \sum_{\substack{(u,v)\in D(\bx_0,L)\\ \gcd(u,v)=1}} \sum_{N(\fa)=F(u,v)}\psi(\fa) \sum_{\substack{t \leq B/\exp(\sqrt{\log(B)})\\ (u,v)\in (B/t)^{1/2}\calB(\bx_0,L)}} \sum_{\substack{N(\ft) = t\\ \gcd(\ft,\overline{\ft})=1 \\ \gcd(\fa,\overline{\ft})=1}} \psi(\ft)^2,$$
    whereas we have that 
    $$E(X_{\Delta,f},\calB(\bx_0,L);0,\psi) = \sum_{\substack{(u,v) \in \cup_t (B/t)^{1/2}\calB(\bx_0,L)\\ \gcd(u,v)=1}}\sum_{N(\fa)=F(u,v)}\psi(\fa) \sum_{\substack{t \leq B/\exp(\sqrt{\log(B)})\\ (u,v)\in (B/t)^{1/2}\calB(\bx_0,L)}} \sum_{\substack{N(\ft) = t\\ \gcd(\ft,\overline{\ft})=1 \\ \gcd(\fa,\overline{\ft})=1}} \psi(\ft)^2.$$
    Thus, the difference is summing over exactly those $(u,v)$ not contained in $D(\bx_0,L)$. 
    Thus, we have that 
    \begin{align*}
        |I(X_{\Delta,f},\calB(\bx_0,L);\psi)&-E(X_{\Delta,f},\calB(\bx_0,L);0,\psi)| \\ &\ll \sum_{\substack{N(\ft)\leq B/\exp(\sqrt{\log(B)})\\ \gcd(\ft,\overline{\ft})=1}} \sum_{\substack{(u,v)\in (B/t)^{1/2}((x_0,y_0+L)+O(\epsilon(B))) \\ \cup (B/t)^{1/2}((x_0+L,y_0)+O(\epsilon(B)))}} (1\star \chi)(F(u,v)) \\
        &\ll \sum_{N(\ft)\leq B/\exp(\sqrt{\log(B)})} \epsilon(B)^2 \cdot \frac{B}{t} \log(B/t)^{\rho_{\Delta,f}-2} \\
        &\ll B\epsilon(B)^2 \log(B/t)^{\rho_{\Delta,f}-1}.
    \end{align*}
    Again, we have applied Theorem \ref{thm: Nair sieve} on the sum over $(u,v)$ with the non-negative multiplicative function $(1\star\chi)(n).$
\end{proof}

Finally, we wish to bound $D(X_{\Delta,f},\calB(\bx_0,L);\psi)$.
Now, we place the absolute values: 
$$D(X_{\Delta,f},\calB(\bx_0,L);\psi) \leq \sum_{T_i} \sum_{\substack{(u,v)\in (B/T_i)^{1/2}\calB(\bx_0,L)\\ \gcd(u,v)=1}} \sum_{N(\fa)=F(u,v)} \left|\sum_{\substack{N(\ft) \sim_\epsilon T_i \\ \gcd(\ft,\overline{\ft})=1 \\ \gcd(\fa,\overline{\ft})=1}} \psi(\ft)^2\right|.$$

\begin{lemma}\label{lem: cancellation of cusp over t}
    For $T\gg_{K} 1$, and $\psi^2\neq 1$, the following bound holds:
    $$\sum_{\substack{N(\ft) \leq T \\ \gcd(\ft,\overline{\ft})=1 \\ \gcd(\fa,\overline{\ft})=1}} \psi(\ft)^2 \ll_K \prod_{\fp\mid \fa} \left(1+\frac{1}{N(\fp)^{2/3}}\right) T^{5/6}.$$
\end{lemma}
\begin{corollary}
Let $\epsilon(B) = o(1)$. Then we have that $$\sum_{\substack{N(\ft) \sim_\epsilon T \\ \gcd(\ft,\overline{\ft})=1 \\ \gcd(\fa,\overline{\ft})=1}} \psi(\ft)^2 \ll_K \prod_{\frakp\mid \fa} \left(1+\frac{1}{N(\fp)^{2/3}}\right)T^{5/6}.$$
\end{corollary}
The above corollary comes from noting that since $\epsilon(B) = o(1)$, 
\begin{multline*}\sum_{\substack{N(\ft) \sim_\epsilon T \\ \gcd(\ft,\overline{\ft})=1 \\ \gcd(\fa,\overline{\ft})=1}} \psi(\ft)^2 
 = \sum_{\substack{N(\ft) \leq (1+\epsilon) T \\ \gcd(\ft,\overline{\ft})=1 \\ \gcd(\fa,\overline{\ft})=1}} \psi(\ft)^2 
 - \sum_{\substack{N(\ft) < T \\ \gcd(\ft,\overline{\ft})=1 \\ \gcd(\fa,\overline{\ft})=1}} \psi(\ft)^2 
 \\ \ll_K \prod_{\frakp \mid \fa} \left(1+\frac{1}{N(\fp)^{2/3}}\right)((1+\epsilon)T)^{5/6} + \prod_{\frakp \mid \fa} \left(1+\frac{1}{N(\fp)^{2/3}}\right)T^{5/6} \ll_K \prod_{\frakp \mid \fa} \left(1+\frac{1}{N(\fp)^{2/3}}\right)T^{5/6}.\end{multline*}

\begin{proof}[Proof of Lemma \ref{lem: cancellation of cusp over t}]
    Define the Dirichlet series: 
    $$\xi_\fa(s,\psi^2) := \sum_{\substack{\fb\subset \calO_K \\ \gcd(\fb,\overline{\fb})=1 \\ \gcd(\ba,\overline{\fb})=1}} \frac{\psi(\fb)^2}{N(\fb)^s}.$$
    This series converges absolutely for $\Re(s)>1$. We note that $$\xi_\fa(s,\psi^2) = \prod_{\fp\mid \ba} \left(1-\frac{\psi(\fp)^2}{N(\fp)^s}\right) \cdot \xi(s,\psi^2),$$
    where we define 
    $$\xi(s,\psi^2) : = \sum_{\substack{\fb\subset \calO_K \\ \gcd(\fb,\overline{\fb})=1}} \frac{\psi(\fb)^2}{N(\fb)^s}.$$
    We can relate $\xi(s,\psi^2)$ to a Hecke $L$-function as such: $$\xi(s,\psi^2) \cdot \sum_{\substack{I\subset \calO_K \\ I\textrm{ principal}}}\frac{1}{N(I)^s} = \xi(s,\psi^2) \zeta(2s) = L_K(s,\psi^2) := \sum_{\fb\subset \calO_K} \frac{\psi(\fb)^2}{N(\fb)^s},$$
    where $L_K(s,\psi^2)$ is a Hecke L-function. 
    Thus, we have that $$\xi_\fa(s,\psi^2) = f_\ba(s) L_K(s,\psi^2),$$
    where $$f_\ba(s) = \zeta(2s)^{-1}\cdot  \prod_{\fp\mid \fa} \left(1-\frac{\psi(\fp)^2}{N(\fp)^s}\right).$$
    Consequently, $\xi_\ba(s,\psi^2)$ is analytic for $\Re(s)>1/2$ (since $\psi^2\neq 1$, $L_K(s,\psi^2)$ converges at $s=1$).

    Let us consider the following weighted sum: 
    $$S(T) = \sum_{\substack{N(\ft)\leq T \\ \gcd(\ft,\overline{\ft})=1 \\ \gcd(\fa,\overline{\ft})=1}} \psi(\ft)^2 (T-N(\ft)).$$
    By Perron's formula, we have that for any $\sigma>1$ the following holds:
    $$S(T) = \frac{1}{2\pi i}\int_{\sigma-i\infty}^{\sigma+i\infty}f_\fa(s) \zeta(2s) L_K(s,\psi^2) \frac{T^{1+s}}{s(s+1)} ds.$$
    Take $1/2<\alpha<1$. Let us shift the contour above to the line $\Re(s)=\alpha$. Since $\psi^2\neq 1$, we know that $L(s,\psi^2)$ will not have a pole at $s=1$. Thus, we have that the above is 
    \begin{align*}S(T) &= \frac{1}{2\pi i} \int_{\alpha-i\infty}^{\alpha+i\infty}f_\ba(s) \zeta(2s) L_K(s,\psi^2)\frac{T^{1+s}}{s(s+1)}ds \\
    &\ll_\alpha \prod_{\fp\mid \fa} \left(1+\frac{1}{N(\fp)^\alpha}\right) T^{1+\alpha} \int_{\alpha-i\infty}^{\alpha+i\infty} |L_K(s,\psi^2)| \frac{ds}{s(s+1)}. 
    \end{align*}

    Now, for $\Re(s)=\alpha$, the convexity bound gives us that $$|L_K(s,\psi^2)|\ll_K |s|^{\frac{1-\alpha}{2}+\epsilon}.$$
    So, we have that the inner integral is bounded above by $$\int_{-\infty}^\infty (\alpha + |t|)^{\frac{1-\alpha+\epsilon}{2} -2 } dt <\infty,$$
    which converges to a constant depending on $\alpha$. Thus, we have that 
    \begin{equation}\label{eq: bound on ST weighted ideal}S(T) \ll_\alpha \prod_{\fp\mid \fa} \left(1+\frac{1}{N(\fp)^\alpha}\right) T^{1+\alpha}.\end{equation}

    Next, we use the following trick. Consider the difference $$S(T+T') - S(T) = T'\sum_{\substack{N(\ft)\leq T \\ \gcd(\ft,\overline{\ft})=1\\ \gcd(\fa,\overline{\ft})=1}} \psi(\ft)^2 + O_K\left(T'^2\right)$$
    where the error term comes from the ideals with $N(\ft)\in (T,T+T']$ and the trivial bound $|\psi^2|\leq 1$. Applying the above bound (\ref{eq: bound on ST weighted ideal}), we also have that $$S(T+T')-S(T) \ll_\alpha \prod_{\fp\mid \fa} \left(1+\frac{1}{N(\fp)^\alpha}\right) (T+T')^{1+\alpha}.$$
    Let us take $\beta = \frac{1+\alpha}{2}$; since $1/2<\alpha<1$, we have that $3/4<\beta < 1$. Then we can see that by comparing the above bounds and taking $T'=T^\beta$, 
    \begin{align*}
       \sum_{\substack{N(\ft)\leq T \\ \gcd(\ft,\overline{\ft})=1\\ \gcd(\fa,\overline{\ft})=1}} \psi(\ft)^2 &\ll T' + \prod_{\fp\mid \fa} \left(1+\frac{1}{N(\fp)^\alpha}\right) T^{1+\alpha}T'^{-1} \\
       &\ll \prod_{\fp\mid \fa} \left(1+\frac{1}{N(\fp)^\alpha}\right) T^\beta .
    \end{align*}
    This brings us to our final choices of $\alpha$ and $\beta$. Let us choose $\alpha=2/3$ and $\beta = 5/6.$ These provide us with the final bound. 
    
\end{proof}

Now, let us plug this into our estimates for $D(X_{\Delta,f},\calB(\bx_0,L);\psi)$. This gives us that: 
$$D(X_{\Delta,f},\calB(\bx_0,L);\psi) \leq \sum_{T_i} T_i^{5/6} \sum_{\substack{(u,v)\in (B/T_i)^{1/2}\calB(\bx_0,L) \\ \gcd(u,v)=1}} (1\star \chi)(F(u,v)) \cdot \prod_{p\mid F(u,v)} (1+1/p^{2/3}).$$
Applying Nair's sieve (Theorem \ref{thm: Nair sieve}) to the inner sum, we get that 
\begin{align*}
    \sum_{(u,v)\in (B/T_i)^{1/2}\calB(\bx_0,L)} &(1\star \chi)(F(u,v)) \prod_{p\mid F(u,v)} (1+p^{-2/3}) \\
    &\ll \frac{BL^2}{T_i} \exp\left(\sum_{p\ll B/T_i} \frac{\varrho_f(p)}{p}\cdot \left((1\star \chi)(p)\cdot (1+p^{-2/3}) - 1\right)\right)\\
    &\ll \frac{BL^2}{T_i} \exp\left((\rho_{\Delta,f}-2)\log\log(B/T_i) + \sum_{p\ll B/T_i} \frac{2\varrho_f(p)}{p^{5/3}}\right) \\
    &\ll \frac{BL^2}{T_i}\log(B)^{\rho_{\Delta,f}-2}.
\end{align*}
Thus, we have that 
\begin{align*}
    D(X_{\Delta,f},\calB(\bx_0,L);\psi) \ll L^2B\log(B)^{\rho_{\Delta,f}-2} \sum_{T_i} T_i^{-1/6}. 
\end{align*}
We note that this inner sum is a geometric series, since $T_i$ is summed over powers of $(1+\epsilon(B))$ up to $B/\exp(\sqrt{\log(B)})$: 
$$\sum_{T_i} T_i^{-1/6} = \sum_{k=0}^{K} (1+\epsilon(B))^{-k/6} < \frac{1}{1-(1+\epsilon(B))^{-1/6}},$$
where $K = (\log(B)-\sqrt{\log(B)})/\log(1+\epsilon(B)).$ Now we know that $$\frac{1}{1-(1+\epsilon(B))^{-1/6}} \ll \epsilon(B)^{-1}.$$
So, we get that $$D(X_{\Delta,f},\calB(\bx_0,L);\psi) \leq \epsilon(B)^{-1} B L^2 \log(B)^{\rho_{\Delta,f}-2}.$$

Finally, we choose $\epsilon(B)$. Let us take $\epsilon(B) = \log(B)^{-1/2}$. So, we get that 
$$D(X_{\Delta,f},\calB(\bx_0,L);\psi) \ll BL^2 \log(B)^{\rho_{\Delta,f}-3/2}.$$
On the other hand, by Lemma 2.2, $$|E(X_{\Delta,f},\calB(\bx_0,L);0,\psi)-D(X_{\Delta,f},\calB(\bx_0,L);\psi)|\ll LB\log(B)^{\rho_{\Delta,f}-3/2}.$$
Hence, 
\begin{equation*}
    E(X_{\Delta,f},\calB(\bx_0,L);0,\psi) \ll L B\log(B)^{\rho_{\Delta,f}-3/2} + BL^2 \log(B)^{\rho_{\Delta,f}-3/2}.
\end{equation*}
Since $L \gg \log(B)^{-10^{-1}}$, we can also safely say that $$E(X_{\Delta,f},\calB(\bx_0,L);0,\psi) \ll L^2 B\log(B)^{\rho_{\Delta,f}-1 - 1/3}.$$
This concludes the proof of Proposition \ref{prop: error term E}.\qed

\subsection{Interpretation with cusp forms}\label{subsec: dihedral interpret cusp form}
Now, let us reframe the above proof using the cusp form interpretation of our sum $E(X_{\Delta,f},\calB(\bx_0,L);0,\psi).$ We consider the expression after our change of variables (Lemma \ref{lem: change of var}), $$D(X_{\Delta,f},\calB(\bx_0,L);\psi) = \sum_{T_i} \sum_{\substack{(u,v)\in (B/T_i)^{1/2}\calB(\bx_0,L) \\ \gcd(u,v)=1}} \sum_{\substack{N(\ft^2\fa)=t^2F(u,v) \\ \gcd(\ft,\overline{\ft})=1\\ \gcd(\fa,\overline{\ft})=1\\ N(\ft)\sim_\epsilon T_i}} \psi(\ft^2\fa).$$
Let $\Xi$ denote the cuspidal automorphic representation on $\GL_2(\A_\Q)$ given by $\psi$; since $\psi$ is a class group character, the representation $\Xi$ has finite image. 
By replacing the condition $\gcd(\ft,\overline{\ft})=1$ with a M\"obius sum over divisors of $N(\ft)$, the above is equivalent to 
\begin{multline*}\sum_{T_i} \sum_{\substack{(u,v)\in (B/T_i)^{1/2}\calB(\bx_0,L)\\ \gcd(u,v)=1}} \sum_{\substack{t\sim_{\epsilon}T_i}} \sum_{d\mid t}\mu(d) \lambda_{\Xi}\left(\frac{t^2F(u,v)}{d^2}\right)\\ 
=\sum_{T_i} \sum_{\substack{(u,v)\in (B/T_i)^{1/2}\calB(\bx_0,L)\\\gcd(u,v)=1}}\sum_{d\leq T_i} \mu(d) \sum_{\substack{t\sim_\epsilon T_i/d}} \lambda_\Xi(t^2 F(u,v)).\end{multline*}
We will return to this sum at the end of the section, but for now let us consider the sum $$\sum_{d\leq X} \mu(d) \sum_{t\leq X/d} \lambda_\Xi(at^2).$$

Let $C(z)$ denote a holomorphic cusp form in $S_k(\Gamma_0(-\Delta),\chi)$ and $\lambda_C(n)$ denote its normalized Fourier coefficients. Let us take $c$ to be a constant such that $|\lambda_C(p)|\leq c$ for all primes $p$. We claim that the following bound holds, \textit{independent of whether of not $C(z)$ is dihedral}.

\begin{lemma}\label{lem: bound on cusp over d and t}
Let $C(z)\in S_k(\Gamma_0(-\Delta),\chi)$ and assume that $|\lambda_C(p)|\leq c$ for all primes $c$. Then we have the bound $$\sum_{d\leq X} \mu(d) \sum_{t\leq X/d} \lambda_C(at^2) \ll (c+1)^{\omega(a)} X \log(X)^{-40c}.$$
\end{lemma}
\begin{remark}
    The exponent of $40c$ is fairly arbitrary, but we choose a number for concreteness that will suffice in future bounds. 
\end{remark}
\begin{proof}
    We define the weighted sum:
    $$S(X) := \sum_{d\leq X} d\mu(d) \sum_{t\leq X/d} \lambda_C(at^2) (X/d - t).$$
    From Perron's formula, we know that for $\sigma>1$,
    $$S(X) = \frac{1}{2\pi i}\sum_{d\leq X} d\mu(d)\int_{\sigma-i\infty}^{\sigma+i\infty} \xi(s;a,C) \cdot \frac{(X/d)^{s+1}}{s(s+1)}ds,$$
    where we define the Dirichlet series $$\xi(s;a,C) := \sum_{n}\frac{\lambda_C(an^2)}{n^s}.$$
    We now relate the above Dirichlet series to the $L$-function $L(s,\sym^2(C))$. Define $\alpha_p$ and $\beta_p$ such that $$1-\lambda_C(p) p^{-s} + \chi(p)p^{-2s} = (1-\alpha_p p^{-s})(1-\beta_p p^{-s}).$$
    Let us define $$g_a(s) := \prod_{p^e\| a} \left(\lambda_C(p^e) + \lambda_C(p^{e+1})p^{-s}+...\right)\left(1-\alpha_p^2 p^{-s}\right)\left(1-\beta_p^2 p^{-s}\right)\left(1-\chi(p) p^{-s}\right).$$
    Then we have the relation for $\Re(s)>1$:
    $$\xi(s;a,C) = g_a(s) L(2s,\chi^2)^{-1} L(s,\sym^2(C)).$$
    In particular, $\xi(s;a,C)$ and $L(s,\sym^2(C))$ share the same order of the pole at $s=1$. We refer the reader to Shimura's \cite{Shimura} classificaton of when $L(s,\sym^2(C))$ has a pole at $s=1$; in our proof, it will turn out that there is cancellation independent of if $C(z)$ is dihedral or not. 

    Let us bring back the M\"obius sum. 
    We can see that $$S(X) = \frac{1}{2\pi i} \int_{\sigma-i\infty}^{\sigma+i\infty} \xi(s;a,C) \frac{X^{s+1}}{s(s+1)}\cdot \sum_{d\leq X} \frac{\mu(d)}{d^s} ds.$$
    Take $\sigma = 1+\frac{1}{\log(X)}>1$, we know that for any $A>0$, $$\sum_{d\leq X} \frac{\mu(d)}{d^{s+1}} = \zeta(s)^{-1} \left(1+O(\log(X)^{-A})\right).$$
    Thus, we have that for any $A'>0$,
    $$S(X) = \frac{1}{2\pi i} \int_{\sigma-i\infty}^{\sigma+i\infty} g_a(s)L(2s,\chi^2)^{-1} L(s,\sym^2(C))\zeta(s)^{-1} \frac{X^{s+1}}{s(s+1)}ds + O\left(\sup_{\Re(s) = \sigma}|g_a(s)| \cdot X^2 \log(X)^{-A'}\right).$$

    Next, we want to add cutoffs at $\sigma+iD$ and $\sigma-iD$ for $D$ to be determined later. This gives us that the integral above is equal to
    \begin{equation*}
         \frac{1}{2\pi i} \int_{\sigma-iD}^{\sigma+iD} \xi(s;a,C) \zeta(s)^{-1} \frac{X^{s+1}}{s(s+1)}ds + O\left(\sup_{\Re(s)=\sigma} |g_a(s)|\cdot \frac{X^2\log(X)^{2}}{D}\right)
    \end{equation*}

    Finally, let us shift our line of integration to $\Re(s) = \alpha$ for some $\alpha>1-\log(D)^{-8}$; in other words, we stay within a zero free region for $\zeta(s)$. We note that if $L(s,\sym^2(C))$ has a pole at $s=1$, then it is a simple pole and will be cancelled by the zero of $\zeta(s)^{-1}$. If $L(s,\sym^2(C))$ is holomorphic at $s=1$, then again $\xi(s;a,C)\zeta(s)^{-1}$ will be holomorphic at $s=1$. Thus, we get the formula 
    $$S(X) = O\left(\sup_{\substack{\Im(s) = \pm D\\ \alpha \leq \Re(s)\leq \sigma}}|g_a(s)|\cdot \frac{X^2\log(X)^2}{D}\right)+ O\left(X^{2-\log(D)^{-8}} \int_{\alpha-iD}^{\alpha+iD} \frac{|\xi(s;a,C)||\zeta(s)|^{-1} }{|s||s+1|}ds\right).$$

    Recall that the convexity bound \cite[(5.8),(5.20)]{IwaniecKowalski} gives us that 
    $$|\xi(s;a,C)| \ll |g_a(s)| (|s|+1)^{1-\alpha}.$$
    We also know from \cite[(3.6.5)]{Titchmarsh} that since $\Re(s) = \alpha > 1-\log(D)^{-8}$, we have that 
    $$|\zeta(s)|^{-1} \ll \log(|s|+3)^7.$$
    Thus, in this region, $$\int_{\alpha-iD}^{\alpha+iD} \frac{|\xi(s;a,C)||\zeta(s)|^{-1}}{|s||s+1|}ds \ll \sup_{\Re(s)=\alpha} |g_a(s)|.$$
    Thus, we have that 
    $$S(X) \ll \sup_{\alpha\leq \Re(s)\leq \sigma}|g_a(s)| \cdot \left( \frac{X^2\log(X)^2}{D} + X^{2-\log(D)^{-9}} \right).$$
    Take $D= \log(X)^{101c}$ and $A' = 101c$. Consequently, we have that \begin{equation}\label{eq: cusp form bound on weighted sum}S(X) \ll \sup_{\alpha\leq \Re(s)\leq \sigma} |g_a(s)| X^2 \log(X)^{-99c}.\end{equation}

    Finally, we need to remove our weight. Let $Y < X$ be determined later and consider 
    $$S(X+Y)-S(X) = \sum_{d\leq X+Y} d\mu(d) \sum_{t\leq (X+Y)/d} \lambda_C(at^2) ((X+Y)/d-t) - \sum_{d\leq X} d\mu(d) \sum_{t\leq X/d} \lambda_C(at^2) (X/d-t).$$
    Rearranging the expression, we get that 
    \begin{multline*}
        S(X+Y)-S(X) = Y\sum_{d\leq X} \mu(d) \sum_{t\leq X/d} \lambda_C(at^2) \\
        + \sum_{d\leq X}d\mu(d) \sum_{X/d<t\leq (X+Y)/d} \lambda_C(at^2) ((X+Y)/d-t) \\
        + \sum_{X<d\leq X+Y} d\mu(d) \sum_{t\leq (X+Y)/d} \lambda_C(at^2) ((X+Y)/d-t).
    \end{multline*}
    Since $Y<X$, the last sum only contains $t=1$ and hence is bounded by $$\ll |g_a(1)| X \cdot \left|\sum_{X\leq d\leq X+Y} \mu(d)\right|.$$

    Consider the second sum. Note that since $t>X/d$, we have that $(X+Y)/d-t \leq Y/d$. So, the second sum is bounded by $$Y\sum_{d\leq X} \sum_{X/d< t\leq (X+Y)/d} |\lambda_C(at^2)| \ll c^{\omega(a)}Y^2 \log(X)^c,$$
    where $c$ is the bound on the prime values of $|\lambda_C(p)|$. 

    Thus, we know that 
    \begin{multline*}
        \sum_{d\leq X} \mu(d) \sum_{t\leq X/d} \lambda_C(at^2) = Y^{-1} (S(X+Y)-S(X)) \\ + O\left(|g_a(1)| X Y^{-1} \left|\sum_{X\leq d\leq X+Y} \mu(d)\right|\right) + O\left(c^{\omega(a)} Y \log(X)^c\right).
    \end{multline*}
    Applying (\ref{eq: cusp form bound on weighted sum}), we have that 
    \begin{multline*}
        \sum_{d\leq X}\mu(d) \sum_{t\leq X/d} \lambda_C(at^2) \ll \sup_{\alpha\leq \Re(s)\leq \sigma} |g_a(s)| \cdot Y^{-1} X^2 \log(X)^{-99c} \\ + |g_a(1)| \cdot Y^{-1} X \left|\sum_{X\leq d\leq X+Y} \mu(d)\right| + c^{\omega(a)} Y\log(X)^{c}.
    \end{multline*}
    We take $Y = X\log(X)^{-42c}.$ We also note that by definition of $g_a(s)$, we know that $$\sup_{\alpha\leq \Re(s)\leq \sigma} |g_a(s)| \ll (c+1)^{\omega(a)}.$$ This gives us that 
    $$\sum_{d\leq X} \mu(d) \sum_{t\leq X/d} \lambda_C(at^2)  \ll  (c+1)^{\omega(a)} X\log(X)^{-40c}.$$

\end{proof}

Finally, let us return to our original sum. Let us consider the non-short interval version for simplicity, but as in Lemma \ref{lem: change of var} and Lemma \ref{lem: cancellation of cusp over t}, the above can be modified for our slightly-short intervals. We consider the sum: 
$$\sum_{\substack{|u|,|v|\leq B^{1/2}\\ \gcd(u,v)=1}} \sum_{t\leq B/\max(|u|,|v|)^2} \sum_{d\mid t} \mu(d) \lambda_\Xi\left(\frac{t^2F(u,v)}{d^2}\right).$$
At this point, we use the fact that if $\Xi$ is induced from a class group character, then $c=1$. Moreover, we can calculate that $$\sup_{\alpha\leq \Re(s)<\sigma} |g_p(1)| \ll \begin{cases}
    1+p^{-\alpha}, & \chi(p)=1, \\ 
    0, & \chi(p) = -1.
\end{cases}$$
We note that we can assume that $\alpha> 1/2$ in the computations above, as this was the lower bound for the zero-free region. 
Thus, the bounds in the proof of Lemma \ref{lem: bound on cusp over d and t} give us a more precise upper bound: 
$$\ll \sum_{|u|,|v|\leq B^{1/2}} \frac{B \log(B/\max(|u|,|v|)^2)^{-40}}{\max(|u|,|v|)^2} \cdot (1\star \chi)(F(u,v)) \cdot \prod_{\substack{p\mid F(u,v)\\ \chi(p)=1}} (1+p^{-2/3}).$$
Now, let us split $[0,B^{1/2}]$ into dyadic regions indexed by $B_i$. The diagonal contribution is clearly the largest above, so we have the upper bound \begin{multline*}\sum_{B_i} \frac{B}{B_i^2} \log(B/B_i^2)^{-40}\sum_{|u|,|v|\sim B_i} (1\star \chi)(F(u,v)) \prod_{\substack{p\mid F(u,v)\\ \chi(p)=1}} (1+p^{-2/3})\\ 
\ll B\log(B)^{\rho_{\Delta,f}-2} \sum_{B_i} \log(B/B_i^2)^{-40} \leq  B\log(B)^{\rho_{\Delta,f}-2} \sum_{n=1}^{\log(B)} n^{-40} \ll B\log(B)^{\rho_{\Delta,f}-2}.\end{multline*}
Compared to the main term of $\log(B)^{\rho_{\Delta,f}-1}$, this is negligible. We end this section by remarking that the process of changing the order of summation in Lemma \ref{lem: change of var} produces a weaker error term due to the extra factor of $\epsilon(B)$, but this process can be similarly done with this perspective.

\section{The cuspidal contribution -- the Gr\"ossen case} \label{sec: cuspidal abs}
In this section, we will prove Proposition \ref{prop: error term A}. Let us recall the definition of $A(X_{\Delta,f},\calB(\bx_0,L);h,\psi)$ (as defined in \ref{eq: def indef A h}):
\begin{equation*}
    A(X_{\Delta,f},\calB(\bx_0,L);h,\psi) = \sum_{\substack{N(\ft)\leq B/\exp(\sqrt{\log(B)}) \\ \gcd(\ft,\overline{\ft})=1}}\sum_{\substack{(u,v)\in (B/t)^{1/2}\calB(\bx_0,L) \\ \gcd(u,v)=1}} \sigma(F(u,v), \Psi^h \psi; \overline{\ft}),
\end{equation*}
where we have defined $\sigma(n,\psi;t)$ as
\begin{equation*}
    \sigma(n,\psi;\ft) := \left|\sum_{\substack{N(\fa) = n\\ \gcd(\fa,\ft)=1}}\psi(\fa)\right|.
\end{equation*}
Recall also that $\psi$ is any class group character and that $h\geq 1$. Moreover, $\Psi$ (defined in Definition \ref{def: fundamental character}) is a Hecke Gr\"ossencharacter of infinite order. 

Now, $\sigma(n;\Psi^h\psi,\overline{\ft})$ is a multiplicative, non-negative function, and bounded by $d(n)$. Thus, we know that by Theorem \ref{thm: Nair sieve}, 
\begin{align*}
    A(X_{\Delta,f},\calB(\bx_0,L);h,\psi) &\leq \sum_{\substack{N(\ft)\leq B/\exp(\sqrt{\log(B)})\\ \gcd(\ft,\overline{\ft})=1}} \sum_{(u,v)\in (B/t)^{1/2}\calB(\bx_0,L)} \sigma(F(u,v); \Psi^h \psi, \overline{\ft})\\
    &\ll  L^2 \sum_{N(\ft)\leq B} \frac{B}{N(\ft)} \prod_{i=1}^r \prod_{p\ll B/t} \left(1+\frac{\varrho_{f_i}(p)(\sigma(p;\Psi^h\psi,\overline{\ft})-1)}{p}\right).
\end{align*}
Here we have used the observation that for all but finitely many primes, $\frac{\varrho_F(p)}{p^2} = \frac{\varrho_f(p)}{p} + O(p^{-2}).$
From now on, let us fix an irreducible factor $g(z) = f_i(z)$; we study $$\prod_{p\ll B/t} \left(1+\frac{\varrho_{g}(p)(\sigma(p;\xi,\overline{\ft})-1)}{p}\right),$$
where we take $\xi =\Psi^h \psi.$

We note that $\sigma(p;\xi,\overline{\ft})=0$ if $\chi(p) = -1$, as in that case there are no prime ideals $\frakp$ such that $N(\frakp)=p$. Additionally, we know that $t$ is product of primes that split in $K$, so we must have that $\gcd(t,p)=1$ for any $p$ such that $\chi(p)=-1.$ Motivated by these choices, we can split the above product as:
\begin{equation}\label{eq: 4 product expansion}\prod_{\substack{p\ll B/t\\ \chi(p)=-1}} \left(1-\frac{\varrho_g(p)}{p}\right)\prod_{\substack{p\ll B/t \\ \chi(p)=1\\ \gcd(p,t)=1}} \left(1+\frac{\varrho_g(p)(\sigma(p;\xi,\overline{\ft})-1)}{p}\right) \prod_{\substack{p\ll B/t\\ p\mid t}} \left(1+\frac{\varrho_g(p)(\sigma(p;\xi,\overline{\ft})-1)}{p}\right).\end{equation}

Consider the primes $p$ such that $p\mid t$. In this case $$\sigma(p;\xi,\overline{\ft}) = \left|\sum_{\substack{N(\frakp)=p\\ \gcd(\frakp,\overline{\ft})=1}}\xi(\frakp)\right| = |\xi(\frakp)| = 1.$$
Here we have used that since $p\mid t$ and $\gcd(\ft,\overline{\ft})=1$, we must have that exactly one of the primes $\frakp$ above $p$ satisfies $\gcd(\frakp,\overline{\ft})=1$. So, the contribution from these primes is given by: 
$$\prod_{\substack{p\ll B/t\\ p\mid t}} \left(1+ \frac{\varrho_g(p)(\sigma(p;\xi,\overline{\ft})-1)}{p}\right)= \prod_{\substack{p\ll B/t \\ p\mid t}}1 = 1.$$

Next, we consider the primes $p$ satisfying that $\chi(p)=1$ and $p\nmid t$. Since $\gcd(p,t)=1$, we know that $$\sigma(p;\xi,\overline{\ft}) = \sigma(p;\xi, 1) = |\xi(\frakp)+\overline{\xi}(\frakp)|= 2|\cos(\theta_p)|,$$
where $\frakp$ is a prime over $p$ and $\theta_p\in [0,2\pi]$ is an angle. By the following identity on cosine, $$2|\cos(\theta)|\leq \frac{3}{2} + \frac{1}{2}\cos(2\theta),$$
we know that $$\sigma(p;\xi,\overline{\ft}) \leq \frac{3}{2} + \frac{1}{4}(\xi(\frakp)^2+\overline{\xi}(\frakp)^2).$$
Thus, we can rewrite the second product as 
\begin{equation*}
    \prod_{\substack{p\ll B/t\\ \chi(p)=1 \\ \gcd(p,t)=1}} \left(1+\frac{\varrho_g(p)(\sigma(p;\xi,\overline{\ft})-1)}{p}\right) \ll \prod_{\substack{p\ll B/t \\ \chi(p)=1}} \left(1+\frac{3\varrho_g(p)}{2p} + \sum_{N(\frakp)=p} \frac{\varrho_g(p)\xi(\frakp)^2}{N(\frakp)} - \frac{\varrho_g(p)}{p} \right).
\end{equation*}

The above analysis tells us that (\ref{eq: 4 product expansion}) can be simplified to the following expression:
\begin{equation}\label{eq: 4 product over primes}
    \prod_{p\ll B/t} \left(1+\frac{\varrho_g(p) (\sigma(p;\xi,\overline{\ft})-1)}{p}\right) \ll \exp\left(\sum_{\substack{p\ll B/t}} \frac{-\varrho_g(p)}{p} + \sum_{\substack{p\ll B/t \\ \chi(p)=1}} \frac{3\varrho_g(p)}{2p} + \sum_{\substack{p\ll B/t\\ N(\fp)=p}} \frac{\varrho_g(N(\frakp))\xi^2(\frakp)}{N(\frakp)}\right).
\end{equation}
In the third summation above, we observe that for all but finitely many primes, the existence of a prime ideal $\fp$ satisfying that $N(\fp)=p$ implies that $\chi(p)=1$.

Since $g(z)$ is irreducible, we have that $$\sum_{p\ll B/t} \frac{\varrho_g(p)}{p} = \log\log(B/t)+O(1).$$
On the other hand, $$\sum_{\substack{p\ll B/t \\ \chi(p)=1}} \frac{3\varrho_g(p)}{2p} = \frac{3}{4}\log\log(B/t) + \frac{3}{4}\cdot \mathbf{1}_{\sqrt{-\Delta}\in \Q[z]/g(z)}\cdot \log\log(B/t) + O(1). $$
The above follows from rewriting the above sum as $$\sum_{p\ll B/t} \frac{3\varrho_g(p)}{4p} + \sum_{p\ll B/t} \frac{3\chi(p)\varrho_g(p)}{4p}$$
and using Corollary \ref{cor: order of pole is varrho} to evaluate the second sum.

It remains to handle the final sum in (\ref{eq: 4 product over primes}). We observe that if $\sqrt{-\Delta}\not\in \Q[z]/g(z)$, then $$\varrho_g(p) = \#\{\frakP\subset \calO_{K_g}: N_{K_g/K}(\frakP)= \frakp\}$$
for any prime $\frakp$ over $p$. Otherwise, $K\subset \Q[z]/g(z)$ and since $[K:\Q]=2$, we have that $$\varrho_g(p) = 2\cdot \#\{\frakP\subset \calO_{K_g}: N_{K_g/K}(\frakP) = \frakp\}$$
for any prime $\frakp$ over $p$. Therefore, we can rewrite this sum as 
$$\sum_{N(\frakp)\ll B/t} \frac{\varrho_g(N(\frakp))\xi^2(\frakp)}{N(\frakp)} \ll  \sum_{\substack{\frakP\subset \calO_{K_g}\\ N(\frakP)\ll B/t}} \frac{\xi^2(N_{K_g/K}(\frakP))}{N(\frakP)} + O(1).$$
Here the factor of $O(1)$ comes from a sum over primes $\frakP$ such that $N_{K/\Q}(\frakP)=p^e$ for $e>1$.

We consider the character $\xi^2 \circ N_{K_g/K}$ on $K_g$; this will be a Hecke character. Since $\Psi$ has infinite order, we can see that since $h\geq 1$, $\Psi^{2h}\psi^2 \circ N_{K_g/K}$ is a nontrivial Hecke character. We can also view $\xi^2\circ N_{K_g/K}$ as the base change of $\xi^2$ to $K_g$. Since $\xi^2\circ N_{K_g/K}$ is a nontrivial Hecke character on $K_g$, the prime ideal theorem with Gr\"ossencharacters (\cite[Theorem 5.13]{IwaniecKowalski}) states the following: for some constant $c>0$, 
$$\sum_{N(\fraka)\leq X} \Lambda(\fraka) \Psi^h(\fraka)\psi(\fraka) \ll X \exp\left(-c\frac{\log(X)}{\sqrt{\log(X)}+\log(h)}\right)\log(xh)^4.$$
Taking $X = B/t$ and $h\leq \log(B)^5$, we consequently have that $$ \sum_{\substack{\frakP\subset \calO_{K_g}\\ N(\frakP) \ll B/t}} \frac{\xi^2(\frakP)}{N(\frakP)} = o(\log\log(B/t)).$$

Consequently, we have that 
\begin{multline*}
    \prod_{p\ll B/t}\left(1+\frac{\varrho_g(p)(\sigma(p;\xi,\overline{\ft})-1)}{p}\right) \ll \exp\left((-\frac{1}{4}+ \frac{3}{4}(\varrho_{\Delta,g}-2))\cdot \log\log(X) + O(1) + o(\log\log(B/t)\right)\\
    \ll \log(B/t)^{\varrho_{\Delta,g}-2 - 1/5}.
\end{multline*}
Finally, we sum over ideals $\ft$: 
\begin{multline*}
    A(X_{\Delta,f},\calB(\bx_0,L);h,\psi) \ll L^2 \sum_{\substack{N(\ft)\ll B/\exp(\sqrt{\log(B)})\\ \gcd(\ft,\overline{\ft})=1}} \frac{B}{N(\ft)} \cdot \log(B/t)^{\rho_{\Delta,f}-2-1/5} \ll L^2 B\log(B)^{\rho_{\Delta,f}-1-1/5}.
\end{multline*}
This completes the proof of Proposition \ref{prop: error term E}.\qed

\begin{example}[A counterexample when $\xi$ has finite order]
In this example, we show an example of when the bound of Proposition \ref{prop: error term A} fails for a character of finite order. 
Let $\Delta=23$ and thus $K = \Q(\sqrt{-23})$. Take $f(z) = z^3-z-1$. Then $f(z)$ generates the Hilbert class field of $\Q(\sqrt{-23})$, which has class group $\Z/3\Z$. Let $\frakp\subset \calO_K$ be a prime ideal such that $N(\fp) \mid F(u,v)$ for some $(u,v)\in (B/t)^{1/2}\calB(\bx_0,L)$. Consider $\Frob_\fp \in \Gal(K_f/K) \cong \Z/3\Z$. We note that $$\#\{\frakP\subset \calO_{K_f}: N(\frakP) = \fp\} = \textrm{Tr}(\Frob_\fp) = \begin{cases}
        3, & \Frob_\fp = 1\\ 
        0, & \textrm{otherwise.}
    \end{cases}$$
    In other words, $\fp$ splits in $K_f$ if and only if $\fp$ is a principal ideal. However, if $N(\fp)\mid F(u,v)$, then we must be in the first case since then $\#\{\frakP\subset \calO_{K_f}:N(\frakP)=\fp\}>0$; in conclusion, for any prime $\fp$ such that $N(\fp)\mid F(u,v)$, $\fp$ is a principal ideal. Thus, we have that for \textit{any class group character $\psi$}, 
    $$\sum_{N(\fa) = F(u,v)} \psi(\fa) = \sum_{N(\fa)=F(u,v)} 1.$$
    In this case, we get that 
    $$\sum_{(u,v)\in (B/t)^{1/2}\calB(\bx_0,L)} \sum_{N(\fa)=F(u,v)}\psi(\fa) = \sum_{(u,v)\in (B/t)^{1/2}\calB(\bx_0,L)} \sum_{N(\fa)=F(u,v)}1,$$
    is exactly the same as the contribution from the Eisenstein part (when $\psi=1$, since if the class group is $\Z/3\Z$ there are no other genus characters).

\end{example}

\begin{remark}
    We observe that in order for our argument above for Proposition \ref{prop: error term A} to achieve a polylogarithmic savings, it suffices for $\xi^2 \circ N_{K_g/K}$ to be nontrivial for every irreducible factor $g(z)$. If $\xi$ has infinite order, this is immediate. For most pairs $(\Delta,f(z))$, the above condition also holds for $\xi$ of finite order. Let $H$ be the Hilbert class field of $K$. The base change of a class group character $\xi^2$ to $K_g$ will also be nontrivial if $K_g \cap H = K.$ 
\end{remark}

\begin{remark}
    As mentioned in \S\ref{subsec: interpret cusp}, we can also interpret $A(X_{\Delta,f},\calB(\bx_0,L);h,\psi)$ as a correlation sum of Fourier coefficients of cusp forms along polynomial values. Let $\Xi$ be a cuspdial automorphic representation of $\GL_2(\A_\Q)$. We can use the argument above to study $$\sum_{n\leq X} \lambda_\Xi(p(n))$$
    for $p(x)\in \Z[x]$ a polynomial. Again, our problem reduces to analytic properties of the base change of $\sym^2(\Xi)$; in particular, one hopes that it is cuspidal. This problem will be studied further in forthcoming work. 
\end{remark}

\section{The Eisenstein contribution}\label{sec: Eisenstein}
For the next four sections, we will work towards the proof of Proposition \ref{prop: main term}. Recall from (\ref{eq: pos def main term}) that we define 
\begin{equation*}
    M_{q_1,q_2}(X_{\Delta,f},\calB(\bx_0,L)) = \sum_{\substack{N(\fk)\leq B/\exp(\sqrt{\log(B)}) \\ \gcd(\fk,\overline{\fk})=1}}\mu(\fk) \sum_{\substack{N(\ft)\leq B/k\exp(\sqrt{\log(B)}) \\ \gcd(\ft\fk,\overline{\ft\fk})=1}} S(X_{\Delta,f},\calB(\bx_0,L);q_1,q_2,t,k), 
\end{equation*}
where we define 
\begin{equation}\label{eq: def S with fixed t}
    S(X_{\Delta,f},\calB(\bx_0,L); q_1,q_2,t,k) = \sum_{\substack{(x,y)\in (B/kt)^{1/2}\calB(\bx_0,L) \\ \gcd(x,y)=1}} \varepsilon_{q_1,q_2}(F(x,y)/k).
\end{equation}
We proceed to analyze $S(X_{\Delta,f},\calB(\bx_0,L);q_1,q_2,t,k)$ in this section and in \S\ref{sec: asymptotic} we resume the summation over $\fk$ and $\ft$. Let us record our final goal for $S(X_{\Delta,f},\calB(\bx_0,L);q_1,q_2,t,k):$
\begin{prop}\label{prop: estimate of S(q_1,q_2,t,k)}
Let $F(x,y)$ be a squarefree binary form of degree $\leq 4$ and $L\gg \log(B)^{-10^{-10}}.$ Then we have that for $t$ and $k$ satisfying $tk\leq B/\exp(\sqrt{\log(B)}),$ the following holds:
\begin{multline*}
S(X_{\Delta,f},\calB(\bx_0,L);q_1,q_2,t,k)= \frac{BL^2}{tk^3} \cdot C_{q_1,q_2,f,k}(B/tk) \cdot  \left(1+O\left(\log(B)^{-10^{-7}}\right)\right) \\+O\left(\frac{ 64^{\omega(k)}\log(k)B}{tk^{3/2}}\cdot \log(B)^{\rho_{\Delta,f}-2-10^{-7}} + \frac{B^{3/4}}{(tk)^{3/4}} \cdot \left(L+(4\deg(F))^{\omega(k)}\nu_1(k)^{1/2}\right)\cdot \log(B)^{-10^{-7}}\right),\end{multline*}
where $C_{q_1,q_2,f,k}(B/tk)$ is an expression that depends on the factorization of $F(x,y)$.
Additionally, for a fixed value of $k$, there exists a constant $c_{q_1,q_2,f,k}$ such that $$C_{q_1,q_2,f,k}(B/tk) = c_{q_1,q_2,f,k}\log(B/tk)^{\rho_{\Delta,f}-2}.$$ 
\end{prop}
\begin{remark}
    At this stage the proof is essentially a combination of the proofs provided by de la Bret\`eche and Tenenbaum in \cite{delaTenenbaum-Manin} (for irreducible quartics and products of two irreducible quadratics) and by Heath-Brown in \cite{HB-linear} (for the product of four linear factors) for $\Delta=1$. In essence, we have reduced ourselves to case when $K$ has class number one.
\end{remark}

First, we must apply a few necessary reductions to put the sum in a more malleable form. We recall our abusive notation that lets $-\Delta$ denote the modulus of the character $\chi$, and is given explicitly by 
\begin{equation*}
    -\Delta:= \begin{cases}
        4\Delta, & \Delta>0,\\
        -\Delta, & \Delta<0.
    \end{cases}
\end{equation*}

\subsection{Multiplicativity of $\varepsilon_{q_1,q_2}$}\label{subsec: multiplicativity}
We recall from (\ref{eq: def eps q_1,q_2}) that when $(n,-\Delta)=1$,$$\varepsilon_{q_1,q_2}(n) = \sum_{\substack{d\mid n}} \chi_{q_1}(d)\chi_{q_2}(n/d).$$
This function is multiplicative, but not totally multiplicative. We would like to write: 
$$\varepsilon_{q_1,q_2}(F(x,y)) \approx \prod_{i=1}^n\varepsilon_{q_1,q_2}(F_i(x,y)),$$
but the above relation is not quite true. Instead, we can use the following multiplicativity relation. 

\begin{lemma}\label{lem: multiplicativity of eisenstein}
Let $n$ and $m$ be integers coprime to $-\Delta$. Then $$\varepsilon_{q_1,q_2}(nm) = \sum_{\substack{c\mid \gcd(m,n)}} \mu(c)\chi(c) \varepsilon_{q_1,q_2}(n/c) \varepsilon_{q_1,q_2}(m/c).$$    
\end{lemma}
\begin{proof}
    Since $\varepsilon_{q_1,q_2}$ is multiplicative, it suffices to check this property on prime powers. Let $p$ be a prime. So, we can assume that $\gcd(p,-\Delta)=1.$ Then expanding the definition, we have that $$\varepsilon_{q_1,q_2}(p^{e+f}) = \sum_{i=0}^{e+f} \chi_{q_1}(p)^{i} \chi_{q_2}(p)^{e+f-i}.$$
    On the other hand, we have that $$\varepsilon_{q_1,q_2}(p^e)\varepsilon_{q_1,q_2}(p^f) = \sum_{i=0}^{e+f} \#\{k\leq e, \ell \leq f: k+\ell = i\} \cdot \chi_{q_1}(p)^i \chi_{q_2}(p)^{e+f-i}.$$
    Now, because $\chi(p) = \chi_{q_1}(p)\chi_{q_2}(p)$, the following equality is true: $$\varepsilon_{q_1,q_2}(p^e)\varepsilon_{q_1,q_2}(p^f) - \chi(p) \varepsilon_{q_1,q_2}(p^{e-1}) \varepsilon_{q_1,q_2}(p^{f-1}) = \varepsilon_{q_1,q_2}(p^{e+f}).$$
\end{proof}

Thus, if $\gcd(F(x,y)/k,-\Delta)=1$, we can write:
$$\varepsilon_{q_1,q_2}(F(x,y)/k) = \sum_{\bc} \mu(\bc) \chi(\bc) \prod_{i=1}^r \varepsilon_{q_1,q_2}(F_i(x,y)/k_ic_i),$$
where $k=k_1...k_r$ and $c_ik_i\mid F_i(x,y)$ for each $i$, and $\bc$ is given by tuples of products of integers dividing $\gcd(F_i(x,y),F_j(x,y))$ for $i\neq j.$ These breakdowns depend on the factorization of $F(x,y)$ (for instance, if $F(x,y)$ is irreducible then there is no need for the above sum). We write down these decompositions explicitly in \S\ref{subsec: proof of Prop S(q_1,q_2,t,k)}.

Moreover, since we will take $\gcd(x,y)=1$, we have that $\gcd(F_i(x,y),F_j(x,y)) = \Res(F_i,F_j)$, the resultant of $F_i$ and $F_j$, for $i\neq j$. Note that $\Res(F_i,F_j)$ is a finite constant only depending on our polynomial $F(x,y)$. Furthermore, this restricts the sum above to those $\bc$ satisfying that $c_i \mid \prod_{i\neq j} \Res(F_i,F_j)$. This fact also allows us to determine that $k$ decomposes into finitely many combinations $k_1\hdots k_r$ where $k_i\mid F_i(x,y)$ -- we choose one without loss of generality and write it as $\bk = (k_1,...,k_r).$  
Plugging the above into (\ref{eq: def S with fixed t}), we get that:
\begin{equation}\label{eq: sum over c in mult}S(X_{\Delta,f},\calB(\bx_0,L);q_1,q_2,t,k) = \sum_{\bc} \mu(\bc)\chi(\bc) S(\bc,\bk),\end{equation}
where we define:
\begin{equation}\label{eq: def of S(c)}
    S(\bc,\bk):=\sum_{\substack{(x,y)\in (B/kt)^{1/2}\calB(\bx_0,L) \\ \gcd(x,y)=1}} \prod_{i=1}^r \varepsilon_{q_1,q_2}(F_i(x,y)/c_ik_i)
\end{equation}

\subsection{Interpretation $\varepsilon_{q_1,q_2}$ as $(1\star \chi)$}\label{subsec: eisenstein to dirichlet convolution}
Next, we would like to make precise our claim that $$\varepsilon_{q_1,q_2}(n)\approx (1\star \chi)(n).$$
As we have alluded to before, these two functions are morally equivalent. 

First, we make the observation that if $n$ satisfies that $p\mid n\implies p\mid -\Delta$, then 
$$\varepsilon_{q_1,q_2}(n) = \sum_{N(\fa)=n} \chi_{q_1,q_2}(\fa) = \chi_{q_1}(p_{q_2}(n)) \chi_{q_2}(p_{q_1}(n)) \mathbf{1}_{\exists \fa: N(\fa)=n}.$$
Since $\varepsilon_{q_1,q_2}(n)$ is multiplicative, we can always write 
$$\varepsilon_{q_1,q_2}(n) = \varepsilon_{q_1,q_2}(p_{-\Delta}(n)) \varepsilon_{q_1,q_2}(p_{\neg-\Delta}(n)) = (\chi_{q_1}(p_{q_2}(n)) \chi_{q_2}(p_{q_1}(n)) \mathbf{1}_{\exists \fa: N(\fa)=p_{-\Delta}(n)}) \varepsilon_{q_1,q_2}(p_{\neg-\Delta}(n)).$$

Recall that $\chi_{q_1}\chi_{q_2} = \chi$ and that these are quadratic characters. For $\gcd(n,-\Delta)=1$, we have the relation: 
\begin{align*}\varepsilon_{q_1,q_2}(n) &= \sum_{\substack{d\mid n}} \chi_{q_1}(n/d)\chi_{q_2}(d) =
\sum_{d\mid n} \chi_{q_1}(n)\chi_{q_1}\chi_{q_2}(d) \\
&= \chi_{q_1}(n) \sum_{d\mid n} \chi(d) 
= \chi_{q_1}(n) (1\star \chi)(p_{\neg-\Delta}(n)).
\end{align*}

We can thus reexpress (\ref{eq: def of S(c)}) as:
\begin{multline*}S(\bc,\bk) = \sum_{\substack{(x,y)\in (B/tk)^{1/2}\calB(\bx_0,L) \\ \gcd(x,y)=1 \\ k_ic_i\mid F_i(x,y)}} \epsilon_{q_1,q_2}(p_{-\Delta}(F(x,y)/kc)) \chi_{q_1}(p_{\neg-\Delta}(F(x,y)/kc))\prod_{i=1}^r (1\star \chi)(p_{\neg -\Delta}(F_i(x,y)/k_ic_i))
\\ = \sum_{\substack{n \ll (B/tk)^2 \\ p\mid n\implies p\mid -\Delta}} \chi_{q_1}(p_{q_2}(n))\chi_{q_2}(p_{q_1}(n)) \mathbf{1}_{\exists \fa: N(\fa) = n}\sum_{\substack{a_i\bmod -\Delta\\ \gcd(a_i,-\Delta)=1}} \prod_{i=1}^r\chi_{q_1}(a_i)\\ \times \sum_{\substack{(x,y)\in (B/tk)^{1/2}\calB(\bx_0,L)\\ \gcd(x,y)=1 \\ k_ic_i\mid F_i(x,y) \\ p_{-\Delta}(F(x,y)/kc)=n \\ p_{\neg -\Delta}(F_i(x,y)/k_ic_i)\equiv a_i\bmod -\Delta}} \prod_{i=1}^r (1\star\chi) (p_{\neg -\Delta}(F_i(x,y)/k_ic_i)).\end{multline*}
Note that since $k =N(\fk)$ where $\gcd(\fk, \overline{\fk})=1$, we must have that $\gcd(k,-\Delta)=1$. 

Let us bound away the contribution from the large $n$. 
\begin{lemma}\label{lem: large n_u}
    Let $F(x,y)$ be a squarefree binary form of degree $\leq 4$. Define for $n$ such that $p\mid n\implies p\mid -\Delta$ and a vector of residues $\ba$ mod $-\Delta$: $$U(n, \ba, \bc,\bk) := \sum_{\substack{(x,y)\in (B/tk)^{1/2}\calB(\bx_0,L) \\ k_ic_i\mid F_i(x,y) \\ p_{-\Delta}(F(x,y)/c)= n \\ p_{\neg -\Delta}(F_i(x,y)/k_ic_i) \equiv a_i \bmod -\Delta}} \prod_{i=1}^r (1\star \chi)(p_{\neg -\Delta}(F_i(x,y)/k_ic_i)).$$
    Then we have that $$U(n, \ba, \bc,\bk) \ll_{\Delta,f} \frac{B L^2\log(B/t)^{\rho_{\Delta,f}-2}(8\deg(F))^{\omega(k)}}{tk^2n}.$$
\end{lemma}
\begin{proof}
    Note that since $(1\star \chi)(n)$ is a non-negative function,  $$U(n,\ba,\bc,\bk)\leq \sum_{n = n_{1}... n_{r}}\sum_{\substack{(x,y)\in (B/tk)^{1/2}\calB(\bx_0,L) \\ k_in_{i}\mid F_i(x,y)}}\prod_{i=1}^r (1\star \chi)(F_i(x,y)/k_ic_i).$$
    Here, we have used the fact that $(1\star \chi)(dp_{\neg u}(n))) = (1\star\chi)(p_{\neg u}(n))$ for any $d$ satisfying that if $p\mid d$ then $p\mid -\Delta$. 
    We apply Theorem \ref{thm: Nair sieve}, we have that the inner sum above is bounded by $$\frac{BL^2}{tk}\cdot  \log(B/tk)^{\rho_{\Delta,f}-2} \cdot \prod_{i=1}^r \frac{(8\deg(F))^{\omega(n_{i}k_i)}}{n_{i}k_i}\ll \frac{BL^2\log(B/tk)^{\varrho_{\Delta,Q}-2}(8\deg(F))^{\omega(k)}}{tk^2n}.$$
    Here we have used Corollary \ref{cor: order of pole is varrho} and that $\omega(n_{i})\leq \omega(\Delta)$ for any value of $n_{i}$. 
\end{proof}

Consequently, we can bound the contribution from $n>N$: 
$$\sum_{n>N}\sum_{\ba\bmod -\Delta} U(n,\ba,\bc,\bk) \ll \frac{BL^2\log(B/tk)^{\rho_{\Delta,f}-2}(8\deg(F))^{\omega(k)}}{tk^2} \cdot  \sum_{\substack{n>N\\ p\mid n\implies p\mid -\Delta}}\frac{1}{n}.$$
Since the sum over $n$ satisfying that $p\mid n\implies p\mid -\Delta$ is a geometric series, we can see that 
\begin{equation}\label{eq: sum over large n}
    \sum_{n>N}\sum_{\ba\bmod -\Delta} U(n,\ba,\bc,\bk) \ll \frac{BL^2\log(B/tk)^{\rho_{\Delta,f}-2}(8\deg(F))^{\omega(k)}}{Ntk^2}.
\end{equation}
For $N= \log(B/tk)^{10^{-7}},$ this error term is sufficient for Propositon \ref{prop: main term}. So, from now on we can assume that $n\leq \log(B/tk)^{10^{-7}} \ll \log(B)^{10^{-7}}$ (since $tk\leq B/\exp(\sqrt{\log(B)})$). 
\begin{remark}
    Since we assume that $n\leq \log(B)^{10^{-7}}$ still grows as $B\rightarrow\infty$, we must be somewhat careful in tracking the dependency of future bounds on $n$. However, $n$ is still quite small, so we can allow for any reasonable polynomial dependency on $n$.
\end{remark}

Finally, let us define a new function for a fixed $n$ and tuple $\ba$ satisfying that $\gcd(a_i,-\Delta)=1$ for all $i$:
\begin{equation}\label{eq: def of S(n_1,n_2,r,a,c)}
    S(n,\ba,\bc,\bk)  = \sum_{\substack{(x,y)\in (B/tk)^{1/2}\calB(\bx_0,L) \\ \gcd(x,y)=1 \\ k_ic_i\mid F_i(x,y) \\ p_{q_i}(F(x,y)/c)=n_i \\ p_{\neg -\Delta}(F_i(x,y)/k_ic_i) \equiv a_i \bmod -\Delta}} \prod_{i=1}^r (1\star \chi)(p_{\neg -\Delta}(F_i(x,y)/k_ic_i)).
\end{equation}

\subsection{Removing the $\gcd$ condition}\label{subsec: gcd condition}
Now, we will remove the condition that $\gcd(x,y)=1$ using M\"obius inversion. In particular, we know that 
$$S(n,\ba,\bc,\bk) = \sum_{b\leq (B/tk)^{1/2}} \mu(b) \sum_{\substack{(x,y)\in (B/tk)^{1/2}\calB(\bx_0,L) \\ b\mid x,y \\ k_ic_i\mid F_i(x,y) \\ p_{q_i}(F(x,y)/c)=n_i \\ p_{\neg -\Delta}(F_i(x,y)/k_ic_i) \equiv a_i \bmod -\Delta}} \prod_{i=1}^r (1\star \chi)(p_{\neg -\Delta}(F_i(x,y)/k_ic_i)).$$

We can bound away the contribution given by large values of $b$:
\begin{lemma}\label{lem: large B}
    Let $F(x,y)$ be a squarefree binary form of degree $\leq 4$. Define for a fixed $b$: 
    $$U(n,\ba,\bc,\bk; b) := \sum_{\substack{(x,y)\in (B/kt)^{1/2}\calB(\bx_0,L) \\ b\mid x,y \\ k_ic_i\mid F_i(x,y) \\ p_{q_i}(F(x,y)/c)= n_i \\ p_{\neg -\Delta}(F_i(x,y)/k_ic_i)\equiv a_i \bmod -\Delta}} \prod_{i=1}^r (1\star \chi)(p_{\neg -\Delta}(F_i(x,y)/k_ic_i)).$$
    Then for any $B\geq X\gg 1$, we have that 
    $$\sum_{\substack{X\leq b\ll (B/kt)^{1/2}}}\mu^2(b) U(n,\ba,\bc,\bk;b) \ll_{\Delta, F} \frac{BL^2\log(B/kt)^{\rho_{\Delta,f}-2}(8\deg(F))^{\omega(k)}}{Xtk^2}.$$ 
\end{lemma}
\begin{proof}
    Again, by the nonnegativity of $(1\star\chi)(n)$, we have that $$U(n, \ba,\bc,\bk; b) \leq \sum_{\substack{(x,y)\in b^{-1}(B/tk)^{1/2}\calB(\bx_0,L)}} \prod_{i=1}^r (1\star \chi) (b^{\deg(F_i)} F_i(x,y)/k_ic_i).$$
    Now we apply Lemma \ref{lem: multiplicativity of eisenstein} to separate the contribution from $b^{\deg(F_i)}$ and $F_i(x,y)/k_ic_i$: 
    $$\sum_{e_i\mid b} \mu(e_i) \chi(e_i) \prod_{i=1}^r(1\star \chi)(b^{\deg(F_i)}/e_i) \sum_{\substack{(x,y)\in b^{-1}(B/tk)^{1/2}\calB(\bx_0,L)\\ e_i\mid F_i(x,y)}} \prod_{i=1}^r (1\star \chi)(F_i(x,y)/k_ic_ie_i).$$
    After using Theorem \ref{thm: Nair sieve}, we have that the above is bounded by $$\sum_{e_i\mid b}\mu^2(e_i) \prod_{i=1}^r (1\star \chi)(b^{\deg(F_i)}/e_i) \cdot \frac{BL^2\log(B/tk)^{\rho_{\Delta,f}-2}(8\deg(F))^{\omega(k)}}{b^2tk^2}\cdot \prod_{i=1}^r\frac{(8\deg(F))^{\omega(e_i)}}{e_i}.$$\\ 

    Summing over $b\geq X$, we get the expression \begin{multline*}\frac{BL^2\log(B/tk)^{\rho_{\Delta,f}-2}(8\deg(F))^{\omega(k)}}{tk^2} \sum_{b\geq X} \frac{\mu^2(b)}{b^2}\cdot \prod_{i=1}^r \sum_{e_i\mid b} \mu^2(e_i)(1\star \chi)(b^{\deg(F_i)}/e_i) \cdot \frac{(8\deg(F))^{\omega(e_i)}}{e_i} \\ 
    \ll \frac{BL^2\log(B/tk)^{\rho_{\Delta,f}-2}\deg(F)^{\omega(k)}}{tk^2} \sum_{b\geq X} \mu^2(b)\cdot \frac{a_b}{b^2},
    \end{multline*}
    where we define:
    $$a_b:= \prod_{i=1}^r \prod_{p\mid b} \left(1+ \frac{8\deg(F)\cdot (1\star \chi)(b^{\deg(F_i)}/p)}{p}\right).$$
    This function $a_b$ is multiplicative in $b$, so we relate this sum to a Dirichlet series with an Euler product expansion. 
    If we consider the singular series: 
    \begin{multline*}\gamma(s) := \sum_{b=1}^\infty \mu^2(b)\cdot \frac{a_b}{b^s} = \prod_{p}\left(1+p^{-s}\prod_{i=1}^r\prod_{p\mid b} \left(1+ \frac{8\deg(F)\cdot (1\star \chi)(b^{\deg(F_i)}/p)}{p}\right) \right)\\
    = \prod_{p} \left(1+p^{-s} + O_{F,\chi}(p^{-1-s})\right).
    \end{multline*}
    Thus, we can see that the singular series will converge at $s=2$ and in fact the convergence is at the same rate as the zeta function. We can now estimate that $$\sum_{b\geq X} \mu^2(b)\cdot \frac{a_b}{b^2}\ll \frac{1}{X}.$$

    Combining this with the earlier bounds, we get that:
    \begin{multline*}
    \frac{BL^2\log(B/tk)^{\rho_{\Delta,f}-2}(8\deg(F))^{\omega(k)}}{tk^2}\sum_{b\geq X} \frac{\mu^2(b)}{b}\cdot \prod_{i=1}^r \sum_{e_i\mid b} \mu^2(e_i)(1\star \chi)(b^{\deg(F_i)}/e_i) \cdot \frac{(8\deg(F))^{\omega(e_i)}}{e_i} \\  \ll \frac{BL^2\log(B/tk)^{\rho_{\Delta,f}-2}(8\deg(F))^{\omega(k)}}{tk^2X}.
    \end{multline*}
    
\end{proof}

If we take $X=\log(B)^{10^{-7}}$, we will get a sufficient error term for Propostion \ref{prop: main term}. So from now on, we assume that $b\leq \log(B)^{10^{-7}}$. Again, we must track how all future error terms depend on $b$, since it does grow with $B$; fortunately, since $b$ is small in comparison with $B$, we can allow any reasonable polynomial dependency on $b$. 

We define a new function to be estimated for a fixed $b$:
\begin{equation}\label{eq: def S(n_1,n_2,a,c,b)}
    S(n,\ba,\bc,\bk;b) = \sum_{\substack{(x,y)\in (B/tk)^{1/2}\calB(\bx_0,L) \\ b\mid x,y \\ k_ic_i\mid F_i(x,y)\\ p_{q_j}(F(x,y)/c)=n_{j}\\ p_{\neg -\Delta}(F_i(x,y)/k_ic_i) \equiv a_i\bmod -\Delta}} \prod_{i=1}^r (1\star \chi)(p_{\neg -\Delta}(F_i(x,y)/k_ic_i)).
\end{equation}

\subsection{Expanding the Dirichlet convolution}\label{subsec: expand dirichlet convolution}
In our final reduction, we will expand the Dirichlet convolution $(1\star \chi)(n)$. We note that for $\ba\bmod -\Delta$ with $\gcd(a_i,-\Delta)=1$ and $p_{\neg-\Delta}(F_i(x,y)/k_ic_i)\equiv a_i\bmod -\Delta$,
\begin{align*}
    (1\star \chi)(p_{\neg -\Delta}(F_i(x,y)/k_ic_i)) &= \sum_{\substack{d_i\mid p_{\neg -\Delta}(F_i(x,y)/k_ic_i) }} \chi(d_i) \\
    &= \sum_{\substack{d_i\mid p_{\neg -\Delta}(F_i(x,y)/k_ic_i)\\ d_i\leq \sqrt{p_{\neg -\Delta}(F_i/k_ic_i)}}}\chi(d_i) + \sum_{\substack{d_i\mid p_{\neg -\Delta}(F_i(x,y)/k_ic_i)\\ d_i\leq \sqrt{p_{\neg -\Delta}(F_i/k_ic_i)}}}\chi(a_i/d_i)\\
    &= (1+\chi(a_i))\sum_{\substack{d_i\mid p_{\neg -\Delta}(F_i(x,y)/k_ic_i)\\ d_i\leq \sqrt{p_{\neg -\Delta}(F_i/k_ic_i)}}}\chi(d_i).
\end{align*} 
We remark that if $p_{\neg -\Delta} (F_i(x,y)/k_ic_i)$ is a square, then the second equality should be modified to be a sum over $d_i\leq \sqrt{p_{\neg -\Delta}}$ and $d_i<\sqrt{p_{\neg -\Delta}}.$ However, the contribution from those $(x,y)$ such that $p_{\neg -\Delta}(F_i(x,y)/k_ic_i)$ is a square can be shown to be negligible via the large sieve. From now on, we ignore these terms. 
Let us define: 
\begin{equation}\label{eq: def S d_i}
    S(n,\ba,\bc,\bk,b; \bd) = \sum_{\substack{(x,y)\in (B/tk)^{1/2}\calB(\bx_0,L)\\ b\mid x,y\\ k_ic_id_i\mid F_i(x,y)\\ p_{q_j}(F(x,y)/c)= n_{j} \\ p_{\neg -\Delta}(F_i(x,y)/k_ic_i)\equiv a_i\bmod -\Delta}} 1.
\end{equation}
Then we can rewrite (\ref{eq: def of S(n_1,n_2,r,a,c)}) as:
$$S(n,\ba,\bc,\bk;b) = \prod_{i=1}^r (1+\chi(a_i))\sum_{d_i\leq \max_{(x,y)\in (B/tk)^{1/2}\calB(\bx_0,L)}\sqrt{|F_i(x,y)|/n_{i}k_ic_i}} \prod_{i=1}^r \chi(d_i) \cdot S(n,\ba,\bc,\bk,b;\bd).$$
Our next goal is to establish either asymptotics on average or upper bounds for $S(n,\ba,\bc,\bk,b;\bd).$ To do so, we will have different treatments for various ranges of $\bd$. For ``small'' $\bd$ (the precise definition of this will be given in \S\ref{subsec: proof of Prop S(q_1,q_2,t,k)}, as it depends on the factorization of $F(x,y)$), we use a level of distribution result established in \S\ref{sec: LoD} -- this will provide our main term. For ``large'' $\bd$, we will prove upper bounds in \S\ref{sec: large moduli} using Nair's sieve that will go into our final error term. 

\begin{remark}\label{rem: size of variables}
    Since there are many variables in the final expression we wish to estimate $S(n,\ba,\bc,\bk;b)$, let us remark now on their approximate ``size'' and what aspects need to be tracked throughout the following proof. 
    \begin{itemize}
        \item $n$ is a small constant ($n\ll \log(B)^{10^{-7}}$; it also satisfies that $p\mid n\implies p\mid -\Delta$ and hence $\sum n^{-e}$ for any $e>0$ will converge to a constant only depending on $-\Delta$. Its importance appears in the computations for evaluating the main term. 
        \item $\ba$ appear with respect to congruences mod $n(-\Delta)$. This aspect is a nuisance when establishing the error term but is relevant for evaluating the main term. 
        \item $\bc$ has entries that must divide their respect resultants $\Res(F_i,F_j)$ for $i\neq j$. Since the pairwise resultants are fixed constants given a form $F$, the dependency on $\bc$ is not needed to be tracked for the error term. However, tracking the dependence on $\bc$ is crucial for evaluating the main term. When $F$ is irreducible, it is not necessary to sum over $\bc$ (i.e. $\bc = 1).$
        \item $\bk$ satisfies that $k\leq B/\exp(\sqrt{\log(B)})$ and its dependence must be tracked throughout the error term computations, as we later sum over $k$ (recall that by construction, for each $k$ there is a unique vector chosen as $\bk$). It is not as relevant for the main term computation. 
        \item $b$ is a small constant $(b\leq \log(B)^{10^{-7}})$ and its dependence must be tracked throughout the error term computations. However, it is not particularly relevant for the main term computations. 
    \end{itemize}
\end{remark}

\section{Level of distribution}\label{sec: LoD}
In this section, we will prove a level of distribution result for squarefree binary forms of degree $\leq 4$ alluded to in the paragraphs above. Let us write it out now in its full technical detail: 

\begin{prop}\label{prop: level of distribution}
    Let $F(x,y) = \prod_{i=1}^r F_i(x,y)$ be a squarefree binary form of degree $\leq 4$, where $F_i(x,y)$ are irreducible factors of $F(x,y)$. For a fixed $n,\ba,\bc,\bk, \mathbf{d}=(d_i)$ and $b$, we define the following local function: 
    \begin{multline*}\varrho_{F}(\mathbf{d};b, n, \ba,\bc,\bk) := \#\{\bx\bmod bcdkn(-\Delta): k_ic_id_i\mid F_i(x,y),   b\mid \gcd(x,y),\\  p_{-\Delta}(F(x,y)) = n, p_{\neg -\Delta}(F_i(x,y)/k_ic_i)\equiv  a_i\bmod -\Delta\}.\end{multline*}
    For a tuple $\bD = (D_i)_{i=1}^r$, we define the error quantity as: 
    $$E(\bD) = \sum_{d_i\sim D_i} \left|S(n,\ba,\bc,\bk, b; \bd) - \frac{\varrho_{F}(\bd;b,n,\ba,\bc,\bk)}{(bcdkn(-\Delta))^2}\cdot \frac{BL^2}{tk}\right|.$$
    Let $D =D_1\dots D_r$. If $F(x,y)$ has no linear factors, then we have that for an explicit constant $\gamma_r>0$, the following bound holds:
    \begin{multline*}
        E(\bD) \ll \frac{\varrho_F^*(\bk)\tau(k)}{\varphi(k)} \cdot D \cdot \log\left(1+\frac{BL^2}{tkD}\right)^9 +
        \frac{64^{\omega(k)}}{k}\cdot \frac{D}{b^2}\cdot \exp(\gamma_r \sqrt{\log_2(B/tk)\log_3(B/tk)})\\
        + \frac{B^{1/2}L}{t^{1/2}}\cdot \frac{64^{\omega(k)}}{k} \cdot \frac{\tau(b^4)}{b} \cdot \sqrt{D} \cdot \exp(\gamma_r\sqrt{\log_2(kcD)\log_3(kcD)}).
    \end{multline*}
    If $F(x,y)$ has a linear factor, then we have that 
    \begin{multline*}E(\bD) \ll \frac{\varrho_F^*(\bk)\tau(k)}{\varphi(k)} \cdot D \cdot \log\left(1+\frac{BL^2}{tkD}\right)^9 +
        \frac{64^{\omega(k)}}{k}\cdot \frac{D}{b^2}\cdot \exp(\gamma_r \sqrt{\log_2(B/tk)\log_3(B/tk)})\\
        + \frac{B^{1/2}L}{t^{1/2}}\cdot \frac{64^{\omega(k)}}{k} \cdot \frac{\tau(b^4)}{b} \cdot \sqrt{D} \cdot \exp(\gamma_r\sqrt{\log_2(kcD)\log_3(kcD)}) \\ 
        + \frac{B^{1/2}L}{(tk)^{1/2}}\cdot \frac{1}{b}\cdot \max_{i:\deg(F_i)=1} D_i\log(kcD)^2 \log\left(1+\frac{BL^2}{tkD}\right)^8.
    \end{multline*}
\end{prop}
\begin{remark}
    The constant $\gamma_r$ is worked out explicitly in the work of Hall and Tenenbaum \cite{HallTenenbaum-DeltaR} on generalized Hooley's $\Delta$-functions. 
\end{remark}

Our proof when $F(x,y)$ has no linear factors follows the proof of the level of distribution result written in de la Bret\`eche and Tenenbaum in \cite{delaTenenbaum-Manin}. In this case, the level of distribution is near optimal in the sense that we can take $D \leq \frac{BL^2}{tk} \log(B)^{-\epsilon}$ for any $\epsilon>0$; in other words, we can let $\epsilon$ be very small, such as $10^{-5}$. If $F(x,y)$ has linear factors, then we get a slightly worse level of distribution result -- we will need $D_i \leq \frac{BL^2}{tk}\log(B)^{-1-\epsilon}$ in the final term. Luckily, our results for large moduli are more flexible in this case. We also would like to point out the work of Marasingha \cite{Marasingha-almostprimes} towards level of distribution results for arbitrary squarefree binary forms. 

\subsection{Lattices of $F(x,y)$}\label{subsec: lattices}
First, we will explore the distribution of lattices corresponding to the divisors of $F(x,y)$ and how they connect to our error term. We start by fixing a tuple $\bd = (d_i)_{i=1}^r$. For a collection of squarefree binary forms $F_1(x,y),...,F_k(x,y)$, we define the following divisor lattices: 
\begin{equation}\label{eq: Lambda(d_1,...,d_k)}\Lambda(d_1,...,d_k) := \{(x,y)\in \Z^2: d_i \mid F_i(x,y)\}.\end{equation}
\begin{equation}\label{eq: Lambda(d_1,...,d_k; b)}
    \Lambda(d_1,...,d_k;b) := \{(x,y)\in \Z^2: d_i\mid F_i(x,y), b\mid x,y\}.
\end{equation}
It is noteworthy that these will be the union of lattices of determinant $d_1...d_n$ or $b^2d_1...d_n$; this observation was made by Daniel in \cite{Daniel}. 
From now on, we use the shorthand notation:
$$\Lambda(\bc\bk\bd;b) := \Lambda(c_1k_1d_1,\hdots, c_rk_rd_r;b).$$

We also will want to consider the lattice: 
\begin{equation}
    \Lambda^*(d_1,\dots,d_k;b) = \{(x,y)\in \Z^2: d_i\mid F_i(x,y), b\mid \gcd(x,y), \gcd(x,y,d)=1\}.
\end{equation}
Similarly, we use the shorthand notation that 
\begin{equation*}
    \Lambda^*(\bc\bk\bd;b) := \Lambda^*(c_1k_1d_1,\dots,c_rk_rd_r;b).
\end{equation*}

Next, we further decompose $\Lambda(\bc\bk\bd; b)$ and $\Lambda^*(\bc\bk\bd;b)$ into lattices corresponding to solutions of $F_i(x,y) \bmod k_ic_id_i$. 
Let $\alpha = (\alpha_1,\alpha_2)$ satisfy that $F_i(\alpha) \equiv 0 \bmod k_ic_id_i$ for all $i$, $b\mid \alpha_1,\alpha_2$ and that $\gcd(\alpha_1,\alpha_2,kcd) = 1.$ Then we define $$\Lambda(\alpha; ckd,b) := \{\bx\in \Z^2: \bx \equiv \lambda \alpha \bmod bcdk \text{ for some $\lambda\in \Z$}\}.$$
It follows that $\Lambda(\alpha;ckd,b)$ is a lattice with determinant $b^2cdk$. We will also consider 
\begin{equation*}
    \Lambda^*(\alpha;ckd;b) := \{(x,y)\in \Z^2: (x,y)\equiv \lambda \alpha \bmod bcdk, \gcd(x,y,ckd)=1\}.
\end{equation*}
Let us write $$\calU(\bc\bk\bd;b) = \{\alpha \bmod bcdk: F_i(\alpha)\equiv 0 \bmod k_ic_id_i, b\mid (\alpha_1,\alpha_2), \gcd(\alpha_1,\alpha_2,kcd)=1\}/ \sim,$$
where equivalence is defined as $\alpha\sim \alpha'$ if and only if $\alpha = \lambda \alpha' \bmod kcd$ for $\lambda\neq 0$.
Then we can see that 
\begin{equation*}
    \#(B/tk)^{1/2}\calB(\bx_0,L) \cap \Lambda^*(\bc\bk\bd;b) = \sum_{\alpha \in \calU(\bc\bk\bd;b)} \#(B/tk)^{1/2}\calB(\bx_0,L) \cap \Lambda^*(\alpha;ckd,b).
\end{equation*}

We also must specify the conditions that $p_{- \Delta}(F(x,y)/c)=n$ and $p_{\neg -\Delta}(F_i(x,y)/k_ic_i)\equiv a_i\bmod -\Delta$. Define $$\Lambda(n,\ba):= \{\bx\in \Z^2: p_{-\Delta}(F(x,y)/c) = n, p_{\neg -\Delta}(F_i(x,y)/k_ic_i)\equiv a_i\bmod -\Delta\}.$$
Observe that the above is not a lattice, despite the notation, but is still defined by congruence conditions mod $n(-\Delta)$ in $\Z^2$ and will be the shift of a finite union of lattices with determinant $n^2(-\Delta)^2$. 
Then, we can see that we would like to count in Proposition \ref{prop: level of distribution}
\begin{equation}\#(B/tk)^{1/2}\calB(\bx_0,L) \cap \Lambda(\bc\bk\bd;b) \cap \Lambda(n,\ba).\end{equation}
It will be advantageous for us to first consider 
\begin{equation*}
    \#(B/tk)^{1/2}\calB(\bx_0,L) \cap \Lambda^*(\bc\bk\bd;b)\cap \Lambda(n,\ba) = \sum_{\alpha\in \calU(\bc\bk\bd;b)} \#(B/tk)^{1/2}\calB(\bx_0,L) \cap \Lambda^*(\alpha;ckd;b) \cap \Lambda(n,\ba).
\end{equation*}

\subsection{A step along the way} 
We first want to consider following sum: 
\begin{equation}\label{eq: def S*(n,a,c,k,b;d)}
    S^*(n,\ba,\bc,\bk,b;\bd) := \sum_{\substack{(x,y)\in (B/tk)^{1/2}\calB(\bx_0,L) \\ b\mid x,y\\ k_ic_id_i\mid F_i(x,y) \\ p_{-\Delta}(F(x,y)/c) = n \\ p_{\neg 
    -\Delta}(F_i(x,y)/k_ic_i)\equiv a_i\bmod -\Delta}}\mathbf{1}_{\gcd(x,y,kcd)=1}. 
\end{equation}
We also define the corresponding local counting function
\begin{multline}\label{eq: def local count with gcd}
    \varrho^*_F(\bd;b,n,\ba,\bc,\bk) := \#\{\bx\bmod bcdkn(-\Delta): k_ic_id_i\mid F_i(x,y), \gcd(x,y,kcd)=1, \\ 
    b\mid (x,y), p_{-\Delta}(F(x,y)/c)=n, p_{\neg -\Delta}(F_i(x,y)/k_ic_i)\equiv a_i\bmod -\Delta\}. 
\end{multline}
Another way to view this count $\varrho^*_F(\bd;b,n,\ba,\bc,\bk)$ is that it is $\varphi(kcd)\cdot \#\calU(\bc\bk\bd;b)$ multiplied against the number of lattices of determinant $n^2(-\Delta)^2$ in $\Lambda(n,\ba).$ First, we will establish that the following lattice counting result. 
\begin{lemma}\label{lem: LoD with gcd}
    Let $F(x,y)=\prod_{i=1}^r F_i(x,y)$ be a squarefree binary form of degree $\leq 4$, where $F_i(x,y)$ are irreducible factors of $F(x,y).$ For a fixed $n,\ba,\bc,\bk$, and a tuple $\bD= (D_i)_{i=1}^r$, we define the error quantity: 
    \begin{equation*}
        E^*(\bD) = \sum_{d_i\sim D_i} \left|S^*(n,\ba,\bc,\bk,b;\bd) - \frac{\varrho_F^*(\bd;b,n,\ba,\bc,\bk)}{(cdbn(-\Delta))^2}\cdot \frac{BL^2}{tk^3}\right|.
    \end{equation*}
    Let $D = D_1\dots D_r$. If $F(x,y)$ has no linear factors, then there is an explicit constant $\gamma_r>0$ such that 
    \begin{multline*}
        E^*(\bD) \ll_{\Delta,f} \frac{\varrho_F^*(\mathbf{k})\tau(k)}{\varphi(k)} \cdot D\cdot \left(1+\log\left(\frac{BL^2}{tk^2cb^2(n(-\Delta))^2D}\right)\right) \\ +  \frac{B^{1/2}L}{t^{1/2}} \cdot \frac{64^{\omega(k)}}{k}\cdot \frac{\tau(b^4)}{b} \cdot \sqrt{D} \cdot \exp(\gamma_r\sqrt{\log_2(kcD)\log_3(kcD)}).
    \end{multline*}
    If $F(x,y)$ has a linear factor, then we have that 
    \begin{multline*}
        E^*(\bD)\ll_{\Delta,f} \frac{\varrho_F^*(\mathbf{k})\tau(k)}{\varphi(k)} \cdot D\cdot \left(1+\log\left(\frac{BL^2}{tk^2cb^2(n(-\Delta))^2D}\right)\right) \\ +  \frac{B^{1/2}L}{t^{1/2}} \cdot \frac{64^{\omega(k)}}{k}\cdot \frac{\tau(b^4)}{b} \cdot \sqrt{D} \cdot \exp(\gamma_r\sqrt{\log_2(kcD)\log_3(kcD)}) \\ 
        + \frac{B^{1/2}L}{(tk)^{1/2}} \cdot \frac{1}{b}\cdot \max_{i: \deg(F_i)=1} D_i \log(kcD)^2.
    \end{multline*}
\end{lemma}
\begin{proof}
    The proof of the lemma above will eventually split into two cases: when $F(x,y)$ contains a linear factor and when it does not (which is equivalent to when $F(x,y)$ has a nontrivial integral solution and when it does not). However, the starting step for both cases is the same. 

    Since we have definitionally the relation that 
    \begin{multline*}
        S^*(n,\ba,\bc,\bk,b;\bd) = \#(B/tk)^{1/2}\calB(\bx_0,L) \cap \Lambda^*(\bc\bk\bd;b)\cap \Lambda(n,\ba) \\ = \sum_{\alpha\in \calU(\bc\bk\bd;b)} \#(B/tk)^{1/2} \calB(\bx_0,L)\cap \Lambda^*(\alpha;ckd;b)\cap \Lambda(n,\ba) \\ 
        = \sum_{\alpha\in \calU(\bc\bk\bd;b)} \sum_{\substack{\bx' \bmod n(-\Delta)\\ p_{-\Delta}(F(x,y)/c)=n \\ p_{\neg -\Delta}(F_i(x,y)/k_ic_i)\equiv a_i\bmod -\Delta}} \#(B/tk)^{1/2}\calB(\bx_0,L) \cap \Lambda^*(\alpha;ckd;b) \cap \Lambda(\bx';n(-\Delta)),
    \end{multline*}
    where the ``shifted lattice'', $\Lambda(\bx';n(-\Delta))$, is of determinant $n^2(-\Delta)^2$ and given by 
    \begin{equation*}
        \Lambda(\bx_0;n(-\Delta)) = \{x,y\in \Z^2: (x,y)\equiv \bx'\bmod n(-\Delta)\}.
    \end{equation*}

    Next, by M\"obius inversion, 
    \begin{equation*}
        \Lambda^*(\alpha,ckd;b) = \sum_{g\mid ckd} \mu(g) \#\{g\bx\in \Lambda(\alpha;ckd;b)\}. 
    \end{equation*}
    Hence, we have that $S^*(n,\ba,\bc,\bk,b;\bd)$ becomes 
    \begin{equation*}
        = \sum_{g\mid ckd} \mu(g) \sum_{\alpha\in \calU(\bc\bk\bd;b)} \sum_{\substack{\bx' \bmod n(-\Delta)\\ p_{-\Delta}(F(x,y)/c)=n \\ p_{\neg -\Delta}(F_i(x,y)/k_ic_i)\equiv a_i\bmod -\Delta}} \#\frac{1}{g}(B/tk)^{1/2}\calB(\bx_0,L) \cap \Lambda(\alpha;ckd/g; b);\cap \Lambda(\bx';n(-\Delta)). 
    \end{equation*}
    Now, by standard lattice counting arguments, we have that 
    \begin{multline*}
        \#\frac{1}{g}(B/tk)^{1/2}\calB(\bx_0,L) \cap \Lambda(\alpha;ckd/g; b)\cap \Lambda(\bx';n(-\Delta)) = \frac{BL^2}{tkg^2} \cdot \frac{g}{ckdb^2} \cdot \frac{1}{(n(-\Delta))^2} \\ + O\left(\frac{B^{1/2}L}{(tk)^{1/2}g}\cdot \frac{1}{\lambda_1(\Lambda(\alpha;ckd/g;b)\cap \Lambda(\bx';n(-\Delta)))} \right)+O\left( \min\left(1, \frac{BL^2}{gtk^2 cdb^2(n(-\Delta))^2}\right)\right),
    \end{multline*}
    where $\lambda_1(\Lambda)$ denotes the first successive minima of a lattice (see \cite[p. 773]{delaTenenbaum-Manin}).

    First, we consider the main term: 
    \begin{multline*}
        \frac{BL^2}{tk^2} \cdot \frac{1}{ckdb^2 n(-\Delta)^2}\cdot \sum_{\alpha\in \calU(\bc\bk\bd;b)} \sum_{\substack{\bx' \bmod n(-\Delta)\\ p_{-\Delta}(F(x,y)/c)=n \\ p_{\neg -\Delta}(F_i(x,y)/k_ic_i)\equiv a_i\bmod -\Delta}}  \sum_{g\mid ckd}\frac{\mu(g)}{g} \\
        = \frac{BL^2}{tk^2} \cdot \frac{1}{ckdb^2 (n(-\Delta))^2}\cdot \sum_{\alpha\in \calU(\bc\bk\bd;b)} \sum_{\substack{\bx' \bmod n(-\Delta)\\ p_{-\Delta}(F(x,y)/c)=n \\ p_{\neg -\Delta}(F_i(x,y)/k_ic_i)\equiv a_i\bmod -\Delta}} \frac{\varphi(ckd)}{ckd}.
    \end{multline*}
    Since $\varrho_F^*(\bd;b,n,\ba,\bc,\bk)$ is exactly $\varphi(ckd)$ times the count being summed over, this gives that 
    \begin{equation*}
        E^*(D) \ll \Sigma_1 + \Sigma_2,
    \end{equation*}
    where we take 
    \begin{equation*}
        \Sigma_1 = \sum_{d_i\sim D_i} \sum_{g\mid ckd}\sum_{\alpha\in \calU(\bc\bk\bd;b)} \sum_{\substack{\bx' \bmod n(-\Delta)\\ p_{-\Delta}(F(x,y)/c)=n \\ p_{\neg -\Delta}(F_i(x,y)/k_ic_i)\equiv a_i\bmod -\Delta}} \frac{B^{1/2}L}{(tk)^{1/2}g}\cdot \frac{1}{\lambda_1(\Lambda(\alpha;ckd/g;b)\cap \Lambda(\bx';n(-\Delta)))}
    \end{equation*}
    and 
    \begin{equation*}
        \Sigma_{2} = \sum_{d_i\sim D_i} \sum_{g\mid ckd} \sum_{\alpha\in \calU(\bc\bk\bd;b)} \sum_{\substack{\bx' \bmod n(-\Delta)\\ p_{-\Delta}(F(x,y)/c)=n \\ p_{\neg -\Delta}(F_i(x,y)/k_ic_i)\equiv a_i\bmod -\Delta}} \min\left(1, \frac{BL^2}{gtk^2 cdb^2(n(-\Delta))^2}\right).
    \end{equation*}

    First, let us treat the second term $\Sigma_2$. First, we observe that the number of $\bx'\bmod n(-\Delta)$ such that $p_{-\Delta}(F(x,y)/c)=n$ and $p_{\neg -\Delta}(F_i(x,y)/k_ic_i)\equiv a_i \bmod -\Delta$ is $O_\Delta(1)$. So, we can safely ignore this summand and see that 
    \begin{equation*}
        \Sigma_2 \ll_{\Delta} \sum_{d_i\sim D_i} \sum_{g\mid cdk} \frac{\varrho^*_F(\bd;b,n,\ba,\bc,\bk)}{\varphi(ckd)} \min\left(1, \frac{BL^2}{gtk^2 cdb^2(n(-\Delta))^2}\right).
    \end{equation*}
    We make the observation that 
    $$\varrho_F^*(\bd;b,n,\ba,\bc,\bk) \ll_\Delta \varrho_F^*(\bd;1,1,\mathbf{0},1,\bk)= \varrho_F^*(\bd\bk).$$ Since $c$ can be treated as a constant and because $\varrho_F^*$ is a submultiplicative function, it suffices to understand: 
    \begin{equation*}
        \sum_{h\mid k} \frac{\varrho_F^*(\bk)}{\varphi(k)}\sum_{d_i\sim D_i} \sum_{g\mid d} \frac{\varrho^*_F(\bd)}{\varphi(d)} \min\left(1, \frac{BL^2}{hgtk^2 cdb^2(n(-\Delta))^2}\right)
    \end{equation*}

    We split up the sum by the contribution when $g\cdot d > \frac{BL^2}{htk^2cb^2(n(-\Delta))^2}$, denoted as $\Sigma_2^*$, and otherwise, denoted as $\Sigma_2^{**}$. In the first setting, we have that 
    \begin{multline*}
        \Sigma_2^* \ll_{\Delta}  \frac{BL^2}{tk^2cb^2(n(-\Delta))^2} \sum_{h\mid k} \frac{\varrho_F^*(\bk)}{\varphi(k)h}\sum_{g> \frac{BL^2}{htk^2cb^2(n(-\Delta))^2 D}} \frac{1}{g}\sum_{\substack{d_i\sim D_i \\ g\mid d\\  d> \frac{BL^2}{hgtk^2cb^2(n(-\Delta))^2}}} \frac{\varrho_F^*(\bd)}{\varphi(d)d}. \\ 
        \ll_\Delta \frac{BL^2}{tk^2 cb^2 n^2} \sum_{h\mid k} \frac{\varrho_F^*(\bk)}{\varphi(k)h}\sum_{g>\frac{BL^2}{htk^2cb^2(n(-\Delta))^2 D}} \frac{\varrho_F^*(\mathbf{g})}{\varphi(g)g^2} \sum_{\substack{d_i'\sim D_i/g_i \\ d'>\frac{BL^2}{g^2htk^2cb^2(n(-\Delta))^2}}} \frac{\varrho^*_F(\bd')}{\varphi(d')d'}.
    \end{multline*}
    Above we have used the following properties: the local counting function $\varrho_F^*$ is sub-multiplicative; we have already seen that although we must divide $g = g_1\dots g_r$ where $g_i\mid d_i$ for each $i$, there are only finitely many such decompositions as any value $g$ such that $g_i\mid F_i(x,y)$ (as common divisors must divide $\Res(F_i,F_j)$) and we choose a decomposition arbitrarily above. 

    Thus, it suffices to understand the sum 
    \begin{equation*}
        \sum_{d'\in [D^1,D^2]} \frac{\varrho^*_F(\bd')}{\varphi(d')d'} .
    \end{equation*} 
    We observe that the Dirichlet series 
    \begin{equation*}
        \xi(s;F) = \sum_{\bd} \frac{\varrho_F^*(\bd;\mathbf{1})}{\varphi(d)d^s}
    \end{equation*}
    has the same analytic behavior as the product of $L$-functions
    \begin{equation*}
        \prod_{i=1}^r \zeta_{\Q[z]/f_i(z)}(s)
    \end{equation*}
    as $s\rightarrow 1^+$. Hence, this gives us that 
    \begin{equation*}
        \sum_{d'\in [D^1,D^2]} \frac{\varrho_F^*(d';\mathbf{1})}{\varphi(d') d'} \ll_{\Delta,f} \log(D^2/D^1)^{r} + 1.
    \end{equation*}

    With the above identity in-hand, and noting that we sum $d_i$ in a dyadic interval, we can see that 
    \begin{multline*}
        \Sigma_2^* \ll_{\Delta,f} \frac{BL^2}{tk^2 cb^2n^2}\cdot \sum_{h\mid k} \frac{\varrho_F^*(\bk)}{\varphi(k)h}\sum_{g>\frac{BL^2}{htk^2cb^2(n(-\Delta))^2 D}} \frac{\varrho_F^*(\mathbf{g})}{\varphi(g)g^2} \\ 
        \ll_{\Delta,f}\frac{BL^2}{tk^2 cb^2n^2}\cdot \sum_{h\mid k} \frac{\varrho_F^*(\bk)}{\varphi(k)h}\cdot \frac{htk^2 cb^2 (n(-\Delta))^2 D}{BL^2}
        \ll_{\Delta,f} \frac{\varrho_F^*(\bk)\tau(k)}{\varphi(k)} \cdot D. 
    \end{multline*}
    Here we have use that 
    \begin{equation*}
        \sum_{\bd} \frac{\varrho_F^*(\bd)}{\varphi(d)d^2}
    \end{equation*}
    will converge absolutely and at the same rate that $\prod_{i=1}^r \zeta_{\Q[z]/f_i(z)}(2)$ converges. 

    For $\Sigma_2^{**}$, we proceed with the opposite contribution:
    \begin{multline*}
        \Sigma_2^{**} \ll_{\Delta} \sum_{h\mid k} \frac{\varrho_F^*(\bk)}{\varphi(k)}\sum_{d_i\sim D_i} \sum_{\substack{g\mid d\\ g\leq \frac{BL^2}{htk^2cb^2(n(-\Delta))^2D}}} \frac{\varrho_F^*(\bd)}{\varphi(d)}\\ 
        \ll_{\Delta} \sum_{h\mid k} \frac{\varrho_F^*(\bk)}{\varphi(k)}\sum_{g\leq \frac{BL^2}{htk^2cb^2(n(-\Delta))^2D}} \frac{\varrho_F^*(\mathbf{g})}{\varphi(g)}\sum_{\substack{d_i\sim D_i/g_i}} \frac{\varrho_F^*(\bd')}{\varphi(d')} . 
    \end{multline*}
    Again, let us use the analytic properties of $\xi(s;F)$ to evaluate the summands: 
    \begin{multline*}
        \Sigma_2^{**} \ll_{\Delta,f} \sum_{h\mid k} \frac{\varrho_F^*(\bk)}{\varphi(k)} \cdot D \cdot \sum_{g\leq \frac{BL^2}{htk^2cb^2(n(-\Delta))^2D}} \frac{\varrho_F^*(\mathbf{g})}{\varphi(g)g}\ll_{\Delta,f} \sum_{h\mid k} \frac{\varrho_F^*(\bk)}{\varphi(k)}\cdot D\cdot \log\left(\frac{BL^2}{htk^2cb^2(n(-\Delta))^2D}\right)\\ 
        \ll_{\Delta,f} \frac{\varrho_F^*(\bk)\tau(k)}{\varphi(k)}\cdot D\cdot \log\left(\frac{BL^2}{tk^2cb^2(n(-\Delta))^2D}\right)  .  
    \end{multline*}
    Together, our estimates give us that 
    \begin{equation}\label{eq: estimate of sigma_2 in LoD}
        \Sigma_2 \ll_{\Delta,f} \frac{\varrho_F^*(\mathbf{k})\tau(k)}{\varphi(k)} \cdot D\cdot \left(1+\log\left(\frac{BL^2}{tk^2cb^2(n(-\Delta))^2D}\right)\right).  
    \end{equation}

    To bound $\Sigma_1$, let us set up a few more facts. Then, we will split into the case when $F(x,y)$ has a linear factor and the case when it does not; this is equivalent to splitting into when $F(x,y)$ has a nontrivial solution over $\Z$ or not. 

    We first note that since $\Lambda(\alpha;ckd/g;b)\cap \Lambda(\bx';n(-\Delta))\subset \Lambda(\alpha;ckd/g;b)$, we know that 
    \begin{equation*}
        \lambda_1(\Lambda(\alpha;ckd/g;b)\cap \Lambda(\bx';n(-\Delta))) \geq \lambda_1(\Lambda(\alpha;ckd/g;b)).
    \end{equation*}
    Thus, we know that 
    \begin{equation*}
        \Sigma_1 \ll_{\Delta} \frac{B^{1/2}L}{(tk)^{1/2}}\sum_{d_i\sim D_i} \sum_{g\mid kcd} \frac{1}{g} \sum_{\alpha\in \calU(\bc\bk\bd;b)} \frac{1}{\lambda_1(\Lambda(\alpha;ckd/g;b))}. 
    \end{equation*}
    Next, since $\det(\Lambda(\alpha;ckd/g;b)) = (ckd/g)b^2$, we know that $\lambda_1\leq b\sqrt{kcd/g}.$ Further, if $\bv$ is the shortest nontrivial vector in $\Lambda(\alpha;ckd/g;b)$, then $\bv$ satisfies that $\gcd(v_1,v_2)=b$ and that $k_ic_id_i/g_i \mid F_i(\bv)$ for each $i$ (for some choice of decomposition of $g=g_1\dots g_r)$. 

    We now replace the summation over $\mathcal{U}(\bc\bk\bd;b)$ as $\bd$ varies with summation over all possible shortest vectors of the lattice: 
    \begin{equation}\label{eq: bound for sigma_1 LoD}
        \Sigma_1 \ll_{\Delta} \frac{B^{1/2}L}{(tk)^{1/2}} 
        \sum_{g\leq kcD}\frac{1}{g}\sum_{\substack{|\bv|\leq b\sqrt{kcD/g}\\ \gcd(v_1,v_2)=b}} \frac{1}{|\bv|} \sum_{d_i\sim D_i} \prod_{i=1}^r \mathbf{1}_{k_ic_id_i/g_i\mid F_i(\bv)}. 
    \end{equation}

    If $F(x,y)$ has no linear factors, i.e. no nontrivial solutions over $\Z$, then we know that $F_i(\bv)\neq 0$ for each $i$. However, if $F_i(x,y)$ is linear for some $i$, then it is not impossible for $F_i(\bv)$ to be zero. Hence, this is the point at which we must split into analyzing the contributions differently in the two cases. 
    
\subsubsection{$F(x,y)$ has no nontrivial solution over $\Z$}
In this case, we can assume that $F_i(\bv)\neq 0$ for all $\bv$ satisfying $\gcd(v_1,v_2)=b$. Hence, 
\begin{equation*}
    \Sigma_1 \ll_{\Delta} \frac{B^{1/2}L}{(tk)^{1/2}} \sum_{g\leq kcD}\frac{1}{g}\sum_{\substack{|\bv|\leq b\sqrt{kcD/g}\\ \gcd(v_1,v_2)=b}} \frac{1}{|\bv|}  \sum_{d_i\sim D_i/g_i} \prod_{i=1}^r \mathbf{1}_{k_ic_id_i\mid F_i(\bv)}.
\end{equation*}
Recall that we use $\Delta$ to notate the Hooley $\Delta$-function and that we use the convention that $\Delta(x) = 0$ if $x\not\in \Z$. Then we have the bound: 
\begin{equation*}
    \Sigma_1 \ll_{\Delta} \frac{B^{1/2}L}{(tk)^{1/2}} \sum_{g\leq kcD} \frac{1}{g}\sum_{\substack{|\bv|\leq b\sqrt{kcD/g}\\ \gcd(v_1,v_2)=b}} \frac{1}{|\bv|}   \prod_{i=1}^r \Delta(F_i(\bv)/k_ic_i).
\end{equation*}
From Corollary \ref{cor: dlB-Tenenbaum with divisor factor}, we can bound the sum over $\bv$ by:
\begin{equation*}
    \frac{\tau(b^{\deg(F)})}{b} \cdot \sqrt{kcD/g} \cdot \frac{(16\deg(F))^{\omega(ck)}}{ck} \prod_{i=1}^r \left(\sum_{s_i\ll\sqrt{kcD/g}} \frac{\varrho_{F_i}^*(s_i)(\Delta(s_i)-1)}{s_i^2}\right).
\end{equation*}
From Theorem \ref{thm: hooley untwisted}, this is subsequently bounded by 
\begin{equation*}
    \frac{\tau(b^{\deg(F)})}{b} \cdot \frac{(16\deg(F))^{\omega(ck)}}{\sqrt{ck}} \cdot \sqrt{\frac{D}{g}} \cdot \exp(\gamma_r \sqrt{\log_2(kcD/g)\log_3(kcD/g)}),
\end{equation*}
where $\gamma_r>0$ is a constant depending on $r$ calculated explicitly in \cite{HallTenenbaum-DeltaR}. 

Recalling that $c$ is a constant only depending on $f$, we know that 
\begin{equation*}
    \Sigma_1 \ll_{\Delta,f} \frac{B^{1/2}L}{t^{1/2}} \cdot \frac{64^{\omega(k)}}{k}\cdot \frac{\tau(b^4)}{b} \cdot \sqrt{D} \cdot \exp(\gamma_r\sqrt{\log_2(kcD)\log_3(kcD)})\sum_{g\leq kcD} \frac{1}{g^{3/2}}.  
\end{equation*}
Thus, in the case when $F(x,y)$ has no nontrivial solutions over $\Z$, we know that 
\begin{equation}\label{eq: bound on sigma_1 LoD no linear factors}
    \Sigma_1 \ll_{\Delta,f} \frac{B^{1/2}L}{t^{1/2}} \cdot \frac{64^{\omega(k)}}{k}\cdot \frac{\tau(b^4)}{b} \cdot \sqrt{D} \cdot \exp(\gamma_r\sqrt{\log_2(kcD)\log_3(kcD)}).
\end{equation}
    
\subsubsection{$F(x,y)$ has a nontrivial solution over $\Z$}
We return to our expression  of $\Sigma_1$ in \eqref{eq: bound for sigma_1 LoD} and split the sum over $\bv$ by when $F_i(\bv)=0$ and otherwise: 
\begin{equation*}
    \Sigma_1 \ll_{\Delta} \frac{B^{1/2}L}{(tk)^{1/2}} \sum_{g\leq kcD} \frac{1}{g} \left(\sum_{\substack{|\bv|\leq b\sqrt{kcD/g}\\ \gcd(v_1,v_2)=b\\ F(\bv)\neq 0}} \frac{1}{|\bv|} \sum_{d_i\sim D_i} \prod_{i=1}^r \mathbf{1}_{k_ic_id_i/g_i\mid F_i(\bv)} + \sum_{\substack{|\bv|\leq b\sqrt{kcD/g}\\ \gcd(v_1,v_2)=b\\ F(\bv)= 0}} \frac{1}{|\bv|} \sum_{d_i\sim D_i} \prod_{i=1}^r \mathbf{1}_{k_ic_id_i/g_i\mid F_i(\bv)}\right).
\end{equation*}
By the argument above, we can see the sum over $F(\bv)\neq 0$ is bounded by 
\begin{equation*}
    \ll_{\Delta,f}\frac{B^{1/2}L}{t^{1/2}} \cdot \frac{64^{\omega(k)}}{k}\cdot \frac{\tau(b^4)}{b} \cdot \sqrt{D} \cdot \exp(\gamma_r\sqrt{\log_2(kcD)\log_3(kcD)}). 
\end{equation*}
So, it remains to handle the contribution when $F(\bv)=0$ -- this implies that $F_i(\bv)=0$ for some linear factor $F_i\mid F$. 

Thus, we can bound the contribution when $F(\bv)=0$ by:
\begin{equation*}
    \frac{B^{1/2}L}{(tk)^{1/2}}\sum_{g\leq kcD}\frac{1}{g}\sum_{\substack{i=1\\ \deg(F_i)=1}}^r \frac{D_i}{g_i} \sum_{\substack{|\bv|\leq b\sqrt{kcD/g}\\ \gcd(v_1,v_2)=b \\ F_i(\bv)=0 }} \frac{1}{|\bv|}.
\end{equation*}
For any linear factor $L(x,y)$, we can compute that the sum:
\begin{equation*}
    \sum_{\substack{|\bv|\leq b\sqrt{kcD/g}\\ L(\bv)=0 \\ \gcd(v_1,v_2)=b}}\frac{1}{|\bv|} \leq \frac{1}{b}\sum_{\substack{|\bv|\leq \sqrt{kcD/g}\\ L(\bv)=0}} \frac{1}{|\bv|} \ll_L \frac{\log(kcD/g)}{b}.
\end{equation*}
Plugging the above estimate in, we have that the contribution from those $F(\bv)=0$ is bounded by:
\begin{equation*}
    \frac{B^{1/2}L}{(tk)^{1/2}} \cdot \frac{1}{b}\cdot \max_{i: \deg(F_i)=1} D_i \log(kcD)^2. 
\end{equation*}
Thus, if $F(x,y)$ has a linear factor: 
\begin{multline*}
    \Sigma_1 \ll_{\Delta,f} \frac{B^{1/2}L}{t^{1/2}} \cdot \frac{64^{\omega(k)}}{k}\cdot \frac{\tau(b^4)}{b} \cdot \sqrt{D} \cdot \exp(\gamma_r\sqrt{\log_2(kcD)\log_3(kcD)}) \\ + \frac{B^{1/2}L}{(tk)^{1/2}} \cdot \frac{1}{b}\cdot \max_{i: \deg(F_i)=1} D_i \log(kcD)^2.
\end{multline*}
This completes the proof of Lemma \ref{lem: LoD with gcd}

\end{proof}

\subsection{From $\Lambda^*$ to $\Lambda$}
We return to the sum that we actually want to evaluate: 
\begin{equation*}
    S(n,\ba,\bc,\bk,b;\bd) = \sum_{\substack{(x,y)\in (B/tk)^{1/2}\calB(\bx_0,L) \\ b\mid x,y\\ k_ic_id_i\mid F_i(x,y) \\ p_{-\Delta}(F(x,y)/c) = n \\ p_{\neg 
    -\Delta}(F_i(x,y)/k_ic_i)\equiv a_i\bmod -\Delta}}1. 
\end{equation*}
In our notation above, this is given by 
\begin{equation*}
    \#(B/tk)^{1/2} \calB(\bx_0,L) \cap \Lambda(\bc\bk\bd;b) \cap \Lambda(n,\ba). 
\end{equation*}
We now relate $\Lambda$ and $\Lambda^*$: for any tuple $\bc\bk\bd$, we have that 
\begin{equation}\label{eq: relate Lambda and Lambda*}
    \Lambda(\bk\bc\bd;b) = \cup_{g_i\mid k_ic_id_i} g\Lambda^*(c_1k_1d_1/g_1^\dagger,\dots,c_rk_rd_r/g_r^\dagger;b),
\end{equation}
where $g= g_1\dots g_r$ and $g_i^\dagger = \gcd(g_i^{\deg(F_i)},k_ic_id_i)$. For notational simplicity, we write the vector above as $\bc\bk\bd/\mathbf{g}^\dagger.$  

Then we claim that 
\begin{multline*}\varrho_{F}(\mathbf{d};b, n, \ba,\bc,\bk) = \#\left\{\bx\bmod bcdkn(-\Delta): \begin{array}{c}
     k_ic_id_i\mid F_i(x,y), b\mid (x,y),\\
    p_{-\Delta}(F(x,y)) = n, p_{\neg -\Delta}(F_i(x,y)/k_ic_i)\equiv  a_i\bmod -\Delta
\end{array}\right\} \\
= \sum_{g_i\mid k_ic_id_i} \varrho_F^*(\bd/\mathbf{g}^\dagger;b,n,\ba,\bc,\bk)\cdot \left(\frac{g^\dagger}{g}\right)^2.
\end{multline*}
Indeed, by \eqref{eq: relate Lambda and Lambda*}, the local count splits into a count of sums of the form:
\begin{multline*}
    \#\left\{ \bx \bmod bcdkn(-\Delta): \begin{array}{c}
     k_ic_id_i\mid F_i(x,y), b\mid (x,y), \gcd(x,y,kcd) = g\\
    p_{-\Delta}(F(x,y)) = n, p_{\neg -\Delta}(F_i(x,y)/k_ic_i)\equiv  a_i\bmod -\Delta
\end{array}\right\} \\ 
= \#\left\{ \bx \bmod bcdkn(-\Delta)/g: \begin{array}{c}
     k_ic_id_i\mid g^{\deg(F_i)} F_i(x,y), b\mid (x,y), \gcd(x,y,kcd/g) = 1\\
    p_{-\Delta}(F(x,y)) = n, p_{\neg -\Delta}(F_i(x,y)/k_ic_i)\equiv  a_i\bmod -\Delta
\end{array}\right\} \\
= \left(\frac{g^\dagger}{g}\right)^2 \cdot \#\left\{ \bx \bmod bcdkn(-\Delta)/g^\dagger: \begin{array}{c}
     k_ic_id_i/g_i^{\dagger}\mid  F_i(x,y), b\mid (x,y), \gcd(x,y,kcd/g^\dagger) = 1\\
    p_{-\Delta}(F(x,y)) = n, p_{\neg -\Delta}(F_i(x,y)/k_ic_i)\equiv  a_i\bmod -\Delta
\end{array}\right\}.
\end{multline*}
This completes the claim.

Let us define a modification of \eqref{eq: def S*(n,a,c,k,b;d)}:
\begin{equation*}
    S^*_g(n,\ba,\bc,\bk,b;\bd) := \sum_{\substack{(x,y)\in \frac{1}{g}(B/tk)^{1/2}\calB(\bx_0,L) \\ b\mid x,y\\ k_ic_id_i\mid F_i(x,y) \\ p_{-\Delta}(F(x,y)/c) = n \\ p_{\neg 
    -\Delta}(F_i(x,y)/k_ic_i)\equiv a_i\bmod -\Delta}}\mathbf{1}_{\gcd(x,y,kcd)=1}. 
\end{equation*}

Consequently, 
\begin{equation*}
    E(\bD) \ll \sum_{d_i\sim D_i} \sum_{g_i\mid k_ic_id_i} \left|S^*_g(n,\ba,\bc,\bk,b;\bd/\mathbf{g}^\dagger) - \frac{\varrho_F^*(\bd/\mathbf{g}^\dagger;b,\ba,\bc,\bk)}{(cdbn(-\Delta))^2} \cdot \frac{BL^2}{tk^3} \cdot \left(\frac{g^\dagger}{g}\right)^2\right|. 
\end{equation*}
Reordering the summation and noting that there are $\leq \tau(g_i^{\deg(F_i)})$ possible values of $d_i$ that achieve the same value of $k_ic_id_i/g_i$, we have: 
\begin{equation*}
    E(\bD) \ll \sum_{g_i \leq k_ic_iD_i} \tau(g_i^{\deg(F_i)})\sum_{d_i\sim D_i/g_i} \left|S^*_g(n,\ba,\bc,\bk,b; \bd) - \frac{\varrho_F^*(\bd;b,\ba,\bc,\bk)}{(cdbn(-\Delta))^2} \cdot \frac{BL^2}{tk^3} \cdot \frac{1}{g^2}\right|. 
\end{equation*}

Let us split the above summation into two ranges by the value \begin{equation*}
    G = 1+\frac{BL^2}{tk D}
\end{equation*}
and write 
\begin{equation*}
    E_{\geq G}(\bD) = \sum_{\substack{g_i \leq k_ic_iD_i\\ g\geq G}} \tau(g_i^{\deg(F_i)})\sum_{d_i\sim D_i/g_i} \left|S^*_g(n,\ba,\bc,\bk,b; \bd) - \frac{\varrho_F^*(\bd;b,\ba,\bc,\bk)}{(cdbn(-\Delta))^2} \cdot \frac{BL^2}{tk^3} \cdot \frac{1}{g^2}\right|,
\end{equation*}
\begin{equation*}
    E_{\leq G}(\bD) = \sum_{\substack{g_i \leq k_ic_iD_i\\ g\leq G}} \tau(g_i^{\deg(F_i)})\sum_{d_i\sim D_i/g_i} \left|S^*_g(n,\ba,\bc,\bk,b; \bd) - \frac{\varrho_F^*(\bd;b,\ba,\bc,\bk)}{(cdbn(-\Delta))^2} \cdot \frac{BL^2}{tk^3} \cdot \frac{1}{g^2}\right|.
\end{equation*}

First, let us bound $E_{\geq G}(\bD).$
For any fixed tuple $\mathbf{g}$, we claim that the following bound holds by Corollary \ref{cor: dlB-Tenenbaum with divisor factor}: 
\begin{multline}
    \sum_{d_i\sim D_i/g_i} \left|S^*_g(n,\ba,\bc,\bk,b; \bd) - \frac{\varrho_F^*(\bd;b,\ba,\bc,\bk)}{(cdbn(-\Delta))^2} \cdot \frac{BL^2}{tk^3} \cdot \frac{1}{g^2}\right| \\ \ll_{\Delta} \sum_{d_i\sim D_i/g_i} S^*_g(n,\ba,\bc,\bk,b;\bd) + \sum_{d_i\sim D_i/g_i} \frac{\varrho_F^*(\bd;b,\ba,\bc,\bk)}{(cdbn(-\Delta))^2} \cdot \frac{BL^2}{tk^3} \cdot \frac{1}{g^2} \\
    \ll_{\Delta,f} \sum_{\substack{(x,y) \in \frac{1}{g}(B/tk)^{1/2}\calB(\bx_0,L)\\ b\mid x,y}} \prod_{i=1}^r \Delta(F_i(x,y)/k_ic_i) + \frac{1}{g^2} \cdot \frac{BL^2}{t} \cdot \frac{\varrho_F^*(\bk)}{k^3} \cdot \left(\log(2D/g)-\log(D/g)+1\right)^r\\ 
    \ll_{\Delta,f} \frac{BL^2}{tkb^2g^2} \cdot \frac{64^{\omega(k)}}{k}\cdot \exp(\gamma_r \sqrt{\log_2(B/tk) \log_3(B/tk)})+ \frac{1}{g^2}\cdot \frac{BL^2}{t}\cdot \frac{\varrho_F^*(\bk)}{k^3}. 
\end{multline}
Thus, we have that 
\begin{multline*}
    E_{\geq G}(\bD) \ll_{\Delta,f}\sum_{\substack{g_i\leq k_ic_iD_i \\ g\geq G}} \frac{\tau(g_i^{\deg(F_i)})}{g^2} \cdot \left(\frac{BL^2}{tkb^2} \cdot \frac{64^{\omega(k)}}{k}\cdot \exp(\gamma_r \sqrt{\log_2(B/tk) \log_3(B/tk)})+  \frac{BL^2}{t}\cdot \frac{\varrho_F^*(\bk)}{k^3}\right) \\ 
    \ll_{\Delta,f} \left(\frac{BL^2}{tkb^2} \cdot \frac{64^{\omega(k)}}{k}\cdot \exp(\gamma_r \sqrt{\log_2(B/tk) \log_3(B/tk)})+  \frac{BL^2}{t}\cdot \frac{\varrho_F^*(\bk)}{k^3}\right) \cdot \frac{\log(G)^{8}}{G}. 
\end{multline*}
By our choice of $G$, we can bound $E_{\geq G}(\bD)$ by:
\begin{equation}\label{eq: bound on E>=G(D)}
    \log\left(1+\frac{BL^2}{tkD}\right)^8\left(\frac{D}{b^2} \cdot \frac{64^{\omega(k)}}{k} \cdot \exp(\gamma_r\sqrt{\log_2(B/tk)\log_3(B/tk)}) + D \cdot \frac{\varrho_F^*(\bk)}{k^2} \right). 
\end{equation}

For $E_{\leq G}(\bD)$, we now apply Lemma \ref{lem: LoD with gcd} to this inner sum, with our region changed to $\frac{1}{g} (B/tk)^{1/2}\calB(\bx_0,L)$. If $F(x,y)$ has no linear factors, then
\begin{multline*}
    E_{\leq G}(\bD) \ll_{\Delta,f} \sum_{\substack{g_i\leq k_ic_iD_i \\ g\leq G}} \tau(g_i^{\deg(F_i)}) \frac{\varrho_F^*(\bk)\tau(k)}{\varphi(k)} \cdot \frac{D}{g} \cdot \log(G) \\ + \sum_{\substack{g_i\leq k_ic_iD_i \\ g\leq G}} \frac{B^{1/2}L}{gt^{1/2}}\cdot \frac{64^{\omega(k)}}{k} \cdot \frac{\tau(b^4)}{b}\cdot \sqrt{\frac{D}{g}} \cdot \exp(\gamma_r\sqrt{\log_2(kcD)\log_3(kcD)}). 
\end{multline*}
Evaluating these summands, we see that 
\begin{multline}\label{eq: bound on E <= G for nonlinear}
    E_{\leq G}(\bD) \ll_{\Delta,f} \frac{\varrho_F^*(\bk)\tau(k)}{\varphi(k)} \cdot  D \cdot \log \left(1+\frac{BL^2}{tk D}\right)^9 \\ + \frac{B^{1/2}L}{t^{1/2}}\cdot \frac{64^{\omega(k)}}{k} \cdot \frac{\tau(b^4)}{b} \cdot \sqrt{D} \cdot \exp(\gamma_r\sqrt{\log_2(kcD)\log_3(kcD)}).
\end{multline}
Now, if $F(x,y)$ has a linear factor, then there is an extra term: 
\begin{multline*}
    \sum_{\substack{g_i\leq k_ic_iD_i\\ g\leq G}} \tau(g_i^{\deg(F_i)}) \cdot \frac{B^{1/2}L}{g(tk)^{1/2}}\cdot \frac{1}{b}\cdot \max_{i:\deg(F_i)=1} \frac{D_i}{g_i} \log(kcD)^2 \\
    \ll_{\Delta,f} \frac{B^{1/2}L}{(tk)^{1/2}} \cdot \frac{1}{b} \cdot \max_{i:\deg(F_i)=1} D_i \log(kcD)^2 \log \left(1+\frac{BL^2}{tkD}\right)^8. 
\end{multline*}
Combining this bound with \eqref{eq: bound on E>=G(D)} and \eqref{eq: bound on E <= G for nonlinear} gives Proposition \ref{prop: level of distribution}.

\section{Large moduli and Hooley's $\Delta$-function}\label{sec: large moduli}
In this section, we will discuss how to handle the sum for larger $\bd$: let us write $$T(n,\ba,\bc,\bk,b;\mathbf{D}) := \sum_{d_i \sim D_i} \chi(d) S(n,\ba,\bc,\bk, b;\bd),$$
where $S(n,\ba,\bc,\bk,b;\bd)$ is defined in (\ref{eq: def S d_i}). If $F(x,y)$ has a linear factor $F_1(x,y)$, let us define 
$$T_L(n, \ba,\bc,\bk,b;D_1) := \sum_{d_1\sim D_1} \sum_{\substack{d_i\ll (B/tk)^{\deg(F_i)/4}(n_ik_i)^{-1/2}}} \chi(d) S(n,\ba,\bc,\bk,b;\bd).$$
We will provide an upper bound for these sums. This proof is based off of de la Bret\`eche and Tenenbaum's \cite{delaTenenbaum-Manin} and Heath Brown's \cite{HB-linear} (for $F(x,y)$ with linear factors) treatment of the large moduli terms.

\begin{prop}\label{prop: large moduli}
Let $F$ be a squarefree binary form. Let $\bD = (D_i)_{i=1}^r$ satisfy that $\log(D)\gg \log(B/kt)$.
Then the following bound holds:
    \begin{multline*}
    T(n,\ba,\bc,\bk, b;\mathbf{D}) \ll  \left(\frac{B}{b^2tk^2} + \frac{B^{3/4}\nu_1(k)^{1/2}}{bt^{3/4}k^{5/4}}\right)\cdot (4\deg(F))^{\omega(k)} \\ \times  \log(B/tk)^{\rho_{\Delta,f}/2-1-10^{-5}}\cdot  \left(\frac{B}{tkD} + \frac{tkD}{B}+\log\log(B/tk)\right)^{1/2} . 
\end{multline*}
\end{prop}

\begin{prop}\label{prop: large moduli linear}
    Let $F$ be a squarefree binary form with a linear factor $F_1(x,y)$. Then, the following bound holds:
    $$T(n,\ba,\bc,\bk, b;D_1) \ll \frac{B(16\deg(F))^{\omega(k)}}{tk^2} \cdot \log(B/tk)^{\rho_{\Delta,f}-2-10^{-5}}.$$
\end{prop}

\subsection{Relation to the twisted Hooley's $\Delta$-function}
We first observe that $$T(n,\ba,\bc,\bk, b;\mathbf{D})\leq \sum_{\substack{(x,y)\in (B/tk)^{1/2}\calB(\bx_0,L) \\ b\mid x,y\\ \exists d_i\sim D_i: k_ic_id_i\mid F_i(x,y) \\  p_{-\Delta}(F(x,y)/c) = n \\ p_{\neg -\Delta}(F_i(x,y)/k_ic_id_i)\equiv a_i\bmod -\Delta}}\prod_{i=1}^r \Delta(F_i(x,y)/k_ic_i,\chi).$$
Since $\Delta(n,\chi)\geq 0$, we can remove a few conditions and ask to bound:
$$\sum_{\substack{|x|,|y|\leq (B/tk)^{1/2} \\ b\mid x,y \\ \exists d_i\sim D_i: k_ic_id_i\mid F_i(x,y)}}\prod_{i=1}^r \Delta(F_i(x,y)/k_ic_i,\chi).$$
After applying Cauchy-Schwarz, we achieve the upper bound of 
\begin{multline*}
    T(n,\ba,\bc,\bk, b;\mathbf{D}) \leq \left(\#\{|x|,|y|\leq (B/tk)^{1/2}: b\mid (x,y), \exists d_i\sim D_i, k_ic_id_i\mid F_i(x,y)\}\right)^{1/2} \\ \times \left(\sum_{\substack{|x|,|y|\leq (B/tk)^{1/2}\\ b\mid x,y}} \prod_{i=1}^r \Delta(F_i(x,y)/k_ic_i,\chi)^2\right)^{1/2}.
\end{multline*}

Let us apply Lemma \ref{lem: small log saving binary} to the first quantity: 
\begin{multline*}
\#\{|x|,|y|\leq (B/tk)^{1/2}: b\mid (x,y), \exists d_i\sim D_i, k_ic_id_i\mid F_i(x,y)\} \\ \leq  
    \#\{|x|,|y|\leq (B/tk)^{1/2}: b\mid (x,y),  k\mid F(x,y), \exists d\sim D, d\mid F(x,y)\} \\ \ll  \left(\frac{\deg(F)^{\omega(k)}}{b^2 k}\cdot \frac{B}{tk} + \nu_1(k)\left(\frac{B}{tk}\right)^{1/2} \right)\cdot \left(\frac{B}{tkD}+\frac{tkD}{B} +\log\log(B/tk)\right)\log(B/tk)^{-10^{-4}}.
\end{multline*}
Here $\nu_1(k) = \sum_{\alpha\in \calU(k)} \lambda_1(\Lambda(\alpha;k))^{-1}$.
We note that tracking through the proof of equation (7.41) in \cite{delaTenenbaum-Manin}, one can add the lattice constraints that $b\mid x,y$ and achieve the upper bound mentioned with an additional factor of $b^{-2}$ (recall that $b\leq \log(B)^{10^{-7}}$ is much smaller that $B/tk \geq \exp(\sqrt{
\log(B)})$). Similarly, after tracking through the proof of (7.41) in \cite{delaTenenbaum-Manin} with the lattice condition that $k\mid F(x,y)$, we achieve the bound above; in \S\ref{sec: LoD}, it was seen that this condition gives $\deg(F)^{\omega(k)}$ such lattices of determinant $k$.

Next, we apply Corollary \ref{cor: dlB-Tenenbaum with divisor factor} to the second sum: 
$$\sum_{\substack{|x|,|y|\leq (B/tk)^{1/2}\\ b\mid x,y}} \prod_{i=1}^r \Delta(F_i(x,y)/k_ic_i,\chi)^2\ll \frac{B}{tk\log(B/tk)} \cdot \frac{(16\deg(F))^{\omega(k)}}{b^2k} \prod_{i=1}^r \left(\sum_{s_i\ll B/tk}\Delta(s_i,\chi)^2 \frac{\varrho_{F_i}(s_i)}{s_i^2}\right).$$
Here we have again restricted to the lattice defined by $b\mid (x,y)$.
If $F_i(x,y)$ is irreducible over the quadratic field $K$, then by Theorem \ref{thm: hooley character bound}, we have that 
\begin{equation*}\sum_{s_i\ll B/tk} \Delta(s_i,\chi)^2 \cdot \frac{\varrho_{F_i}(s_i)}{s_i^2}  \ll \sum_{s_i\ll B/tk} \Delta(s_i,\chi)^2 \cdot \frac{\varrho_{f_i}(s_i)}{s_i}\ll \exp(C\sqrt{\log_2(B)\log_3(B)}),
\end{equation*}
for an explicit $C>0$. Otherwise, we use a trivial bound:
\begin{multline*}\sum_{s_i\ll B/tk} \Delta(s_i,\chi)^2 \cdot \frac{\varrho_{F_i}(s_i)}{s_i^2} \ll\sum_{s_i\ll B/tk} \Delta(s_i)^2 \cdot \frac{\varrho_{F_i}(s_i)}{s_i^2} \ll \log(B/tk) \exp(C\sqrt{\log_2(B)\log_3(B)}).\end{multline*}
Here we have used the results of de la Bret\`eche and Tenenbaum in \cite[Theorem 1.1]{delaTenenbaumHooley} for averages of higher powers of $\Delta(n)$ correlated with local counts of polynomials. 

Combining these two bounds, we have that 
\begin{multline*}
    T(n,\ba,\bc,\bk, b;\mathbf{D}) \leq \left(\frac{B}{b^2tk^2} + \frac{B^{3/4}\nu_1(k)^{1/2}}{bt^{3/4}k^{5/4}}\right)\cdot (4\deg(F))^{\omega(k)} \cdot  \left(\frac{B}{tkD} + \frac{tkD}{B}+\log\log(B/tk)\right)^{1/2} \\ \times \exp(C'\sqrt{\log_2(B)\log_3(B)}) \cdot \log(B/tk)^{\rho_{\Delta,f}/2-1-10^{-4}}. 
\end{multline*}
This completes the proof of Proposition \ref{prop: large moduli}.\qed 

\subsection{Linear factors}
Let us assume that $F(x,y)$ has a linear factor that we denote as $F_1(x,y)=L(x,y)$. Again, we enlarge our region and remove the short box constraints, since the terms being summed are non-negative. In this case, we have that 
\begin{equation*}
    T(n,\ba,\bc,\bk,b;D_1) \leq \sum_{\substack{|x|,|y|\leq (B/tk)^{1/2} \\ b\mid x,y }} \Delta(L(x,y)/k_1r_1c_1,\chi;D_1) \prod_{i=2}^r (1\star \chi)(F_i(x,y)/k_ic_i,\chi),
\end{equation*}
where we define the localized twisted Hooley $\Delta$ function as $$\Delta(n,\chi;D) := \left|\sum_{\substack{d\mid n\\ d\sim D}} \chi(d)\right|.$$
Next, we separate the sum by certain values of $L(x,y)$: 
\begin{equation}\label{eq: large moduli linear sum split}
    T(n,\ba,\bc,\bk,b;D_1) \leq \sum_{\substack{m\ll (B/tk)^{1/2}/k_1r_1c_1 \\ b\mid m\\ \exists d_1\sim D_1, d_1\mid m}} \Delta(m,\chi;D_1) \sum_{\substack{|x|,|y|\leq (B/tk)^{1/2}\\ b\mid x,y \\ L(x,y) = mk_1r_1c_1}}\prod_{i=2}^r (1\star \chi)(F_i(x,y)/k_ic_i,\chi).
\end{equation}

For now let us fix a $m$. The restriction $L(x,y) = mk_1r_1c_1$ allows us to rewrite the polynomial $F_i(x,y)$ as a single variable polynomial $F_{i,m}(x).$ In particular, we have that $$\sum_{\substack{|x|,|y|\leq (B/tk)^{1/2} \\ b\mid x,y\\ L(x,y) = mk_1c_1}} \prod_{i=2}^r (1\star \chi)(F_i(x,y)/k_ic_i,\chi) \ll \sum_{\substack{x\ll (B/tk)^{1/2} \\ b\mid x}}\prod_{i=2}^r (1\star \chi)(F_{i,m}(x)/k_ic_i,\chi).$$
So, we apply Corollary \ref{cor: dlB-Tenenbaum with divisor factor} to the sum above: 
$$\ll \frac{B^{1/2}}{b(tk)^{1/2}}\cdot \frac{(16\deg(F))^{\omega(k)}}{k} \prod_{i=1}^r \exp\left(\sum_{p_i\ll B/tk} \frac{\varrho_{F_{i,m}}(p_i)\chi(p_i)}{p_i}\right).$$

Now, we can see that if $\varrho_{F_i}'(p)$ is the original number of roots of $F_i(x,1)\bmod p$, then for $p$ not dividing $m$, we have that $$\varrho_{F_i,m}(p) = \varrho_{F_i}(p).$$
So, we write the above sum as $$\sum_{i=2}^r \sum_{\substack{p_i\ll X \\ p_i\nmid m}} \frac{\varrho_{F_{i}}'(p_i)\chi(p_i)}{p_i} + \sum_{i=2}^r \sum_{p_i\mid m} \frac{\varrho_{F_{i,m}}(p_i)\chi(p_i)}{p_i}.$$

We consider $$\exp\left(\sum_{i=2}^r \sum_{p\mid m} \frac{\varrho_{F_{i,m}}(p_i)\chi(p_i)}{p_i} - \sum_{p_i\mid m} \frac{\varrho_{F_i}'(p_i)\chi(p_i)}{p_i}\right)\ll \prod_{i=2}^r \prod_{p_i\mid m} \left(\frac{1-\varrho_{F_{i,m}}(p_i)/p_i}{1-\varrho_{F_i}'(p_i)/p_i}\right)\ll \prod_{i=2}^r \left(\frac{\sigma(m)}{m}\right)^{\deg(F_i)},$$
where we use the notation $\sigma(m) = \sum_{d\mid m} d.$ Here, we have used \cite[(4.6)]{HB-linear}. Now since $m\ll B/tk$, the above is bounded by $$\ll \prod_{i=2}^r \log\log(B/tk)^{\deg(F_i)} \ll \log\log(B/tk)^{\deg(F)-1}.$$

In summary, we have that 
\begin{multline*}\sum_{\substack{x\ll (B/tk)^{1/2} \\ b\mid x}}\prod_{i=2}^r (1\star \chi)(F_{i,m}(x)/k_ic_i,\chi) \ll \frac{B^{1/2}\log\log(B/tk)^{\deg(F)-1} }{b(tk)^{1/2}}  \cdot  \frac{(8\deg(F))^{\omega(k)}}{k} \\ \times \exp\left(\sum_{i=2}^r \sum_{p_i\ll B/tk}\frac{\varrho_{F_i}'(p_i)\chi(p_i)}{p_i}\right).\end{multline*}
Since $L(x,y)$ was a linear factor of $F(x,y)$, we know that the Picard rank satisfies $$\rho_{\Delta,F} = \rho_{\Delta, F_2...F_r}.$$
Thus, we get that 
$$\sum_{\substack{x\ll (B/tk)^{1/2}\\ b\mid x}} \prod_{i=2}^r (1\star \chi)(F_{i,m}(x)/k_ic_i,\chi) \ll \frac{B^{1/2}\log(B/tk)^{\rho_{\Delta,F}-2}\log\log(B/tk)^{\deg(F)-1}}{b(tk)^{1/2}}\cdot \frac{(16\deg(F))^{\omega(k)}}{k}.$$
Here we have again used Corollary \ref{cor: order of pole is varrho} that the order of the pole of $\xi(s;F,\chi)$ is given by $\rho_{\Delta,f}-2.$

Applying the above bound to the inner sum of (\ref{eq: large moduli linear sum split}), we see that 
\begin{multline*}T(n,\ba,\bc,\bk,b;D_1) \leq \frac{B^{1/2}\log(B/tk)^{\rho_{\Delta,f}-2}\log\log(B/tk)^{\deg(F)-1}}{b(tk)^{1/2}}\cdot \frac{(16\deg(F))^{\omega(k)}}{k} \\ \times \sum_{\substack{m\ll (B/tk)^{1/2}/k_1c_1 \\ b\mid m \\ \exists d_1\sim D_1, d_1\mid m}} \Delta(m,\chi;D_1).\end{multline*}
Applying the Cauchy-Schwarz inequality to the second sum, this final sum is bounded by: 
$$\ll \#\{m\ll (B/tk)^{1/2}/k_1c_1: b\mid m,\exists d_1\sim D_1, d_1\mid m\}^{1/2} \times \left(\sum_{\substack{m\ll (B/tk)^{1/2}/k_1c_1 \\ b\mid m}} \Delta(m,\chi)^2\right)^{1/2}.$$
Now we apply Lemma \ref{lem: small log saving } to the first sum, and Theorem \ref{thm: Nair single sieve} and Theorem \ref{thm: hooley character bound} to the second to get the upper bound:
\begin{equation*}
    T(n,\ba,\bc,\bk,b;D_1) \leq \frac{B\log(B/t)^{\rho_{\Delta,f}-2-10^{-5}}}{tk} \cdot \frac{(16\deg(F))^{\omega(k)}}{k} .
\end{equation*}
This completes the proof of Proposition \ref{prop: large moduli linear}.

\section{The asymptotics}\label{sec: asymptotic}

In this section, we will finally prove Proposition \ref{prop: estimate of S(q_1,q_2,t,k)} and establish the asymptotic in Proposition \ref{prop: main term}. To do so, we bring in ingredients from the preceding sections. Unfortunately, it is in this section where we are confronted with a notationally-challenging problem -- while morally equivalent, the proofs for various factorizations of $F(x,y)$ must be treated slightly differently. As such, we have to do casework on the factorization-type of $F(x,y)$ in \S\ref{subsec: proof of Prop S(q_1,q_2,t,k)}. 

We also highlight that the calculations in \S\ref{subsec: proof of Prop main term} are the jumping-off point for computing the leading constant in \S\ref{sec: constant}. In particular, the notation that is used in \S\ref{sec: constant} is established in this section. 

\subsection{Proof of Proposition \ref{prop: estimate of S(q_1,q_2,t,k)}}\label{subsec: proof of Prop S(q_1,q_2,t,k)}
Let us first recall the statement of Proposition \ref{prop: estimate of S(q_1,q_2,t,k)}:
\begin{prop*}
Let $F(x,y)$ be a squarefree binary form of degree $\leq 4$ and $L\gg \log(B)^{-10^{-10}}.$ Then we have that for $t$ and $k$ satisfying $tk\leq B/\exp(\sqrt{\log(B)}),$ the following holds:
\begin{multline*}
S(X_{\Delta,f},\calB(\bx_0,L);q_1,q_2,t,k)= \frac{BL^2}{tk^3} \cdot C_{q_1,q_2,f,k}(B/tk) \cdot  \left(1+O\left(\log(B)^{-10^{-7}}\right)\right) \\+O\left(\frac{ 64^{\omega(k)}\log(k) B}{tk^{3/2}}\cdot \log(B)^{\rho_{\Delta,f}-2-10^{-7}} + \frac{B^{3/4}}{(tk)^{3/4}} \cdot \left(L+(4\deg(F))^{\omega(k)}\nu_1(k)^{1/2}\right)\cdot \log(B)^{-10^{-7}}\right),\end{multline*}
where $C_{q_1,q_2,f,k}(B/tk)$ is an expression that depends on the factorization of $F(x,y)$.
Additionally, for a fixed value of $k$, there exists a constant $c_{q_1,q_2,f,k}$ such that $$C_{q_1,q_2,f,k}(B/tk) = c_{q_1,q_2,f,k}\log(B/tk)^{\rho_{\Delta,f}-2}.$$ 
\end{prop*}
\begin{remark}
Note that since $tk\leq B/\exp(\sqrt{\log(B)}),$ we have that $\log(B)\ll \log(B/tk) \ll \log(B).$ We use these interchangeably in this section when referring to error terms. 
\end{remark}

Before delving into the case work for this proof, let us outline the general strategy. After applying the reductions discussed in \S\ref{sec: Eisenstein}, the sum that we wish to analyze takes the form: 
\begin{multline*}
    S(X_{\Delta,f},\calB(\bx_0,L);q_1,q_2,t,k) \\ \sim  \sum_{c_{ij}\mid \Res(F_i,F_j)}\mu(c)\chi(c) \sum_{b\leq \log(B)^{10^{-7}}} \mu(b) \cdot \sum_{\substack{n\leq \log(B)^{10^{-7}} \\ p\mid n \implies p\mid -\Delta}} \chi_{q_1}(p_{q_2}(n)) \chi_{q_2}(p_{q_1}(n)) \mathbf{1}_{\exists \fa: N(\fa)=n}\\  \sum_{\substack{\ba \bmod -\Delta \\ \gcd(a_i,-\Delta)=1}} \prod_{i=1}^r\chi_{q_1}(a_i) (1+\chi(a_i))  \sum_{d_i\ll  (B/tk)^{\deg(F_i)/2}} \prod_{i=1}^r \chi(d_i) S(n, \ba, \bc,\bk, b;\bd).
\end{multline*}
We proceed by splitting up the sum over $d_i$ into sums over ``small'' and ``large'' moduli -- either $d \leq (B/tk)^{\deg(F)/4}\log(B)^{-\epsilon}$ or $d\geq (B/tk)^{\deg(F)/4}\log(B)^{-\epsilon}$. We remark that if $F(x,y)$ has linear factors, the cutoff point for being considered a ``small'' or ``large'' modulus is shifted by $\log(B)$.

For ``small'' moduli, our level of distribution result (Proposition \ref{prop: level of distribution}) enables us to achieve the main term: 
\begin{multline*}
    \frac{BL^2}{tk^3} \cdot \sum_{c_{ij}\mid \Res(F_i,F_j)}\frac{\mu(c)\chi(c)}{c} \sum_{b\leq \log(B)^{10^{-7}}} \frac{\mu(b)}{b^2} \cdot \sum_{\substack{n\leq \log(B)^{10^{-7}} \\ p\mid n \implies p\mid -\Delta}} \frac{\chi_{q_1}(p_{q_2}(n)) \chi_{q_2}(p_{q_1}(n))\mathbf{1}_{\exists \fa: N(\fa)=n}}{n^2} \\  \sum_{\substack{\ba \bmod -\Delta \\ \gcd(a_i,-\Delta)=1}} \frac{\prod_{i=1}^r\chi_{q_1}(a_i) (1+\chi(a_i))}{(-\Delta)^2}  \sum_{d_i\ll  (B/tk)^{\deg(F_i)/2}(\log B)^{-\epsilon}}\prod_{i=1}^r \chi(d_i) \frac{ \varrho_F(\bd;b,n, \ba, \bc, \bk)}{d^2}. 
\end{multline*}

For ``large'' moduli, the analysis on the twisted Hooley $\Delta$-function (Propostion \ref{prop: large moduli} or Proposition \ref{prop: large moduli linear}) gives us that these terms contribute negligibly. For these terms, we establish a bound of the form:
$$\sum_{d\geq (B/tk)^{\deg(F)/2}\log(B)^{-\epsilon}} \chi(d) S(n, \ba,\bc,\bk, b;\bd) \ll \log(B)^{\rho_{\Delta,f}-2-10^{-5}} \cdot \left(\frac{B}{b^2tk^2} + \frac{B^{3/4}\nu_1(k)^{1/2}}{(tk)^{3/4}}\right).$$
Note that since $b,n\leq \log(B)^{10^{-7}}$ compared to our saving of $\log(B)^{-10^{-5}},$ we will get a sufficient error term here. 

After careful analysis, we can rewrite this main term as $$\frac{BL^2}{tk^2}\cdot \cdot \sum_{c_{ij}\mid \Res(F_i,F_j)}\mu(c)\chi(c)\cdot \BM_{q_1,q_2}(\bc) \cdot V(\bc; B/tk),$$
where $\BM_{q_1,q_2}(\bc)$ is a constant that detects local obstructions at primes $p\mid -\Delta$ for $X^*_{\Delta,f_1,...,f_r,\bc}$(formally defined in (Definition \ref{def: BM})) and $V(\bc;B/tk)$ is a sum related to the Dirichlet series $\xi(s;F,\chi)$ (defined in (\ref{eq: def of xi(s;F,chi)})). We determine that $V(\bc;B/tk)$ satisfies that $$V(\bc;B/tk) \sim c_{q_1,q_2,f,k,\bc} \log(B/tk)^{\rho_{\Delta,f}-2}.$$
Additionally, in \S\ref{sec: constant}, we will analyse the constant $c_{q_1,q_2,f,k,\bc}$ further.

Now, we proceed with the full proof of Proposition \ref{prop: estimate of S(q_1,q_2,t,k)}. We start with the case work on the factorizaton of $F(x,y)$.
\subsubsection{$F(x,y)$ is irreducible}
Since $F(x,y)$ is irreducible, we can skip the reductions in \S\ref{subsec: multiplicativity}. Applying the reductions of \S\ref{subsec: eisenstein to dirichlet convolution}, \S\ref{subsec: gcd condition}, and \S\ref{subsec: expand dirichlet convolution}, we achieve: 
\begin{align*}
    &S(X_{\Delta,f},\calB(\bx_0,L);q_1,q_2,t,k) \\
    &= \sum_{\substack{n\ll \log(B)^{10^{-7}} \\ p\mid n\implies p\mid -\Delta}}\chi_{q_1}(p_{q_2}(n))\chi_{q_2}(p_{q_1}(n))\mathbf{1}_{\exists \fa: N(\fa)=n}\sum_{\substack{a\bmod -\Delta\\ \gcd(a,-\Delta)=1}} \chi_{q_1}(a) S(n,a,1,k) \\ &\hspace{5cm}+ O\left(\frac{BL^2}{tk^2}\cdot \log(B)^{\rho_{\Delta,f}-2-10^{-7}}\right)\\
    &= \sum_{b\leq \log(B)^{10^{-7}}}\mu(b)\sum_{\substack{n\ll \log(B)^{10^{-7}} \\ p\mid n\implies p\mid -\Delta}}\chi_{q_1}(p_{q_2}(n))\chi_{q_2}(p_{q_1}(n))\mathbf{1}_{\exists \fa: N(\fa)=n}\sum_{\substack{a\bmod -\Delta\\ \gcd(a,-\Delta)=1}} \chi_{q_1}(a) S(n,a,1,k;b) \\ &\hspace{5cm} + O\left(\frac{BL^2}{tk^2}\log(B)^{\rho_{\Delta,f}-2-10^{-7}}\right) \\
    &= \sum_{b\leq \log(B)^{10^{-7}}}\mu(b)\sum_{\substack{n\ll \log(B)^{10^{-7}} \\ p\mid n\implies p\mid -\Delta}}\chi_{q_1}(p_{q_2}(n))\chi_{q_2}(p_{q_1}(n)) \mathbf{1}_{\exists \fa: N(\fa)=n}\sum_{\substack{a\bmod -\Delta\\ \gcd(a,-\Delta)=1}} \chi_{q_1}(a) (1+\chi(a)) \\ &\hspace{1cm}\times \sum_{d\leq \max\sqrt{|F(x,y)|/n_1n_2k}}\chi(d) S(n,a,1, k,b;d) + O\left(\frac{BL^2}{tk^2}\log(B)^{\rho_{\Delta,f}-2-10^{-7}}\right).
\end{align*}
Above, the maximum is taken over $(x,y)\in (B/tk)^{1/2}\calB(\bx_0,L).$
For ``small'' moduli $d\leq D:= (B/tk)^{\deg(F)/4}(nk)^{-1/2}\log(B)^{-10^{-5}},$ we apply Proposition \ref{prop: level of distribution}, our level of distribution result. From this, we can see that
\begin{multline*}
    \sum_{d \leq D}\chi(d) S(n,a,1,k,b;d) = \frac{BL^2}{tk^3}\sum_{d\leq D}\frac{\chi(d) \varrho_F(d;b,n,a,1,k)}{d^2 b^2n^2(-\Delta)^2} \\
    + O\Bigg(  \frac{\varrho_F(k)\tau(k) B}{\varphi(k) k^{3/2}tn^{1/2}}\log(B)^{-10^{-5}}(\log\log(B)+\log(k)) + \frac{64^{\omega(k)}B}{tk^{5/2} n^{1/2}b^2}\log(B)^{-10^{-6}} \\ + \frac{64^{\omega(k)} BL}{tk^{7/4}}\cdot \frac{\tau(b^4)}{b}\log(B)^{-10^{-6}}\Bigg).
\end{multline*}

\begin{remark}
    If $F(x,y)$ is an irreducible cubic form, the level of distribution error for any $D\leq (B/tk)^{3/4}(nk)^{-1/2}$ is sufficient. So, in this case we do not need to separately handle these ``large'' moduli. 
\end{remark}
On the other hand, Proposition \ref{prop: large moduli} for ``large'' moduli tells us that 
\begin{multline*}\sum_{(B/tk)^{\deg(F)/4}(nk)^{-1/2}  \gg d \geq D}\chi(d) S(n,a,1,k,b;d)\\ \ll \left(\frac{B}{b^2tk^{3/2}} + \frac{B^{3/4}\nu_1(k)^{1/2}}{bt^{3/4}k^{3/4}}\right)\cdot (4\deg(F))^{\omega(k)}\cdot \log(B)^{\rho_{\Delta,f}/2-1-10^{-6}}.\end{multline*}
Together, this gives us that 
\begin{multline*}
    S(X_{\Delta,f},\calB(\bx_0,L);q_1,q_2,t,k) = \frac{BL^2}{tk^3}\cdot \sum_{b\leq \log(B)^{10^{-7}}}\mu(b)\sum_{\substack{n\ll \log(B)^{10^{-7}} \\ p\mid n\implies p\mid -\Delta}}\chi_{q_1}(p_{q_2}(n))\chi_{q_2}(p_{q_1}(n))\mathbf{1}_{\exists \fa: N(\fa)=n} \\ \times  \sum_{\substack{a\bmod -\Delta\\ \gcd(a,-\Delta)=1}} \chi_{q_1}(a) (1+\chi(a)) 
    \sum_{d\leq D}\frac{\chi(d) \varrho_F(d;b,n,a,1,k)}{d^2 b^2n^2(-\Delta)^2} \\ + O\left(64^{\omega(k)}\log(k) \cdot \left(\frac{B}{tk^{3/2}} + \frac{B^{3/4}\nu_1(k)^{1/2}}{t^{3/4}k^{3/4}}\right)\log(B)^{\rho_{\Delta,f}-2-10^{-7}}\right). 
\end{multline*}
Here we have used that $n,b\leq \log(B)^{10^{-7}}$ are small constants. 
Note that $\rho_{\Delta,f}\geq 2$ by definition and hence $\rho_{\Delta,f}-2 \geq \rho_{\Delta,f}/2-1$. \\ 

Let us rework this main term. First, we can  extend the sum over $n,b\leq \log(B)^{10^{-7}}$ to one over all $n$ and all $b$ by again bounding away the contribution from large $n$ and large $b$ with the Nair sieve (Theorem \ref{thm: Nair sieve}). This leaves us with the term 
\begin{equation*}
    \sum_{b} \frac{\mu(b)}{b^2}\sum_{d\leq D} \frac{\chi(d)\varrho_F(d;b,n,a,1,k)}{d^2 n^2(-\Delta)^2}.
\end{equation*}

We can also use multiplicativity of $\varrho_F$ to write:
$$\varrho_F(d;b,n,a,1,k) = \varrho_F(p_{\neg -\Delta}(b),kd) \varrho_F(p_{-\Delta}(b),n,a),$$
where we have defined: 
$$\varrho_F(p_{\neg-\Delta} (b), kd) = \#\{\bx \bmod p_{\neg -\Delta}(b) kd: p_{\neg-\Delta}(b) \mid (x,y), F(x,y)\equiv 0 \bmod kd\} $$
$$\varrho_F(p_{-\Delta}(b), n,a) = \#\{\bx \bmod p_{-\Delta}(b)n(-\Delta): p_{-\Delta}(b) \mid (x,y), p_{-\Delta}(F(x,y)) = n, p_{\neg -\Delta}(F(x,y))\equiv a \bmod -\Delta\}.$$
Note that we have used that $\gcd(k,-\Delta)=1$ by construction and $\gcd(d,-\Delta)=1$ for those $d$ with nonzero contribution to the sum. 
Let us use the notation that $b=b_{\neg -\Delta}b_{-\Delta}$ where $\gcd(b_{\neg-\Delta},-\Delta)=1$ and $p\mid b_{-\Delta} \implies p\mid -\Delta$. Consequently, the main term can be rewritten as: 
\begin{equation*}
\frac{BL^2}{tk^3}\cdot\BM_{q_1,q_2} \cdot   V(k,D),   
\end{equation*}
where we define: 
\begin{multline}\label{eq: def of BM irred}
\BM_{q_1,q_2} := \sum_{b_{-\Delta}} \frac{\mu(b_{-\Delta})}{b_{-\Delta}^2}\sum_{\substack{p\mid n \implies p\mid -\Delta}} \frac{\chi_{q_1}(p_{q_2}(n))\chi_{q_2}(p_{q_2}(n))\mathbf{1}_{\exists \fa: N(\fa)=n}}{n^2}\\ \times \sum_{\substack{a \bmod -\Delta \\ \gcd(a,-\Delta)=1}} \frac{\chi_{q_1}(a)(1+\chi(a))}{(-\Delta)^2} \cdot \varrho_F(b_{-\Delta},n,a).
\end{multline}
\begin{equation}\label{eq: def of V irred}
    V(k,D) = \sum_{b_{\neg -\Delta}} \frac{\mu(b_{\neg-\Delta})}{b_{\neg-\Delta}^2}\sum_{d\leq D} \frac{\chi(d)\varrho_F(b_{\neg-\Delta},kd)}{d^2}.
\end{equation}
Observe that the sum over $b_{\neg -\Delta}$ converges absolutely above in $V(k,D)$. 
Both the terms $\BM_{q_1,q_2}$ and $V(k,D)$ will be generalized in the more complicated cases below; however, we can immediately rearrange $\BM_{q_1,q_2}$ into a less complicated expression.

Noting that for $\mu(b_{-\Delta})$ to be nonzero, $b_{-\Delta}$ must be squarefree and hence divides $-\Delta$. Hence, we can rewrite:
\begin{align*}
    \BM_{q_1,q_2} &= \sum_{p\mid n\implies p\mid -\Delta}  \frac{\chi_{q_1}(p_{q_2}(n))\chi_{q_2}(p_{q_2}(n))\mathbf{1}_{\exists \fa: N(\fa)=n}}{n^2} \\
    & \hspace{1cm} \times \sum_{\substack{a \bmod -\Delta \\ \gcd(a,-\Delta)=1}} \frac{\chi_{q_1}(a)+\chi_{q_2}(a)}{(-\Delta)^2} \sum_{\substack{\bx\bmod n(-\Delta)^2\\ p_{-\Delta}(F(x,y)) = n\\ 
    p_{\neg -\Delta}(F(x,y))\equiv a \bmod -\Delta}} \sum_{\substack{b\mid x,y\\ b\mid -\Delta}} \frac{\mu(b)}{b^2}\cdot \left(\frac{b}{-\Delta}\right)^2 \\
    &= \sum_{p\mid n\implies p\mid -\Delta}  \frac{\chi_{q_1}(p_{q_2}(n))\chi_{q_2}(p_{q_2}(n))\mathbf{1}_{\exists \fa: N(\fa)=n}}{n^2}
    \times \sum_{\substack{a \bmod -\Delta \\ \gcd(a,-\Delta)=1}} \frac{\chi_{q_1}(a)+\chi_{q_2}(a)}{(-\Delta)^2} \varrho_F^*(n,a), 
\end{align*}
where we use M\"obius inversion on the sum over $b$ and hence 
\begin{equation*}
    \varrho_F^*(n,a) = \#\{\bx\bmod n(-\Delta): \gcd(x,y,-\Delta)=1, p_{-\Delta}(F(x,y))=n, p_{\neg -\Delta}(F(x,y))\equiv a \bmod -\Delta\}.
\end{equation*}

We can also rearrange the sum over $b$ for $V(k,D)$: 
\begin{multline*}
    V(k,D) = \sum_{d\leq D} \frac{\chi(d)}{d^2} \sum_{\gcd(b_{\neg -\Delta},-\Delta)=1} \frac{\mu(b_{\neg -\Delta})}{b_{\neg -\Delta}^2} \cdot \varrho_F(b_{\neg -\Delta}, kd)\\
    = \sum_{d\leq D} \frac{\chi(d)}{d^2}\sum_{\substack{\bx_0 \bmod kd \\ kd\mid F(\bx_0)}} \sum_{\gcd(b_{\neg -\Delta},-\Delta)=1} \frac{\mu(b_{\neg -\Delta})}{b_{\neg -\Delta}^2} \cdot \#\{\bx\bmod b_{\neg -\Delta}kd: b_{\neg -\Delta}\mid \gcd(x,y), \bx\equiv \bx_0 \bmod kd\}.
\end{multline*}
This local count can be determined precisely: 
\begin{equation}\label{eq: sec 11 counting congruence case work}
    \#\{\bx\bmod b_{\neg -\Delta}kd: b_{\neg -\Delta}\mid \gcd(x,y), \bx\equiv \bx_0 \bmod kd\} = \begin{cases}
        b'^2, & b'\mid \gcd(kd,b_{\neg -\Delta}), b'\mid \bx_0,\\ 
        1, & \text{otherwise.}
    \end{cases}
\end{equation}
Consequently, we see that 
\begin{multline*}
    \sum_{\substack{\bx_0 \bmod kd \\ kd\mid F(\bx_0)}} \sum_{\gcd(b_{\neg -\Delta},-\Delta)=1} \frac{\mu(b_{\neg -\Delta})}{b_{\neg -\Delta}^2} \cdot \#\{\bx\bmod b_{\neg -\Delta}kd: b_{\neg -\Delta}\mid \gcd(x,y), \bx\equiv \bx_0 \bmod kd\} \\
    = \sum_{\substack{\bx_0\bmod kd\\ kd\mid F(\bx_0)}} \prod_{p\mid -\Delta} \left(1-p^{-2} \cdot \begin{cases}
        p^2, & p\mid kd,p\mid \bx_0, \\ 
        0, & p\mid kd, p\nmid \bx_0, \\
        1, & p\nmid kd,
    \end{cases}\right)
     = \prod_{p\nmid kd(-\Delta)} \left(1-\frac{1}{p^2}\right) \varrho_F^*(kd).
\end{multline*}

Hence, we have that 
\begin{equation*}
    V(k,D) = \sum_{d\leq D} \frac{\chi(d)\varrho_F^*(kd)}{d^2}\prod_{p\nmid kd(-\Delta)}\left(1-\frac{1}{p^2}\right) = \prod_{p\nmid -\Delta} \left(1-\frac{1}{p^2}\right)\sum_{d\leq D}\prod_{p\mid kd} \left(1-\frac{1}{p^2}\right)^{-1} \frac{\chi(d)\varrho_F^*(kd)}{d^2}.
\end{equation*}

\subsubsection{$F(x,y)$ is the product of two irreducible factors} Write $F(x,y) = F_1(x,y)F_2(x,y)$ as the decomposition of $F(x,y)$ into irreducible factors. If one of the factors is linear, assume without loss of generality that it is $F_1$. Following the reductions of \S\ref{sec: Eisenstein}, we have that: 
\begin{align*}
    &S(X_{\Delta,f},\calB(\bx_0,L);q_1,q_2,t,k) =  \sum_{c\mid \Res(F_1,F_2)}\mu(c)\chi(c)  S(\bc = (c,c),\bk) \\
    &= \sum_{c\mid \Res(F_1,F_2)}\mu(c)\chi(c)  \times  \sum_{\substack{n\ll \log(B)^{10^{-7}} \\ p\mid n\implies p\mid -\Delta}} \chi_{q_1}(p_{q_2}(n))\chi_{q_2}(p_{q_1}(n))\mathbf{1}_{\exists \fa: N(\fa)=n} \sum_{\substack{\ba \bmod -\Delta \\ \gcd(a_i,-\Delta)=1}}\chi_{q_1}(a) S(n,\ba,\bc,\bk) \\ &\hspace{5cm}+ O\left(\frac{BL^2}{tk^2}\log(B)^{\rho_{\Delta,f}-2-10^{-7}}\right)\\
    &= \sum_{c\mid \Res(F_1,F_2)}\mu(c)\chi(c)  \sum_{\substack{n\ll \log(B)^{10^{-7}} \\ p\mid n \implies p\mid -\Delta}} \chi_{q_1}(p_{q_2}(n))\chi_{q_2}(p_{q_1}(n))\mathbf{1}_{\exists \fa: N(\fa)=n}   \sum_{\substack{\ba \bmod -\Delta \\ \gcd(a_i,-\Delta)=1}}\chi_{q_1}(a) \\ &\hspace{2cm}\times \sum_{b\leq \log(B)^{10^{-7}}}\mu(b) S(n,\ba ,\bc,\bk;b) + O\left(\frac{BL^2}{tk^2}\log(B)^{\rho_{\Delta,f}-2-10^{-7}}\right)\\
    &=   \sum_{c\mid \Res(F_1,F_2)}\mu(c)\chi(c)  \sum_{\substack{n\ll \log(B)^{10^{-7}} \\ p\mid n \implies p\mid -\Delta}} \chi_{q_1}(p_{q_2}(n))\chi_{q_2}(p_{q_1}(n))\mathbf{1}_{\exists \fa: N(\fa)=n}   \sum_{\substack{\ba \bmod -\Delta \\ \gcd(a_i,-\Delta)=1}}\chi_{q_1}(a) (1+\chi(a_1))(1+\chi(a_2))\\ &\hspace{2cm}\times \sum_{b\leq \log(B)^{10^{-7}}}\mu(b)
     \sum_{d_i\leq \max \sqrt{|F_i(x,y)|/n_ik_i}}\chi(d)S(n,\ba,\bc,\bk, b;\bd) \\ &\hspace{6cm}+ O\left(\frac{BL^2}{tk^2}\log(B)^{\rho_{\Delta,f}-2-10^{-7}}\right).
\end{align*}

If $F_1$ and $F_2$ are both irreducible quadratics, we write that $\bd\leq \bD $ if $\bd = (d_1,d_2)$ satisfies that $d_1d_2\leq D:=(B/tk)^{\deg(F)/4}(nk)^{-1/2}\log(B)^{-10^{-5}}$. If $F_1$ is linear, we define $D_1:= (B/tk)^{1/4}(n_1k_1)^{-1/2}\log(B)^{-2-10^{-5}}$.  We write $\bd\leq \bD$ when $d_1\leq D_1$ -- for such vectors $\bd$, we automatically have that $d\leq (B/tk)^{\deg(F)/4} (nk)^{-1/2}\log(B)^{-2-10^{-5}}.$ Proposition \ref{prop: level of distribution} tells us that for $\bd\leq \bD$, we have that:
\begin{multline*}
    \sum_{d_i\leq D_i} \chi(d)S(n,\ba,\bc,\bk, b;\bd) = \frac{BL^2}{tk^3} \cdot \sum_{\bd\leq \bD} \frac{\chi(d) \varrho_F(\bd; b,n, \ba, \bc,\bk)}{c^2d^2b^2n^2(-\Delta)^2} \\
    + O\Bigg(\frac{\varrho_F(k)\tau(k) B}{ t\varphi(k)k^{3/2}n^{1/2}}\log(B)^{-10^{-5}}(\log\log(B)+\log(k)) + \frac{64^{\omega(k)}B}{tk^{5/2}n^{1/2}b^2}\log(B)^{-10^{-6}} \\ +\frac{64^{\omega(k)}BL}{tk^{7/4}}\cdot \frac{\tau(b^4)}{b}\cdot \log(B)^{-10^{-6}} + \frac{B^{3/4}L}{(tk)^{3/4}(k_1n_1)^{1/2}b}\log(B)^{-10^{-6}}\Bigg).
\end{multline*}
Note that this last term does not exist when $F_1(x,y)$ and $F_2(x,y)$ are both irreducible quadratics. Also we observe an abuse of notation above -- here $c$ denotes the product of the entries of the vector $\bc$ compared to the entries of the vector earlier on this page.

Applying either Proposition \ref{prop: large moduli} (if $F_1,F_2$ are both quadratic) or Proposition \ref{prop: large moduli linear} (if $F_1$ or $F_2$ is linear), we achieve the upper bound:
\begin{multline*}\sum_{\bd\not\leq \bD} \chi(d) S(n,\ba,\bc,\bk, b;\bd)\ll \frac{B(16\deg(F))^{\omega(k)}}{tk^2} \cdot \log(B/tk)^{\rho_{\Delta,f}-2-10^{-5}} \\ + \left(\frac{B}{b^2tk^{3/2}} + \frac{B^{3/4}\nu_1(k)^{1/2}}{bt^{3/4}k^{3/4}}\right)\cdot (4\deg(F))^{\omega(k)}\cdot \log(B)^{\rho_{\Delta,f}/2-1-10^{-6}}.
\end{multline*}
 Here we have used that $0\leq \rho_{\Delta,f}/2-1 \leq \rho_{\Delta,f}-2.$

Recalling that $n,b\leq \log(B)^{10^{-7}}$, we get the estimate:
\begin{multline*}S(X_{\Delta,f},\calB(\bx_0,L);q_1,q_2,t,k)= \frac{BL^2}{tk^3} \cdot C_{q_1,q_2,f,k}(B/tk) \cdot (1+O(\log(B)^{-10^{-7}})) \\+O\left(\frac{64^{\omega(k)}\log(k) B}{tk^{3/2}}\log(B)^{\rho_{\Delta,f}-2-10^{-7}} + \frac{B^{3/4}\cdot ((4\deg(F))^{\omega(k)}\nu_1(k)^{1/2}+L)}{(tk)^{3/4}}\cdot \log(B)^{-10^{-7}}\right),\end{multline*}
where we define: 
\begin{multline*}
    C_{q_1,q_2,f,k}(B/tk) := \sum_{c'\mid \Res(F_1,F_2)}\frac{\mu(c')\chi(c') }{(c')^4} \sum_{\substack{p\mid n \implies p\mid \Delta}} \frac{\chi_{q_1}(p_{q_2}(n))\chi_{q_2}(p_{q_1}(n))\mathbf{1}_{\exists \fa: N(\fa)=n}}{n^2} \\ \times  \sum_{\substack{\ba \bmod \Delta \\ \gcd(a_i,\Delta)=1}}\frac{\chi_{q_1}(a) (1+\chi(a_1))(1+\chi(a_2))}{(-\Delta)^2}    \sum_{b} \frac{\mu(b)}{b^2}
     \sum_{\bd\leq \bD} \frac{\chi(d) \varrho_F(\bd;b,n, \ba, \bc,\bk)}{d^2}.
\end{multline*}
It is of note here that our abusive notation has come back to bite us -- indeed, in this situation $\bc = (c',c')$ and hence $c = (c')^2$ (in previous reductions, we used the letter $c$ to denote the variable $c'$ above).

Now, we generalize the definitions of $\BM_{q_1,q_2}$ and $V(k,D)$ from before (in (\ref{eq: def of BM irred}), (\ref{eq: def of V irred})). 
\begin{defn}\label{def: BM}
We define the constant:
    $$\BM_{q_1,q_2}(\bc) := \sum_{\substack{p\mid n \implies p\mid -\Delta}} \frac{\chi_{q_1}(p_{q_2}(n))\chi_{q_2}(p_{q_1}(n))\mathbf{1}_{\exists \fa: N(\fa)=n}}{n^2}\sum_{\substack{\ba \bmod -\Delta \\ \gcd(a_i,-\Delta)=1}} \frac{\chi_{q_1}(a)}{(-\Delta)^2} \times  \prod_{i=1}^r (1+\chi(a_i)) \varrho_F^*(n,\bc\ba).$$
Here the function $\varrho_f(n,\ba)$ denotes:
$$\varrho_F^*(n,\ba) = \#\{\bx\bmod n(-\Delta): \gcd(x,y,-\Delta)=1, p_{-\Delta}(F(x,y))=n, p_{\neg -\Delta}(F_i(x,y)) \equiv c_ia_i\bmod -\Delta\}.$$
\end{defn}
This constant is independent of $\bk$, and will detect when there is an obstruction to rational points, as will be explained in \S\ref{sec: constant}. We observe that it does contain a twist depending on the values of $\bc$; these twists are nonexistence in the case when $F$ is irreducible, as seen in (\ref{eq: def of BM irred}).

We also define the generalization of $V(k,D)$ from before. 
\begin{defn}\label{eq: def of V(c,k,D)}
    \begin{equation}
    V(\bc,\bk,D) :=  \prod_{p\nmid -\Delta}\left(1-\frac{1}{p^2}\right) \sum_{\bd\leq \bD} \frac{\chi(d)\varrho_F^*(\bc\bk\bd)}{d^2}\prod_{p\mid cdk} \left(1-\frac{1}{p^2}\right)^{-1},
\end{equation}
where $\bd\leq \bD$ denotes summing over the appropriate range of ``small'' moduli and $$\varrho_F^*(\bc\bk\bd) := \#\{\bx\bmod ckd: \gcd(x,y,ckd)=1, c_ik_id_i\mid F_i(x,y)\}. $$
\end{defn}

We now claim that 
\begin{equation*}
    C_{q_1,q_2,f,k}(B/tk) = \sum_{c'\mid \Res(F_1,F_2)} \frac{\mu(c')\chi(c')}{(c')^4}\cdot \BM_{q_1,q_2}(\bc)\cdot V(\bc,\bk,D). 
\end{equation*}
By the Chinese remainder theorem (and that $\gcd(cdk,-\Delta)=1$), it is clear that 
\begin{equation*}
    \varrho_F(\bd;b,n,\ba,\bc,\bk)= \varrho_F(p_{-\Delta}(b), n, \bc\ba) \varrho_F(p_{\neg-\Delta}(b), \bc\bk\bd),.
\end{equation*}
where we define 
\begin{equation*}
    \varrho_F(p_{-\Delta}(b),n,\bc\ba) = \#\left\{\bx\bmod p_{-\Delta}(b)n(-\Delta): p_{-\Delta}(b) \mid \gcd(x,y), \begin{array}{c}
        p_{-\Delta}(F(x,y))=n \\
        p_{\neg-\Delta}(F_i(x,y))\equiv c_ia_i\bmod -\Delta
    \end{array}\right\},
\end{equation*}
and recall that
\begin{equation*}
    \varrho_F(p_{\neg-\Delta}(b),\bc\bk\bd) = \#\left\{\bx\bmod p_{\neg-\Delta}(b)ckd: p_{\neg-\Delta}(b)\mid \gcd(x,y), c_ik_id_i\mid F_i(x,y)\right\}.
\end{equation*}
Rewriting each value of $b$ as an element coprime to $-\Delta$ and one whose prime factors are exclusively those of $-\Delta$, we see that 
\begin{multline}\label{eq: def of C_q  from Prop 8.1}
    C_{q_1,q_2,f,k}(B/tk) =  \sum_{c'\mid \Res(F_1,F_2)} \frac{\mu(c')\chi(c')}{(c')^4} \sum_{\gcd(b,-\Delta)=1} \frac{\mu(b)}{b^2} \sum_{\bd\leq \bD} \frac{\chi(d)\varrho_F(b,\bc\bk\bd)}{d^2}\\ 
    \times \sum_{\substack{p\mid n \implies p\mid -\Delta}} \frac{\chi_{q_1}(p_{q_2}(n))\chi_{q_2}(p_{q_1}(n))\mathbf{1}_{\exists \fa: N(\fa)=n}}{n^2} \sum_{\substack{\ba \bmod -\Delta \\ \gcd(a_i,-\Delta)=1}} \frac{\chi_{q_1}(a)}{(-\Delta)^2}   \prod_{i=1}^r (1+\chi(a_i)) \sum_{b'\mid -\Delta} \frac{\mu(b')}{b'^2} \varrho_F(b',n,\bc\ba). 
\end{multline}
First, we consider 
\begin{equation*}
    \sum_{\bd\leq \bD}\frac{\chi(d)}{d^2} \sum_{\gcd(b,-\Delta)=1}\frac{\mu(b)}{b^2} \varrho_F(b,\bc\bk\bd). 
\end{equation*}
By \eqref{eq: sec 11 counting congruence case work}, it is a straightforward calculation to evaluate this expression as: 
\begin{equation*}
    \sum_{\bd\leq \bD}\frac{\chi(d)}{d^2} \sum_{\substack{\bx_0\bmod ckd\\ c_ik_id_i\mid F_i(\bx_0)}} \prod_{p\nmid ckd(-\Delta)}\left(1-\frac{1}{p^2}\right) \mathbf{1}_{\gcd(x,y,p)=1}.
\end{equation*}
Hence, this factor is indeed equal to $V(\bc,\bk,\bD)$. 

Second, we observe that the following identity holds:
\begin{equation*}
    \sum_{b\mid -\Delta}\frac{\mu(b)}{b^2} \varrho_F(b,n,\bc\ba) = \varrho_F^*(n,\bc\ba).
\end{equation*}
With this in hand, we can conclude that the second line in \eqref{eq: def of C_q  from Prop 8.1} is equal to $\BM_{q_1,q_2}(\bc).$ This completes the proof of the claim.

\subsubsection{$F(x,y)$ is a product of three irreducible factors}
The proofs for the case of three and four irreducible factors are essentially equivalent to the proof when there are two irreducible factors; we write them for completeness. 

Since $\deg(F(x,y))\leq 4$, we must have at least two linear factors. We denote these as $F_1(x,y)$ and $F_2(x,y)$. Applying the reductions of \S\ref{sec: Eisenstein}, we have the following estimates:
\begin{align*}
    S(X_{\Delta,f},&\calB(\bx_0,L);q_1,q_2,t,k) = \sum_{c_1c_2\mid \Res(F_1F_2,F_3)} \mu(c_1c_2)\chi(c_1c_2) \sum_{c\mid \Res(F_1,F_2)}\mu(c)\chi(c) S(\bc = (cc_1,cc_2,c_1c_2),\bk) \\
    &= \sum_{c_1c_2\mid \Res(F_1F_2,F_3)} \mu(c_1c_2)\chi(c_1c_2) \sum_{c\mid \Res(F_1,F_2)}\mu(c)\chi(c) \sum_{\substack{n\ll \log(B)^{10^{-7}}\\ p\mid n\implies p\mid -\Delta}}\chi_{q_1}(p_{q_2}(n))\chi_{q_2}(p_{q_1}(n)) \mathbf{1}_{\exists \fa: N(\fa)=n} \\ & \hspace{2cm} \times \sum_{\substack{\ba\bmod -\Delta \\ \gcd(a_i,-\Delta)=1}}\chi_{q_1}(a) S(n,\ba,\bc,\bk) + O\left(\frac{BL^2}{tk^2}\log(B)^{\rho_{\Delta,f}-2-10^{-7}}\right) \\
    &= \sum_{c_1c_2\mid \Res(F_1F_2,F_3)} \mu(c_1c_2)\chi(c_1c_2) \sum_{c\mid \Res(F_1,F_2)}\mu(c)\chi(c) \sum_{\substack{n\ll \log(B)^{10^{-7}}\\ p\mid n\implies p\mid -\Delta}}\chi_{q_1}(p_{q_2}(n))\chi_{q_2}(p_{q_1}(n)) \mathbf{1}_{\exists \fa: N(\fa)=n} \\ &\hspace{2cm} \times
    \sum_{\substack{\ba\bmod -\Delta \\ \gcd(a_i,-\Delta)=1}}\chi_{q_1}(a)(1+\chi(a_1))(1+\chi(a_2))(1+\chi(a_3)) 
     \sum_{b\ll \log(B)^{10^{-7}}}\mu(b)\\ &\hspace{4cm}\times  \sum_{d_i\leq \sqrt{\max|F_i(x,y)|/(k_in_i)^{1/2}}} \chi(d) S(n,\ba,\bc,\bk, b;\bd)\\
    &\hspace{6cm}+O\left(\frac{BL^2}{tk^2}\log(B)^{\rho_{\Delta,f}-2-10^{-7}}\right).
\end{align*}
Let $D_i := (B/tk)^{1/4}(n_ik_i)^{-1/2}\log(B)^{-2-10^{-5}}$ for each $i$ such that $F_i$ is linear. We write that $\bd\leq \bD$ if $d_i\leq D_i$ for each linear factor $F_i$. Otherwise, we write that $\bd\not\leq \bD$. Applying Proposition \ref{prop: level of distribution}, we get that
\begin{multline*}
    \sum_{d_i\leq D_i} \chi(d)S(n,\ba,\bc,\bk, b;\bd) = \frac{BL^2}{tk^3} \cdot \sum_{\bd\leq \bD} \frac{\chi(d) \varrho_F(\bd; b,n, \ba, \bc,\bk)}{c^2d^2b^2n^2(-\Delta)^2} \\
    + O\left(\frac{64^{\omega(k)}\log(k) B}{tk^{5/2}n^{1/2}}\log(B)^{-10^{-6}} + \frac{64^{\omega(k)}BL}{tk^{7/4}n^{1/4}}\log(B)^{-10^{-6}} + \frac{B^{3/4}L}{(tk)^{3/4}\min((k_1n_1)^{1/2},(k_2n_2)^{1/2})}\log(B)^{-10^{-6}}\right).
\end{multline*}
On the other hand, Proposition \ref{prop: large moduli linear} implies that the following upper bound holds:
$$\sum_{\bd\not\leq \bD} \chi(d) S(n,\ba,\bc,\bk, b;\bd)\ll \frac{B(16\deg(F))^{\omega(k)}}{tk^{3/2}b^2}\log(B)^{\rho_{\Delta,f}-2-10^{-5}}.$$

Together (again recalling that $n,b\leq \log(B)^{10^{-7}}$), we get the estimate that:
\begin{multline*}
    S(X_{\Delta,f},\calB(\bx_0,L);q_1,q_2,t,k)= \frac{BL^2}{tk^3} \cdot C_{q_1,q_2,f,k}(B/tk) \cdot (1+O(\log(B)^{-10^{-7}})) \\+O\left(\frac{64^{\omega(k)}\log(k) B}{tk^{3/2}}\log(B)^{\rho_{\Delta,f}-2-10^{-7}} + \frac{B^{3/4}L}{(tk)^{3/4}}\log(B)^{-10^{-7}}\right),
\end{multline*}
where we have that:
$$C_{q_1,q_2,f,k}(B/tk) := \sum_{c_1c_2\mid \Res(F_1F_2,F_3)} \frac{\mu(c_1c_2)\chi(c_1c_2)}{c_1^4c_2^4}\sum_{c\mid \Res(F_1,F_2)} \frac{\mu(c)\chi(c)}{c^4} \cdot \BM_{q_1,q_2}(\bc) \cdot V(\bc,\bk,D).$$
(Recall the defintion of $\BM_{q_1,q_2}(\bc)$ from Definition \ref{def: BM} and $V(\bc,\bk,\bD)$ from Definition \ref{eq: def of V(c,k,D)}).

\subsubsection{$F(x,y)$ is the product of four irreducible factors}
Let $F(x,y) = F_1F_2F_3F_4$, where all four are linear facts. Applying the reductions of \S\ref{sec: Eisenstein}, we have the following estimate for $S(X_{\Delta,f},\calB(\bx_0,L);q_1,q_2,t,k)$:
\begin{align*}
     & \sum_{\substack{c\mid \Res(F_1F_2,F_3F_4)\\ c=c_1c_2\\ c=c_3c_4}}\mu(c)\chi(c) \sum_{\substack{e\mid \Res(F_1,F_2)\\ f\mid \Res(F_3,F_4)}} \mu(e)\chi(e)\mu(f)\chi(f) S(\bc=(c_1e,c_2e,c_3f,c_4f),\bk) \\
    &= \sum_{\substack{c\mid \Res(F_1F_2,F_3F_4)\\ c=c_1c_2\\ c=c_3c_4}}\mu(c)\chi(c) \sum_{\substack{e\mid \Res(F_1,F_2)\\ f\mid \Res(F_3,F_4)}} \mu(e)\chi(e)\mu(f)\chi(f)
    \\ &\hspace{2cm}\times \sum_{\substack{n\ll \log(B)^{10^{-7}}\\ p\mid n\implies p\mid -\Delta}} \chi_{q_1}(p_{q_2}(n))\chi_{q_2}(p_{q_1}(n)) \mathbf{1}_{\exists \fa: N(\fa)=n} \sum_{\substack{\ba\bmod -\Delta \\ \gcd(a_i,-\Delta)=1}} \chi_{q_1}(a) S(p_{q_1}(n),p_{q_2}(n),\ba,\bc,\bk) \\ 
    &\hspace{4cm} + O\left(\frac{BL^2}{tk^2}\log(B)^{\rho_{\Delta,f}-2-10^{-7}}\right)\\
    &= \sum_{\substack{c\mid \Res(F_1F_2,F_3F_4)\\ c=c_1c_2\\ c=c_3c_4}}\mu(c)\chi(c) \sum_{\substack{e\mid \Res(F_1,F_2)\\ f\mid \Res(F_3,F_4)}} \mu(e)\chi(e)\mu(f)\chi(f)
    \\ &\hspace{2cm}\times \sum_{\substack{n_j\ll \log(B)^{10^{-7}}\\ p\mid n\implies p\mid -\Delta}} \chi_{q_1}(p_{q_2}(n))\chi_{q_2}(p_{q_1}(n))  \mathbf{1}_{\exists \fa: N(\fa)=n}\sum_{\substack{\ba\bmod -\Delta \\ \gcd(a_i,-\Delta)=1}} \chi_{q_1}(a) \prod_{i=1}^4 (1+\chi(a_i)) \sum_{b\ll \log(B)^{10^{-7}}} \mu(b)\\ &\hspace{4cm} \times \sum_{d_i\leq \sqrt{\max|F_i|/(k_in_i)^{1/2}}} \chi(d)S(p_{q_1}(n),p_{q_2}(n),\ba,\bc,\bk,b;\bd) \\
    &\hspace{6cm} + O\left(\frac{BL^2}{tk^2}\log(B)^{\rho_{\Delta,f}-2-10^{-7}}\right).
\end{align*}
Let $D_i := (B/tk)^{1/4}(n_ik_i)^{-1/2}\log(B)^{-2-10^{-5}}$ for each $i$. We write that $\bd\leq \bD$ if $d_i\leq D_i$ for each linear factor $F_i$. Otherwise, we write that $\bd\not\leq \bD$. Applying Proposition \ref{prop: level of distribution}, we get that
\begin{multline*}
    \sum_{d_i\leq D_i} \chi(d)S(n,\ba,\bc,\bk, b;\bd) = \frac{BL^2}{tk^3} \cdot \sum_{\bd\leq \bD} \frac{\chi(d) \varrho_F(\bd; b,n, \ba, \bc,\bk)}{c^2d^2b^2n^2(-\Delta)^2} \\
    + O\left(\frac{64^{\omega(k)}\log(k) B}{tk^{5/2}n^{1/2}}\log(B)^{-10^{-6}} + \frac{64^{\omega(k)}BL}{tk^{7/4}n^{1/4}}\log(B)^{-10^{-6}} + \frac{B^{3/4}L}{(tk)^{3/4}\min_i((k_in_i)^{1/2})}\log(B)^{-10^{-6}}\right).
\end{multline*}
On the other hand, Proposition \ref{prop: large moduli linear} implies that the following upper bound holds:
$$\sum_{\bd\not\leq \bD} \chi(d) S(n,\ba,\bc,\bk, b;\bd)\ll \frac{B(16\deg(F))^{\omega(k)}}{tk^{3/2}b^2}\log(B)^{\rho_{\Delta,f}-2-10^{-5}}.$$

Together (again recalling that $n,b\leq \log(B)^{10^{-7}}$), we get the estimate that:
\begin{multline*}
    S(X_{\Delta,f},\calB(\bx_0,L);q_1,q_2,t,k)= \frac{BL^2}{tk^3} \cdot C_{q_1,q_2,f,k}(B/tk) \cdot (1+O(\log(B)^{-10^{-7}})) \\+O\left(\frac{64^{\omega(k)}\log(k)B}{tk^{3/2}}\log(B)^{\rho_{\Delta,f}-2-10^{-7}} + \frac{B^{3/4}L}{(tk)^{3/4}}\log(B)^{-10^{-7}}\right),
\end{multline*}
where we have that:
$$C_{q_1,q_2,f,k}(B/tk) := \sum_{\substack{c\mid \Res(F_1F_2,F_3F_4)\\ c=c_1c_2\\ c=c_3c_4}}\frac{\mu(c)\chi(c)}{c^4} \sum_{\substack{e\mid \Res(F_1,F_2)\\ f\mid \Res(F_3,F_4)}} \frac{\mu(e)\chi(e)\mu(f)\chi(f)}{e^4f^4} \cdot  \BM_{q_1,q_2}(\bc)\cdot V(\bc,\bk,D).$$

\subsubsection{Completion of the proof of Proposition \ref{prop: estimate of S(q_1,q_2,t,k)}}
Noting that $\BM_{q_1,q_2}(\bc)$ is a finite constant, and since $F(x,y)$ is a squarefree form we know that the sum over $\bc$ is given by a finite sum, we see that in order to establish Proposition \ref{prop: estimate of S(q_1,q_2,t,k)} it suffices to prove that for any fixed $\bc$ we have that: 
\begin{equation}\label{eq: goal asymp V}
V(\bc,\bk,D) = c_{q_1,q_2,f,k,\bc} \log(B/tk)^{\rho_{\Delta,f}-2}(1+O(\log(B)^{-1/2})).\end{equation}
To do so, it suffices to understand the growth of the factor:
\begin{equation*}
    \sum_{\bd\leq \bD} \frac{\chi(d)\varrho_F^*(\bc\bk\bd)}{d^2}\prod_{p\mid ckd} \left(1-\frac{1}{p^2}\right)^{-1}. 
\end{equation*}
Let us define the Dirichlet series:
\begin{equation}\label{eq: def of xi(s;F,chi,c,k)}
    \xi(s;F,\chi,\bc,\bk) = \sum_{\bd} \frac{\chi(d) \varrho_F^*(\bc\bk\bd)}{d^{2s}}\prod_{p\mid ckd}\left(1-\frac{1}{p^2}\right)^{-1}.
\end{equation}
Since $\varrho_F^*(n)\leq \deg(F)^{\omega(n)}n$ for all but finitely many values of $n$, we can see that the above Dirichlet series converges for $\Re(s)>1.$ Moreover, for $\Re(s)>1$, we have the Euler product expansion: 
$$\xi(s;F,\chi,\bc,\bk) = \fS_{\bc,\bk}'(s) \prod_{p\nmid ck} \left(1+(1-p^{-2})^{-1}\cdot \left(\frac{\chi(p)}{p^{2s}} \cdot (\varrho_F^*((p,1,...,1)) + \hdots + \varrho_F^*((1,...,p)))+ \frac{\chi(p)^2}{p^{4s}}\cdot \hdots\right)\right),$$
where we define: 
$$\fS_{\bc,\bk}'(s) := \prod_{\substack{p\mid ck\\ p^{e_i}\| c_ik_i}} (1-p^{-2})^{-1}\cdot \left(\varrho_F^*((p^{e_1},...,p^{e_r}))+\frac{\chi(p)}{p^{2s}} \cdot (\varrho_F^*((p^{e_1+1},...,p^{e_r}))+ \hdots + \varrho_F^*((p^{e_1},...,p^{e_r+1}))) + \hdots\right).$$
Note that this product is over a finite number of primes (depending on $\bc$ and $\bk$) and hence converges. The local factors at $p\nmid ck$ are close to those in the Euler product expansion of $\xi(s;F,\chi)$ (as defined in (\ref{eq: def of xi(s;F,chi)})): in particular, we know that for such primes $\varrho_F^*((p,1,\dots,1))+\dots \varrho_F^*((1,\dots,p)) = \varrho_F^*(p)$. So, the local factors are equivalent up to a factor of $O(p^{-2s-1}+p^{-4s+2}).$
 Thus, we get that for $\Re(s)>1$, $$\xi(s;F,\chi,\bc,\bk) = \fS_{\bc,\bk}(s) \cdot \fS_F(s)\cdot \xi(s;F,\chi),$$
where we take:
$$\fS_{\bc,\bk} (s):= \fS_{\bc,\bk}'(s) \cdot \prod_{p\mid ck} \left(1+\frac{\chi(p)}{p^{2s}} \cdot (\varrho_F^*((p,1,...,1)) + \hdots + \varrho_F^*((1,...,p)))+ \frac{\chi(p)^2}{p^{4s}}\cdot \hdots\right)^{-1},$$
and $\fS_F(s)$ converges absolutely for $\Re(s)>3/4$. 

By Lemma \ref{lem: poles have same order}, we can see that $\xi(s;F,\chi,\bc,\bk)$ will have a meromorphic continuation to $\Re(s)>1/2$ and a pole of order $\rho_{\Delta,f}-2$ at $s=1$. Thus, we have that:
\begin{multline*}V(\bc,\bk,D) = \sum_{\bd\leq \bD} \frac{\chi(d)\varrho_F^*(\bc\bk\bd)}{d^2} \prod_{p\mid ckd} \left(1-\frac{1}{p^2}\right)^{-1} \\ = \fS_{\bc,\bk}(1) \fS_F(1)\Res_{s=1} \xi(s;F,\chi)\cdot \log(D)^{\rho_{\Delta,f}-2}(1+O(\log(D)^{-1})).\end{multline*}
Recalling that $D = (B/tk)^{\deg(F)/4}k^{-1/2} \log(B)^{-\varepsilon}$, where $\varepsilon = 10^{-5}$ if there are no linear factors and $\varepsilon = -2-10^{-5}$ if $F$ has linear factors, we establish the claim (\ref{eq: goal asymp V}) for a fixed value of $k$ and $\bc$:
$$V(\bc,\bk,D) = c_{q_1,q_2,f,k,\bc} \log(B/tk)^{\rho_{\Delta,f}-2} (1+O(\log(B)^{-1/2})),$$
where $c_{q_1,q_2,f,k,\bc}= \fS_{\bc,\bk}(1) \fS_F(1)\cdot \Res_{s=1}\xi(s;F,\chi) \cdot (\deg(F)/4)^{\rho_{\Delta,f}-2}.$\qed

\subsection{Proof of Proposition \ref{prop: main term}}\label{subsec: proof of Prop main term}
Let us first recall the statement of Proposition \ref{prop: main term}. 
\begin{prop*}
    Let us assume that $L\gg \log(B)^{-10^{-10}}.$ There exists a constant $C_{q_1,q_2,f}$ such that as $B\rightarrow\infty$, we have that $$M_{q_1,q_2}(X_{\Delta,f},\calB(\bx_0,L)) = C_{q_1,q_2,f} L^2\cdot B\log(B)^{\rho_{\Delta,f}-1} (1+ O(\log(B)^{-10^{-8}})).$$
\end{prop*}

In this subsection, we will explain how to boost the estimate of Proposition \ref{prop: estimate of S(q_1,q_2,t,k)} to Proposition \ref{prop: main term} by summing over $t$ and $k$. Additionally, we shall keep track of our leading constant to be plugged into the analysis of \S\ref{sec: constant}. 

\begin{proof}
Recall that 
$$M_{q_1,q_2}(X_{\Delta,f},\calB(\bx_0,L)) = \sum_{\substack{N(\fk)\leq B/\exp(\sqrt{\log(B)})\\\gcd(\fk,\overline{\fk})=1}} \mu(\fk) \sum_{\substack{N(\ft)\leq B/k\exp(\sqrt{\log(B)}) \\ \gcd(\ft,\overline{\ft})=1}}S(X_{\Delta,f},\calB(\bx_0,L);q_1,q_2,N(\ft),N(\fk)).$$
Applying Proposition \ref{prop: estimate of S(q_1,q_2,t,k)}, we can see that this is given by: 
\begin{multline*}
    M_{q_1,q_2}(X_{\Delta,f},\calB(\bx_0,L)) = BL^2\sum_{\substack{N(\fk)\leq B/\exp(\sqrt{\log(B)})\\ \gcd(\fk,\overline{\fk})=1}} \frac{\mu(\fk)}{N(\fk)^3} \sum_{\substack{N(\ft)\leq B/k\exp(\sqrt{\log(B)})\\ \gcd(\ft,\overline{\ft})=1}} \frac{1}{N(\ft)}C_{q_1,q_2,f,N(\fk)}(B/N(\ft\fk)) \\ 
    + O\left(B\log(B)^{\rho_{\Delta,f}-2-10^{-7}} \sum_{N(\fk)\leq B/\exp(\sqrt{\log(B)})}\frac{64^{\omega(k)}\log(k)}{N(\fk)^{3/2}} \sum_{N(\ft)\leq B/k\exp(\sqrt{\log(B)})}\frac{1}{N(\ft)}\right)\\
    + O\left(B^{3/4}\log(B)^{-10^{-7}} \sum_{N(\fk)\leq B/\exp(\sqrt{\log(B)})} \frac{L+(4\deg(F))^{\omega(N(\fk))}\nu_1(N(\fk))^{1/2}}{N(\fk)^{3/4}} \sum_{N(\ft)\leq B/k\exp(\sqrt{\log(B)})} \frac{1}{N(\ft)^{3/4}}\right).
\end{multline*}

First, let us bound away the error terms. Since we know that $$\sum_{N(\fa)\leq X} \frac{1}{N(\fa)} \ll_\Delta \log(X),
\sum_{N(\fa)\leq X} \frac{C^{\omega(N(\fa))}\log(N(\fa))}{N(\fa)^\alpha}\ll_\Delta 1$$
for any $\alpha>1$ and any $C\in \mathbb{N}$, we can see that the first and second terms are both bounded by $B\log(B)^{\rho_{\Delta,f}-1-10^{-7}}.$ 

For the final error term, we see that: 
\begin{multline*}
    \sum_{N(\fk)\leq B/\exp(\sqrt{\log(B)})} \frac{L+(4\deg(F))^{\omega(N(\fk))}\nu_1(N(\fk))^{1/2}}{N(\fk)^{3/4}} \sum_{N(\ft)\leq B/k\exp(\sqrt{\log(B)})} \frac{1}{N(\ft)^{3/4}} \\ \ll_{\Delta} \sum_{N(\fk)\leq B/\exp(\sqrt{\log(B)})} \frac{L+(4\deg(F))^{\omega(N(\fk))}\nu_1(N(\fk))^{1/2}}{N(\fk)^{3/4}} \cdot \left(\frac{B}{N(\fk)}\right)^{1/4}  \\
    \ll_{\Delta}B^{1/4}L\cdot \log(B) +  B^{1/4} \sum_{k\leq B/\exp(\sqrt{\log(B)})} \frac{(4\deg(F))^{\omega(k)}\nu_1(k)^{1/2}}{k} \cdot \#\{\fk: N(\fk) = k\}.
\end{multline*}
Observe that $\#\{\fk: N(\ft) = k\} \ll 2^{\omega(k)}.$ On the other hand, we have that 
\begin{multline*}
    \sum_{k\leq B/\exp(\sqrt{\log(B)})} \frac{64^{\omega(k)}\nu_1(k)^{1/2}}{k} \cdot \#\{\fk: N(\fk) = k\} \\ \leq \left(\sum_{k\leq B/\exp(\sqrt{\log(B)})} \frac{1}{k}\right)^{1/2} \left(\sum_{k\leq B/\exp(\sqrt{\log(B)})} \frac{\nu_1(k)C_2^{\omega(k)}}{k}\right),
\end{multline*}
where $C_2 = 2^{14}$.
The first term above is bounded by $\log(B)^{1/2}.$ Expanding out the second term, we have 
\begin{equation*}
    \sum_{k\leq B/\exp(\sqrt{\log(B)})} \frac{\nu_1(k)C_2^{\omega(k)}}{k} \leq  \sum_{k\leq B} \frac{C_2^{\omega(k)}}{k}\sum_{\alpha\in \calU(k)} \frac{1}{\lambda_1(\Lambda(\alpha))}. 
\end{equation*}
Since we know that 
$$\sum_{k\leq B} \sum_{\alpha\in \calU(k)}\frac{1}{\lambda_1(\Lambda(\alpha))} \ll \sum_{0\neq |\bv|\leq \sqrt{B}}\frac{\#\{k\leq B: k\mid F(\bv)\}}{|\bv|}\ll \sqrt{B}\log(B),$$
by partial summation the second term is bounded absolutely. In summary, we have that 
$$\sum_{k\leq B/\exp(\sqrt{\log(B)})} \frac{64^{\omega(k)}\nu_1(k)^{1/2}}{k} \cdot \#\{\fk: N(\fk) = k\}\ll \log(B)^{1/2}.$$
Thus, this final error term is again upper bounded by $B\log(B)^{\rho_{\Delta,f}-1-10^{-7}}.$ For all of the error terms above, we recall that $L\gg \log(B)^{-10^{-10}}.$

\medskip

It remains to analyze the main term by expanding the definition of $C_{q_1,q_2,f,N(\fk)}$:
\begin{multline}\label{eq: main term of Prop 4.10}
    BL^2  \cdot \sum_{c_{ij}\mid \Res(F_i,F_j)} \frac{\mu(\bc)\chi(\bc) }{c^2}\cdot \BM_{q_1,q_2} (\bc)\sum_{\substack{N(\fk)\leq B/\exp(\sqrt{\log(B)})\\ \gcd(\fk,\overline{\fk})=1}} \frac{\mu(\fk)}{N(\fk)^3} \sum_{\substack{N(\ft)\leq B/k\exp(\sqrt{\log(B)})\\ \gcd(\ft,\overline{\ft})=1}} \frac{1}{N(\ft)}V(\bc,\bk,D).
\end{multline}
Here we use $\mu(\bc)\chi(\bc)$ to denote the appropriate product depending on the factorization of $F(x,y)$ (as written down explicitly in \S\ref{subsec: proof of Prop S(q_1,q_2,t,k)}).
Let us expand out the definition of $V(\bc,\bk,D)$ and change the order of summation (using $\ft'=\ft\fk$): 
\begin{multline*}
    \sum_{\substack{N(\fk)\leq B/\exp(\sqrt{\log(B)})\\ \gcd(\fk,\overline{\fk})=1}} \frac{\mu(\fk)}{N(\fk)^3} \sum_{\substack{N(\ft)\leq B/k\exp(\sqrt{\log(B)})\\ \gcd(\ft,\overline{\ft})=1}} \frac{1}{N(\ft)}V(\bc,\bk,D) \\ = \sum_{\substack{N(\ft')\leq B/\exp(\sqrt{\log(B)})\\ \gcd(\ft',\overline{\ft'})=1}}\frac{1}{N(\ft')}\sum_{\fk\mid \ft'} \frac{\mu(\fk)}{N(\fk)^2} \sum_{\bd\leq (B/N(\ft)')^{\deg(F)/4}} \frac{\chi(d)\varrho_F^*(\bc\bk\bd)}{d^2} \prod_{p\mid ckd} \left(1-\frac{1}{p^2}\right)^{-1}.
\end{multline*}
Note that since $k=N(\fk)\mid t=N(\ft)$, we know that if $p\mid k$ then $\chi(p)=1$; hence, $\chi(k) = 1$.

Let us consider the multiple Dirichlet series: 
\begin{equation}\label{eq: def of xi(s_1,s_2)}
    \xi(s_1,s_2;F,\chi,\bc) := \sum_{\gcd(\ft,\overline{\ft})=1} \frac{1}{N(\ft)^{s_1}} \sum_{\fk\mid \ft} \frac{\mu(\fk)}{N(\fk)^2} \sum_{\bd} \frac{\chi(d)\varrho_F^*(\bc\bk\bd)}{d^{2s_2}}\prod_{p\mid ckd} \left(1-\frac{1}{p^2}\right)^{-1}.
\end{equation}
We want to establish what the analytic behavior of this series looks like at $(s_1,s_2)=(1,1)$. Since the series converges in the region $\Re(s_1),\Re(s_2)>1$, in this regime we have the Euler product: 
\begin{equation*}
    \xi(s_1,s_2;F,\chi,\bc) = \prod_{p\nmid -\Delta} L_p(s_1,s_2;F,\chi,\bc),
\end{equation*}
where if $p\nmid \prod_{i\neq j} \Res(F_i,F_j)$, we define the local Euler factors $L_p(s_1,s_2;F,\chi,\bc)$ as: 
\begin{equation*}
     1+(1-p^{-2})^{-1}\left(\frac{\chi(p)+1}{p^{s_1}} \left(1-\frac{\varrho_F^*(p)}{p^2}\right) + \frac{\chi(p)\varrho_F^*(p)}{p^{2s_2}} + \frac{\chi(p)+1}{p^{s_1+2s_2}} \left(\varrho_F^*(p) - \frac{\varrho_F^*(p^2)}{p^2} \right)+\hdots\right).  
\end{equation*}
Observe that since $\gcd(c,p)=1$, the above local factors are independent of $\bc$. Here we have also used that if $p\nmid \prod_{i\neq j}\Res(F_i,F_j)$, then 
\begin{equation*}
    \varrho_F^*((p,1,\dots,1)) + \dots + \varrho_F^*((1,\dots,p)) = \varrho_F^*(p) = \#\{(x,y)\bmod p: \gcd(x,y,p)=1, p\mid F(x,y)\}.
\end{equation*}
Now, if $p\mid \prod_{i\neq j}\Res(F_i,F_j)$, then the expression for $L_p(s_1,s_2;F,\chi,\bc)$ is more complicated; however, this is a finite product and so we can consider this (\textit{for now}) an absolute constant:
\begin{equation*}
    \gamma(F,\chi,\bc) := \prod_{p\mid \prod_{i\neq j} \Res(F_i,F_j)} L_p(1,1;F,\chi,\bc).
\end{equation*}

Inspecting these local factors further, it becomes evident that the analytic behavior of $\xi(s_1,s_2;F,\chi,\bc)$ at $s_1=1$, $s_2=2$ matches that of: 
\begin{equation*}
    \zeta_K(s_1) \cdot \prod_{p} \left(1+\frac{\chi(p)\varrho_F^*(p)}{p^{2s_2}}\right).
\end{equation*}
Note that to conclude the above statement, we have used that $\varrho_F^*(n)\ll_F 4^{\omega(n)}n$ to state that the difference between $\xi(s_1,s_2;F,\chi,\bc)$ and the above product converges absolutely for $\Re(s_1),\Re(s_2)>1-\delta$. At $s_1=1$, we know that $\zeta_K(s_1)$ has a simple pole. 

On the other hand, we can see that 
\begin{equation*}
    \xi^*(s;F,\chi) = \prod_p \left(1+\frac{\chi(p)\varrho_F^*(p)}{p^{2s_2}}\right)
\end{equation*}
will share the same behavior at $s_2=1$ as $\xi(s;F,\chi)$, which is the same behavior as that of the $L$-function 
\begin{equation*}
    L(s;F,\chi) =  \prod_{i=1}^r L(s;f_i\times \chi) = \prod_{i=1}^r \prod_{\fp\subset \Q[z]/f_i(z)} \left(1-\frac{\chi(N(\fp))}{N(\fp)^{s_2}}\right)^{-1}.
\end{equation*}
In particular, it will have a pole of order $\#\{1\leq i\leq r: K\subset \Q[z]/f_i(z)\}=\rho_{\Delta,f}-2$ at $s_2=1.$ This follows from Lemma \ref{lem: poles have same order} and its corollary.

With these two facts in hand, we can now return to our sum:
\begin{equation*}
\sum_{\substack{N(\ft')\leq B/\exp(\sqrt{\log(B)})\\ \gcd(\ft',\overline{\ft'})=1}}\frac{1}{N(\ft')}\sum_{\fk\mid \ft'} \frac{\mu(\fk)}{N(\fk)^2} \sum_{\bd\leq (B/N(\ft)')^{\deg(F)/4}} \frac{\chi(d)\varrho_F^*(\bc\bk\bd)}{d^2}\prod_{p\mid ckd}\left(1-\frac{1}{p^2}\right)^{-1}    
\end{equation*}
Indeed, since $\xi(s_1,s_2;F,\chi,\bc)$ has a simple pole at $s_1=1$ and a pole of order $\rho_{\Delta,f}-2$ at $s_2=2$, this sum can be evaluated as:
\begin{equation*}
    \gamma(F,\chi,\bc)\cdot C'_{F,\chi,\bc} \cdot \log(B)^{\rho_{\Delta,f}-1}+O_{F,\chi,\bc}(\log(B)^{\rho_{\Delta,f}-2}).
\end{equation*}
Finally, we can evaluate our main term \eqref{eq: main term of Prop 4.10}: 
\begin{equation*}
    BL^2 \sum_{c_{ij}\mid \Res(F_i,F_j)} \frac{\mu(\bc)\chi(\bc)}{c^2}\cdot \BM_{q_1,q_2}(\bc) \left(\gamma(F,\chi,\bc)\cdot C'_{F,\chi,\bc} \cdot \log(B)^{\rho_{\Delta,f}-1} + O_{F,\chi,\bc}(\log(B)^{\rho_{\Delta,f}-2}\right).
\end{equation*}
Since the sum over $c_{ij}\mid \Res(F_i,F_j)$ is finite, we can set 
\begin{equation*}
    C_{q_1,q_2,f} = \sum_{c_{ij}\mid \Res(F_i,F_j)} \frac{\mu(\bc)\chi(\bc)}{c^2}\cdot \BM_{q_1,q_2}(\bc) \cdot \gamma(F,\chi,\bc)\cdot C'_{F,\chi,\bc}, 
\end{equation*}
and complete the proof of Proposition \ref{prop: main term}.

\end{proof}

\section{Analysis of the leading constant}\label{sec: constant}
Finally, we want to establish that our leading constant if $c_{\Delta,f}$ is zero then there is a local or Brauer-Manin obstruction to rational points on $X_{\Delta,f}$. We can see from the reductions of section \S\ref{sec: reductions} that it suffices to classify when
$$\sum_{q_1q_2=-\Delta}' C_{q_1,q_2,f}=0.$$
In particular, we want to deduce exactly when the following expression is zero:
\begin{equation}\label{eq: leading constant}
   \sum_{\bc} \frac{\mu(\bc)\chi(\bc)}{c^2} \cdot \sum_{q_1q_2=-\Delta}'\BM_{q_1,q_2}(\bc) \Res_{s_1=1}\Res_{s_2=1} \xi(s_1,s_2;F,\chi,\bc),
\end{equation}
where $\xi(s_1,s_2;F,\chi,\bc)$ is the series defined in \eqref{eq: def of xi(s_1,s_2)}.
Here we use $\Res_{s=1}$ to denote the first nonzero coefficient in the Laurent series expansion at $s=1$. 
We will begin our analysis gradually by starting with the combinatorially simplest case: when $f(z)$ is irreducible. In this case, it is known by \cite{ColliotSSDI} that $X_{\Delta,f}$ satisfies the Hasse principle; we aim for the stronger result that if there are no local obstructions to $X_{\Delta,f}$ then there are $\asymp B\log(B)^{\rho_{\Delta,f}-2}$ rational points of height at most $B$. 

Second, we introduce the auxiliary varieties that are inherent to our analysis -- this comes to the sum over $c_{ij}$. These varieties also appear in the finite and effective computation of \cite{ColliotTheleneCoraySansuc} and \cite{CTSSchinzel} for determining if $X_{\Delta,f}(\Q) = \emptyset.$ In this subsection, we also establish that if $f(z)$ has a linear factor then there exists an auxiliary variety containing a rational point. 

Finally, we will classify when $c_{\Delta,f}=0$ by handling the cases when $f(z)$ factorizes. This will include the case when $f(z)$ factorizes as a product of two quadratics -- the only scenario when the Hasse principle may fail.

\subsection{The irreducible case}\label{subsec: constant irred case}
Let $f(z)$ be an irreducible polynomial. Then there are no torsors $c$ involved and our expression (\ref{eq: leading constant}) becomes:
\begin{equation*}
    \sum_{q_1q_2=-\Delta}' \BM_{q_1,q_2}(1) \cdot \Res_{s_1=1}\Res_{s_2=1}\xi(s_1,s_2;F,\chi,1).
\end{equation*}
In Lemma \ref{lem: BM constant}, we find that if $$\BM(1) = \sum_{q_1q_2=-\Delta}' \BM_{q_1,q_2}(1) = 0$$
then for some $p\mid -\Delta$, $X_{\Delta,f}(\Q_p)\neq \emptyset.$ In other words, $\BM(1)$ detects the local obstructions for primes dividing $-\Delta$. 

So, we consider 
\begin{equation}\label{eq: irreducible sum of singular series}
    \Res_{s_1=1,s_2=1}\xi(s_1,s_2;F,\chi,1) = \Res_{s_1=1,s_2=1} \sum_{\gcd(\ft,\overline{\ft})=1} \frac{1}{N(\ft)^{s_1}}\sum_{\fk\mid \ft} \frac{\mu(\fk)}{N(\fk)^2} \sum_{d} \frac{\chi(d)\varrho_F^*(kd)}{d^{2s_2}}\prod_{p\mid kd}\left(1-\frac{1}{p^2}\right)^{-1}. 
\end{equation}
Let us define the local factor for $p\nmid -\Delta$:
\begin{equation}
    L_p(s_1,s_2;F,\chi,1) = 1+(1-p^{-2})^{-1}\cdot \left(\frac{\chi(p)+1}{p^{s_1}} \left(1-\frac{\varrho_F^*(p)}{p^2}\right) + \frac{\chi(p)\varrho_F^*(p)}{p^{2s_2}}+\hdots\right).
\end{equation}
It suffices to show a uniform lower bound on 
$$(1-p^{-1})^{\rho_{\Delta,f}-1}L_p(1,1;F,\chi,1).$$
Here we have recalled that $\xi(s_1,s_2;F,\chi,1)$ has a simple pole near $s_1=1$ and a pole of order $\rho_{\Delta,f}-2$ at $s_2=2.$

First, observe that since $\rho_F^*(n)\leq 4^{\omega(n)}n$, for $p\gg 1$ we have a uniform lower bound on $L_p(1,1;F,\chi,1)$ given by the geometric series: 
\begin{equation*}
    (1-p^{-1})^{\rho_{\Delta,f}-1} \left(1- \left(1-\frac{1}{p^2}\right)^{-1}\left(\frac{4}{p} +\frac{16}{p^2} - +\frac{2\cdot (4p-8)}{p^3} + \dots \right)\right). 
\end{equation*}
In other words, for $p\gg 1$, the above expression is uniformly bounded below by some constant $c'>0$. With this bound in hand, it suffices to show that for any prime $p\nmid -\Delta$, we have that $L_p(1,1;F,\chi,1)\neq 0.$

For any exponent $e$, we know definitionally that 
\begin{equation}
    \varrho_F^*(p^e) - \frac{\varrho_F^*(p^{e+1})}{p^2} \geq 0,
\end{equation}
as any element satisfying $p^{e+1}\mid F(x,y)$ must satisfy that $p^e \mid F(x,y)$. Further, we know that 
\begin{equation}
    (p^2-1) - \varrho_F^*(p)\geq 0,
\end{equation}
since the count $\varrho_F^*(p)$ excludes when $(x,y)=(0,0)\bmod p$. 
Together, this tells us that if $\chi(p)=1$, then $L_p(1,1;F,\chi,1)>0$, as it is the sum of 1 and only nonnegative terms. Indeed, we do not expect any local obstructions to occur when $\chi(p)=1$.  

If $\chi(p)=-1$, then (since $\chi(p)+1=0$ and the terms dependent on $s_1$ disappear) we see that 
\begin{equation*}
    L_p(1,1;F,\chi,1) = \left(1 - (1-p^{-2})^{-1} \frac{\varrho_F^*(p)}{p^2}\right)+(1-p^{-2})^{-1}\left(\frac{\varrho_F^*(p^2)}{p^4}- \frac{\varrho_F^*(p^3)}{p^6}\right)+\dots 
\end{equation*}
Using the grouping indicated above, we again see that this term must be nonnegative. If $L_p(1,1;F,\chi,1)=0$, then we must additionally have that 
\begin{equation*}
    \varrho_F^*(p) = p^2-1, \quad \varrho_F^*(p^{2k}) = \frac{\varrho_F^*(p^{2k+1})}{p^2}, \forall k\in \Z_{\geq 1}. 
\end{equation*}
The left (right) equality holds if and only if $p$ divides all values of $F(x,y)$ ($p^{2k+1}$ divides all values of $F(x,y)$). In other words, for some $k\geq 1$, we have that $p^{2k-1} \mid F(x,y)$ for all $x,y\in \Z$, but there are no solutions to $F(x,y)\equiv 0 \bmod p^{2k}$ with $\gcd(x,y,p)=1$. Since $\chi(p)=-1$, this tells us that there are no nontrivial solutions to the equation $$x^2+\Delta y^2 = F(x,y) \bmod p^{2k},$$
and thus there must be a local obstruction, i.e. $X_{\Delta,f}(\Q_p)=\emptyset.$

Thus, we have proved the following result. 
\begin{prop}
    Let $f(z)$ be irreducible. Then if the leading constant $c_{\Delta,f}$ is zero then there is a local obstruction  for $X_{\Delta,f}$. Consequently, if there are no local obstructions, then $c_{\Delta,f}\neq 0$ and hence $X_{\Delta,f}(\Q) \neq \emptyset$. Thus, the Hasse principle holds for $X_{\Delta,f}.$
\end{prop}

\subsection{Descent and auxiliary varieties}
To analyze (\ref{eq: leading constant}) when $f(z)$ factorizes, we must first introduce the auxiliary varieties $X^*_{\Delta,f_1,...,f_r,\alpha_1,...,\alpha_r}$ (defined in (\ref{eq: def of variety X_Delta, f_1,...,f_r})). Fortunately, these auxiliary varieties are well-studied and appear as crucial objects in Colliot-Th\'el\`ene, Coray, and Sansuc's work on the Brauer-Manin obstruction and descent. Specifically, in a sequence of papers \cite{CTSansucDescentI,CTSansucDescentII}, Colliot-Th\'el\`ene and Sansuc established the method of descent. In our proof, we would like to use the fibrations discussed by Colliot-Th\'el\`ene, Coray, and Sansuc in \cite{ColliotTheleneCoraySansuc} when $f(z)$ factorizes as the product of two irreducible quadratics.  
It is of note that in the proof of \cite[Proposition 3]{CTSSchinzel} it is established that a similar set of auxiliary surfaces are intimately tied to these questions; observe that in \cite[Proposition 3]{CTSSchinzel}, these varieties are denoted as $W_g$ and take into account extra homogenization when $f(z)$ has an odd degree irreducible factor. 

Let us recall from (\ref{eq: def of variety X_Delta, f_1,...,f_r}) the definition of our auxiliary varieties:
\begin{equation*}
X^*_{\Delta,f_1,...,f_r,\alpha_1,...,\alpha_r} = \{(z,(x_i,y_i)_{i=1}^r: x_i^2 + \Delta y_i^2 = \alpha_i f_i(z)\neq 0\}.
\end{equation*}
In the next subsection, we describe how $X^*_{\Delta,f_1,...,f_r,\alpha_1,...,\alpha_r}$ relates to $X_{\Delta,f}$ in many ways. Most importantly is the following consequence of Harari's formal lemma \cite[Corollary 2.6.1]{Harari} (see also \cite[Theorem 13.4.3]{CTSkoroBGBook} and the example \cite[p. 343]{CTSkoroBGBook}) and was demonstrated to the author by Colliot-Th\'el\`ene in \cite{CTPrivate}. In \cite{ColliotTheleneCoraySansuc} and \cite{CTSSchinzel}, certain cases are also shown under certain degree parity conditions. 

\begin{thm}[Colliot-Th\'el\`ene \cite{CTPrivate}]\label{thm: descent}
    Assume that $X_{\Delta,f}(\Q_v)\neq \emptyset$ for every place $v$ and there is no Brauer-Manin obstruction to the Hasse principle for $X_{\Delta,f}$. Then there exists an $(\alpha_1,...,\alpha_r)\in (\Q^\times)^r$ satisfying $\alpha_1\dots \alpha_r=1$, $X^*_{\Delta,f_1,...,f_r,\alpha_1,...,\alpha_r}(\Q_v)\neq \emptyset$ for all places $v$. 
\end{thm}
\begin{proof}
    By assumption, let $(x_v,y_v,z_v)\in X_{\Delta,f}(\A_\Q)^{\textrm{Br}(X)}$; we can furthermore assume that $(x_v,y_v,z_v)\in X^*_{\Delta,f}(\A_\Q)$. Let $U_{\Delta,f}$ denote the smooth compactification of $X^*_{\Delta,f}$ and hence $U_{\Delta,f}\subset X_{\Delta,f}$. Then let $S$ denote the finite set of places of $\Q$ such that $$S := \{v: \exists \gamma \in \textrm{Br}(U_{\Delta,f}), \gamma(x_v,y_v,z_v)\neq 0\}.$$ 
    Let $\beta_i = (-\Delta, f_i)\in \textrm{Br}(U_{\Delta,f})$ denote an element of the Brauer group of $U_{\Delta,f}$, for each $i=1,\hdots, r-1$. Then Harari's formal lemma \cite[Theorem 13.4.3]{CTSkoroBGBook} gives an extension of $(x_v,y_v,z_v)_{v\in S}$ to an adelic point $(M_v)_{v\in \mathscr{P}\cup \{\infty\}}$ such that 
    $$\sum_{v\in \mathscr{P}\cup \{\infty\}} \beta_i(M_v) = 0, \forall i=1,\hdots,r-1.$$
    Note that Harari's formal lemma specifically applies to those $\beta_i \in \textrm{Br}(U_{\Delta,f})\backslash \textrm{Br}(X_{\Delta,f})$; however, this equality trivially holds for any element of $\textrm{Br}(X_{\Delta,f})$ since $(x_v,y_v,z_v)\in X_{\Delta,f}(\A_\Q)^{\textrm{Br}(X)}.$

    Next, we consider the exact sequence 
    $$1\rightarrow \Q^\times/ N_{K/\Q}(K^\times) \rightarrow \oplus_{v\in \mathscr{P}\cup\{\infty\}} \Q_v^\times / N_{K_v/\Q_v}(K_v^\times) \xrightarrow{\gamma} \Z/2\Z \rightarrow 1.$$
    Consider for $i\in \{1,\hdots, r-1\}$ the map $$(M_v)\mapsto (f_i(z_v)) \in \oplus_{v\in \mathscr{P}\cup\{\infty\}} \Q_v^\times/N_{K_v/\Q_v} (K_v^\times).$$ 
    Since we know that 
    $$\gamma((f_i(z_v))_v) = \sum_{v\in \mathscr{P}\cup\{\infty\}} \beta_i(M_v) =0 \in \Z/2\Z,$$
    there must exist an element $\alpha_i \in \Q^\times$ such that 
    $$f_i(z_v) = \alpha_i N_{K/\Q}(x_i+y_i\sqrt{-\Delta}) \neq 0.$$
    Furthermore, the above holds simultaneously for all $\beta_1,\hdots,\beta_{r-1}$, so we must have that the auxiliary variety
    $$\begin{array}{c}
         x_1^2+\Delta y_1^2 = \alpha_1 f_1(z) \neq 0 \\
         \vdots \\
         x_{r-1}^2 + \Delta y_{r-1}^2 = \alpha_{r-1} f_{r-1}(z) \neq 0\\
         x_r^2 + \Delta y_r^2 = (\alpha_1\cdots \alpha_{r-1})^{-1} f_{r}(z)\neq 0
    \end{array}$$
    contains an adelic point; in other words, for $\alpha_r = (\alpha_1\cdots \alpha_{r-1})^{-1}$, we have that  $(M_v)_{v\in \mathscr{P}\cup \{\infty\}} \in X^*_{\Delta,f_1,...,f_r,\alpha_1,...,\alpha_r}(\A_\Q).$ 
\end{proof}

\begin{remark}\label{rem: construction of a rational point from aux}
We note that for any point $(z,(x_i,y_i)_{i=1}^r)\in X^*_{\Delta,f_1,...,f_r,\alpha_1,...,\alpha_r}(\Q)$, we can create a point in $X_{\Delta,f}(\Q)$ given by $$x+y\sqrt{-\Delta} = \prod_{i=1}^r (x_i+y_i\sqrt{-\Delta}), \textrm{ } z=z.$$    
Thus, if the family of varieties $X^*_{\Delta,f_1,...,f_r,\alpha_1,...,\alpha_r}$ satisfies the Hasse principle, Theorem \ref{thm: descent} gives that the Brauer-Manin obstruction is the only possible obstruction to the Hasse principle. In fact, this is exactly the approach Colliot-Th\'el\`ene, Coray, and Sansuc \cite{ColliotTheleneCoraySansuc} used to prove that when $f=f_1(z)f_2(z)$ is the product of two irreducible quadratics the only possible obstruction to the Hasse principle is the Brauer-Manin obstruction. 
Moreover, in \cite[Proposition 3 and 4]{CTSSchinzel}, Colliot-Th\'el\`ene and Sansuc used this strategy with a variation of the auxiliary varieties to prove that under the assumption of Schinzel's Hypothesis H, the only possible obstruction to the Hasse principle for Ch\^atelet surfaces is the Brauer-Manin obstruction. 

However, in the resolution of the question of the Hasse principle for Ch\^atelet surfaces, Colliot-Th\'el\`ene, Sansuc, and Swinnerton-Dyer \cite{ColliotSSDI}\cite{ColliotSSDII}, use a different parameterization of the rational points on $X_{\Delta,f}$ -- this parameterization is given by torsors. This allows them to approach the hardest case, when $f(z)$ is irreducible, which is circuitous with the auxiliary varieties above. Fortunately for us, this case when $f(z)$ is irreducible is actually the clearest example for our argument, as demonstrated in the previous subsection.

\end{remark}

Next, we make some more observations on connections between rational points on $X_{\Delta,f}$ and those on $X^*_{\Delta,f_1,...,f_r,\alpha_1,...,\alpha_r}.$

\begin{lemma}\label{lem: existence of rat point on torsor}
    Assume that $X_{\Delta,f}(\Q) \neq \emptyset$. Then there exists $(\alpha_1,...,\alpha_r)\in (\Q^\times)^r$ such that $$X^*_{\Delta,f_1,...,f_r,\alpha_1,...,\alpha_r}(\Q)\neq \emptyset.$$
\end{lemma}
\begin{proof}
First note that if $X_{\Delta,f}(\Q)\neq \emptyset$, then $X_{\Delta,f}$ must be unirational (see \cite[Proposition 9.8.(iii)]{ColliotSSDII} and \cite[Theorem A]{CorayTsfasman}) and hence we have that $X_{\Delta,f}(\Q)$ is Zariski-dense; so, $X^*_{\Delta,f}(\Q)\neq \emptyset.$
    Assume that for some $(x_0,y_0,z_0)\in \Q$ we have that $x_0^2+\Delta y_0^2 = f(z_0)\neq 0$. 
    Then set $\alpha_i = f_i(z_0)\neq 0$ for each $i$ and consider $$x_i^2 + \Delta y_i^2 = \alpha_i f_i(z).$$
    This variety clearly has a rational point at $z=z_0$ for each $i$. Thus, $X^*_{\Delta,f_1,...,f_r,\alpha_1,...,\alpha_r}(\Q) \neq \emptyset.$ 
\end{proof}

For the remainder of the section, we aim to prove the following proposition:
\begin{prop}\label{prop: leading constant auxiliary}
    Let $\Delta\in \Z$ be a squarefree integer satisfying that $\sqrt{-\Delta}\not\in \Q$. Let $f(z)\in \Z[z]$ be a separable polynomial of degree 3 or 4. Let $X_{\Delta,f}$ be the corresponding Ch\^atelet surface. Then as $B\rightarrow\infty,$ $$N(X_{\Delta,f},B) = c_{\Delta,f} B\log(B)^{\rho_{\Delta,f}-1} + O(B\log(B)^{\rho_{\Delta,f}-1-10^{-10}})$$
    with $c_{\Delta,f}$ nonzero if there exists an $(\alpha_1,...,\alpha_r)\in (\Q^\times)^r$ such that $\alpha_1\dots\alpha_r=1$ and $X^*_{\Delta,f_1,...,f_r,\alpha_1,...,\alpha_r}(\Q_v)\neq \emptyset$ for all places $v$. 
\end{prop}
By Theorem \ref{thm: descent}, this tells us that $c_{\Delta,f}=0$ if and only if there is either a local or Brauer-Manin obstruction to rational points on $X_{\Delta,f}.$

\subsubsection{Finite classes of torsors}
Now, we show that we actually only need to consider a finite set of values $(\alpha_1,...,\alpha_r)$ to understand $X_{\Delta,f}$. 
In \cite[Proposition 3]{CTSSchinzel}, it is shown that one must only look at a finite set of values $(\alpha_1,...,\alpha_r)$ that belong to the kernels of certain natural maps. This sentiment of restricting to a finite set of representatives $(\alpha_1,...,\alpha_r)$ is repeated in the lemma below, but first we shall homogenize our varieties. 

Let us fix a tuple $(\alpha_1,...,\alpha_r)\in \Q^r$ satisfying $\alpha_1...\alpha_r = 1.$ We observe that we can instead look at the integral points on the homogenization of these variety $X^*_{\Delta,f_1,...,f_r,\alpha_1,...,\alpha_r}$: 
$$Y^*_{\Delta,f_1,...,f_r,\alpha_1,...,\alpha_r} := \{x_i^2 + \Delta y_i^2 = \alpha_i t_i^2 F_i(u,v)\neq 0, \gcd(x_i,y_i,t_i)=\gcd(u,v)=1,\forall i\}.$$
    A version of this reduction is described explicitly by de la Bret\`eche and Browning in the derivation of \cite[Lemma 4.2]{delaBrowningHasse} for $\Delta=1$. Next, we note that if $\alpha_i = \alpha_i' \beta_i$, where $\beta_i \in \Z$ and satisfies that $\beta_i = b_i^2$, i.e. $\beta_i$ is a square, then we have that $$X_{\Delta, f_1,...,f_r,\alpha_1,...,\alpha_r}^*\cong X^*_{\Delta,f_1,...,f_r,\alpha_1',...,\alpha_r'}.$$ 
    By the equivalence of these varieties, we can assume that $\alpha_i\in \Z$ are squarefree. We can also assume that $f_1,\dots,f_r$ all satisfy that the coefficients of $f_i$ are at most divisible by a squarefree integer. Moreover, we also know that the product of all of the $\alpha_i$ is a square.
    
    Further, if $\alpha_i = \alpha_i' p_i$ for $\chi(p_i)=1$, then we fix some representative $c_{g_i}$ for the class group element $g_i = [\fp_i]\in C$. If $\gamma_i = N(c_{g_i})$, then we also have that 
    \begin{equation*}
        X_{\Delta, f_1,...,f_r,\alpha_1,...,\alpha_r}^*\cong X^*_{\Delta,f_1,...,f_r,\alpha_1'\gamma_1,...,\alpha_r'\gamma_r}.
    \end{equation*}
    So, we can assume, up to some fixed elements depending on the class group $C$, that $\alpha_i\in \Z$ satisfy that if $p\mid \alpha_i$ then $\chi(p)=-1$. We assume that $N(c_{[1]})=1$. 

    \begin{lemma}\label{lem: finite num of torsors}
        Let $\alpha_1,...,\alpha_r\in \Z$ be squarefree and assume that if $p\mid \alpha_i$ then either $\chi(p)=-1$ or $p$ divides some norm of a fixed set of class group representatives. Moreover, assume that $\alpha_1\dots\alpha_r$ is a square. If $Y^*_{\Delta,f_1,...,f_r,\alpha_1,...,\alpha_r}(\Q)\neq \emptyset$, then $$\alpha_i \mid \prod_{j\neq i} \Res(F_i,F_j).$$
    \end{lemma}
    \begin{proof}
    Let $((x_i,y_i,t_i),u,v)$ be a point on $Y^*_{\Delta,f_1,...,f_r,\alpha_1,...,\alpha_r}$.
        Since $p\mid \alpha_i$ implies that $\chi(p)=-1$ or $p$ divides the norm of a nontrivial class group representative, we must have that $\alpha_i\mid F_i(u,v)$ for each $i$ in order for the equation $$x_i^2+\Delta y_i^2 = \alpha_i t_i^2 F_i(u,v)\neq 0$$
        to be solvable. Let us assume that $\alpha_i$ has the following prime factorization: 
        $$\alpha_i = \prod_{\chi(p)=-1} p^{e_{p,i}},$$
        with $e_{p,i}\in \{0,1\}.$
        Since the product of all of the $\alpha_i$ must be a square, we must have that $$e_{p,i} = \sum_{j\neq i} e_{p,i,j}$$
        where $p^{e_{p,i,j}} \mid \gcd(\alpha_i,\alpha_j).$ Then we also have that $p^{e_{p,i,j}}\mid F_j(u,v)$ as well. Since $\gcd(u,v)=1$, this implies that $p^{e_{p,i,j}}\mid \Res(F_i,F_j)$. Since this holds for every pair $i$ and $j$ and all prime factors of $\alpha_i,\alpha_j$, we have that $$\alpha_i \mid \prod_{j\neq i} \Res(F_i,F_j).$$
    \end{proof}
    \begin{remark}
        We note that if there is some $b_i$ that divides all of the coefficients of $F_i$, then $b_i\mid \Res(F_i,F_j)$ for any $j\neq i$. This is analogous to tracking the primes that divide the coefficients of $f(z)$ when $f(z)$ is irreducible in the case described above. 
    \end{remark}

From Lemma \ref{lem: finite num of torsors}, we see that it suffices to consider a finite set of tuples $(\alpha_1,...,\alpha_r)$ in order to prove Proposition \ref{prop: leading constant auxiliary}. In particular, we can restrict ourselves to those $\alpha_1,...,\alpha_r$ that are squarefree and divide the pairwise resultants. These line up with the tuples $\bc$ that we sum over in (\ref{eq: leading constant}).

\subsubsection{Possible local obstructions}
First, we want to determine for which primes $p$ we can have the possibility of local obstructions for $Y^*_{\Delta,f_1,...,f_r,\alpha_1,...,\alpha_r}.$

    \begin{lemma}\label{lem: local criteria Hasse}
        Fix $\alpha_1,...,\alpha_r\in \Z$ as in Lemma \ref{lem: finite num of torsors}. Let $\ell$ be a finite prime. Then the following are equivalent: 
        \begin{enumerate}
            \item $Y_{\Delta,f_1,\dots,f_r,\alpha_1,\dots,\alpha_r}^*(\Q_\ell)\neq \emptyset$, 
            \item There exists a solution to \begin{equation*}
            x_i^2+\Delta y_i^2 = \alpha_i F_i(u,v), \forall i=1,\dots,r
        \end{equation*}
        over $\Z_\ell$ with $\gcd(u,v,\ell)=1$.
        \end{enumerate} 
        Further, if $\ell\mid \prod_{i=1}^r \alpha_i$ and $\chi(\ell)=1$ then $Y^*_{\Delta,f_1,\dots,f_r,\alpha_1,\dots,\alpha_r}(\Q_\ell)\neq \emptyset$. If $\chi(\ell)=-1$ and $\ell^{e_i} \| \alpha_i$, then $Y^*_{\Delta,f_1,...,f_r,\alpha_1,...,\alpha_r}(\Q_\ell)\neq \emptyset$ if and only if \begin{equation*}
            \#\{ (u,v) \in \F_\ell: \ell^{e_i\bmod 2} \mid F_i(u,v), \gcd(u,v,\ell)=1\}>0.
        \end{equation*}
    \end{lemma}
    
    \begin{remark}
        By our assumptions on $\alpha_i$, we have that $e_i\in \{0,1\}.$
    \end{remark}

    \begin{proof}
        First, let us start with the equation $$x^2 + \Delta y^2 = c.$$
        For $\ell \nmid c(-\Delta)$, there will always be solutions over $\Q_\ell$ by the Lang-Weil theorem and Hensel's Lemma; hence there is no chance of a local obstruction at these places. 


        Consider if $\ell \mid c$ and $\gcd(\ell,-\Delta)=1$. Now, if $\chi(\ell)=-1$, we know that there can only be a solution to $x^2+\Delta y^2 = c$ if $c = \ell^{2e}m,$ for some $e\geq 1$ and  $\gcd(m,\ell)=1$. On the other hand, if $\chi(\ell)=1$, $x^2+\Delta y^2$ factorizes as a product of linear factors mod $\ell$. Thus, there is a nonzero solution to $x^2+\Delta y^2 = c\bmod \ell$, and by Hensel's lemma this can be lifted to a solution in $\Q_\ell$.  
        
        Now, we return to the original variety $Y^*_{\Delta,f_1,...,f_r,\alpha_1,...,\alpha_r}.$ Now, we return to the original variety $Y^*_{\Delta,f_1,...,f_r,\alpha_1,...,\alpha_r}.$ The first part of the lemma statement is follows from the definition of $Y_{\Delta,f_1,\dots,f_r,\alpha_1,\dots,\alpha_r}^*$. 
        Let $\ell \mid \prod_{i=1}^r \alpha_i$. If $\chi(\ell)=1$, then since $x^2+\Delta y^2$ factors as a product of linear products mod $\ell$, by a combination of Lang-Weil and Hensel's Lemma there will always be a solution over $\Q_\ell$. 
        
        So, we can assume from now on that $\chi(\ell)=-1$. Let $e_i$ be the power such that $\ell^{e_i}\| \alpha_i$. We can see that $Y^*_{\Delta,f_1,...,f_r,\alpha_1,...,\alpha_r}(\Q_\ell)\neq \emptyset$ if and only if there is a point $(u,v)\in \Z_\ell^2$  with $\gcd(u,v,\ell)=1$ such that $\ell^{e_i\bmod 2}\mid F_i(u,v)$ (as we need an even power of $\ell$ to divide $\alpha_i F_i(u,v)$). In other words, we must have that 
        \begin{equation*}
            \#\{ (u,v) \in \F_\ell: \ell^{e_i\bmod 2} \mid F_i(u,v), \gcd(u,v,\ell)=1\}>0.
        \end{equation*}This is in fact sufficient because from such a point with $\ell^{e_i\bmod 2} \mid F_i(u,v)$, we can build a solution in $Y^*_{\Delta,f_1,...,f_r,\alpha_1,...,\alpha_r}(\Q_\ell)$.

    \end{proof}

\subsection{Computing the constant}
Let us return to analyzing (\ref{eq: leading constant}) in the general case when $f(z)$ factorizes; in this section, we prove Proposition \ref{prop: leading constant auxiliary}. First, let us fix a $\bc$ and define the local factor at a prime $p\nmid -\Delta$ as $L_p(s_1,s_2;F,\chi,\bc)$; when $p\mid -\Delta$, we use the convention that $L_p(s_1,s_2;F,\chi,\bc)=1.$ 
We recall that we are looking at the expression: 
\begin{equation*}
    \sum_{c_{ij}\mid \Res(F_i,F_j)} \frac{\mu(\bc)\chi(\bc)}{c^2}\left(\sum_{q_1q_2=-\Delta}' \BM_{q_1,q_2}(\bc)\right)\times \Res_{s_1=1}\Res_{s_2=2}\prod_{p\mid c}L_p(s_1,s_2;F,\chi,\bc) \prod_{p\nmid c} L_p(s_1,s_2;F,\chi,\bc).
\end{equation*}
Define $R = \prod_{i\neq j} \Res(F_i,F_j)$ and for any prime $p\mid R$, $\bc_p$ the vector containing the $p$-parts of $\bc$. Additionally, define 
$$\BM(\bc):=\sum_{q_1q_2=-\Delta}' \BM_{q_1,q_2}(\bc).$$
Then we can rearrange the above expression as: 
\begin{equation}\label{eq: leading constant with good primes separated}
    \Res_{s_1=1}\Res_{s_2=1} \prod_{p\nmid R} L_p(s_1,s_2;F,\chi,1)\times \left(\sum_{c_{ij}\mid \Res(F_i,F_j)} \frac{\mu(\bc)\chi(\bc)}{c^2} \BM(\bc) \times  \left(\prod_{p\mid R}L_p(1,1;F,\chi,\bc_p)\right)\right) .
\end{equation}

To understand the above expression better, our plan is as follows: first, we will analyze the product over primes $p\nmid R$; this reduces our computation to the finite sum over $\bc$. 
Second, we will determine when $\BM(\bc)=0$ (and this will in turn end up telling us if there is a local obstruction at $Y^*_{\Delta,f_1,\dots,f_r,\bc}(\Q_{\ell})$ for any $\ell\mid -\Delta$). Third, we must look at this sum over $\bc$ and determine when there can be perfect cancellation (which will correspond to local obstructions for primes dividing $c$).

\begin{lemma}\label{lem: residue of L function not dividing Res is nonzer}
    Let $R = \prod_{i\neq j} \Res(F_i,F_j)$. Then, if 
    \begin{equation*}
        \Res_{s_1=1}\Res_{s_2=1} \prod_{p\nmid R} L_p(s_1,s_2;F,\chi,1) = 0,
    \end{equation*}
    then for some prime $p\nmid R(-\Delta)$ with $\chi(p)=-1$, $Y_{f_1,\dots,f_r,\bc}^*(\Q_p)=\emptyset.$
\end{lemma}
\begin{proof}
    For $p\nmid R$, we have the following expression for $L_p(s_1,s_2;F,\chi,1):$
    \begin{equation*}
        1+(1-p^{-2})^{-1}\left(\frac{\chi(p)+1}{p^{s_1}} \left(1-\frac{\varrho_F^*(p)}{p^2}\right) + \frac{\chi(p)\varrho_F^*(p)}{p^{2s_2}} + \frac{\chi(p)+1}{p^{s_1+2s_2}} \left(\varrho_F^*(p) - \frac{\varrho_F^*(p^2)}{p^2}\right)+\dots\right) 
    \end{equation*}
    In other words, this is the factor that appeared in the case when $F$ is irreducible. Let us repeat the proof for completion. 

    Note that it suffices to show a uniform lower bound on $L_p(1,1;F,\chi,1)$ for $p\nmid R$. Indeed when $p\gg 1$, there is a uniform lower bound given by the geometric series (as $\varrho_F*(n)\leq 4^{\omega(n)}n$)
    \begin{equation*}
    (1-p^{-1})^{\rho_{\Delta,f}-1} \left(1- \frac{4}{p} +\frac{16}{p^2} - +\frac{2\cdot (4p-8)}{p^3} + \dots \right). 
\end{equation*}
In other words, for $p\gg 1$, the above expression is uniformly bounded below by some constant $c'>0$. With this bound in hand, it suffices to classify when we have that $L_p(1,1;F,\chi,1)= 0$ for some $p\nmid R(-\Delta).$

Since $\varrho_F^*(n)<n^2$ (as to exclude the from the count when $(x,y)\equiv 0 \bmod n$), 
we know that 
\begin{equation*}
    1-(1-p^{-2})^{-1}\frac{\varrho_F^*(p)}{p^2} > 0.
\end{equation*}
Additionally, we have that for any exponenet $e$, 
\begin{equation*}
    \varrho_F^*(p^e) - \frac{\varrho_F^*(p^{e+1})}{p^2} \geq 0,
\end{equation*}
as any element satisfying $p^{e+1}\mid F(x,y)$ must satisfy that $p^e \mid F(x,y)$. Together, these two facts give that $L_p(1,1;F,\chi,1)>0$ for any prime $p$ such that $\chi(p)=1$, as the local factor is a sum of $1$ with exclusively nonnegative terms. 

If $\chi(p)=-1$, then the term depending on $s_1$ disappears and we see that 
\begin{equation*}
    L_p(1,1;F,\chi,1) = \left(1 - (1-p^{-2})^{-1} \frac{\varrho_F^*(p)}{p^2}\right)+(1-p^{-2})^{-1}\left(\frac{\varrho_F^*(p^2)}{p^4}- \frac{\varrho_F^*(p^3)}{p^6}\right)+\dots 
\end{equation*}
Using the grouping indicated above, we again see that this term must be nonnegative. If $L_p(1,1;F,\chi,1)=0$, then we must additionally have that 
\begin{equation*}
    \varrho_F^*(p) = p^2-1, \quad \varrho_F^*(p^{2k}) = \frac{\varrho_F^*(p^{2k+1})}{p^2}, \forall k\in \Z_{\geq 1}. 
\end{equation*}
The left (right) equality holds if and only if $p$ divides all values of $F(x,y)$ ($p^{2k+1}$ divides all values of $F(x,y)$). In other words, for some $k\geq 1$, we have that $p^{2k-1} \mid F(x,y)$ for all $x,y\in \Z$, but there are no solutions to $F(x,y)\equiv 0 \bmod p^{2k}$ with $\gcd(x,y,p)=1$. Since $\chi(p)=-1$, this tells us that there are no nontrivial solutions to the equation $$x^2+\Delta y^2 = F(x,y) \bmod p^{2k},$$
and thus there must be a local obstruction, i.e. $X_{\Delta,f}(\Q_p)=\emptyset.$
\end{proof}

Next, we want to understand when $\BM(\bc)$ can be zero (and when it can be negative).

\begin{lemma}\label{lem: BM constant}
    Let $\bc$ be a fixed vector satisfying $\gcd(c_i,-\Delta)=1$ for each $i=1,\dots,r$. Then $$\BM(\bc):=\sum_{q_1q_2=-\Delta}' \BM_{q_1,q_2}(\bc)\geq 0.$$
    Additionally, $\BM(\bc)=0$ implies that there exists $\ell \mid -\Delta$ such that $Y^*_{\Delta,f_1,...,f_r,\bc}(\Q_\ell) = \emptyset.$ 
\end{lemma}

\begin{proof}
    Let us expand out the definition of $\BM_{q_1,q_2}(\bc)$: 
    \begin{multline*}
        \sum_{q_1q_2=-\Delta}' \BM_{q_1,q_2}(\bc) = \\ \sum_{q_1q_2=-\Delta}' \sum_{p\mid n\implies p\mid \Delta} \frac{\chi_{q_1}(p_{q_2}(n))\chi_{q_2}(p_{q_1}(n))\mathbf{1}_{\exists \fa: N(\fa)=n}}{n^2(-\Delta)^2} \sum_{\substack{\ba \bmod -\Delta \\ \gcd(a_i,-\Delta)=1}} \prod_{i=1}^r (1+\chi(a_i)) \chi_{q_1}(a)\varrho_F^*(n,\bc\ba)\\
        = \sum_{p\mid n\implies p\mid -\Delta}\frac{1}{n^2(-\Delta)^2}\sum_{\substack{\ba \bmod -\Delta \\ \gcd(a_i,-\Delta)=1}} \prod_{i=1}^r (1+\chi(a_i)) \varrho_F^*(n,\bc\ba) \sum_{q_1q_2=-\Delta}'
\chi_{q_1}(p_{q_2}(n))\chi_{q_2}(p_{q_1}(n)) \chi_{q_1}(a)\mathbf{1}_{\exists \fa: N(\fa)=n}.
\end{multline*}
Observe that if $n$ satisfies that $p\mid n\implies p\mid -\Delta$, then we have that $$\varepsilon_{q_1,q_2}(n) =\chi_{q_1}(p_{q_2}(n))\chi_{q_2}(p_{q_1}(n))\mathbf{1}_{\exists \fa: N(\fa)=n}.$$
    In particular, if $p_{\neg -\Delta}(F_i(u,v)/c_i)\equiv a_i\bmod -\Delta$ and that $p_{-\Delta}(F(u,v))=n$, then we know that 
    $$\varepsilon_{q_1,q_2}(F(u,v)/c) = \chi_{q_1}(p_{q_2}(n))\chi_{q_2}(p_{q_1}(n)) \chi_{q_1}(a)(1\star \chi)(F(u,v)/cp_{-\Delta}(F(u,v))).$$

    Thus, we can rewrite the above expression as:
    \begin{multline}\label{eq: BM written with r_G}
        \BM(\bc) = \sum_{p\mid n\implies p\mid -\Delta} \frac{1}{n^2(-\Delta)^2} \sum_{\substack{\ba\bmod -\Delta\\ \gcd(a_i,-\Delta)=1}}\prod_{i=1}^r (1+\chi(a_i)) \sum_{\substack{u,v\bmod n(-\Delta) \\ \gcd(u,v,-\Delta)=1 \\ p_{-\Delta}(F(u,v))=n\\
        p_{\neg-\Delta}(F_i(u,v)/c_i)\equiv a_i\bmod -\Delta}} \sum_{q_1q_2=-\Delta}' \frac{\varepsilon_{q_1,q_2}(F(u,v)/c)}{(1\star \chi)(F(u,v)/c)}.
    \end{multline}
    We recall that $$\sum_{q_1q_2=-\Delta}' \varepsilon_{q_1,q_2}(n) = r_{[1]\calG}(n)=\sum_{\substack{N(\fa)=n\\ [\fa] \in [1]\calG}}1 \geq 0,$$
    where $[1]\calG$ denotes that $[\fa]$ sits in the identity coset of $C/\calG$, where $\calG$ denotes the genus subgroup. Moreover, $r_{\calG}(n)=0$ if $(1\star \chi)(n)=0$, so the above expression is well-defined and non-negative. This completes the proof that $\BM(\bc)\geq 0.$ \\

    Let us determine when $\BM(\bc)$ is zero. We inspect \eqref{eq: BM written with r_G} further: 
    \begin{multline*}
        \BM(\bc) = \sum_{p\mid n\implies p\mid -\Delta} \frac{\mathbf{1}_{\exists \fa: N(\fa)=n}}{n^2(-\Delta)^2} \sum_{\substack{\ba\bmod -\Delta\\ \gcd(a_i,-\Delta)=1}}\prod_{i=1}^r (1+\chi(a_i)) \\ \times \sum_{\substack{u,v\bmod n(-\Delta) \\ \gcd(u,v,-\Delta)=1 \\ p_{-\Delta}(F(u,v))=n\\
        p_{\neg-\Delta}(F_i(u,v)/c_i)\equiv a_i\bmod -\Delta}} \delta_{F(u,v)/c} \mathbf{1}_{\exists \fa\in [1]\calG, N(\fa)=F(u,v)/c},
    \end{multline*}
    where $\delta_{F(u,v)/c}>0$ is some positive constant. We immediately see that for $\BM(\bc)=0$, a few events could occur: 
    \begin{enumerate}
        \item When $n$ satisfies that $p\mid n\implies p\mid -\Delta$ and that $\exists \fa, N(\fa)=n$, for all $\ba$ such that $\varrho_F^*(n,\bc\ba)\neq 0$, there exists some $i$ such that $\chi(a_i)=-1$. 
        \item When $n$ satisfies that $p\mid n\implies p\mid -\Delta$ and that $\exists \fa, N(\fa)=n$, for all $\ba \bmod \Delta$ satisfying that $\chi(a_i)=1$ for all $i$, for all $(u,v)\bmod n(-\Delta)$ such that $\gcd(u,v,-\Delta)=1$, $p_{-\Delta}(F(u,v)) =n$ and $p_{\neg -\Delta}(F_i(u,v)/c_i)\equiv a_i\bmod -\Delta$, there does \textit{not} exist an ideal $\fa\in [1]\mathcal{G}$ satisfying that 
        $N(\fa) = F(u,v)/c.$ As $[1]\mathcal{G}$ is the set of quadratic form with discriminant $-\Delta$ that are locally equivalent to $x^2+\Delta y^2$, this implies that 
        $$X_{u,v,c} := \{x^2 + \Delta y^2 = F(u,v)/c\}$$
        has no solution over $\Q$, and since this is a quadratic form, this implies that $X_{u,v,c}(\A_\Q)=\emptyset.$
    \end{enumerate}

    Let us compare these events to the condition laid out in Lemma \ref{lem: local criteria Hasse}; in particular, we want to show that both events (1) and (2) above lead to the conclusion that there does not exist a solution to 
    \begin{equation*}
        x_i^2+\Delta y_i^2 = c_i F_i(u,v), \forall i=1,\dots,r,
    \end{equation*}
    over $\Z_\ell$ with $\gcd(u,v,\ell)=1$ for some $\ell\mid -\Delta$.

    We prove the contrapositive. Assume that there exists a solution $((x_i,y_i)_{i=1}^r, u,v)\in \Z_\ell$ with $\gcd(u,v,\ell)=1$  for each $\ell \mid -\Delta$. Let $\ell^{e_i} \| F_i(u,v)$ for each $i=1,\dots,r$ and set 
    $$n = \prod_{\ell\mid -\Delta} \ell^{e_i}.$$
    Then there is a solution to 
    \begin{equation*}
        x_i^2+\Delta y_i^2 = c_i F_i(u,v) \bmod n(-\Delta), \forall i=1,\dots,r,
    \end{equation*}
    satisfying that $\gcd(u,v,-\Delta)=1$ and $p_{-\Delta}(F(u,v))=n$. This implies a solution to 
    \begin{equation*}
        x^2+\Delta y^2 = cF(u,v) \bmod n(-\Delta)
    \end{equation*}
    with $\gcd(u,v,-\Delta)=1$. Consequently, since the only possible local obstructions to $X_{u,v,c}$ lie at the primes dividing $-\Delta$, we know that there exists an ideal $\fa\in [1]\mathcal{G}$ satisfying $N(\fa) = F(u,v)/c$, and hence an ideal $\fa'\mid \fa$ satisfying $N(\fa)=n$. Set $a_i = p_{\neg -\Delta}(F_i(u,v)/c_i)\bmod -\Delta$. If we can show that $\chi(a_i)=1$ for each $i$, then we know that events (1) and (2) both cannot occur. 

    Since we know that there exists $(x_i,y_i)\bmod n(-\Delta)$ such that 
    \begin{equation*}
        x_i^2+\Delta y_i^2 = c_iF_i(u,v) \bmod n(-\Delta),
    \end{equation*}
    we must have that for each $i=1,\dots,r$, $\chi(p_{\neg -\Delta}(c_iF_i(u,v)))=1$. Since $\gcd(c_i,-\Delta)=1$, we therefore have that 
    $$\chi(p_{\neg -\Delta}(c_iF_i(u,v))) = \chi(p_{\neg -\Delta}(F_i(u,v)/c_i)) = \chi(a_i) = 1.$$
    Thus, $\chi(a_i)=1$ for each $i$ and events (1) and (2) cannot occur. Consequently, we know that $\BM(\bc)=0$ implies that there is a local obstruction at some prime $\ell \mid -\Delta$.

\end{proof}

\begin{lemma}\label{lem: BM change by split primes}
    For any $p_i$ such that $\chi(p_i)=1$, $$\BM(\bc) = \BM((p_1c_1,...,p_rc_r)).$$
\end{lemma}
\begin{proof}
        We note that by a change of variables,
    \begin{multline*}
        \sum_{\substack{\ba\bmod \Delta\\ \gcd(a_i,\Delta)=1}} \prod_{i=1}^r (1+\chi(a_i)) \varrho_F(n,(a_ic_ip_i)) \times  \sum_{q_1q_2=-\Delta}' \chi_{q_1}(p_{q_2}(n))\chi_{q_2}(p_{q_1}(n)) \chi_{q_1}(a) 
        \\ = \sum_{\substack{\ba \bmod \Delta\\ \gcd(a_i,\Delta)=1}} \prod_{i=1}^r (1+\chi(a_ip_i)) \varrho_F(n,\bc\ba) \times \sum_{q_1q_2=-\Delta}' \chi_{q_1}(p_{q_2}(n))\chi_{q_2}(p_{q_1}(n)) \chi_{q_1}(ap_1...p_r) \\
        = \sum_{\substack{\ba \bmod \Delta\\ \gcd(a_i,\Delta)=1}} \prod_{i=1}^r (1+\chi(a_i)) \varrho_F(n,\bc\ba) \times \sum_{q_1q_2=-\Delta}' \chi_{q_1}(p_{q_2}(n))\chi_{q_2}(p_{q_1}(n)) \chi_{q_1}(a) 
    \end{multline*}
    Here we use that $\chi(p_i)=1$ for each $i$, and since $q_1\mid -\Delta$, $\chi_{q_1}(p_i)=1$ for each $i$ as well. Thus, we have that $\BM(\bc) = \BM((c_1p_1,...,c_rp_r)).$
    
\end{proof}

Having completed the above lemmas, we return to \eqref{eq: leading constant with good primes separated}; recall that we are tasked with determining when the following expression is zero:
\begin{equation*}
    \sum_{c_{ij}\mid \Res(F_i,F_j)} \frac{\mu(\bc)\chi(\bc)}{c^2}\BM(\bc) \times \left(\prod_{p\mid R} L_p(1,1;F,\chi,\bc_p)\right),
\end{equation*}
where $R = \prod_{i\neq j} \Res(F_i,F_j)$. Observe that we can restrict to a sum over vectors $\bc$ with $\gcd(c_i,-\Delta)=1$ for each $i=1,\dots,r$, as the contribution from those vectors that are not coprime is zero. By Lemma \ref{lem: BM change by split primes}, we can further rearrange the expression as: 
\begin{equation*}
    \left(\sum_{\substack{c_{ij}\mid \Res(F_i,F_j)\\ p\mid c\implies \chi(p)=-1 \\ \mu^2(\bc)=1}} \frac{\BM(\bc)}{c^2}\cdot \prod_{p\mid c} L_p(1,1;F,\chi,\bc_p)\right) \times \left(\sum_{\substack{c_{ij}'\mid \Res(F_i,F_j) \\ p\mid c'\implies \chi(p)=1}}\frac{\mu(\bc')}{c'^2} \prod_{p\mid c'} L_p(1,1;F,\chi,\bc_p)\right).
\end{equation*}

We split this into two lemmas. 
\begin{lemma}\label{lem: leading constant split primes}
    If the following expression is nonzero
    \begin{equation*}
        \sum_{\substack{c_{ij}'\mid \Res(F_i,F_j) \\ p\mid c'\implies \chi(p)=1}}\frac{\mu(\bc')}{c'^2} \prod_{p\mid c'} L_p(1,1;F,\chi,\bc_p).
    \end{equation*}
\end{lemma}

\begin{lemma}\label{lem: leading constant nonsplit primes}
    If the following expression is zero
    \begin{equation*}
        \sum_{\substack{c_{ij}\mid \Res(F_i,F_j)\\ p\mid c\implies \chi(p)=-1 \\ \mu^2(\bc)=1}} \frac{\BM(\bc)}{c^2}\cdot \prod_{p\mid c} L_p(1,1;F,\chi,\bc_p),
    \end{equation*} 
    then for all $\bc$ counted above, $Y_{\Delta,f_1,\dots,f_r,\bc}^*(\A_\Q) = \emptyset;$ in other words, there is a local obstruction.  
\end{lemma}
With these two lemmas in hand, in addition to Lemma \ref{lem: finite num of torsors}, Lemma \ref{lem: residue of L function not dividing Res is nonzer} and Theorem \ref{thm: descent}, this completes the proof of the statement of Theorem \ref{thm: Manin}. In other words, this guarantees that if $c_{\Delta,f}=0$ then there is a local or Brauer-Manin obstruction to rational points on $X_{\Delta,f}$. 

\begin{proof}[Proof of Lemma \ref{lem: leading constant split primes}]
    Let $p\mid R$ be a prime satisfying that $\chi(p)=1$. Fix for now a vector $\bc$ in our summation. Assume that $\bc_p = (p^{e_i})_{i=1}^r$. Then we can expand out the local $L$-function factor (here, we recall that $\bk$ is a vector of length $r$ where exactly one entry is given by $p$ and all other entries are $1$ and we have a choice of decomposing $k$ such that $k_i \mid F_i(x,y)$; so, without loss of generality we choose that entry to be the first): 
    \begin{multline*}
        L_p(1,1;F,\chi,\bc_p) = (1-p^{-2})^{-1} \sum_{n=1}^\infty \frac{2}{p^n} \left(\sum_{i_j\geq e_j} \frac{\varrho_F^*((p^{i_1},\dots,p^{i_r}))}{p^{2(i_1+\dots+i_r)}} - \sum_{i_j\geq e_j} \frac{\varrho_F^*((p^{i_1+1},\dots,p^{i_r}))}{p^{2(i_1+\dots+i_r+1)}}\right)\\
        = (1-p^{-2})^{-1} \sum_{n=1}^\infty \frac{2}{p^n} \cdot \sum_{\substack{i_j\geq e_j \\ j\geq 2}} \frac{\varrho_F^*((p^{e_1},p^{i_2},\dots,p^{i_r}))}{p^{2(e_1+i_2+\dots+i_r)}}. 
    \end{multline*}
    Thus, we can write our summand as: 
    \begin{multline*}
        \sum_{\substack{c_{ij} \mid \Res(F_i,F_j)\\ p\mid c\implies\chi(p)=1}} \frac{\mu(\bc)}{c^2} \prod_{p\mid c} L_p(1,1;F,\chi,\bc_p) \\ = \prod_{\substack{p\mid R \\ \chi(p)=1}} \frac{2}{(1-p^{-1})(1-p^{-2})} \times   \sum_{e_j \leq \max_{i\neq j}\mathbf{1}_{p\mid \Res(F_i,F_j)}} \prod_{j=1}^r \frac{(-1)^{e_j}}{p^{2e_j}} \sum_{\substack{f_j\geq e_j \\ j\geq 2}} \frac{\varrho_F^*((p^{e_1},p^{f_2},\dots,p^{f_r}))}{p^{2(e_1+f_2+\dots+f_r)}}.
    \end{multline*}
    Let $\calI_p$ denote the indices $j$ for which $p\nmid \prod_{i\neq j}\Res(F_i,F_j)$. Then by inclusion-exclusion, the sum above collapses to: 
    \begin{equation*}
        \sum_{\substack{j\in \calI_p\setminus\{1\}\\ f_j\geq 0}} \frac{\varrho_F^*((1,\dots,p^{f_r})}{p^{2(\sum_{j\in \calI_p} f_j)}} - \mathbf{1}_{1\not\in \calI_p} \cdot \sum_{\substack{j\in \calI_p\setminus\{1\}\\ f_j\geq 0}} \frac{\varrho_F^*((p,\dots,p^{f_r}))}{p^{2(1+\sum_{j\in \calI_p} f_j)}}.
    \end{equation*}
    Now if $1\in \calI_p$, then it is clear that this factor is strictly positive, as the summands are nonnegative and $\varrho_F^*((1,1,\dots,1)) = 1$. So, it remains to check when $1\not\in \calI_p$. Then we note that for any fixed set of $f_j$, we have that definitionally,
    \begin{equation*}
        \frac{\varrho_F^*((1,\dots,p^{f_r}))}{p^{2(\sum_{j\in \calI_j} f_j)}} - \frac{\varrho_F^*((p,\dots,p^{f_r}))}{p^{2(1+\sum_{j\in \calI_p}f_j)}} \geq 0.
    \end{equation*}
    Thus, the summands are all nonnegative. Further, the summands are strictly positive unless $p\mid F_1$. In this case, for each $j>1$ we must have that $p\mid \Res(F_1,F_j) \mid \prod_{i\neq j}\Res(F_i,F_j)$; we also that that $p\mid \prod_{i\neq 1} \Res(F_i,F_1)$ as well. Hence $\calI_p = \emptyset$. So, the above expression becomes: 
    \begin{equation*}
        \varrho_F^*((1,\dots,1)) - \frac{\varrho_F^*((p,1,\dots,1))}{p^2} = 1 - \frac{p^2-1}{p^2}>0,
    \end{equation*}
    which is strictly positive. Here, we have used that if $p\mid F_1$, then $\varrho_F^*((p,1,\dots,1)) = p^2-1$, since we require $\gcd(u,v,p)=1$ in the count.  
\end{proof}

\begin{proof}[Proof of Lemma \ref{lem: leading constant nonsplit primes}]
First, let us fix a vector $\bc$ and a prime $p\mid c$ such that $\chi(p)=-1$. Let us notate that $\bc_p = (p^{e_i})_{i=1}^r$. We consider the local factor: 
\begin{equation*}
   L_p(1,1;F,\chi,\bc_p) = (1-p^{-2})^{-1} \sum_{i_j\geq 0} \frac{\chi(p)^{i_1+\dots+i_r}\varrho_F^*((p^{i_1+e_1},\dots,p^{i_r+e_r}))}{p^{2(i_1+\dots+i_r)}}. 
\end{equation*}

We now claim that if $\bc$ satisfies that $\mu^2(\bc)=1$ and if $p\mid c$ then $\chi(p)=-1$, then the following expression is nonnegative:
\begin{equation*}
    c^{-2} \prod_{p\mid c} L_p(1,1;F,\chi,\bc_p) = \prod_{p\mid c} (1-p^{-2})^{-1} \left(\sum_{i_j\geq 0} \frac{(-1)^{(i_1+\dots+i_r)}\varrho_F^*((p^{i_1+e_1},\dots,p^{i_r+e_r}))}{p^{2(i_1+\dots+i_r+e_1+\dots+e_r)}}\right).
\end{equation*}
We note that we can write this sum as 
\begin{multline*}
    \sum_{i=0}^{\infty} \left(\frac{\varrho_F^*((p^{e_1+2i},p^{e_2},\dots, p^{e_r}))}{p^{2(e_1+2i+\dots+e_r)}} - \frac{\varrho_F^*((p^{e_1+2i+1},p^{e_2},\dots,p^{e_r}))}{p^{2(e_1+2i+1+\dots+e_r)}} \right)\\ 
    - \sum_{i=0}^\infty  \left(\frac{\varrho_F^*((p^{e_1+2i},p^{e_2+1},\dots, p^{e_r}))}{p^{2(e_1+2i+e_2+1+\dots+e_r)}} - \frac{\varrho_F^*((p^{e_1+2i+1},p^{e_2+1},\dots,p^{e_r}))}{p^{2(e_1+2i+1+e_2+1+\dots+e_r)}}\right)\\
    + \sum_{i=0}^\infty  \left(\frac{\varrho_F^*((p^{e_1+2i},p^{e_2+2},\dots, p^{e_r}))}{p^{2(e_1+2i+e_2+2+\dots+e_r)}} - \frac{\varrho_F^*((p^{e_1+2i+1},p^{e_2+2},\dots,p^{e_r}))}{p^{2(e_1+2i+1+e_2+2+\dots+e_r)}}\right) - \hdots 
\end{multline*}
We now make a few observations: (1) Each of these summands is nonnegative; in particular 
\begin{equation*}
    \frac{\varrho_F^*((p^{i_1},p^{i_2},\dots,p^{i_r}))}{p^{2(i_1+\dots+i_r)}} - \frac{\varrho_F^*((p^{i_1+1},p^{i_2},\dots,p^{i_r}))}{p^{2(i_1+1+\dots+i_r)}}\geq 0.
\end{equation*}
This relation comes from the definition of $\varrho_F^*$, for if $p^{i_1+1}\mid F_1$ then $p^{i_1}\mid F_1$. (2) We also have the relation that for any tuple $(i_1,\dots,i_r)$, the following inequality holds:
\begin{multline*}
    \frac{\varrho_F^*((p^{i_1},p^{i_2},\dots,p^{i_r}))}{p^{2(i_1+\dots+i_r)}} - \frac{\varrho_F^*((p^{i_1+1},p^{i_2},\dots,p^{i_r}))}{p^{2(i_1+1+\dots+i_r)}} \\ \geq \frac{\varrho_F^*((p^{i_1},p^{i_2+1},\dots,p^{i_r}))}{p^{2(i_1+i_2+1+\dots+i_r)}} - \frac{\varrho_F^*((p^{i_1+1},p^{i_2+1},\dots,p^{i_r}))}{p^{2(i_1+1+i_2+1+\dots+i_r)}}.
\end{multline*}
This relation comes from the fact that one of the conditions counted in $\varrho_F^*((p^{i_1},\dots,p^{i_r}))$ is that $p^{i_j}\mid F_j(u,v)$, which is an independent condition for each $j$. Finally we make the observation that (1) and (2) hold irregardless of which indices we choose to compare (in other words, it did not need to be the first and second tuple). 

Together, this gives us that 
\begin{equation*}
    \sum_{i_j\geq 0} \frac{(-1)^{(i_1+\dots+i_r)}\varrho_F^*((p^{i_1+e_1},\dots,p^{i_r+e_r}))}{p^{2(i_1+\dots+i_r+e_1+\dots+e_r)}}\geq 0
\end{equation*}
and hence the claim is shown. 

Returning to our original expression: 
\begin{equation*}
    \textrm{NS}:= \sum_{\substack{c_{ij}\mid \Res(F_i,F_j)\\ p\mid c\implies \chi(p)=-1 \\ \mu^2(\bc)=1}} \frac{\BM(\bc)}{c^2}\cdot \prod_{p\mid c} L_p(1,1;F,\chi,\bc_p),
\end{equation*}
we can see that for NS to be zero, we must have that for every $\bc$ vector counted, either $\BM(\bc)=0$ or that 
\begin{equation*}
    c^{-2}\prod_{p\mid c}L_p(1,1;F,\chi,\bc_p)=0.
\end{equation*}
In Lemma \ref{lem: BM constant}, we saw that if $\BM(\bc)=0$, then there exists some $\ell\mid -\Delta$ such that there is a local obstruction, i.e. $Y^*_{\Delta,f_1,\dots,f_r,\bc}(\Q_\ell) = \emptyset.$ So, to conclude this lemma, we need to show that if the product of local factors is zero, then for some $p\mid c$ we have that $Y^*_{\Delta,f_1,\dots,f_r,\bc}(\Q_p) = \emptyset$. 

Let $p\mid c$ and $\chi(p)=-1$. It suffices to show that if $L_p(1,1;F,\chi,\bc_p)=0$ then $Y^*_{\Delta,f_1,\dots,f_r,\bc}(\Q_p) = \emptyset$. If $\varrho_F^*((p^{e_1},\dots,p^{e_r})) = 0$, then clearly $L_p(1,1;F,\chi,\bc_p)=0$. Since $\mu^2(\bc)=1$, we know that $e_j\in \{0,1\}$ for each $j\in \{1,\dots,r\}$, and that $e_j=1$ for some index $j$ (as $p\mid c$), so by Lemma \ref{lem: local criteria Hasse} would imply that $Y^*_{\Delta,f_1,\dots,f_r,\bc_p}(\Q_p)=\emptyset.$ 

So from now on assume that $\varrho_F^*((p^{e_1},\dots,p^{e_r}))>0$ and that $L_p(1,1;F,\chi,\bc_p)=0$. We note that in both observations (1) and (2) above, equality holds if and only if $p^{i_1+1} \mid F_1$ in the first case and $p^{i_2+1}\mid F_2$ in the second, or if all terms in the expressions are zero. Hence, if $L_p(1,1;F,\chi,\bc_p)=0$, then for some $j\in \{1,\dots,r\}$ and for some $i\geq 0$  we have that $p^{e_j+2i+1}\mid F_j$ and that there are no $(u,v)\bmod p^{e_j+2i+2}$ such that $\gcd(u,v,p)=1$ and $F_j(u,v)\equiv 0 \bmod p^{e_j+2i+2}$. Consider then the equation: 
\begin{equation*}
    x_j^2+\Delta y_j^2 = c_j F_j(u,v),
\end{equation*}
over $\Q_p$. We can see that this implies that $v_p(c_jF_j(u,v)) = p^{2e_j+2i+1}$ for all $(u,v)\in \Q_p$, which by the argument of Lemma \ref{lem: local criteria Hasse}, gives us again that $Y^*_{\Delta,f_1,\dots,f_r,\bc}(\Q_p)=\emptyset.$

\end{proof}

Thus,\textit{ if $c_{\Delta,f}=0$ then for all of the $\bc$ above from Lemma \ref{lem: finite num of torsors}, $Y^*_{\Delta,f_1,...,f_r,\bc}(\Q_\ell) = \emptyset$ for some place $\ell$}. This completes the proof of Proposition \ref{prop: leading constant auxiliary}, and thus Theorem \ref{thm: Manin}. \qed

\appendix
\section{A special case of the leading constant}\label{sec: appendix peyre}
Assume that $\Delta>0$ is an odd prime. Further, let us assume that $f(x)$ is an irreducible quartic that remains irreducible over $\Q(\sqrt{-\Delta}).$ Our goal for this appendix to relate the leading constant to Peyre's constant in this particular case. For $\Delta = 1$, this was previously completed in \cite{delaTenenbaum-Manin}. We would like to remind the reader at this point of our abusive notation: since $\Delta>0$, we use $-\Delta$ to denote the modulus of $\chi_{-\Delta}$, which is truly of size $4\Delta$.

In this case, we have the following expression for our leading constant: 
\begin{equation*}
    c_{\Delta,f} = \frac{2|\calU|}{|\mathcal{C}|} \left(\sum_{q_1q_2=-\Delta}' \BM_{q_1,q_2}(1)\right) \prod_{p\nmid -\Delta} \left(1-\frac{1}{p^2}\right)\cdot \Res_{s_1=1} \xi(s_1,1;F,\chi,1),
\end{equation*}
where $\xi(s_1,1;F,\chi,1)$ was defined in \ref{eq: def of xi(s_1,s_2)} as: 
\begin{equation*}
    \xi(s_1,1;F,\chi,1) = \sum_{\gcd(\ft,\overline{\ft})=1} \frac{1}{N(\ft)^{s_1}} \sum_{\fk\mid \ft} \frac{\mu(\fk)}{N(\fk)^2} \sum_{d} \frac{\chi(d) \varrho_F^*(kd)}{d^2}\prod_{p\mid kd}\left(1-\frac{1}{p^2}\right)^{-1}.
\end{equation*}

Before manipulating the intimidating expression above, let us first recall the definition of Peyre's constant in this case. 
Peyre's prediction for the leading constant in Manin's conjecture in this case is given by $\omega_\infty\prod_{p<\infty}w_p$,
where $\omega_p$ and $\omega_\infty$ are the Hardy-Littlewood densities. At the real place, one can show that the archimedean denisty $\omega_\infty$ is equivalent to the following limit: 
\begin{equation}
    \omega_\infty = \lim_{\delta\rightarrow 
    0} \lim_{\epsilon\rightarrow 0} \frac{\vol\{|x^2+\Delta y^2 -t^2 F(u,v)|<\epsilon/2: |tu^2|,|tv^2|\leq 1, \delta^{-1}<|t|\leq 1, x^2+\Delta y^2 \leq \nu_{\Delta,f}^2\}}{\epsilon\log(1/\delta)}
\end{equation}
It is also possible to express $\omega_\infty$ in terms of a singular integral; the expression above will have a clean and simple evaluation in terms of $\Delta$ and $f(z)$ later on. 
The $p$-adic densities $\omega_p$ are defined as follows: 
\begin{equation}
    \omega_p = \lim_{v\rightarrow\infty} \frac{1}{p^{4v}} \sum_{\substack{(u,v,x,y,t)\in (\Z/p^v\Z)^5\\ u^2+\Delta v^2 = t^2 F(x,y) \bmod p^v \\ p\nmid (x,y)\\ p\nmid (u,v,t)}} 1.
\end{equation}
The goal of this section is to prove the following proposition. 
\begin{prop}\label{prop: Peyre for Delta positive prime}
    Let $\Delta>0$ be an odd prime and let $f(z)$ be an irreducible quartic that remains irreducible over $\Q(\sqrt{-\Delta})$. Then, 
    \begin{equation*}
        c_{\Delta,f} = \frac{1}{4}\omega_\infty\prod_{p<\infty} \omega_p.
    \end{equation*}
\end{prop}
\begin{remark}
    We point out that our original expression in \eqref{eq: def of the height} is that 
    \begin{multline*}
        N(X_{\Delta,f},B) := \frac{1}{2}\cdot \#\Big\{ ((x,y),(u,v),t)\in \Z^2\times \Z^2 \times \Z_+: x^2 + \Delta y^2 = t^2 F(u,v) \\ \gcd(x,y,t)=\gcd(u,v)=1, \max(|tu^2|,|tv^2|,\|x\pm y\sqrt{-\Delta}\|)\leq \nu_{\Delta,f}B\Big\}.
    \end{multline*}
    So, this factor of $1/4$ is expected since we do not restrict $t$ to be nonnegative when we compute the archimedean density. 
\end{remark}

We start the proof by first analyzing the Dirichlet series $\xi(s_1,1;F,\chi,1)$ further. It will turn out that this Dirichlet series contains the local densities for all primes not dividing $-\Delta$ (i.e. $p\neq \Delta, 2$); additionally, we see $L(1,\chi)$ appear as part of the expression.
\begin{lemma}\label{lem: Peyre constant good primes}
    Let $\Delta>0$ be an odd prime and let $f(z)$ be an irreducible quartic that remains irreducible over $\Q(\sqrt{-\Delta})$. Then, 
    \begin{equation*}
        \Res_{s_1=1}\xi(s_1,1;F,\chi,1) = L(1,\chi)\prod_{p\mid -\Delta}\left(1-\frac{1}{p}\right) \prod_{p\nmid -\Delta}\omega_p\cdot \left(1-\frac{1}{p^2}\right)^{-1}. 
    \end{equation*}
\end{lemma}
\begin{proof}
    Let us first define the following Dirichlet series for $\Re(s)>1$: 
    \begin{equation*}
        L_{\not\square,K}(s) = \sum_{\gcd(\ft,\overline{\ft})=1}N(\ft)^{-s} = \prod_{p} \left(1+(\chi(p)+1)(p^{-s}+p^{-2s}+\hdots )\right). 
    \end{equation*}
    This series will have a simple pole at $s=1$; this can be seen by nice relation: $$L_{\not\square,K}(s) = \zeta_K(s)\zeta(2s)^{-1}.$$
    Let us compute the residue of $L_{\not\square,K}(s)$ at $s=1$: 
    \begin{align*}
        \Res_{s=1} L_{\not\square,K}(s) &= \lim_{s\rightarrow1^+} \zeta(s)^{-1} L_{\not\square,K}(s)\\
        &= \lim_{s\rightarrow 1^+} \prod_{p\mid -\Delta}\left(1-\frac{1}{p^{s}}\right) \prod_{p\nmid -\Delta} \left(1-\frac{1}{p^{s}}\right)\left(1+\frac{\chi(p)+1}{p^{s}(1-p^{-s})}\right)\\
        &= \prod_{p\mid -\Delta} \left(1-\frac{1}{p}\right)\lim_{s\rightarrow 1^+} \prod_{p\nmid -\Delta}\left(1+\frac{\chi(p)}{p^{s}}\right)  \\
        &= \prod_{p\mid -\Delta} \left(1-\frac{1}{p}\right)\prod_{p\nmid -\Delta}\left(1-\frac{1}{p^2}\right) L(1,\chi). 
    \end{align*}

    Now, we return to our original Dirichlet series $\xi(s_1,1;F,\chi,1)$, which also has a simple pole at $s=1$. We can see that 
    \begin{align*}
        \Res_{s=1} \xi(s,1;F,\chi,1) &= \Res_{s=1} L_{\not\square, K}(s) \cdot \lim_{s\rightarrow 1^+} L_{\not\square,K}(s)^{-1} \xi(s,1;F,\chi,1)\\
        &= \Res_{s=1} L_{\not\square,K}(s) \cdot \prod_{p\nmid -\Delta} \left(1+\frac{\chi(p)+1}{p-1}\right)^{-1} L_p(1,1;F,\chi,1),
    \end{align*}
    where $L_p(1,1;F,\chi,1)$ are the local factors given by: 
    \begin{multline*}
        L_p(1,1;F,\chi,1) = 1+ \left(1-\frac{1}{p^2}\right)^{-1}\left(\frac{\chi(p)\varrho_F^*(p)}{p^2} + \frac{\chi(p)^2 \varrho_F^*(p^2)}{p^4}+\dots \right)\\
        + \frac{\chi(p)+1}{p}\left(1+\left(1-\frac{1}{p^2}\right)^{-1}\left(\frac{\chi(p)\varrho_F^*(p)}{p^2} + \frac{\chi(p)^2 \varrho_F^*(p^2)}{p^4}+\dots -\frac{\varrho_F^*(p)}{p^2}-\frac{\chi(p)\varrho_F^*(p^2)}{p^4}-\dots\right)\right)\\
        + \frac{\chi(p)+1}{p^2} \left(1+\left(1-\frac{1}{p^2}\right)^{-1}\left(\frac{\chi(p)\varrho_F^*(p)}{p^2} + \frac{\chi(p)^2 \varrho_F^*(p^2)}{p^4}+\dots -\frac{\varrho_F^*(p)}{p^2}-\frac{\chi(p)\varrho_F^*(p^2)}{p^4}-\dots\right)\right)+\dots 
    \end{multline*}
    For simplicity, let us define: 
    \begin{equation*}
        \ell_p := \left(1-\frac{1}{p^2}\right)^{-1}\left(\frac{\chi(p)\varrho_F^*(p)}{p^2} + \frac{\chi(p)^2 \varrho_F^*(p^2)}{p^4}+\dots\right).
    \end{equation*}
    Then we can write these local factors as: 
    \begin{align*}
        L_p(1,1;F,\chi,1) &= 1+\ell_p + \frac{\chi(p)+1}{p}\left(1+\ell_p - \chi(p) \ell_p\right) + \frac{\chi(p)+1}{p^2} \left(1+\ell_p-\chi(p)\ell_p\right)+\dots \\
        &= 1+\ell_p + \frac{(\chi(p)+1)(1+(1-\chi(p))\ell_p)}{p-1}\\
        &= \begin{cases}
            1+\ell_p, & \chi(p)=-1, \\
            1+\ell_p + \frac{\chi(p)+1}{p-1}, & \chi(p)=1. 
        \end{cases}
    \end{align*}
    Thus, we can evaluate further our residue as: 
    \begin{equation*}
        \Res_{s=1}L_{\not\square,K}(s) \cdot \prod_{\chi(p)=-1} (1+\ell_p) \prod_{\chi(p)=-1} \left(1+\ell_p\cdot \left(1+\frac{\chi(p)+1}{p-1}\right)^{-1}\right).
    \end{equation*}
    It now behooves us to split into the cases when $\chi(p)=1$ and $\chi(p)=-1$.

    We start with the simpler case: fix a prime $p$ such that $\chi(p)=-1.$ Then our local factor is given by 
    \begin{equation*}
        1+\ell_p = \left(1-\frac{1}{p^2}\right)^{-1} \left(1-\frac{1}{p^2}+\frac{\chi(p)\varrho_F^*(p)}{p^2}+\frac{\chi(p)^2\varrho_F^*(p^2)}{p^4}+\dots \right).
    \end{equation*}
    On the other hand, we have that 
    \begin{equation*}
        \omega_p = \left(1-\frac{1}{p}\right) \lim_{v\rightarrow\infty} \frac{1}{p^{3v}} \sum_{\substack{x^2+\Delta y^2 = F(u,v) \\ p\nmid (u,v)}} 1.
    \end{equation*}
    We split our sum by the value $\gamma$ such that $v_p(F(u,v))=\gamma$, i.e. we count 
    \begin{equation*}
        \lim_{v\rightarrow\infty}\frac{1}{p^{3v}} \sum_{\gamma=0}^{v} \sum_{\substack{x^2+\Delta y^2 = F(u,v) \\ p\nmid (u,v) \\ p^\gamma \| F(u,v)}}1 = \sum_{\gamma=0}^\infty \frac{\#\{(x,y,u,v)\bmod p^{\gamma+1} : p\nmid (u,v), p^{\gamma} \| F(u,v) , x^2+\Delta y^2 = F(u,v)\}}{p^{3(\gamma+1)}}.
    \end{equation*}
    Since $\chi(p)=-1$, we note that if $\gamma$ is odd, then this count is zero, i.e. there will be no solutions to $x^2+\Delta y^2 = F(u,v)$ for $p^\gamma \| F(u,v)$. So, it suffices to consider $\gamma$ even. Let $\gamma>0$. In that case, we see that: 
    \begin{multline*}
        \#\{(x,y,u,v)\bmod p^{\gamma+1} : p\nmid (u,v), p^{\gamma} \| F(u,v) , x^2+\Delta y^2 = F(u,v)\} \\ 
        = \sum_{\gcd(a,p)=1} \#\{x^2+\Delta y^2 \equiv  p^\gamma a\bmod p^{\gamma+1}\}\cdot \#\{u,v: p\nmid (u,v), F(u,v)\equiv p^\gamma a\bmod p^{\gamma+1}\} \\
        = (p+1) \left(p^{\gamma/2}\right)^2\cdot  \left(\sum_{\substack{u,v\bmod p^{\gamma+1}\\ p\nmid (u,v)\\ p^{\gamma} \mid F(u,v)}}1-\sum_{\substack{u,v\bmod p^{\gamma+1}\\ p\nmid (u,v)\\ p^{\gamma+1} \mid F(u,v)}}1\right) = (p+1)p^\gamma \left(p^2\varrho_F^*(p^\gamma) - \varrho_F^*(p^{\gamma+1})\right).
    \end{multline*}
    Here we have used that there will always be $(p+1)$ solutions to $x^2+\Delta y^2\equiv a \bmod p$ for $\gcd(a,p)=1$. Note that if $\gamma=0$, then we have similarly that 
    \begin{equation*}
        \#\{(x,y,u,v)\bmod p: p\nmid (u,v), p^\gamma\|F(u,v), x^2+\Delta y^2 = F(u,v)\} = (p+1)(p^2-1-\varrho_F^*(p)).
    \end{equation*}
    Plugging the above expression in, we have that 
    \begin{equation*}
        \lim_{v\rightarrow\infty}\frac{1}{p^{3v}} \sum_{\gamma=0}^{v} \sum_{\substack{x^2+\Delta y^2 = F(u,v) \\ p\nmid (u,v) \\ p^\gamma \| F(u,v)}}1  = \frac{p+1}{p}\left( 1-\frac{1}{p^2} - \frac{\varrho_F^*(p)}{p^2}+\sum_{\gamma=1}^\infty \cdot \left(\frac{\varrho_F^*(p^{\gamma})}{p^{2\gamma}} - \frac{\varrho_F^*(p^{\gamma+1})}{p^{2(\gamma+1)}}\right)\right).
    \end{equation*}
    Thus, we can express the local density as: 
    \begin{equation*}
        \omega_p = \left(1-\frac{1}{p^2}\right) \left(1-\frac{1}{p^2}+ \frac{\chi(p)\varrho_F^*(p)}{p^2}+ \frac{\chi(p)^2\varrho_F^*(p^2)}{p^4}+\dots\right).
    \end{equation*}
    (Recall that $\chi(p)=-1$). Hence, we have that 
    \begin{equation}\label{eq: Peyre chi(p)=-1}
        1+\ell_p = \left(1-\frac{1}{p^2}\right)^{-2} \omega_p.
    \end{equation}

    Let us now fix a prime $p$ such that $\chi(p)=1$ and inspect the local density further. 
    \begin{equation*}
        \omega_p = \lim_{v\rightarrow\infty} \frac{1}{p^{4v}} \sum_{\substack{(u,v,x,y,t)\bmod p^{v}\\ x^2+\Delta y^2 = t^2F(u,v) \\ p\nmid (u,v) \\ p\nmid (x,y,t)}}1.
    \end{equation*}
    We split up our sum by the values of $v_p(t)$ and $v_p(F(u,v))$. \begin{equation*}
        \omega_p = \lim_{v\rightarrow\infty} \sum_{\gamma_1=0}^v\sum_{\gamma_2=0}^v \frac{\#\{(u,v,x,y,t)\bmod p^v: \begin{array}{c}
        x^2+\Delta y^2 = t^2 F(u,v), p\nmid (u,v), p\nmid (x,y,t) \\ 
        p^{\gamma_1}\| t, p^{\gamma_2}\|F(u,v)
        \end{array}\}}{p^{4v}}.
    \end{equation*}  
    If $2\gamma_1+\gamma_2< v$ and $\gamma_1>0$, then this count becomes: 
    \begin{multline*}
        \sum_{\gcd(a,p)=1} \#\{(x,y): x^2+\Delta y^2 = p^{2\gamma_1+\gamma_2}a \bmod p^v, \gcd(x,y,p)=1\} \\ \times \#\{(u,v,t): t^2F(u,v) \equiv p^{2\gamma_1+\gamma_2}a \bmod p^{v}, p^{\gamma_1}\| t, p^{\gamma_2}\|F(u,v),\gcd(u,v,p)=1\}. 
    \end{multline*}
    For any value of $a$ coprime to $p$, we have that 
    \begin{equation*}
        \#\{(x,y)\bmod p^{k+1}: x^2+\Delta y^2 = p^k a\bmod p^{k+1}, \gcd(x,y,p)=1\} = \begin{cases}
            p-1, & k=0,\\
            2(p-1)p^{k}, & k\geq 1.
        \end{cases}.
    \end{equation*}
    Notably, this expression is independent of $a$. Thus, we have that since $\gamma_1>0$, 
    \begin{multline*}
        2(p-1)p^{v-1} \cdot \#\{(u,v,t)\bmod p^v: p^{\gamma_1}\|t, p^{\gamma_2}\|F(u,v), \gcd(u,v,p)=1\}\\ 
        = 2(p-1)p^{v-1} \cdot (p-1)p^{v-\gamma_1-1}\cdot \left(p^{2(v-\gamma_2)} \varrho_F^*(p^{\gamma_2}) - p^{2(v-\gamma_2-1)} \varrho_F^*(p^{\gamma_2+1})\right).
    \end{multline*}

    On the other hand, if $2\gamma_1+\gamma_2\geq v$ and $\gamma_1>0$, then our count becomes 
    \begin{multline*}
        \#\{(x,y)\bmod p^v : x^2+\Delta y^2 \equiv 0 \bmod p^v, \gcd(x,y,p)=1\} \\ \times \#\{(u,v,t)\bmod p^v: p^{\gamma_1}\|t, p^{\gamma_2}\| F(u,v), \gcd(u,v,p)=1\}.
    \end{multline*}
    Now, we count for $v>0$: 
    \begin{equation*}
        \#\{(x,y)\bmod p^v: x^2+\Delta y^2 \equiv 0 \bmod p^v, \gcd(x,y,p)=1\} =2(p-1)p^{v-1}.
    \end{equation*}
    So, we get that for $\gamma_1>0$ and $2\gamma_1+\gamma_2\geq v$, the expression: 
    \begin{equation*}
        2(p-1)p^{v-1} \cdot (p-1)p^{v-\gamma_1-1}\cdot \left(p^{2(v-\gamma_2)} \varrho_F^*(p^{\gamma_2}) - p^{2(v-\gamma_2-1)} \varrho_F^*(p^{\gamma_2+1})\right).
    \end{equation*}
Finally, if $\gamma_1=0$, we evaluate that the count is equal to 
    \begin{multline*}
        \begin{cases}
            (p-1)^2 (p^2-1-\varrho_F^*(p)) p^{4(v-1)}, & \gamma_2=0, \\ 
            (\gamma_2+1)(p-1)^2p^{2(v-1)}(p^{2(v-\gamma_2}) \varrho_F^*(p^{\gamma_2}) - p^{2(v-\gamma_2-1)}\varrho_F^*(p^{\gamma_2+1})), & \gamma_2>0.
        \end{cases}
    \end{multline*}

    Thus, we can now evaluate $\omega_p$: 
    \begin{multline*}
        \omega_p = \lim_{v\rightarrow \infty} \left(1-\frac{1}{p}\right)^2 \left(1-\frac{1}{p^2}-\frac{\varrho_F^*(p)}{p^2}\right) + \left(1-\frac{1}{p}\right)^2 \sum_{\gamma_2=1}^{v-1} (\gamma_2+1)\left(\frac{\varrho_F^*(p^{\gamma_2})}{p^{2\gamma_2}} - \frac{\varrho_F^*(p^{\gamma_2+1})}{p^{2(\gamma_2+1)}}\right) \\
        + 2\left(1-\frac{1}{p}\right)^2 \sum_{\gamma_1=1}^{v/2} \frac{1}{p^{\gamma_1}}\left(1-\frac{1}{p^2}-\frac{\varrho_F^*(p)}{p^2}\right) + 2\left(1-\frac{1}{p}\right)^2 \sum_{\substack{2\gamma_1+\gamma_2<v \\ \gamma_1,\gamma_2>0}} \frac{1}{p^{\gamma_1}} \left(\frac{\varrho_F^*(p^{\gamma_2})}{p^{2\gamma_2}} - \frac{\varrho_F^*(p^{\gamma_2+1})}{p^{2(\gamma_2+1)}}\right)\\ + 2\left(1-\frac{1}{p}\right)^2 \sum_{\gamma_1\geq v/2}\frac{1}{p^{\gamma_1}}\left(1-\frac{1}{p^2}-\frac{\varrho_F^*(p)}{p^2}\right)
        + 2 \left(1-\frac{1}{p}\right)^2\sum_{\substack{2\gamma_1+\gamma_2\geq v\\ \gamma_1,\gamma_2>0}}\frac{1}{p^{\gamma_1}} \left(\frac{\varrho_F^*(p^{\gamma_2})}{p^{2\gamma_2}} - \frac{\varrho_F^*(p^{\gamma_2+1})}{p^{2(\gamma_2+1)}}\right). 
    \end{multline*}

    Taking the limit as $v\rightarrow\infty,$ we get: 
    \begin{multline*}
        \omega_p = \left(1-\frac{1}{p}\right)^2\left(1-\frac{1}{p^2}-\frac{\varrho_F^*(p)}{p^2}\right) + \left(1-\frac{1}{p}\right)^2\left(\frac{\varrho_F^*(p)}{p^2}+\sum_{\gamma_2=1}^\infty \frac{\varrho_F^*(p^{\gamma_2})}{p^{2\gamma_2}}\right)\\
        + \left(1-\frac{1}{p}\right)^2 \cdot \frac{2}{p-1} \cdot \left(1-\frac{1}{p^2}-\frac{\varrho_F^*(p)}{p^2}\right) + \left(1-\frac{1}{p^2}\right)\cdot \frac{2}{p-1} \cdot \frac{\varrho_F^*(p)}{p^2}\\
         = \left(1-\frac{1}{p}\right)^2 \left(1+\frac{2}{p-1}\right) \left(1-\frac{1}{p^2}\right) + \left(1-\frac{1}{p}\right)^2 \sum_{\gamma_2=1}^\infty \frac{\varrho_F^*(p^{\gamma_2})}{p^{2\gamma_2}}.
    \end{multline*}

    It is time to return to the expression that we get for our local factor at $p$ when $\chi(p)=1$; that is 
    $$1+\ell_p\cdot \left(1+\frac{2}{p-1}\right)^{-1}.$$
    Plugging in our definition of $\ell_p$ when $\chi(p)=1$, we have:
    \begin{multline*}
        \left(1-\frac{1}{p^2}\right)^{-1} \left(1-\frac{1}{p^2} + \left(1+\frac{2}{p-1}\right)^{-1}\sum_{\gamma=1}^\infty \frac{\varrho_F^*(p^{\gamma})}{p^{2\gamma}}\right) \\
        = \left(1-\frac{1}{p^2}\right)^{-1} \left(1+\frac{2}{p-1}\right)^{-1} \left(\left(1+\frac{2}{p-1}\right)\left(1-\frac{1}{p^2}\right)+\sum_{\gamma=1}^\infty \frac{\varrho_F^*(p^{\gamma})}{p^{2\gamma}}\right).
    \end{multline*}
    Thus, we have that 
    \begin{equation}\label{eq: Peyre chi(p)=1}
        1+\ell_p\cdot \left(1+\frac{2}{p-1}\right)^{-1} = \left(1-\frac{1}{p^2}\right)^{-1} \left(1+\frac{2}{p-1}\right)^{-1} \left(1-\frac{1}{p}\right)^{-2} \omega_p.
    \end{equation}

    Collecting our terms together (\eqref{eq: Peyre chi(p)=-1} and \eqref{eq: Peyre chi(p)=1}), we achieve that 
    \begin{multline*}
        \Res_{s=1}\xi(s_1,1;F,\chi,1) \\ = \Res_{s=1}L_{\not\square,K}(s) \cdot \prod_{p\nmid -\Delta}\left(1-\frac{1}{p^2}\right)^{-1} \omega_p \cdot \prod_{\chi(p)=-1} \left(1-\frac{1}{p^2}\right)^{-1} \prod_{\chi(p)=1} \left(1+\frac{2}{p-1}\right)^{-1} \left(1-\frac{1}{p^2}\right)^{-2}.
    \end{multline*}
    Noting that this expression for $\chi(p)=1$ is equivalent to $(1-p^{-2})^{-1}$, and plugging in the values of the residue at $s=1$ of $L_{\not\square,K}(s)$, we get: 
    \begin{equation*}
        \Res_{s=1}\xi(s_1,1;F,\chi,1) = \prod_{p\mid -\Delta}\left(1-\frac{1}{p}\right)\prod_{p\nmid -\Delta}\left(1-\frac{1}{p^2}\right)^{-1} L(1,\chi) \prod_{p\nmid -\Delta}\omega_p.
    \end{equation*}
    This completes the proof of the lemma.

\end{proof}

With Lemma \ref{lem: Peyre constant good primes} in hand, we see that 
\begin{equation*}
    c_{\Delta,f} = \frac{2|\calU|}{|\mathcal{C}|} L(1,\chi) \left(\sum_{q_1q_2=-\Delta}' \BM_{q_1,q_2}(1)\right) \prod_{p\mid -\Delta} \left(1-\frac{1}{p}\right) \prod_{p\nmid -\Delta}\omega_p.
\end{equation*}
By Dirichlet's class number formula, we know that since $\Delta>0$, 
\begin{equation*}
    L(1,\chi) = \frac{2\pi |\mathcal{C}|}{|\calU| \sqrt{\Delta}}.
\end{equation*}
So, this leaves us with the following expression for $c_{\Delta,f}$:
\begin{equation}\label{eq: Peyre constant with the good primes and L(1,chi)}
    c_{\Delta,f} = \frac{4\pi}{\sqrt{\Delta}}\prod_{p\nmid -\Delta} \omega_p \times \left(\sum_{q_1q_2=-\Delta}' \BM_{q_1,q_2}(1)\right) \prod_{p\mid -\Delta} \left(1-\frac{1}{p}\right). 
\end{equation}

It is now time to approach those primes dividing $-\Delta$ (i.e. $4\Delta$). Recall that we chose $\Delta$ to be itself an odd prime, so these are simply $2$ and $\Delta$. Inspecting the sum over genus characters further, we see that if $\Delta \equiv 3 \bmod 4$, then we only sum over $q_1=1$ and $q_2=4\Delta$. On the other hand, if $\Delta \equiv 1\bmod 4$, then we sum over two pairs of $(q_1,q_2)$: $(1,4\Delta)$ and $(4, \Delta)$. 

\subsection{$\Delta\equiv 3 \bmod 4$}\label{subsec: 3 mod 4}
We review our definition of $\BM_{q_1,q_2}(1)$ as in Definition \ref{def: BM}: 
\begin{equation*}
    \BM_{q_1,q_2}(1) = \sum_{p\mid n\implies p\mid 4\Delta} \frac{\chi_{q_1}(p_{q_2}(n)) \chi_{q_2}(p_{q_1}(n)) \mathbf{1}_{\exists \fa: N(\fa)=n}}{n^2} \sum_{\substack{a\bmod 4\Delta \\ \gcd(a,4\Delta)=1}} \frac{\chi_{q_1}(a)}{(4\Delta)^2}\times \prod_{i=1}^r (1+\chi(a_i)) \varrho_F^*(n,\ba), 
\end{equation*}
where the local count is given by:
\begin{equation*}
    \varrho_F^*(n,a) = \#\{\bx\bmod 4n\Delta: \gcd(x,y,4\Delta)=1, p_{4\Delta}(F(x,y))=n, p_{\neg 4\Delta} (F(x,y)) \equiv a_i\bmod 4\Delta\}.
\end{equation*}
For $(q_1,q_2)=(1,4\Delta)$, this reduces to 
\begin{equation*}
    \sum_{p\mid n\implies p\mid 4\Delta} \frac{\mathbf{1}_{\exists \fa: N(\fa)=n}}{n^2} \sum_{\substack{\ba \bmod 4\Delta \\ \gcd(a,4\Delta)=1}} \prod_{i=1}^r(1+\chi(a_i)) \frac{\varrho_F^*(n,\ba)}{(4\Delta)^2}.
\end{equation*}
Fix a value of $n$. It is a straightforward computation to see when $\Delta\equiv 3\bmod 4$, the following identity holds:
\begin{equation*}
    \#\{x^2+\Delta y^2 = na \bmod 4n\Delta\} = (1+\chi(a)) \cdot 4n\Delta.
\end{equation*}
So, we can rewrite the above expression as: 
\begin{equation*}
    \sum_{p\mid n\implies p\mid 4\Delta} \mathbf{1}_{\exists \fa: N(\fa)=n} \cdot \frac{\#\left\{(u,v,x,y) \bmod 4n\Delta: \begin{array}{c}
         x^2+\Delta y^2 = F(u,v) \bmod 4n\Delta\\ 
         \gcd(u,v,4\Delta)=1, p_{4\Delta}(F(u,v))=n
    \end{array}\right\}}{(4n\Delta)^3}.
\end{equation*}
From the Chinese remainder theorem, this can be rewritten as $w_2w_\Delta$ where 
\begin{equation*}
    w_2 = \sum_{k=0}^{\infty} \mathbf{1}_{\exists\fa: N(\fa) = 2^k} \frac{\#\left\{(u,v,x,y) \bmod 2^{k+2}: \begin{array}{c}
         x^2+\Delta y^2 = F(u,v) \bmod 2^{k+2}\\ 
         \gcd(u,v,2)=1, v_{2}(F(u,v))=k
    \end{array}\right\}}{2^{3(k+2)}}
\end{equation*}
and we have that 
\begin{equation*}
    w_\Delta = \sum_{k=0}^\infty \frac{\#\left\{(u,v,x,y) \bmod \Delta^{k+1}: \begin{array}{c}
         x^2+\Delta y^2 = F(u,v) \bmod \Delta^{k+1}\\ 
         \gcd(u,v,\Delta)=1, v_{\Delta}(F(u,v))=k
    \end{array}\right\}}{\Delta^{3(k+1)}}.
\end{equation*}

Observe that if $x^2+\Delta y^2 = t^2 F(u,v)$ and $\Delta \mid t$, then automatically we have that $\Delta \mid x,y$. So, we have that 
\begin{align*}
    w_\Delta &= \left(1-\frac{1}{\Delta}\right)^{-1}  \sum_{k=0}^\infty \frac{\#\left\{(u,v,x,y) \bmod \Delta^{k+1}: \begin{array}{c}
         x^2+\Delta y^2 = t^2F(u,v) \bmod \Delta^{k+1}\\ 
         \gcd(u,v,\Delta)=1, \gcd(x,y,t,\Delta)=1, \\ v_{\Delta}(F(u,v))=k
    \end{array}\right\}}{\Delta^{4(k+1)}} \\ 
    &= \left(1-\frac{1}{\Delta}\right)^{-1} \omega_\Delta. 
\end{align*}
So, it remains to relate $w_2$ to $\omega_2$, the 2-adic density. 

We now observe that if there does not exist an ideal $\fa$ such that $N(\fa)=2^k$, then there are no solutions to $x^2+\Delta y^2 = t^2 F(u,v)$ with $v_2(F(u,v))=k$ over $\Z_2$. So, we have that $w_2$ is the $2$-adic density of the variety given by $x^2+\Delta y^2 = F(u,v)$ over $\Z_2$. However, since $\Delta \equiv 3 \bmod 4$, there do exist solutions to $x^2+\Delta y^2 = t^2 F(u,v)$ with $2\mid t$ and $\gcd(x,y,t,2)=1$. Indeed, after splitting by $\gamma$ such that $2^\gamma \| t$, this gives us that 
\begin{equation*}
    \omega_2 = w_2 \cdot \left(\frac{1}{2}+\frac{1}{4}+\dots \right) = w_2. 
\end{equation*}

In summary, we have that for $\Delta \equiv 3 \bmod 4$, 
\begin{equation}\label{eq: Peyre BM for 3 mod 4}
    \prod_{p\mid 4\Delta}\left(1-\frac{1}{p}\right) \sum_{q_1q_2=4\Delta}' \BM_{q_1,q_2}(1) = \frac{1}{2}\omega_2 \omega_{\Delta}.
\end{equation}

\subsection{$\Delta\equiv 1 \bmod 4$}
We must now range over two genus characters: when $(q_1,q_2) = (1,4\Delta)$ and $(4,\Delta)$. By \eqref{eq: Peyre constant with the good primes and L(1,chi)}, our next step is to evaluate:
\begin{equation*}
    \prod_{p\mid 4\Delta}\left(1-\frac{1}{p}\right) \left(\BM_{1,4\Delta}(1)+\BM_{4,\Delta}(1)\right).
\end{equation*}
Expanding the definition of $\BM_{q_1,q_2}(1)$ from Definition \ref{def: BM}, we have that 
\begin{multline*}
    \BM_{1,4\Delta}(1)+\BM_{4,\Delta}(1) = \sum_{p\mid n\impliedby p\mid 4\Delta} \mathbf{1}_{\exists \fa:N(\fa)=n}\sum_{\gcd(a,4\Delta)=1} \frac{\varrho_F^*(n,a)}{(4n\Delta)^2} \\
    \times \left(1+\chi(a)+\chi_4(p_{\Delta}(n))\chi_{\Delta}(p_2(n))(\chi_{4}(a)+\chi_{\Delta}(a))\right). 
\end{multline*}
Observe that since $\Delta\equiv 1 \bmod 4$, we know that $\chi_4(p_{\Delta}(n))=1$. Additionally, since $\Delta \equiv 1 \bmod 4$, we have that there exists $\fa,\fa'$ such that $N(\fa)=\Delta$ and $N(\fa')=2$. So, we can also remove the indicator function. This leaves us with the summation: 
\begin{equation*}
    \BM_{1,4\Delta}(1)+\BM_{4,\Delta}(1) = \sum_{p\mid n\impliedby p\mid 4\Delta} \sum_{\gcd(a,4\Delta)=1} \frac{\varrho_F^*(n,a)}{(4n\Delta)^2} \\
    \times \left(1+\chi(a)+\chi_{\Delta}(p_2(n))(\chi_{4}(a)+\chi_{\Delta}(a))\right). 
\end{equation*}
Next, we want to relate this bottom expression to the number of solutions to $x^2+\Delta y^2 = na\bmod 4n\Delta$.

\begin{lemma}\label{lem: Peyre sum of genus characters}
    Assume that $p\mid n\implies p\mid 4\Delta$ and that $\gcd(a,4\Delta)=1$. Then 
    \begin{equation*}
        \#\{x^2+\Delta y^2 = na \bmod 4n\Delta\} = \left(1+\chi(a)+\chi_{\Delta}(p_2(n))(\chi_{4}(a)+\chi_{\Delta}(a))\right) \cdot 4n\Delta.
    \end{equation*}
\end{lemma}
\begin{proof}
    Assume that $n=2^b\Delta^c$. Then by the Chinese remainder theorem, we want to understand the following counts: 
    \begin{equation*}
        x^2+\Delta y^2 = 2^b \Delta^c a \bmod \Delta^{c+1}, \quad x^2+\Delta y^2 = 2^b \Delta^c a \bmod 2^{b+2}.
    \end{equation*}
    In the first scenario, if $c=0$, then there are exactly 
    $(1+\legendre{2^b a}{\Delta}) \cdot \Delta^{c+1} = (1+\chi_{\Delta}(p_2(n)) \chi_{\Delta}(a)) \cdot \Delta^{c+1}$
    such solutions to the equation. On the other hand, if $c>0$, then $\Delta \mid x$. Then, after a change of variables, we are left with the equation 
    \begin{equation*}
        \Delta x^2 + y^2 = 2^b \Delta^{c-1} a \bmod \Delta^c
    \end{equation*}
    and lifting such solutions to $\Z/\Delta^{c+1}\Z$. We repeat the procedure and deduce that again there are $(1+\chi_{\Delta}(p_2(n))\chi_\Delta(a))\cdot \Delta^{c+1}$ such solutions. 

    On the other hand, since $\Delta$ is odd and $\gcd(a,4\Delta)=1$, if $b=0$ then the number of solutions to 
    \begin{equation*}
        x^2+\Delta y^2 = x^2+y^2 = \Delta^c a \bmod 4
    \end{equation*}
    is given by $(1+\chi_{4}(\Delta^c a)) \cdot 4.$ Now assume that $b>0$. Observe that $1+\Delta \equiv 2 \bmod 4$, let us write $\delta$ such that $1+\Delta \equiv 2 \delta \bmod 8.$ In other words, we can take $\delta = \frac{\Delta-1}{2}$. Then if $x^2+\Delta y^2 \equiv 2^b \Delta^c a\bmod 2^{b+2}$, we can consider 
    \begin{equation*}
        x'+y'\sqrt{-\Delta} = \frac{x+y\sqrt{-\Delta}}{(1+\sqrt{-\Delta})^2}. 
    \end{equation*}
    Observe that $x'^2+\Delta y'^2 \equiv \delta^b \Delta^c a\bmod 2^{b+2}.$ Thus, we can count the number of solutions to $x^2+\Delta y^2 \equiv  2^b \Delta^c a \bmod 2^{b+2}$ as 
    \begin{equation*}
        (1+\chi_{4}( a) \chi_{4}(\delta)^b) \cdot 2^{b+2}.
    \end{equation*}
    Recall that $\Delta\equiv 1 \bmod 4$ and so $\chi_4(\Delta)=1.$

    We now claim that 
    \begin{equation*}
        (1+\chi(a)+\chi_{\Delta}(p_2(n))(\chi_4(a)+\chi_\Delta(a))) \cdot 4n\Delta = (1+\chi_{\Delta}(p_2(n))\chi_{\Delta}(a))(1+\chi_4(a)\chi_4(\delta)^b) \cdot 2^{b+2}\Delta^{c+1}.
    \end{equation*}
    To see this, we expand out the product on the right hand side: 
    \begin{equation*}
        1+\chi_{\Delta}(p_2(n))\chi_\Delta(a) + \chi_4(a)\chi_4(\delta)^b + \chi_\Delta(p_2(n)a)\chi_4(a)\chi_4(\delta)^b.
    \end{equation*}
    Comparing this with $1+\chi(a)+\chi_{\Delta}(p_2(n))(\chi_4(a)+\chi_\Delta(a))$, we can see that it suffices to show that 
    \begin{equation*}
        \chi(a) + \chi_{\Delta}(p_2(n)) \chi_4(a) = \chi_4(a)\chi_4(\delta)^b+ \chi_{\Delta}(p_2(n)a)\chi_4(a)\chi_4(\delta)^b. 
    \end{equation*}
    At this point, recall that $\Delta\equiv 1 \bmod 4$ and that 
    \begin{equation*}
        \chi_4(\delta) = \legendre{2}{\Delta} = \begin{cases}
            1, & \Delta\equiv 1 \bmod 8, \\
            -1, & \Delta \equiv 5 \bmod 8.
        \end{cases}
    \end{equation*}
    Thus, we have that 
    \begin{equation*}
        \chi_4(\delta^b \Delta^c a) = \legendre{2}{\Delta}^b \legendre{a}{4} = \chi_\Delta(p_2(n))\chi_4(a),
    \end{equation*}
    since by definition $\chi_\Delta(\cdot) = \legendre{\cdot}{\Delta}$ and $p_2(n)= 2^b$. On the other hand, 
    \begin{equation*}
        \chi_\Delta(p_2(n) a) \chi_4(a)\chi_4(\delta)^b  = \legendre{2}{\Delta}^b \chi_\Delta(a) \legendre{2}{\Delta}^b \chi_4(a) = \chi(a).
    \end{equation*}
    Here we have recalled that $\chi = \chi_{4\Delta}.$

\end{proof}

With Lemma \ref{lem: Peyre sum of genus characters} in hand, we can see that 
\begin{equation*}
    \BM_{1,4\Delta}(1) + \BM_{4,\Delta}(1) = \sum_{p\mid n\implies p\mid 4\Delta} \frac{\#\left\{(u,v,x,y) \bmod 4n\Delta: \begin{array}{c}
         x^2+\Delta y^2 = F(u,v) \bmod 4n\Delta\\ 
         \gcd(u,v,4\Delta)=1, p_{4\Delta}(F(u,v))=n
    \end{array}\right\}}{(4n\Delta)^3}.
\end{equation*}
Now the computation follows in the same way as when $\Delta\equiv 3 \bmod 4$ and we again achieve that when $\Delta\equiv 1 \bmod 4$
\begin{equation}\label{eq: Peyre BM for 1 mod 4}
    \prod_{p\mid 4\Delta} \left(1-\frac{1}{p}\right)\sum_{q_1q_2=4\Delta}' \BM_{q_1,q_2}(1) = \frac{1}{2}\omega_2\omega_\Delta.
\end{equation}

\subsection{Evaluation at the real place}
We now calculate $\omega_\infty$: we use the volume formula for an ellipse of the shape $\{x^2+\Delta y^2 \leq a\}$:
\begin{multline*}
    \omega_\infty = \lim_{\delta\rightarrow 0} \lim_{\epsilon\rightarrow 0} \log(1/\delta)^{-1}\epsilon^{-1} \int_{\delta<|t|\leq 1} \int_{|u|,|v|\leq |t|^{-1/2}} \vol(\{x^2+\Delta y^2 \in t^2F(u,v)+[-\epsilon/2, \epsilon/2]\}) dudvdt\\
    = \lim_{\delta\rightarrow 0}\lim_{\epsilon\rightarrow 0} \log(1/\delta)^{-1} \epsilon^{-1} \int_{|t|\leq 1} \int_{|u|,|v|\leq |t|^{-1/2}}  \frac{\pi \epsilon}{\sqrt{\Delta}}dudvdt\\
    = \frac{2\pi}{\sqrt{\Delta}} \lim_{\delta \rightarrow 0} \frac{1}{\log(1/\delta)} \int_{1/\delta < t\leq 1}\int_{|u|,|v|\leq t^{-1/2}} dudvdt = \frac{8\pi}{\sqrt{\Delta}}.
\end{multline*}

Putting together Lemma \ref{lem: Peyre constant good primes} and our calculations from \eqref{eq: Peyre BM for 3 mod 4} and \eqref{eq: Peyre BM for 1 mod 4}, we see that 
\begin{equation*}
    c_{\Delta,f} = \frac{4\pi}{\sqrt{\Delta}} \cdot \frac{1}{2}\omega_2\omega_{\Delta}\prod_{p\nmid -\Delta} \omega_p  = \frac{1}{4}\omega_\infty \prod_{p<\infty} \omega_p.
\end{equation*}
This completes the proof of Proposition \ref{prop: Peyre for Delta positive prime}. \qed

\bibliographystyle{abbrv}
\bibliography{biblio}

@article {Daniel,
    AUTHOR = {Daniel, Stephan},
     TITLE = {On the divisor-sum problem for binary forms},
   JOURNAL = {J. Reine Angew. Math.},
  FJOURNAL = {Journal f\"{u}r die Reine und Angewandte Mathematik. [Crelle's
              Journal]},
    VOLUME = {507},
      YEAR = {1999},
     PAGES = {107--129},
      ISSN = {0075-4102},
   MRCLASS = {11N37},
  MRNUMBER = {1670278},
MRREVIEWER = {G. Greaves},
       DOI = {10.1515/crll.1999.010},
       URL = {https://doi.org/10.1515/crll.1999.010},
}

@article {dlB-Browning-sieve,
    AUTHOR = {de la Bret\`eche, R. and Browning, T. D.},
     TITLE = {Sums of arithmetic functions over values of binary forms},
   JOURNAL = {Acta Arith.},
  FJOURNAL = {Acta Arithmetica},
    VOLUME = {125},
      YEAR = {2006},
    NUMBER = {3},
     PAGES = {291--304},
      ISSN = {0065-1036},
   MRCLASS = {11N37 (11N32)},
  MRNUMBER = {2276196},
MRREVIEWER = {Ognian Trifonov},
       DOI = {10.4064/aa125-3-6},
       URL = {https://doi.org/10.4064/aa125-3-6},
}

@book {Browning-book,
    AUTHOR = {Browning, Timothy D.},
     TITLE = {Quantitative arithmetic of projective varieties},
    SERIES = {Progress in Mathematics},
    VOLUME = {277},
 PUBLISHER = {Birkh\"{a}user Verlag, Basel},
      YEAR = {2009},
     PAGES = {xiv+160},
      ISBN = {978-3-0346-0128-3},
   MRCLASS = {11-02 (11D45 11G35 14G05 14G25)},
  MRNUMBER = {2559866},
MRREVIEWER = {Robert Juricevic},
       DOI = {10.1007/978-3-0346-0129-0},
       URL = {https://doi.org/10.1007/978-3-0346-0129-0},
}

@article {delaTenenbaum-Manin,
    AUTHOR = {de la Bret\`eche, R\'{e}gis and Tenenbaum, G\'{e}rald},
     TITLE = {Sur la conjecture de {M}anin pour certaines surfaces de
              {C}h\^{a}telet},
   JOURNAL = {J. Inst. Math. Jussieu},
  FJOURNAL = {Journal of the Institute of Mathematics of Jussieu. JIMJ.
              Journal de l'Institut de Math\'{e}matiques de Jussieu},
    VOLUME = {12},
      YEAR = {2013},
    NUMBER = {4},
     PAGES = {759--819},
      ISSN = {1474-7480},
   MRCLASS = {11D72 (11D09 11D45 11E12 11G35 11N37 11P55 14G05 14G25)},
  MRNUMBER = {3103132},
MRREVIEWER = {D. R. Heath-Brown},
       DOI = {10.1017/S1474748012000886},
       URL = {https://doi.org/10.1017/S1474748012000886},
}

@article {delaTenenbaum-delta,
    AUTHOR = {de la Bret\`eche, R. and Tenenbaum, G.},
     TITLE = {Oscillations localis\'{e}es sur les diviseurs},
   JOURNAL = {J. Lond. Math. Soc. (2)},
  FJOURNAL = {Journal of the London Mathematical Society. Second Series},
    VOLUME = {85},
      YEAR = {2012},
    NUMBER = {3},
     PAGES = {669--693},
      ISSN = {0024-6107},
   MRCLASS = {11N56 (11D45 11L40 11N37)},
  MRNUMBER = {2927803},
MRREVIEWER = {D. R. Heath-Brown},
       DOI = {10.1112/jlms/jdr058},
       URL = {https://doi.org/10.1112/jlms/jdr058},
}

@article {delaTenenbaum-sieve,
    AUTHOR = {de la Bret\`eche, R\'{e}gis and Tenenbaum, G\'{e}rald},
     TITLE = {Moyennes de fonctions arithm\'{e}tiques de formes binaires},
   JOURNAL = {Mathematika},
  FJOURNAL = {Mathematika. A Journal of Pure and Applied Mathematics},
    VOLUME = {58},
      YEAR = {2012},
    NUMBER = {2},
     PAGES = {290--304},
      ISSN = {0025-5793},
   MRCLASS = {11A25 (11D45 11N37 11N56)},
  MRNUMBER = {2965973},
MRREVIEWER = {Olivier Bordell\`es},
       DOI = {10.1112/S0025579311002154},
       URL = {https://doi.org/10.1112/S0025579311002154},
}

@article {Marasingha-almostprimes,
    AUTHOR = {Marasingha, Gihan},
     TITLE = {Almost primes represented by binary forms},
   JOURNAL = {J. Lond. Math. Soc. (2)},
  FJOURNAL = {Journal of the London Mathematical Society. Second Series},
    VOLUME = {82},
      YEAR = {2010},
    NUMBER = {2},
     PAGES = {295--316},
      ISSN = {0024-6107},
   MRCLASS = {11N36 (11E76 11N05 11N32 11R42)},
  MRNUMBER = {2725041},
MRREVIEWER = {S. W. Graham},
       DOI = {10.1112/jlms/jdq026},
       URL = {https://doi.org/10.1112/jlms/jdq026},
}

@incollection {HB-linear,
    AUTHOR = {Heath-Brown, D. R.},
     TITLE = {Linear relations amongst sums of two squares},
 BOOKTITLE = {Number theory and algebraic geometry},
    SERIES = {London Math. Soc. Lecture Note Ser.},
    VOLUME = {303},
     PAGES = {133--176},
 PUBLISHER = {Cambridge Univ. Press, Cambridge},
      YEAR = {2003},
   MRCLASS = {11N25},
  MRNUMBER = {2053459},
MRREVIEWER = {G. Greaves},
}

@article {HallTenenbaum-DeltaR,
    AUTHOR = {Hall, R. R. and Tenenbaum, G.},
     TITLE = {The average orders of {H}ooley's {$\Delta_r$}-functions. {II}},
   JOURNAL = {Compositio Math.},
  FJOURNAL = {Compositio Mathematica},
    VOLUME = {60},
      YEAR = {1986},
    NUMBER = {2},
     PAGES = {163--186},
      ISSN = {0010-437X},
   MRCLASS = {11N37 (11N56 11P05)},
  MRNUMBER = {868136},
MRREVIEWER = {R. W. K. Odoni},
       URL = {http://www.numdam.org/item?id=CM_1986__60_2_163_0},
}

@article {Nair,
    AUTHOR = {Nair, Mohan},
     TITLE = {Multiplicative functions of polynomial values in short
              intervals},
   JOURNAL = {Acta Arith.},
  FJOURNAL = {Acta Arithmetica},
    VOLUME = {62},
      YEAR = {1992},
    NUMBER = {3},
     PAGES = {257--269},
      ISSN = {0065-1036},
   MRCLASS = {11N32 (11N64)},
  MRNUMBER = {1197420},
MRREVIEWER = {Michael Filaseta},
       DOI = {10.4064/aa-62-3-257-269},
       URL = {https://doi.org/10.4064/aa-62-3-257-269},
}

@article {delaBrownPeyre-maninChatelet,
    AUTHOR = {de la Bret\`eche, R\'{e}gis and Browning, Tim and Peyre, Emmanuel},
     TITLE = {On {M}anin's conjecture for a family of {C}h\^{a}telet surfaces},
   JOURNAL = {Ann. of Math. (2)},
  FJOURNAL = {Annals of Mathematics. Second Series},
    VOLUME = {175},
      YEAR = {2012},
    NUMBER = {1},
     PAGES = {297--343},
      ISSN = {0003-486X},
   MRCLASS = {11G35 (11G50 14G25)},
  MRNUMBER = {2874644},
MRREVIEWER = {Michel Gros},
       DOI = {10.4007/annals.2012.175.1.8},
       URL = {https://doi.org/10.4007/annals.2012.175.1.8},
}

@article{HBclassnum,
  doi = {10.48550/ARXIV.2212.12587},
  
  url = {https://arxiv.org/abs/2212.12587},
  
  author = {Heath-Brown, D. R.},

journal = {https://arxiv.org/abs/2212.12587},
  
  title = {Equidistribution for {S}olutions of {$p+m^2+n^2=N$}, and for {C}hâtelet {S}urfaces},
  
  publisher = {arXiv},
  
  year = {2022},
  
  copyright = {Creative Commons Attribution 4.0 International}
}

@article {Blomer,
    AUTHOR = {Blomer, Valentin},
     TITLE = {On cusp forms associated with binary theta series},
   JOURNAL = {Arch. Math. (Basel)},
  FJOURNAL = {Archiv der Mathematik},
    VOLUME = {82},
      YEAR = {2004},
    NUMBER = {2},
     PAGES = {140--146},
      ISSN = {0003-889X},
   MRCLASS = {11F11 (11E45 11N37)},
  MRNUMBER = {2047667},
MRREVIEWER = {Gergely Harcos},
       DOI = {10.1007/s00013-003-4806-x},
       URL = {https://doi.org/10.1007/s00013-003-4806-x},
}

@book {IwaniecKowalski,
    AUTHOR = {Iwaniec, Henryk and Kowalski, Emmanuel},
     TITLE = {Analytic number theory},
    SERIES = {American Mathematical Society Colloquium Publications},
    VOLUME = {53},
 PUBLISHER = {American Mathematical Society, Providence, RI},
      YEAR = {2004},
     PAGES = {xii+615},
      ISBN = {0-8218-3633-1},
   MRCLASS = {11-02 (11Fxx 11Lxx 11Mxx 11Nxx)},
  MRNUMBER = {2061214},
MRREVIEWER = {K. Soundararajan},
       DOI = {10.1090/coll/053},
       URL = {https://doi.org/10.1090/coll/053},
}

@article {PeyreConstant,
    AUTHOR = {Peyre, Emmanuel},
     TITLE = {Hauteurs et mesures de {T}amagawa sur les vari\'{e}t\'{e}s de {F}ano},
   JOURNAL = {Duke Math. J.},
  FJOURNAL = {Duke Mathematical Journal},
    VOLUME = {79},
      YEAR = {1995},
    NUMBER = {1},
     PAGES = {101--218},
      ISSN = {0012-7094},
   MRCLASS = {11G35 (14G05 14J45)},
  MRNUMBER = {1340296},
MRREVIEWER = {Shouwu Zhang},
       DOI = {10.1215/S0012-7094-95-07904-6},
       URL = {https://doi.org/10.1215/S0012-7094-95-07904-6},
}

@article {NairTenenbaum,
    AUTHOR = {Nair, Mohan and Tenenbaum, G\'{e}rald},
     TITLE = {Short sums of certain arithmetic functions},
   JOURNAL = {Acta Math.},
  FJOURNAL = {Acta Mathematica},
    VOLUME = {180},
      YEAR = {1998},
    NUMBER = {1},
     PAGES = {119--144},
      ISSN = {0001-5962},
   MRCLASS = {11N37 (11N25 11N32)},
  MRNUMBER = {1618321},
MRREVIEWER = {A. J. Hildebrand},
       DOI = {10.1007/BF02392880},
       URL = {https://doi.org/10.1007/BF02392880},
}

@inproceedings{vanderWaall,
  title={On a conjecture of {D}edekind on zeta-functions},
  author={van der Waall, Robert W},
  booktitle={Indagationes Mathematicae (Proceedings)},
  volume={78},
  number={1},
  pages={83--86},
  year={1975},
  organization={North-Holland}
}

@article {Uchida,
    AUTHOR = {Uchida, K\^{o}ji},
     TITLE = {On {A}rtin {$L$}-functions},
   JOURNAL = {Tohoku Math. J. (2)},
  FJOURNAL = {The Tohoku Mathematical Journal. Second Series},
    VOLUME = {27},
      YEAR = {1975},
     PAGES = {75--81},
      ISSN = {0040-8735},
   MRCLASS = {12A70},
  MRNUMBER = {369323},
MRREVIEWER = {W. Narkiewicz},
       DOI = {10.2748/tmj/1178241036},
       URL = {https://doi.org/10.2748/tmj/1178241036},
}

@article {BatyrevManin,
    AUTHOR = {Batyrev, V. V. and Manin, Yu. I.},
     TITLE = {Sur le nombre des points rationnels de hauteur born\'{e} des
              vari\'{e}t\'{e}s alg\'{e}briques},
   JOURNAL = {Math. Ann.},
  FJOURNAL = {Mathematische Annalen},
    VOLUME = {286},
      YEAR = {1990},
    NUMBER = {1-3},
     PAGES = {27--43},
      ISSN = {0025-5831},
   MRCLASS = {11G35 (14G10 14G40)},
  MRNUMBER = {1032922},
MRREVIEWER = {Marc Hindry},
       DOI = {10.1007/BF01453564},
       URL = {https://doi.org/10.1007/BF01453564},
}

@article {delaTenenbaumHooley,
    AUTHOR = {de la Bret\`eche, R\'{e}gis and Tenenbaum, G\'{e}rald},
     TITLE = {Two upper bounds for the {E}rd{\H{o}}s-{H}ooley delta-function},
   JOURNAL = {Sci. China Math.},
  FJOURNAL = {Science China. Mathematics},
    VOLUME = {66},
      YEAR = {2023},
    NUMBER = {12},
     PAGES = {2683--2692},
      ISSN = {1674-7283},
   MRCLASS = {11N37 (11K65)},
  MRNUMBER = {4670143},
MRREVIEWER = {Peter Shiu},
       DOI = {10.1007/s11425-022-2193-8},
       URL = {https://doi.org/10.1007/s11425-022-2193-8},
}

@article {Destagnol,
    AUTHOR = {Destagnol, Kevin},
     TITLE = {La conjecture de {M}anin pour certaines surfaces de
              {C}h\^{a}telet},
   JOURNAL = {Acta Arith.},
  FJOURNAL = {Acta Arithmetica},
    VOLUME = {174},
      YEAR = {2016},
    NUMBER = {1},
     PAGES = {31--97},
      ISSN = {0065-1036},
   MRCLASS = {11D45 (11D57 11N37)},
  MRNUMBER = {3517531},
MRREVIEWER = {Abderrahmane Nitaj},
       DOI = {10.4064/aa8312-2-2016},
       URL = {https://doi.org/10.4064/aa8312-2-2016},
}

@article {BrowningLinear,
    AUTHOR = {Browning, T. D.},
     TITLE = {Linear growth for {C}h\^{a}telet surfaces},
   JOURNAL = {Math. Ann.},
  FJOURNAL = {Mathematische Annalen},
    VOLUME = {346},
      YEAR = {2010},
    NUMBER = {1},
     PAGES = {41--50},
      ISSN = {0025-5831},
   MRCLASS = {11G50 (11G35 14G05)},
  MRNUMBER = {2558885},
MRREVIEWER = {Hizuru Yamagishi},
       DOI = {10.1007/s00208-009-0383-z},
       URL = {https://doi-org.ezproxy.princeton.edu/10.1007/s00208-009-0383-z},
}

@article {ColliotSSDII,
    AUTHOR = {Colliot-Th\'{e}l\`ene, Jean-Louis and Sansuc, Jean-Jacques and
              Swinnerton-Dyer, Peter},
     TITLE = {Intersections of two quadrics and {C}h\^{a}telet surfaces. {II}},
   JOURNAL = {J. Reine Angew. Math.},
  FJOURNAL = {Journal f\"{u}r die Reine und Angewandte Mathematik. [Crelle's
              Journal]},
    VOLUME = {374},
      YEAR = {1987},
     PAGES = {72--168},
      ISSN = {0075-4102},
   MRCLASS = {11G35 (11E81 14J20 14J26)},
  MRNUMBER = {876222},
MRREVIEWER = {Noriko Yui},
}

@article {ColliotSSDI,
    AUTHOR = {Colliot-Th\'{e}l\`ene, Jean-Louis and Sansuc, Jean-Jacques and
              Swinnerton-Dyer, Peter},
     TITLE = {Intersections of two quadrics and {C}h\^{a}telet surfaces. {I}},
   JOURNAL = {J. Reine Angew. Math.},
  FJOURNAL = {Journal f\"{u}r die Reine und Angewandte Mathematik. [Crelle's
              Journal]},
    VOLUME = {373},
      YEAR = {1987},
     PAGES = {37--107},
      ISSN = {0075-4102},
   MRCLASS = {11G35 (11E81 14J20 14J26)},
  MRNUMBER = {870307},
MRREVIEWER = {Noriko Yui},
}

@article {ErdosTuran,
    AUTHOR = {Erd\"{o}s, P. and Tur\'{a}n, P.},
     TITLE = {On a problem in the theory of uniform distribution. {I}},
   JOURNAL = {Nederl. Akad. Wetensch., Proc.},
  FJOURNAL = {Proceedings. Akadamie van Wetenschappen Amsterdam.
              North-Holland, Amsterdam},
    VOLUME = {51},
      YEAR = {1948},
     PAGES = {1146--1154 = Indagationes Math. {\bf 10, 370--378}},
      ISSN = {0370-0348},
   MRCLASS = {41.1X},
  MRNUMBER = {27895},
}

@book {MontgomeryTenLectures,
    AUTHOR = {Montgomery, Hugh L.},
     TITLE = {Ten lectures on the interface between analytic number theory
              and harmonic analysis},
    SERIES = {CBMS Regional Conference Series in Mathematics},
    VOLUME = {84},
 PUBLISHER = {Published for the Conference Board of the Mathematical
              Sciences, Washington, DC; by the American Mathematical
              Society, Providence, RI},
      YEAR = {1994},
     PAGES = {xiv+220},
      ISBN = {0-8218-0737-4},
   MRCLASS = {11-02 (11Kxx 11L07 11Mxx 11Nxx)},
  MRNUMBER = {1297543},
MRREVIEWER = {John B. Friedlander},
       DOI = {10.1090/cbms/084},
       URL = {https://doi.org/10.1090/cbms/084},
}

@incollection {Chatelet1,
    AUTHOR = {Ch\^{a}telet, Fran\c{c}ois},
     TITLE = {Points rationnels sur certaines surfaces cubiques},
 BOOKTITLE = {Les {T}endances {G}\'{e}om. en {A}lg\`ebre et {T}h\'{e}orie des
              {N}ombres},
    SERIES = {Colloq. Internat. CNRS, No. 143},
     PAGES = {67--75},
 PUBLISHER = {CNRS, Paris},
      YEAR = {1966},
   MRCLASS = {14.40},
  MRNUMBER = {213353},
MRREVIEWER = {T. Ono},
}

@article {Chatelet2,
    AUTHOR = {Ch\^{a}telet, F.},
     TITLE = {Points rationnels sur certaines courbes et surfaces cubiques},
   JOURNAL = {Enseign. Math. (2)},
  FJOURNAL = {L'Enseignement Math\'{e}matique. Revue Internationale. 2e S\'{e}rie},
    VOLUME = {5},
      YEAR = {1959},
     PAGES = {153--170 (1960)},
      ISSN = {0013-8584},
   MRCLASS = {14.40 (14.49)},
  MRNUMBER = {130218},
MRREVIEWER = {J. W. S. Cassels},
}

@article {delaBrowningChatelet,
    AUTHOR = {de la Bret\`eche, R. and Browning, T. D.},
     TITLE = {Binary forms as sums of two squares and {C}h\^{a}telet surfaces},
   JOURNAL = {Israel J. Math.},
  FJOURNAL = {Israel Journal of Mathematics},
    VOLUME = {191},
      YEAR = {2012},
    NUMBER = {2},
     PAGES = {973--1012},
      ISSN = {0021-2172},
   MRCLASS = {11D25 (11D45 14G05 14G25)},
  MRNUMBER = {3011504},
MRREVIEWER = {\c{S}. A. Basarab},
       DOI = {10.1007/s11856-012-0019-y},
       URL = {https://doi.org/10.1007/s11856-012-0019-y},
}

@article {delaBrowningHasse,
    AUTHOR = {de la Bret\`eche, R. and Browning, T. D.},
     TITLE = {Density of {C}h\^{a}telet surfaces failing the {H}asse principle},
   JOURNAL = {Proc. Lond. Math. Soc. (3)},
  FJOURNAL = {Proceedings of the London Mathematical Society. Third Series},
    VOLUME = {108},
      YEAR = {2014},
    NUMBER = {4},
     PAGES = {1030--1078},
      ISSN = {0024-6115},
   MRCLASS = {11G25 (11G50 14G17)},
  MRNUMBER = {3198755},
MRREVIEWER = {Tony Shaska},
       DOI = {10.1112/plms/pdt060},
       URL = {https://doi.org/10.1112/plms/pdt060},
}

@article {ColliotTheleneCoraySansuc,
    AUTHOR = {Colliot-Th\'{e}l\`ene, Jean-Louis and Coray, Daniel and Sansuc,
              Jean-Jacques},
     TITLE = {Descente et principe de {H}asse pour certaines vari\'{e}t\'{e}s
              rationnelles},
   JOURNAL = {J. Reine Angew. Math.},
  FJOURNAL = {Journal f\"{u}r die Reine und Angewandte Mathematik. [Crelle's
              Journal]},
    VOLUME = {320},
      YEAR = {1980},
     PAGES = {150--191},
      ISSN = {0075-4102},
   MRCLASS = {14G05 (10C02)},
  MRNUMBER = {592151},
MRREVIEWER = {A. Pfister},
}

@book {IwaniecTopics,
    AUTHOR = {Iwaniec, Henryk},
     TITLE = {Topics in classical automorphic forms},
    SERIES = {Graduate Studies in Mathematics},
    VOLUME = {17},
 PUBLISHER = {American Mathematical Society, Providence, RI},
      YEAR = {1997},
     PAGES = {xii+259},
      ISBN = {0-8218-0777-3},
   MRCLASS = {11Fxx (11-02)},
  MRNUMBER = {1474964},
MRREVIEWER = {B. Ramakrishnan},
       DOI = {10.1090/gsm/017},
       URL = {https://doi.org/10.1090/gsm/017},
}

@article {HolowinskyShiftedConvolution,
    AUTHOR = {Holowinsky, Roman},
     TITLE = {A sieve method for shifted convolution sums},
   JOURNAL = {Duke Math. J.},
  FJOURNAL = {Duke Mathematical Journal},
    VOLUME = {146},
      YEAR = {2009},
    NUMBER = {3},
     PAGES = {401--448},
      ISSN = {0012-7094},
   MRCLASS = {11N36 (11F30 11M36)},
  MRNUMBER = {2484279},
MRREVIEWER = {Don Redmond},
       DOI = {10.1215/00127094-2009-002},
       URL = {https://doi.org/10.1215/00127094-2009-002},
}

@article {Rankin,
    AUTHOR = {Rankin, R. A.},
     TITLE = {Contributions to the theory of {R}amanujan's function
              {$\tau(n)$} and similar arithmetical functions. {I}. {T}he
              zeros of the function {$\sum^\infty_{n=1}\tau(n)/n^s$} on
              the line {${\Re(s)}=13/2$}. {II}. {T}he order of the
              {F}ourier coefficients of integral modular forms},
   JOURNAL = {Proc. Cambridge Philos. Soc.},
  FJOURNAL = {Proceedings of the Cambridge Philosophical Society},
    VOLUME = {35},
      YEAR = {1939},
     PAGES = {351--372},
      ISSN = {0008-1981},
   MRCLASS = {10.0X},
  MRNUMBER = {411},
MRREVIEWER = {A. E. Ingham},
}

@article {FomenkoQuad,
    AUTHOR = {Fomenko, O. M.},
     TITLE = {Distribution of values of {F}ourier coefficients of modular
              forms of weight {$1$}},
   JOURNAL = {Zap. Nauchn. Sem. S.-Peterburg. Otdel. Mat. Inst. Steklov.
              (POMI)},
  FJOURNAL = {Rossi\u{\i}skaya Akademiya Nauk. Sankt-Peterburgskoe Otdelenie.
              Matematicheski\u{\i} Institut im. V. A. Steklova. Zapiski Nauchnykh
              Seminarov (POMI)},
    VOLUME = {226},
      YEAR = {1996},
    NUMBER = {Anal. Teor. Chisel i Teor. Funktsi\u{\i}. 13},
     PAGES = {196--227, 240},
      ISSN = {0373-2703},
   MRCLASS = {11F30 (11E45)},
  MRNUMBER = {1433357},
MRREVIEWER = {Pavel Guerzhoy},
       DOI = {10.1007/BF02358541},
       URL = {https://doi.org/10.1007/BF02358541},
}

@article {BatyrevTschinkel,
    AUTHOR = {Batyrev, Victor V. and Tschinkel, Yuri},
     TITLE = {Manin's conjecture for toric varieties},
   JOURNAL = {J. Algebraic Geom.},
  FJOURNAL = {Journal of Algebraic Geometry},
    VOLUME = {7},
      YEAR = {1998},
    NUMBER = {1},
     PAGES = {15--53},
      ISSN = {1056-3911},
   MRCLASS = {11G50 (11G35 14G05 14G40 14M25)},
  MRNUMBER = {1620682},
}

@article {CTSSchinzel,
    AUTHOR = {Colliot-Th\'{e}l\`ene, Jean-Louis and Sansuc, Jean-Jacques},
     TITLE = {Sur le principe de {H}asse et l'approximation faible, et sur
              une hypoth\`ese de {S}chinzel},
   JOURNAL = {Acta Arith.},
  FJOURNAL = {Polska Akademia Nauk. Instytut Matematyczny. Acta Arithmetica},
    VOLUME = {41},
      YEAR = {1982},
    NUMBER = {1},
     PAGES = {33--53},
      ISSN = {0065-1036},
   MRCLASS = {10C02 (10B40 12B20 14G25)},
  MRNUMBER = {667708},
MRREVIEWER = {D. J. Lewis},
       DOI = {10.4064/aa-41-1-33-53},
       URL = {https://doi.org/10.4064/aa-41-1-33-53},
}

@misc{delaTenenbaumHooley24,
      title={Note on the mean value of the {E}rd{\H{o}}s--Hooley Delta-function}, 
      author={Régis de la Bretèche and Gérald Tenenbaum},
      year={2024},
      eprint={2309.03958},
      archivePrefix={arXiv},
      primaryClass={math.NT},
      url={https://arxiv.org/abs/2309.03958}, 
}

@article {Shimura,
    AUTHOR = {Shimura, Goro},
     TITLE = {On the holomorphy of certain {D}irichlet series},
   JOURNAL = {Proc. London Math. Soc. (3)},
  FJOURNAL = {Proceedings of the London Mathematical Society. Third Series},
    VOLUME = {31},
      YEAR = {1975},
    NUMBER = {1},
     PAGES = {79--98},
      ISSN = {0024-6115},
   MRCLASS = {10D15},
  MRNUMBER = {382176},
MRREVIEWER = {I. Piatetski-Shapiro},
       DOI = {10.1112/plms/s3-31.1.79},
       URL = {https://doi.org/10.1112/plms/s3-31.1.79},
}

@book {Titchmarsh,
    AUTHOR = {Titchmarsh, E. C.},
     TITLE = {The theory of the {R}iemann zeta-function},
   EDITION = {Second},
      NOTE = {Edited and with a preface by D. R. Heath-Brown},
 PUBLISHER = {The Clarendon Press, Oxford University Press, New York},
      YEAR = {1986},
     PAGES = {x+412},
      ISBN = {0-19-853369-1},
   MRCLASS = {11M06},
  MRNUMBER = {882550},
MRREVIEWER = {Matti Jutila},
}

@article {CTSansucDescentII,
    AUTHOR = {Colliot-Th\'{e}l\`ene, Jean-Louis and Sansuc, Jean-Jacques},
     TITLE = {La descente sur les vari\'{e}t\'{e}s rationnelles. {II}},
   JOURNAL = {Duke Math. J.},
  FJOURNAL = {Duke Mathematical Journal},
    VOLUME = {54},
      YEAR = {1987},
    NUMBER = {2},
     PAGES = {375--492},
      ISSN = {0012-7094},
   MRCLASS = {11G35 (14G25)},
  MRNUMBER = {899402},
MRREVIEWER = {Daniel Coray},
       DOI = {10.1215/S0012-7094-87-05420-2},
       URL = {https://doi.org/10.1215/S0012-7094-87-05420-2},
}

@incollection {CTSansucDescentI,
    AUTHOR = {Colliot-Th\'{e}l\`ene, J.-L. and Sansuc, J.-J.},
     TITLE = {La descente sur les vari\'{e}t\'{e}s rationnelles},
 BOOKTITLE = {Journ\'{e}es de {G}\'{e}ometrie {A}lg\'{e}brique d'{A}ngers, {J}uillet
              1979/{A}lgebraic {G}eometry, {A}ngers, 1979},
     PAGES = {223--237},
 PUBLISHER = {Sijthoff \& Noordhoff, Alphen aan den Rijn---Germantown, Md.},
      YEAR = {1980},
   MRCLASS = {14G05 (14C15)},
  MRNUMBER = {605344},
MRREVIEWER = {Daniel Coray},
}

@book {CTSkoroBGBook,
    AUTHOR = {Colliot-Th\'{e}l\`ene, Jean-Louis and Skorobogatov, Alexei N.},
     TITLE = {The {B}rauer-{G}rothendieck group},
    SERIES = {Ergebnisse der Mathematik und ihrer Grenzgebiete. 3. Folge. A
              Series of Modern Surveys in Mathematics [Results in
              Mathematics and Related Areas. 3rd Series. A Series of Modern
              Surveys in Mathematics]},
    VOLUME = {71},
 PUBLISHER = {Springer, Cham},
      YEAR = {[2021] \copyright 2021},
     PAGES = {xv+453},
      ISBN = {978-3-030-74247-8; 978-3-030-74248-5},
   MRCLASS = {14F22 (14E08 14G05 14G12 14K05)},
  MRNUMBER = {4304038},
MRREVIEWER = {Thomas Benedict Williams},
       DOI = {10.1007/978-3-030-74248-5},
       URL = {https://doi.org/10.1007/978-3-030-74248-5},
}

@article {Harari,
    AUTHOR = {Harari, David},
     TITLE = {M\'{e}thode des fibrations et obstruction de {M}anin},
   JOURNAL = {Duke Math. J.},
  FJOURNAL = {Duke Mathematical Journal},
    VOLUME = {75},
      YEAR = {1994},
    NUMBER = {1},
     PAGES = {221--260},
      ISSN = {0012-7094},
   MRCLASS = {11G35 (14G25 14M10)},
  MRNUMBER = {1284820},
MRREVIEWER = {Takeshi Ooe},
       DOI = {10.1215/S0012-7094-94-07507-8},
       URL = {https://doi.org/10.1215/S0012-7094-94-07507-8},
}

@article {CorayTsfasman,
    AUTHOR = {Coray, D. F. and Tsfasman, M. A.},
     TITLE = {Arithmetic on singular {D}el {P}ezzo surfaces},
   JOURNAL = {Proc. London Math. Soc. (3)},
  FJOURNAL = {Proceedings of the London Mathematical Society. Third Series},
    VOLUME = {57},
      YEAR = {1988},
    NUMBER = {1},
     PAGES = {25--87},
      ISSN = {0024-6115},
   MRCLASS = {11G35 (11E04 11E95 11G25 14J17 14J20 14J26)},
  MRNUMBER = {940430},
MRREVIEWER = {Noriko Yui},
       DOI = {10.1112/plms/s3-57.1.25},
       URL = {https://doi.org/10.1112/plms/s3-57.1.25},
}

@Misc{CTPrivate,
  author       = {Colliot-Th\'{e}l\`ene, Jean-Louis },
  howpublished = {{Private Correspondence}},
  year         = {2024}
}

@article {CTChateletsurvey,
    AUTHOR = {Colliot-Th\'{e}l\`ene, Jean-Louis},
     TITLE = {Les grands th\`emes de {F}ran\c{c}ois {C}h\^{a}telet},
   JOURNAL = {Enseign. Math. (2)},
  FJOURNAL = {L'Enseignement Math\'{e}matique. Revue Internationale. 2e S\'{e}rie},
    VOLUME = {34},
      YEAR = {1988},
    NUMBER = {3-4},
     PAGES = {387--405},
      ISSN = {0013-8584},
   MRCLASS = {01A70 (11G35)},
  MRNUMBER = {979649},
MRREVIEWER = {Andrew Bremner},
}

@article {IwaniecMunshiChatelet,
    AUTHOR = {Iwaniec, Henryk and Munshi, Ritabrata},
     TITLE = {Cubic polynomials and quadratic forms},
   JOURNAL = {J. Lond. Math. Soc. (2)},
  FJOURNAL = {Journal of the London Mathematical Society. Second Series},
    VOLUME = {81},
      YEAR = {2010},
    NUMBER = {1},
     PAGES = {45--64},
      ISSN = {0024-6107},
   MRCLASS = {11N36 (11D45 11G50 14G05)},
  MRNUMBER = {2580453},
MRREVIEWER = {Ulrich Derenthal},
       DOI = {10.1112/jlms/jdp056},
       URL = {https://doi.org/10.1112/jlms/jdp056},
}

\end{document}